\documentclass[english]{amsart}

\usepackage{multirow}
\usepackage{enumerate}\usepackage{tikz}
\usetikzlibrary{cd,arrows,positioning} 
\tikzset{>=stealth}
\tikzcdset{arrow style=tikz}
\tikzset{link/.style={column sep=1.8cm,row sep=0.16cm}}
\tikzset{map/.style={row sep=0em, column sep=0em}}
\usepackage{amssymb}
\usepackage{array}
\usepackage{mathrsfs}
\usepackage{xspace}
\usepackage{MnSymbol}
\usepackage{mathtools}
\usepackage{color}
 \usepackage[normalem]{ulem} 
\usepackage[curve,all]{xy}
\xyoption{matrix}
 \xyoption{curve}
 \xyoption{color}
 \xyoption{line}
\xyoption{arc}

\usepackage{soul}

\usepackage[shortlabels]{enumitem}
\setlist{wide}
\setlist[enumerate]{label=\rm{(\arabic*)}}
\setlist[enumerate,2]{label=\rm({\it\roman*})}
\setlist[itemize]{label=\raisebox{0.25ex}{\tiny$\bullet$}}

\usepackage[backref, colorlinks, linktocpage, citecolor = blue, linkcolor = blue]{hyperref}

\newcommand{\ps}{ \, \tikz[baseline=-.6ex] \draw[->,dotted,line width=.6] (0,0) -- +(.5,0); \,}
\theoremstyle{plain}
\newtheorem{theoremA}{Theorem}
\newtheorem{propositionA}[theoremA]{Proposition}
\newtheorem{corollaryA}[theoremA]{Corollary}

\newtheorem{theorem}{Theorem}[subsection]
\newtheorem*{theoremaux}{Theorem \theoremauxnum}
\gdef\theoremauxnum{1}

\newtheorem{proposition}[theorem]{Proposition}
\newtheorem*{propositionaux}{Proposition \propositionauxnum}
\gdef\propositionauxnum{1}

\newtheorem{lemma}[theorem]{Lemma}
\newtheorem*{lemmaaux}{Lemma \lemmaauxnum}
\gdef\lemmaauxnum{1}

\newtheorem*{theorem*}{Theorem}
\newtheorem*{definition*}{Definition}
\newtheorem*{remark*}{Remark}

\newtheorem{corollary}[theorem]{Corollary}
\newtheorem{corollary*}{Corollary}

\theoremstyle{definition}
\newtheorem{definition}[theorem]{Definition}
\newtheorem{notation}[theorem]{Notation}
\newtheorem{example}[theorem]{Example}

\theoremstyle{remark}
\newtheorem{remark}[theorem]{Remark}

\newcommand{\Rb}{1.40cm}

\newcommand\iso{\stackrel{\simeq}{\longrightarrow}}

\newcommand{\incl}[1][r]{\ar@<-0.2pc>@{^(-}[#1] \ar@<+0.2pc>@{-}[#1]}
\newcommand{\tr}[1]{\vphantom{#1}^{t}\!#1}

\newcommand{\I}{\ensuremath{\mathrm{I}}\xspace}
\newcommand{\II}{\ensuremath{\mathrm{II}}\xspace}
\newcommand{\III}{\ensuremath{\mathrm{III}}\xspace}
\newcommand{\IV}{\ensuremath{\mathrm{IV}}\xspace}

\renewcommand{\P}{\mathbb{P}}
\newcommand{\PP}{\mathcal{P}}
\newcommand{\FF}{\mathcal{F}}

\renewcommand{\div}{\mathrm{div}}
\newcommand{\flip}{\mathrm{flip}}
\newcommand{\flop}{\mathrm{flop}}
\newcommand{\antiflip}{\mathrm{antiflip}}

\newcommand{\Spec}{\mathrm{Spec}}
\renewcommand{\Im}{\mathrm{Im}}
\newcommand{\V}{\mathcal{V}}
\renewcommand{\SS}{\mathcal{S}}
\newcommand{\TT}{\hat{\mathcal{S}}}
\newcommand{\p}{\mathbb{P}}

\newcommand{\pr}{\mathrm{pr}}
\newcommand{\Q}{\mathbb{Q}}
\newcommand{\QQ}{\mathcal{Q}}
\newcommand{\RR}{\mathcal{R}}

\renewcommand{\k}{\mathrm{k}}

\newcommand{\A}{\mathbb{A}}

\newcommand{\C}{\mathbb{C}}
\newcommand{\F}{\mathbb{F}}
\newcommand{\N}{\mathbb{N}}

\newcommand{\U}{\mathcal{U}}
\newcommand{\W}{\mathcal{W}}
\newcommand{\Z}{\mathbb{Z}}

\newcommand{\G}{\mathbb{G}}
\newcommand{\Gm}{\mathbb{G}_m}

\renewcommand{\O}{\mathcal{O}}

\newcommand{\OP}{\mathcal{O}_{\mathbb{P}^1}}

\newcommand{\car}{\mathrm{char}}

\newcommand{\OFa}{\mathcal{O}_{\mathbb{F}_a}}
\newcommand{\s}[1]{s_{#1}}

\DeclareMathOperator{\SL}{SL}
\DeclareMathOperator{\Aut}{Aut}

\DeclareMathOperator{\Autz}{Aut^{\circ}}
\DeclareMathOperator{\PGL}{PGL}

\DeclareMathOperator{\PSO}{PSO}
\DeclareMathOperator{\rk}{rk}

\DeclareMathOperator{\BPic}{\bf Pic}
\DeclareMathOperator{\Pic}{Pic}
\DeclareMathOperator{\NE}{NE}
\DeclareMathOperator{\BNS}{\bf NS}
\DeclareMathOperator{\NS}{NS}
\DeclareMathOperator{\Gal}{Gal}
\DeclareMathOperator{\GL}{GL}
\DeclareMathOperator{\Bir}{Bir}

\DeclareMathOperator{\Ker}{Ker}

\title[Connected algebraic groups acting on 3-dimensional Mori fibrations]{Connected algebraic groups acting on three-dimensional Mori fibrations}
\date{\today}

\author{J\'er\'emy Blanc}
\address{Universit{\"a}t Basel,
Departement Mathematik und Informatik,
Spiegelgasse 1,
CH-4051, Basel,
Switzerland}
\email{Jeremy.Blanc@unibas.ch}

\author{Andrea Fanelli}
\address{Institut de Math\'ematiques de Bordeaux, UMR 5251 CNRS, Universit\'e de Bordeaux, 33405 Talence cedex, France}
\email{andrea.fanelli.1@u-bordeaux.fr}

\author{Ronan Terpereau}
\address{Institut de Math\'{e}matiques de Bourgogne, UMR 5584 CNRS, Universit\'{e} Bourgogne Franche-Comt\'{e}, F-21000 Dijon, France}
\email{ronan.terpereau@u-bourgogne.fr}


\begin{document}

\begin{abstract}
We study the connected algebraic groups acting on Mori fibrations $X \to Y$ with $X$ a rational threefold and $\mathrm{dim}(Y) \geq 1$. More precisely, for these fibre spaces we consider  the neutral component of their automorphism groups  and study their equivariant birational geometry.
This is done using, inter alia, minimal model program and Sarkisov program and allows us to determine the maximal connected algebraic subgroups of $\mathrm{Bir}(\mathbb{P}^3)$, recovering most of the classification results of Hiroshi Umemura in the complex case.
\end{abstract}

\maketitle
\tableofcontents

\section{Introduction}

\subsection{Aim and scope}
In this article we work over a fixed algebraically closed field $\k$. Our main goal is to study the connected algebraic subgroups of the Cremona group $\Bir(\p^3)$ (see Definition~\ref{Defi:AlgSubgroupsBir} for a precise definition), up to conjugation, via a new geometric approach involving explicit birational geometry of rational threefolds.

When $\k$ is the field of complex numbers $\C$, a classification of the maximal connected algebraic subgroups of $\Bir(\P^3)$  has been stated by Enriques and Fano \cite{EF98} and achieved by Umemura in a series of four papers \cite{Ume80,Ume82a,Ume82b,Ume85}. In more than $150$ pages, detailed arguments are given and a finite list of families is precisely established. The proof of Umemura uses a result of Lie that gives a classification of analytic actions on complex threefolds (see \cite[Theorem~1.12]{Ume80}) to derive a finite list of algebraic groups acting rationally on $\p^3$. 

Umemura, together with Mukai, studied in \cite{MU83,Ume88} the \emph{minimal} smooth rational projective threefolds (a smooth projective variety $X$ is called \emph{minimal} if any birational morphism $X \to X'$ with $X'$ smooth is an isomorphism). For each subgroup $G\subseteq \Bir(\p^3)$ of the list of maximal connected algebraic subgroups of $\Bir(\p^3)$, they determine the minimal smooth rational projective threefolds $X$ such that $\varphi^{-1} G\varphi=\Autz(X)$ for some birational map $\varphi\colon X\dashrightarrow \p^3$; this gives a detailed story of $95$ pages additional to Umemura's classification.

\smallskip

In this paper we will not use the long work of Umemura or any analytic method. We will rather use another strategy to recover both the maximal connected algebraic subgroups of $\Bir(\P^3)$  and the minimal rational projective threefolds on which they act, based on the \emph{minimal model program} (MMP), as we now explain. 

If $G$ is a connected algebraic subgroup of $\Bir(\P^3)$, then the regularisation theorem of Weil (recalled in \upshape\S~\ref{SubSec:Reg}) gives the existence of a birational map $\varphi\colon X \dashrightarrow \P^3$ such that $G \subseteq \varphi\Autz(X) \varphi^{-1}$. Also, one can always compactify $X$ equivariantly and assume that it is projective. Supposing moreover that the base field $\k$ is of characteristic zero, we may even assume that $X$ is smooth. And finally we can run an MMP (which is always $\Autz(X)$-equivariant; see Remark~\ref{rk:MMP G eq}) to reduce to the case where $X$ is a Mori fibre space (see Theorem~\ref{th alg subg of Cr3 are aut of Mori fib}).

This observation justifies our strategy: we start from a rational projective threefold, take an equivariant resolution of singularities if $\car(\k)=0$ or assume it is smooth otherwise, run an MMP (which is valid for any smooth projective threefold provided that $\car(\k) \notin \{2,3\}$), and then study which of the possible outcomes $X \to Y$ (with  $0 \leq \dim(Y) < \dim(X)=3$) provide maximal algebraic subgroups in $\Bir(\P^3)$. We distinguish between three cases:
\begin{enumerate}
\item if $\dim(Y)=2$, then $X \to Y$ is a \emph{Mori conic bundle} over a rational projective surface with canonical singularities. This case is studied in \upshape\S \ref{Sec:ConicBundles};
\item if $\dim(Y)=1$, then $X \to Y=\P^1$ is a \emph{Mori del Pezzo fibration} over $\P^1$. This case is studied in \upshape\S \ref{Sec:dP}; and
\item if $\dim(Y)=0$, then $X$ is a $\Q$-factorial Fano threefold of Picard rank $1$ with $($at worse$)$ terminal singularities. 
\end{enumerate}

When $\car(\k)=0$, our results provide a full description of all the possible maximal connected algebraic groups acting on rational three-dimensional Mori fibre spaces (and not just the smooth models), except when the basis of the Mori fibration is trivial (i.e.~$\dim(Y)=0$); see Theorem~\ref{th:Ea} for a precise statement. As consequence of the classification we also prove that each connected algebraic subgroup of $\Bir(\p^3)$ is contained in a maximal one (Corollary~\ref{corF}). This fact  seems difficult to prove without the classification, is unknown for $\Bir(\p^n)$ when $n \geq 4$ and is false for $\Bir(\p^1\times C)$ with $C$ is a non-rational curve (see \cite{Fong19}). The hypothesis on the characteristic is needed in several steps to obtain the classification (for instance, we use our previous work \cite{BFT}).

It turns out that most of such connected algebraic groups are conjugated to algebraic subgroups of automorphism groups of certain $\P^1$-bundles over smooth rational surfaces; these were studied thoroughly in \cite{BFT}. Therefore, in the present article we mostly focus on the neutral component of the automorphism groups of conic bundles over rational surfaces that are not $\P^1$-bundles and of del Pezzo fibrations over $\P^1$. It is striking to see that, even though there are many rational three-dimensional Mori fibre spaces, in the end, only very few of them give rise to maximal connected algebraic subgroups in the Cremona group $\Bir(\P^3)$; see Corollary~\ref{corF}.  This is for instance completely different from the dimension $2$ case where each rational Mori fibration $X\to Y$, with $X$ a minimal surface, gives a maximal connected algebraic subgroup of $\Bir(\p^2)$.

Several results obtained in \cite{BFT} are used in this article, but nevertheless we tried to make it mostly self-contained. For specialists of the MMP and the Sarkisov program, our results can be seen as a natural geometric application of the ideas and techniques from these theories, together with some explicit birational geometry of threefolds, to determine and understand the maximal connected algebraic subgroups of the Cremona group. 

For non-specialists interested in the birational geometry of threefolds, this paper could represent a good source of natural examples of Mori fibrations. In addition, the geometric restriction given by the connected algebraic group actions allows us to describe explicitly all the equivariant links between those.

The birational geometry of rational threefolds is incredibly rich and our results also enlighten certain geometric features for the different classes of Mori fibrations such as conic bundles and del Pezzo fibrations. For instance, we prove that conic bundles over rational surfaces whose generic fibre is not $\P^1$ and del Pezzo fibrations over $\P^1$ whose generic fibre is of degree $\leq 6$ have quite few automorphisms; see Theorems~\ref{Thm:conic bundles}-\ref{Thm:MainQuadric} for precise results.

\subsection{Statement of the main results}
In all the article the base field $\k$ is supposed to be algebraically closed. Our first main result is the following:

\begin{theoremA} \label{th:A} \emph{(see \upshape\S~\ref{proof of thA} for the proof)} 
Assume that $\car(\k)\notin\{2,3,5\}$ and let $\hat X$ be a smooth rational projective threefold. Then there is an $\Autz(\hat X)$-equivariant birational map $ \hat X \dashrightarrow X$, where $X$ is a Mori fibre space that satisfies one of the following conditions:  
\begin{enumerate}
\item\label{thAP1} $X$ is a $\p^1$-bundle over $\p^2$, $\p^1\times\p^1$ or a Hirzebruch surface $\F_a$ with $a\ge 2$; or
\item\label{thAQuadricP2} $X$ is either a $\p^2$-bundle over $\P^1$ or a smooth Umemura quadric fibration $\QQ_g$ over $\P^1$ $($see Definition~$\ref{def:QQg})$ with $g \in \k[u_0,u_1]$ a square-free homogeneous polynomial of degree $2n \geq 2$; or 
\item\label{thAFano} $X$ is a rational $\Q$-factorial Fano threefold of Picard rank $1$ with terminal singularities.
\end{enumerate}  
\end{theoremA}

Let us comment on the possible cases of Theorem~\ref{th:A}:
\begin{enumerate}
\item
As explained before, very few Mori fibrations appear in Theorem~\ref{th:A}. In particular, there is no conic bundle that is not a $\p^1$-bundle and no del Pezzo fibration of degree $d\le 7$ (see Theorems~\ref{Thm:conic bundles} and \ref{Thm:MainQuadric} for more details).
\item
In this text a $\p^n$-bundle is always a Zariski locally trivial $\p^n$-bundle.  The $\p^2$-bundles over $\p^1$ are simply given by $\RR_{m,n}= \P(\O_{\p^1}(-m) \oplus \O_{\p^1}(-n) \oplus \O_{\p^1})$ for some $m,n\ge 0$ (see {\upshape\S~\ref{sec:first classification}\ref{DefiRmn}}). Their automorphism groups are easy to describe (see Lemma~\ref{lem:aut P2 bundle}).
 \item
 There are many distinct families of $\p^1$-bundles over $\p^2$ or over a Hirzebruch surface $\F_a$. In \cite{BFT}, we focused on these and gave a classification of maximal ones in the case where $\car(\k)=0$.
\item
The Umemura quadric fibrations $\QQ_g\to \p^1$ are studied in Section~\ref{quadric fibrations}. They are parametrised by classes of hyperelliptic curves, and then form an infinite dimensional family. Their automorphism group is however simply $\PGL_2$ (or an extension by $\G_m$ in a very special case, but in this case $\Autz(\QQ_g)$ is conjugated to a strict subgroup of $\Autz(Q)=\PSO_5$ with $Q \subseteq \P^4$ a smooth quadric hypersurface).
 \item
The case of Fano threefolds is less understood. There is for the moment no complete classification of their automorphism groups, except in the smooth case and over an algebraically closed field of characteristic zero \cite[Theorem 1.1.2]{KPS}.
 \end{enumerate}

Along our way to prove Theorem~\ref{th:A}, Theorem~\ref{th:Ea}, and Theorem~\ref{th:Eb} below, we prove the following three results (Proposition~\ref{propB} and Theorems~\ref{Thm:conic bundles}-\ref{Thm:MainQuadric}), which we believe are interesting on their own.

The next result, whose proof is elementary, is certainly well-known from specialists but we could not find a suitable reference. Therefore we chose to recall it and write a complete proof.

\begin{propositionA} \label{propB} \emph{(Lemma~\ref{lemma:linear} and Corollary~\ref{cor Auts cubic threefold})}
Let $X$ be a rationally connected variety $($i.e.~two general points of $X$ are connected by a rational curve$)$. Then every algebraic subgroup $G \subseteq   \Bir(X)$ is a linear algebraic group.

Suppose moreover that $\car(\k)=0$, $\dim(X)=3$, and $X$ is not rational $($for instance $X$ is a smooth projective cubic threefold$)$. Then every connected algebraic subgroup of $\Bir(X)$ is trivial. In particular, $\Autz(X)$ is trivial.
\end{propositionA}

The next two results concern automorphism groups of certain conic bundles over surfaces and del Pezzo fibrations over $\P^1$; these are key-ingredients in the proof of Theorem~\ref{th:A}.

\begin{theoremA} \label{Thm:conic bundles}  \emph{(see \upshape\S~\ref{sec:proof of Prop B} for the proof)}
Assume that $\car(\k)\neq 2$,  let $X$ be a normal rationally connected threefold, and let $\pi\colon X\rightarrow S$ be a conic bundle.
\begin{enumerate}
\item\label{Thm:conic bundlesCaseP1b}
If the generic fibre of $\pi$ is isomorphic to $\p^1_{\k(S)}$, then there is an $\Autz(X)$-equivariant commutative diagram 
\[\xymatrix@R=4mm@C=2cm{
    \hat X \ar@{-->}[r]^{\psi} \ar[d]_{\hat\pi}  & X \ar[d]^\pi \\
    \hat S\ar@{-->}[r]^{\eta} & S
  }\]  
where $\psi$ and $\eta$ are a birational maps, $\hat S$ is a smooth projective surface with no $(-1)$-curve, and the morphism $\hat\pi\colon \hat X\to \hat S$ is a $\p^1$-bundle.
\item\label{Thm:conic bundlesCaseTorus}
If the generic fibre of $\pi$ is not isomorphic to $\p^1_{\k(S)}$, the action of $\Autz(X)$ on $S$ gives an exact sequence $($see {\upshape\S~\ref{Mori fibrations}} for the notation$)$
\[1 \to \Autz(X)_S \to \Autz(X) \to H \to 1,\]
where $H \subseteq \Autz(S)$ and $\Autz(X)_S$ is a finite group, isomorphic to $(\Z/2\Z)^r$ for some $r\in \{0,1,2\}$. Moreover, the following hold:
\begin{enumerate}
\item\label{Thm:conic bundlesCaseTorusSrat}
 If $S$ is rational, which is always true if $\car(\k)=0$, then both $H$ and $\Autz(X)$ are tori of dimension at most two.
 \item\label{Thm:conic bundlesCaseTorusXrat}
 If $X$ is rational, then $S$ is rational and there is an $\Autz(X)$-equivariant birational map $\varphi\colon X \dashrightarrow \P^3$ such that $\varphi \Autz(X) \varphi^{-1} \subsetneq \Aut(\p^3)=\PGL_4$.
 \end{enumerate}
\end{enumerate}
\end{theoremA}

\begin{theoremA}\label{Thm:MainQuadric}
Assume that $\car(\k)\notin\{2,3,5\}$. Let $\pi_X\colon X\to \p^1$ be a Mori del Pezzo fibration of degree $d$. Then, $d\in \{1,2,3,4,5,6,8,9\}$ and the following hold:
\begin{enumerate}
\item\label{PropMainQuadric_le5}
If $d\le 5$ $($resp.~$d=6)$, then $\Autz(X)$ is a torus of dimension $\le 1$ $($resp.~$\le 3)$. 
\item\label{PropMainQuadric_mainpoint}
If $\Autz(X)$ is not isomorphic to a torus, there is an $\Autz(X)$-equivariant commutative diagram
\[\xymatrix@R=3mm@C=1cm{
    X \ar@{-->}[rr]^{\psi} \ar[rd]_{\pi_X}  && Y \ar[ld]^{\pi_Y} \\
    & \p^1
  }\]
  such that $\psi$ is a birational map, $\Autz(X)$ acts regularly on $Y$, and either
\begin{enumerate}
\item \label{case1propB} $\pi_{Y}\colon Y\to  \p^1$ is a $\p^2$-bundle; or 
\item \label{case2propB} there is a square-free homogeneous polynomial $g\in \k[u_0,u_1]$ of degree $2n$ $($with $n \geq 1)$ such that $(Y,\pi_{Y})=(\QQ_{g},\pi_{g})$.
\end{enumerate}  
Moreover, in Case~\ref{case2propB}, the group $\psi\Autz(X)\psi^{-1}\subseteq \Autz(\QQ_g)$ is either equal to $\PGL_2$ $($see Lemma~$\ref{Lem:ActiononQQg1}\ref{ActionPGL2})$ if $n \geq 2$ or to $\PGL_2\times\G_m$ $($see Example~$\ref{Example:Qgu0u1}$, with $g=u_0u_1)$ if $n=1$.
\end{enumerate}
\end{theoremA}

\begin{remark} 
The main reason for the restriction on the characteristic of $\k$ in Theorem~\ref{Thm:MainQuadric} comes from the fact that the generic fibre of a del Pezzo fibration $X \to \P^1$ can be non-smooth in small characteristic (see Lemma~\ref{Lem:bound_char_dP}).
\end{remark}

Once we have proven Theorem~\ref{th:A}, we use results obtained in \cite{BFT} together with the Sarkisov program (recalled in \upshape\S~\ref{subsec:Sarkisov}) to prove the following results.

\begin{theoremA}  \label{th:Ea}
Assume that $\car(\k)=0$ and let $\hat X$ be a rational projective threefold. Then there is an $\Autz(\hat X)$-equivariant birational map $ \hat X \dashrightarrow X$, where $X$ is one of the following Mori fibre spaces $($see {\upshape\S~\ref{sec:first classification}} for the notation$)$:  
\begin{center}
\scalebox{0.9}{
\begin{tabular}{lllrclll}
$\hypertarget{th:D_a}{(a)}$& A decomposable &$\p^1$-bundle & $\FF_a^{b,c}$&\hspace{-0.3cm}$\longrightarrow$& \hspace{-0.2cm}$\F_a$& with $a,b\ge 0$, $a\not=1$, $c\in \Z$, and \\
&&&&&& $(a,b,c)=(0,1,-1)$; or\\
&&&&&& $a=0$, $c \neq 1$, $b \geq 2$, $b\ge \lvert c\rvert $; or\\
&&&&&& $-a<c<a(b-1)$; or\\
&&&&&& $b=c=0$.\\
$\hypertarget{th:D_b}{(b)}$& A decomposable &$\p^1$-bundle &$\PP_b$&\hspace{-0.3cm}$\longrightarrow$& \hspace{-0.2cm}$\p^2$& for some $b \geq 2$.\\
$\hypertarget{th:D_c}{(c)}$& An Umemura &$\p^1$-bundle &$\U_a^{b,c}$&\hspace{-0.3cm}$\longrightarrow$& \hspace{-0.2cm}$\F_a$& for some $a,b\ge 1, c\ge 2$ with\\
&&&&&& $c<b$ if $a=1$; and\\
&&&&&& $c-2<ab$  and $c-2 \neq a(b-1)$ if $a \geq 2$.\\
$\hypertarget{th:D_d}{(d)}$ &A Schwarzenberger\!\! &$\p^1$-bundle &$\SS_b$&\hspace{-0.3cm}$\longrightarrow$& \hspace{-0.2cm}$\p^2$& for some $b=1$ or $b\ge 3$.\\
$\hypertarget{th:D_e}{(e)}$ &A &$\p^1$-bundle &$\V_{b}$&\hspace{-0.3cm}$\longrightarrow$& \hspace{-0.2cm}$\p^2$& for some $b\ge 3$.\\
$\hypertarget{th:D_W}{(f)}$& A singular &$\p^1$-fibration &$\W_b$&\hspace{-0.3cm}$\longrightarrow$& \hspace{-0.2cm}$\P(1,1,2)$& for some $b \geq 2$.\\
$\hypertarget{th:D_RR}{(g)}$ &A decomposable &$\p^2$-bundle &$\RR_{m,n}$&\hspace{-0.3cm}$\longrightarrow$& \hspace{-0.2cm}$\p^1$& for some $m \geq n\ge 0$, \\
&&&&&& with $(m,n) \neq (1,0)$ and\\ 
&&&&&&  $m=n$ or $m>2n$.\\
$\hypertarget{th:D_QQg}{(h)}$ &An Umemura &quadric fibration  &$\QQ_g$&\hspace{-0.3cm}$\longrightarrow$& \hspace{-0.2cm}$\p^1$& for some homogeneous\\ 
&&&&&&    polynomial $g \in \k[u_0,u_1]$ of\\
&&&&&&    even degree with at least \\ 
&&&&&& four roots of odd multiplicity.\\
$\hypertarget{th:D_bP3}{(i)}$& The &projective space & \multicolumn{3}{c}{$\P^3$.}  &\\
$\hypertarget{th:D_bQ3}{(j)}$& The smooth & quadric & \multicolumn{3}{c}{$Q_3\subset \P^4$.}  &\\
$\hypertarget{th:D_bP1112}{(k)}$& The weighted &projective space & \multicolumn{3}{c}{$\P(1,1,1,2)$.}  &\\
$\hypertarget{th:D_bP1123}{(l)}$& The weighted &projective space & \multicolumn{3}{c}{$\P(1,1,2,3)$.} &\\
$\hypertarget{th:D_Fano}{(m)}$& \multicolumn{6}{l}{A  rational $\Q$-factorial Fano threefold of Picard rank $1$ with terminal singularities,}\\
&\multicolumn{6}{l}{{not isomorphic to any of the cases \hyperlink{th:D_bP3}{$(i)$}-\hyperlink{th:D_bQ3}{$(j)$}-\hyperlink{th:D_bP1112}{$(k)$}-\hyperlink{th:D_bP1123}{$(l)$}.}}
\end{tabular} } \vspace{-2.7mm}
\end{center}
\end{theoremA}

\begin{remark*}
In Theorem~\ref{th:Ea}, Families \hyperlink{th:D_a}{$(a)$}, \hyperlink{th:D_b}{$(b)$}, \hyperlink{th:D_c}{$(c)$}, \hyperlink{th:D_d}{$(d)$}, \hyperlink{th:D_e}{$(e)$}, \hyperlink{th:D_RR}{$(g)$}, \hyperlink{th:D_bP3}{$(i)$}, \hyperlink{th:D_bQ3}{$(j)$}  correspond to smooth varieties. 
A variety $\QQ_g$ from Family \hyperlink{th:D_QQg}{$(h)$} is smooth if and only if the polynomial $g$ is square-free.
Families \hyperlink{th:D_W}{$(f)$}, \hyperlink{th:D_bP1112}{$(k)$}, and \hyperlink{th:D_bP1123}{$(l)$} correspond to singular varieties. 
{Family \hyperlink{th:D_Fano}{$(m)$} contains smooth varieties and singular varieties.}
\end{remark*}

\begin{remark*}
A description of the automorphism groups of the Mori fibre spaces listed in Theorem~\ref{th:Ea} can be found in \cite[\S~3.1, \S~3.6, \S~4.1, and \S~4.2]{BFT} for the $\P^1$-bundles, in \upshape\S~$\ref{quadric fibrations}$ for the Umemura quadric fibrations, in \upshape\S~\ref{sec:non-maximality} for  the $\P^2$-bundles over $\P^1$, in Proposition~$\ref{prop:Fabc maximality}$  for the $\P^1$-fibrations $\W_b \to \P(1,1,2)$, and in \cite[\S~8]{AlAm} for the weighted projective spaces. $($See also \cite[\S~4]{Ume85} for an alternative description of these automorphism groups in the smooth cases.$)$
\end{remark*}

\begin{theoremA}  \label{th:Eb}
Assume that $\car(\k)=0$. Let $X_1$ and $X_2$ be two Mori fibre spaces such that $X_1$ belongs to one of the Families \hyperlink{th:D_a}{$(a)$}--\hyperlink{th:D_bP1123}{$(l)$} of Theorem~\ref{th:Ea}. If there exists an $\Autz(X_1)$-equivariant birational map $\varphi\colon X_1 \dashrightarrow X_2$, then  $X_2$ also belongs to one of the Families \hyperlink{th:D_a}{$(a)$}--\hyperlink{th:D_bP1123}{$(l)$}, and $\varphi \Autz(X_1) \varphi^{-1}=\Autz(X_2)$. Moreover, $\varphi$ is a composition of isomorphisms of Mori fibrations and of the following equivariant Sarkisov links $($or their inverses$)$:
\begin{enumerate}[$(S1)$]
\item\label{SA1}
$\p^1\times\p^1\times\p^1/(\p^1\times \p^1)\iso \p^1\times\p^1\times\p^1 /(\p^1 \times \p^1)$ $($exchange of factors$)$;
\item\label{SA2}
$\p^1\times\p^2/\p^1\iso \p^1\times\p^2 /\p^2$ ;
\item\label{SA3}
$\SS_1/\p^2 \iso\SS_1/\p^2 $ $($automorphism of order two exchanging fibrations, Prop.~$\ref{Prop:HomSpaces})$;
\item\label{SA4}
$\FF_0^{b,0}\iso \F_b\times\p^1  \iso \FF_b^{0,0}$ for $b\ge 2$ $($isomorphism, Lemma~$\ref{Lemma:F0bcListLinks})$;
\item\label{SA5}
$\SS_b\dasharrow \SS_b$ for $b\ge 3$ $($birational involution, Prop.~$\ref{prop:Schwarzenberg involution})$;
\item\label{SA6}
$\PP_2\to \P(1,1,1,2)$ $($reduced blow-up of the singular point of $\P(1,1,1,2))$;
\item\label{SA7}
$\FF_{m-n}^{1,-n} \to \RR_{m,n}$ for $m=n\ge 1$ or $m>2n\ge 2$ $($blow-up of a section, Prop.~$\ref{prop:Rmn list of links})$;
\item\label{SA8}
$\RR_{1,1}\dasharrow \RR_{1,1}$ $($birational involution which is a flop, Prop.~$\ref{prop:Rmn list of links})$;
\item\label{SA9}
$\P(1,1,2,3)\dasharrow \RR_{3,1}$ $($reduced blow-up of $[0:0:1:0]$ followed by a flip, Lemma.~$\ref{LemP1123})$;
\item\label{SA10}
$\P(1,1,2,3)\dasharrow \W_2$ $($weighted blow-up of $[0:0:0:1]$, Lemma~$\ref{LemP1123})$;
\item\label{SA11}
$\FF_a^{b,c}\dasharrow \FF_a^{b+1,c+a}$ for all $a,b,c\in \Z$, $a,b\ge 0$, $a(c+a)>0$ and either $ab>0$ or $ac<0$ $($Lemma~$\ref{LinksIIbetweenFU})$;
\item\label{SA12}
$\U_a^{b,c}\dasharrow \U_a^{b+1,c+a}$ for each Umemura bundle $\U_a^{b,c}$ $($Lemma~$\ref{LinksIIbetweenFU})$;
\item\label{SA13}
$\U_1^{b,2}\to \V_b$ for each $b\ge 3$ $($blow-up of a point, Lemma~$\ref{lem:sequence of links for Umemura bundles})$;
\item\label{SA14}
$\W_b\dasharrow \FF_2^{b-1,-1}$ for each $b\ge 2$ $($blow-up of a singular point followed by a flip, Example~$\ref{WbFF2bm1})$;
\item\label{SA15}
$\W_b\dasharrow \FF_2^{b,1}$ for each $b\ge 2$ $($blow-up of a singular point followed by a flip, Example~$\ref{WbFF2b})$; and
\item\label{SA16}
$\QQ_g\dasharrow \QQ_{gh^2}$ for each $g,h\in \k[u_0,u_1]$ homogeneous polynomials of degree $2n \geq 4$ and $1$ respectively and such that $g$ has at least three roots (blow-up of a singular point followed by a divisorial contraction, Lemma~$\ref{Lem:LinksQQgh})$.
\end{enumerate}
\end{theoremA}

\begin{remark*}
We conjecture that, for any singular rational $\Q$-factorial Fano threefold $X$ of Picard rank $1$ with terminal singularities $($other than $\P(1,1,1,2)$ and $\P(1,1,2,3))$, there is always an $\Autz(X)$-equivariant birational map $X \dashrightarrow X'$ with $X'$ either being a smooth rational Fano threefold of Picard rank $1$ or belonging to one of the Families \hyperlink{th:D_a}{$(a)$}--\hyperlink{th:D_bP1123}{$(l)$} of Theorem~\ref{th:Ea}. 

{When $\k=\mathbb{C}$, Umemura obtained the full classification of the maximal connected algebraic subgroups in $\Bir(\P^3)$ in \cite{Ume80,Ume82a,Ume82b,Ume85} using analytic methods $($see \upshape\S~\ref{sec:Umemura classification} for a comparison with our classification$)$. He then described, together with Mukai, the possible regular actions of these groups on smooth projective algebraic varieties in \cite{MU83,Ume88}. Combined with our results, this proves the conjecture above in the case $\k=\mathbb{C}$. However, an independent geometric proof of this conjecture would lead, together with our results and \cite{KPS}, to an alternative proof of Umemura's classification via methods from modern birational geometry.}
\end{remark*}

Finally, the next result follows readily from Theorems~\ref{th:Ea} and~\ref{th:Eb}, and from a simple instance of the BAB conjecture (the boundedness of terminal $\Q$-factorial Fano threefolds). A proof of this corollary is given at the end of the article.

\begin{corollaryA} \label{corF}
Assume that $\car(k)=0$. Let $G$ be a connected algebraic subgroup of $\Bir(\P^3)$. Then there exists a birational map $\varphi\colon X \dashedrightarrow \P^3$ such that $\varphi^{-1} G\varphi \subseteq  \Autz(X)$, where $X$ is one of Mori fibre spaces listed in Theorem~\ref{th:Ea} and such that the connected algebraic subgroup $\varphi \Autz(X) \varphi^{-1} \subseteq \Bir(\P^3)$ is maximal for the inclusion. 

Moreover, for each variety $Y$  that belongs to one of the Families  \hyperlink{th:D_a}{$(a)$}-\hyperlink{th:D_bP1123}{$(l)$}, and for each birational map $\psi\colon Y\dasharrow \p^3$, the connected algebraic subgroup $\psi \Autz(Y) \psi^{-1} \subseteq \Bir(\P^3)$ is maximal for the inclusion.
\end{corollaryA}

\subsection{Comparison with Umemura's classification}\label{sec:Umemura classification}
In this section we recall the classification of the connected algebraic subgroups of $\Bir(\p_\C^3)$ obtained by Umemura  \cite{Ume80,Ume82a,Ume82b,Ume85} and compare it with our results. The corresponding smooth relatively minimal models were determined by Mukai and Umemura in \cite{MU83,Ume88}. Their main result can be expressed as follows.

\begin{theorem*} \emph{\cite{Ume80,Ume82a,Ume82b,MU83,Ume85,Ume88}}\\
Let $G$ be a connected algebraic subgroup of $\Bir(\p_\C^3)$. Then $G$ is conjugated to an algebraic subgroup of $\Autz(X)$, where $X$ is a smooth rational Mori fibre space of dimension $3$. 
Moreover, the conjugacy classes of algebraic subrgoups of $\Bir(\p_\C^3)$ parametrized by the families $[P1], [P2],\ldots, [J6;m,n],\ldots, [J12]$ in the table below are the maximal conjugacy classes of algebraic subgroups of $\Bir(\p_\C^3)$.
\end{theorem*}

\scalebox{0.87}{
\begin{tabular}{|l|l|l|}
\hline
$[P1]$ & & the projective space $\P^3$ \\[0pt]
\hline
$[P2]$ & & the smooth quadric $Q \subseteq \P^4$\\[0pt]
\hline
$[E1]$ & &the Fano threefold $X_5$\\[0pt]   
\hline  
$[E2]$ && the Mukai-Umemura Fano threefold $X_{22}^{MU}$\\[0pt] 
\hline
$[J1]$ && $\p^2\times\p^1$\\[0pt]
\hline
$[J2]$ & &$\p^1\times\p^1\times\p^1$\\[0pt]
\hline
$[J3;m]$& $m \geq 2$ &$\p^1\times\F_m$ \\[0pt]
\hline
$[J4]$ & &the projectivisation of the tangent bundle of $\p^2$ \\[0pt]
\hline
$[J5;m]$ & $m \geq 3$& the Schwarzenberger $\p^1$-bundle $S_m \to \p^2$ \\[0pt]
\hline
$[J6;m,n]$& $m \geq 2, -2 \geq n$& $\P(\O_{\p^1\times\p^1} \oplus \O_{\p^1\times\p^1}(-m,-n))$ \\[0pt]
\hline
$[J7;m]$ & $m \geq 2$& $\P(\O_{\p^2} \oplus \O_{\p^2}(-m))$\\[0pt]
\hline
$[J8;m,n]$ & $m \geq n \geq 2$ &$\P(\O_{\p^1\times\p^1} \oplus \O_{\p^1\times\p^1}(-m,-n))$ \\[0pt]
         & $m \geq 1$& $\P(\O_{\p^1\times\p^1} \oplus \O_{\p^1\times\p^1}(-m,-1))$ \\[0pt]
                         &  & or $\P(\O_{\p^1}(-m)\oplus \O_{\p^1}(-m) \oplus \O_{\p^1})$ \\[0pt]
\hline
$[J9;m,n]$ & $m\wedge n \geq 2, m \geq 2n \geq 4$& $\P(\O_{\F_n} \oplus \O_{\F_n}(-mf-ks_{-n}))$ with $k\ge \lfloor \frac{m}{n}\rfloor$  \\[0pt]
         & $m \wedge n  \geq 2, 2n>m >n $ & $\P(\O_{\F_n} \oplus \O_{\F_n}(-mf-ks_{-n}))$ with $k\ge \lfloor \frac{m}{n}\rfloor$  \\[0pt]
            &&  or $\P(\O_{\p^1}(-m) \oplus \O_{\p^1}(-(m-n)) \oplus \O_{\p^1})$\\
         & $m \wedge n=1, m \geq 2n \geq 4$& $\P(\O_{\F_n} \oplus \O_{\F_n}(-mf-ks_{-n}))$ with $k\ge \lfloor \frac{m}{n}\rfloor$ \\[0pt]
             &&  or the so-called \emph{Euclidean model} $U_{m,n}$\\  \relax
         & $m \wedge n=1, 2n>m >n $& $\P(\O_{\F_n} \oplus \O_{\F_n}(-mf-ks_{-n}))$ with $k\ge \lfloor \frac{m}{n}\rfloor$   \\[0pt]
            &&  or $\P(\O_{\p^1}(-m) \oplus \O_{\p^1}(-(m-n)) \oplus \O_{\p^1})$\\[0pt]
              &&  or the so-called \emph{Euclidean model} $U_{m,n}$\\[0pt]
\hline
$[J10;m]$ &   $m \geq 2$ & $\P(\O_{\p^1}(-m) \oplus \O_{\p^1} \oplus \O_{\p^1})$\\[0pt]
\hline
$[J11;m,l]$ & $m \geq 2, l \geq 2$&  the Umemura $\p^1$-bundle $\U_{m}^{l,j} \to \p^2$ with $j \geq l$ \\[0pt]
          & $m=1, l \geq 3$ &  the Umemura $\p^1$-bundle $\U_{m}^{l,j} \to \p^2$ with $j \geq l$ \\[0pt]
          \hline
$[J12;g(t)]$ &$g \in \k[t]$ of even degree & the smooth quadric fibration $\QQ_g \to \p^1$ of \upshape\S~\ref{quadric fibrations} \\[0pt]
\hline\end{tabular}}

\ \\

The conjugacy classes [P1], [P2], [E1], and [E2] are determined in \cite{Ume80,Ume82a} and correspond to the case of Fano threefolds with Picard rank $1$. This first part of the classification of the maximal connected algebraic subgroup of $\Bir(\p_\C^3)$ is based on the classification of {law chunks of analytic actions (see \cite[\S~1]{Ume80} for the definition, these are the analytic counterpart of the rational actions introduced in Definition~\ref{Defi:AlgSubgroupsBir})} due to Lie combined with the classical theory of algebraic groups and invariant theory.

The other conjugacy classes [J1],\ldots,[J12] correspond to non-trivial Mori fibrations (i.e.~conic bundles over a surface and del Pezzo fibrations over $\p^1$). They are determined in \cite{Ume82b,Ume85} by studying the properties of linear algebraic groups of small rank, their linear representations, and their homogeneous  spaces.  

Let us note that Umemura first obtained an explicit representative for each conjugacy classes via group-theoretic arguments, and only then Mukai and Umemura determined the relatively minimal models $X$ realising each maximal conjugacy class of $\Bir(\p_\C^3)$ in \cite{MU83,Ume88} via the Mori theory. 

The Euclidean models that appear in Umemura's paper \cite{Ume88} do \emph{not} appear in our classification, since they are not Mori fibre spaces; those smooth models can however be recovered from our list, as explained in Remark~\ref{Rem:EuclideanModels}. 

\smallskip

We now give the correspondence between Umemura's classification and the one obtained in Theorem~\ref{th:Ea}. 
\begin{itemize}
\item 
Family \hyperlink{th:D_a}{$(a)$} corresponds to Umemura's Families [J2], [J3], [J6], and [J8] when $a=0$ and  to Umemura's families [J3] and [J9]  when $a\geq 2$. 

\item 
Family \hyperlink{th:D_b}{$(b)$} corresponds to Umemura's Family [J7].

\item 
Family \hyperlink{th:D_c}{$(c)$}  corresponds to Umemura's Family [J11].

\item 
Family \hyperlink{th:D_d}{$(d)$}  corresponds to Umemura's Families [J4] and [J5].

\item
Family \hyperlink{th:D_e}{$(e)$}  was overlooked in the work of Umemura. But this family corresponds to conjugacy classes of algebraic subgroups of $\Bir(\P^3_\C)$ contained in [J11].

\item 
Family \hyperlink{th:D_RR}{$(g)$} corresponds to Umemura's Families [J1], [J8], [J9], and [J10].

\item 
Family \hyperlink{th:D_QQg}{$(h)$} corresponds to Umemura's Family [J12].

\item
{Elements [P1], [P2], [E1], and [E2]  are smooth Fano threefolds with Picard rank $1$ and they belong to Families \hyperlink{th:D_bP3}{$(i)$}, \hyperlink{th:D_bQ3}{$(j)$} and \hyperlink{th:D_Fano}{$(m)$} . }
\end{itemize}

\subsection{Content of the sections} Let us specify the content of each section.

Section~\ref{sec:preliminaries} is dedicated to preliminaries. In \upshape\S~\ref{Mori fibrations} we recall the notion of \emph{Mori fibration} and consider in particular the case of threefolds. In \upshape\S~\ref{subsec:Sarkisov} we recall how equivariant birational maps between Mori fibrations can be factorised in equivariant \emph{Sarkisov links}, with a particular focus on the three-dimensional case. Then in \upshape\S~\ref{SubSec:AlgSubgroups} we recall the definition and characterize the \emph{algebraic subgroups} of the group of birational transformations $\Bir(X)$ for a variety $X$. 
In \upshape\S~\ref{SubSec:Reg} we prove the first part of Proposition~\ref{propB} (Lemma~\ref{lemma:linear}), then we recall the famous \emph{regularization theorem} of Andr\'e Weil and apply it to prove Theorem~\ref{th alg subg of Cr3 are aut of Mori fib} (which is the starting point in the proofs of Theorems~\ref{th:A} and~\ref{th:Ea}). In \upshape\S~\ref{tori_cremona} we prove the second part of Proposition~\ref{propB} (see Corollary~\ref{cor Auts cubic threefold}).
 
Section~\ref{Sec:ConicBundles} is dedicated to the proof of Theorem~\ref{Thm:conic bundles}. 
In \upshape\S~\ref{first reduction} we explain how to reduce the case of \emph{standard conic bundles}, which are Mori conic bundles with nice geometric features. In \upshape\S~\ref{conic bundles} we study standard conic bundles whose generic fibre is not $\P^1$, and finally in \upshape\S~\ref{sec:proof of Prop B} we prove Theorem~\ref{Thm:conic bundles}.

Section~\ref{Sec:dP} is dedicated to the proof of Theorem~\ref{Thm:MainQuadric}. In \upshape\S~\ref{subsec generalities} we recall some general results on del Pezzo surfaces and del Pezzo fibrations. In \upshape\S~\ref{red_delPezzo} we study del Pezzo fibrations of small degree over $\P^1$ and prove the first sentence in Theorem~\ref{Thm:MainQuadric}. In \upshape\S~\S~\ref{subsec:P2 fibrations}-\ref{quadric fibrations} we consider the case of $\P^2$-fibrations and quadric fibrations over $\P^1$ respectively. The proof of Theorem~\ref{Thm:MainQuadric} is given in \upshape\S~\ref{proof of thA}. Once we have proven Theorems~\ref{Thm:conic bundles} and~\ref{Thm:MainQuadric}, we easily prove Theorem~\ref{th:A}, also in \upshape\S~\ref{proof of thA}.

Section~\ref{subsec:first refinement} is an intermediate step towards the proof of Theorem~\ref{th:Ea}. In \upshape\S~\ref{sec:first classification} we introduce some families of $\P^1$-bundles over $\P^2$ and Hirzebruch surfaces $\F_a$ (with $a \geq 0$), and we recall the main result proven in \cite{BFT} (which is necessary to prove Theorem~\ref{th:Ea}). Then in \upshape\S~\ref{sec:non-maximality} we prove a series of lemmas that are also useful to prove Theorem~\ref{th:Ea}.

Section~\ref{sec:maximality} is dedicated to the proofs of Theorem~\ref{th:Ea}, Theorem~\ref{th:Eb} and Corollary~\ref{corF}. We first describe in \upshape\S~\S~\ref{subsec:homog case}-\ref{subsec Umemura quadric fib} all the equivariant Sarkisov links starting from the non-trivial Mori fibrations listed in Theorem~\ref{th:first classification maximality}, and then we easily deduce from this the proofs of Theorem~\ref{th:Ea}, Theorem~\ref{th:Eb}, and Corollary~\ref{corF} in \upshape\S~\ref{subsec:proof of th D}.

\section{Preliminaries} \label{sec:preliminaries}

\noindent \textbf{Notation.}
In this article we work over a fixed algebraically closed field $\k$. To the extent possible, we make no assumption on the characteristic of $\k$. Each time a restriction on the characteristic of $\k$ is required we write it down explicitly.
A \emph{variety} is an integral separated scheme of finite type over a field; in particular, varieties are always irreducible.   
An \emph{algebraic group} is a group scheme over a field that is smooth, or equivalently, geometrically reduced. By an algebraic subgroup, we always mean a closed and reduced subgroup scheme. The neutral component of an algebraic group $G$ is the connected component containing the identity element, denoted as $G^\circ$; this is a normal subgroup scheme of
$G$, and the quotient $G/G^\circ$ is a finite group scheme. When the base field of our varieties, rational maps, and algebraic groups is not specified, we work over the fixed algebraically closed field $\k$.
In this article, a $\P^n$-bundle is always assumed to be locally trivial for the Zariski topology; in particular, it is the projectivisation of a rank $n+1$ vector bundle when working over a regular Noetherian scheme.

\subsection{Mori fibrations}  \label{Mori fibrations}
In this subsection we recall some notions from the Mori theory / MMP; see \cite{KM98,Mat02,K13} for more details.
\begin{definition}
A normal projective Gorenstein variety $Z$ defined over an arbitrary field is called \emph{Fano} if the anticanonical bundle $\omega_Z^{\vee}$ of $Z$ is ample. A \emph{del Pezzo surface} is a surface that is a Fano variety.
\end{definition}

\begin{definition}Let $X$ be a proper scheme over {an arbitrary field $K$}. Then one can associate to $X$ the \emph{Picard scheme} $\BPic_{X/K}$ and its neutral component $\BPic_{X/K}^0$, which is a connected group scheme of finite type parametrising the algebraically trivial line bundles on $X$ (see \cite[n.~232, Sect.~6]{FGA}, \cite{M64}, and \cite{O62}).
\\The \emph{N\'eron-Severi scheme} $\BNS_{X/K}$ is defined via the following exact sequence of abelian group schemes:
\begin{equation*}  
0 \to \BPic_{X/K}^0 \to \BPic_{X/K} \to \BNS_{X/K} \to 0.
\end{equation*} 
The abelian group of $K$-points $\BNS_{X/K}(K)$ is denoted by $\NS(X)$ and the \emph{Picard rank} of $X$ is defined as $\rho(X):=\dim_\Q  \NS(X)_\Q.$
\end{definition}

\begin{definition}\label{Df:MoriFibration}
Let $\pi\colon X \to Y$ be a dominant projective morphism of normal projective varieties. Then $\pi$ is called a \emph{Mori fibration}, and the variety $X$ a \emph{Mori fibre space},  if the following conditions are satisfied:
\begin{enumerate}[$a)$]
\item\label{MoriFibrationDefa} $\pi_*(\O_X)=\O_Y$ and $\dim(Y) < \dim(X)$;
\item\label{MoriFibrationDefb} $X$ is $\Q$-factorial with terminal singularities; and
\item\label{MoriFibrationDefc} $\omega_X^{\vee}$ is $\pi$-ample and the relative Picard number $\rho(X/Y)$ of $\pi$, that is, the rank of $\NS(X/Y)=\NS(X)/\pi^* \NS(Y)$, is one.
\end{enumerate}
\end{definition}

In the rest of this article, we will consider only the case where $X$ is a rational threefold. The MMP for {smooth projective threefolds} has been established over a field of characteristic zero in \cite{Mor82} and recently over a field of characteristic $\geq 5$ (see for instance { \cite{HX15,CTX15, Bir16, BW,HW19}). Consequently, if $X$ is a smooth rational threefold and $\car(\k) =0$ or $\geq 5$}, then we can {run} an MMP  to produce a Mori fibration.

If $X$ is a rational threefold and $X \to Y$ is a Mori fibration, then we distinguish between three cases according to the dimension of the basis $Y$.
\begin{itemize}
\item \emph{$\dim(Y) =2$}. The Mori fibration $\pi$ is a \emph{conic bundle}, that is, a general fibre of $\pi$ is isomorphic to $\p^1$ (hence the generic fibre is a geometrically irreducible conic). Also, the surface $Y$ is rational with only canonical singularities.  
\item \emph{$\dim(Y)=1$}. The Mori fibration $\pi$ is a \emph{del Pezzo fibration}, that is, a Mori fibration whose general fibre is a del Pezzo surface (which is smooth if $\car(\k)=0$, but can be singular in low characteristic, see \upshape\S~\ref{Sec:dP}). Also, the curve $Y$ is isomorphic to $\p^1$.
\item \emph{$\dim(Y)=0$}. The Mori fibration is trivial and $X$ is a rational Fano threefold with Picard rank $1$ and terminal singularities.
\end{itemize}  

Let us note that a conic bundle (resp.~a del Pezzo fibration) is not necessarily a Mori fibration (because of the Picard rank condition).

\begin{definition}
A \emph{Mori conic bundle} (resp.~a \emph{Mori del Pezzo fibration}) is a conic bundle (resp.~a del Pezzo fibration) which is also a Mori fibration.
\end{definition}

We recall a result due to Blanchard \cite{Bla56} in the setting of complex geometry, whose proof has been adapted to the setting of algebraic geometry.

\begin{proposition} \emph{\cite[Proposition~4.2.1]{BSU13}} \label{blanchard}
Let $f\colon X \to Y$ be a proper morphism between varieties such that $f_*(\O_X)=\O_Y$. If a connected algebraic group $G$ acts regularly on $X$ {$($i.e.~$G$ acts on $X$ through a morphism of algebraic groups $G \mapsto \Autz(X))$}, then there exists a unique regular action of $G$ on $Y$ such that $f$ is $G$-equivariant.
\end{proposition}

Let us note that, if $X$ is a projective variety, then $\Autz(X)$ is a connected algebraic group; see \cite{MO67} for the algebraic group structure of the neutral component of the automorphism group of a proper scheme. We now consider two particular situations of high interest for us where Proposition~\ref{blanchard} applies. 
\begin{itemize}
\item Let $f\colon X \to Y$ be a divisorial contraction between projective varieties. By Proposition \ref{blanchard}, the algebraic group $\Autz(X)$ acts on $Y$ and $f$ is $\Autz(X)$-equivariant. This means that there is an inclusion of algebraic groups 
\[ f \Autz(X) f^{-1} \subseteq \Autz(Y),\] 
where we write $f^{-1}$ to denote the birational map which is the inverse of the birational morphism $f$. This observation will be useful in \upshape\S~\S~\ref{sec:non-maximality}, \ref{sec:P2bundles}, and \ref{decompo P1 over Fa}.
\item Let now $\pi\colon X \to Y$ be a Mori fibration. By Proposition \ref{blanchard}, the algebraic group $G:=\Autz(X)$ acts on $Y$ and $\pi$ is $G$-equivariant. 
To study $G=\Autz(X)$, we will often consider the exact sequence
\begin{equation}  \label{exact sequence}
1 \to \Autz(X)_Y \to \Autz(X) \to H \to 1,
\end{equation} 
where $H$ is the image of the natural homomorphism $G \to \Autz(Y)$, and $\Autz(X)_Y$ is the (possibly disconnected) subgroup scheme of $\Autz(X)$ which preserves every fibre of the Mori fibration $\pi$.
\end{itemize}

\begin{remark} \label{rk: inclusion into aut group of generic fibre}
Let $K=\k(Y)$, and consider the cartesian square 
\[\xymatrix@R=4mm@C=2cm{
    X_K \ar[r]^{p_1} \ar[d]_{p_2}  & X \ar[d]^{f} \\
    \Spec(K) \ar[r]_-{q} & Y
  }\]
defining the generic fibre of the Mori fibration $\pi\colon X \to Y$. Since, for any $\phi \in \Autz(X)_Y$, we have $q \circ p_2=f \circ p_1=f \circ \phi \circ p_1$, the universal property of Cartesian squares yields the existence of $\tilde \phi \in \Aut(X_K)$ such that $p_1 \circ \tilde \phi =\phi \circ p_1$ and $p_2 \circ \tilde \phi=p_2$. The map $\phi \mapsto \tilde \phi$ is an injective group homomorphism $\Autz(X)_Y \hookrightarrow \Aut(X_K)$.
\end{remark}

\begin{remark}\label{rk:MMP G eq}
Let $G$ be a connected algebraic group.
It follows from Proposition~\ref{blanchard} that an MMP applied to a smooth projective $G$-variety is automatically $G$-equivariant. Indeed, any contraction morphism $\phi\colon X \to Y$ associated with an extremal ray of $\overline{\NE}(X)_{K_X<0}$ satisfies the assumptions of Proposition~\ref{blanchard}, hence is $G$-equivariant. 
Moreover, the finite type $\O_Y$-algebra $\mathcal{A}:=\bigoplus_{m \geq 0}\phi_* \O_X(mK_X)$ is canonically a \emph{$G$-equivariant sheaf} (see \cite[Tag 03LE]{stacks-project} for the definition), hence the variety $X^+:=\mathrm{Proj}(\mathcal{A})$ is endowed with a $G$-action and the birational map $X^+ \dashrightarrow X$ is $G$-equivariant.
\end{remark}

\subsection{The equivariant Sarkisov program for threefolds} \label{subsec:Sarkisov}
In this subsection we recall some classical facts about the \emph{Sarkisov program}. 
This is used to factorise birational maps between Mori fibrations in easy links. 
We discuss here the three-dimensional case, following the approach by Corti \cite{C95}.

\smallskip

The following notion of isomorphism is often used implicitly in the literature. For instance, in \cite{C95,HM}, the authors consider linear systems instead of rational maps and implicitly study Mori fibrations up to such isomorphisms.
\begin{definition}\label{IsoMfs}
Let $\pi\colon X \to Y$ and $\pi'\colon X' \to Y'$ be two Mori fibrations. An isomorphism $\varphi\colon X\to X'$ is called \emph{isomorphism of Mori fibrations} if there is a commutative diagram
\[
\xymatrix@R=5mm@C=2cm{
     X  \ar[r]^-{\varphi} \ar[d]_-{\pi} & X' \ar[d]^-{\pi'} \\
     Y  \ar[r]^-{\tau}& Y'
  },
\]
where $\tau\colon Y \to Y'$ is an isomorphism.
\end{definition}
\begin{definition}\label{sarkisov-links}
A birational map 
$$
\xymatrix@R=5mm@C=2cm{
     X  \ar@{-->}[r]^-{\varphi} \ar[d]_-{\pi} & X' \ar[d]^-{\pi'} \\
     Y  & Y'
  },
$$
where $\pi\colon X \to Y$ and $\pi'\colon X' \to Y'$ are two Mori fibrations, is a \emph{Sarkisov link} if it has one of the following four forms:
$$ \text{\scriptsize (type \I)}
  
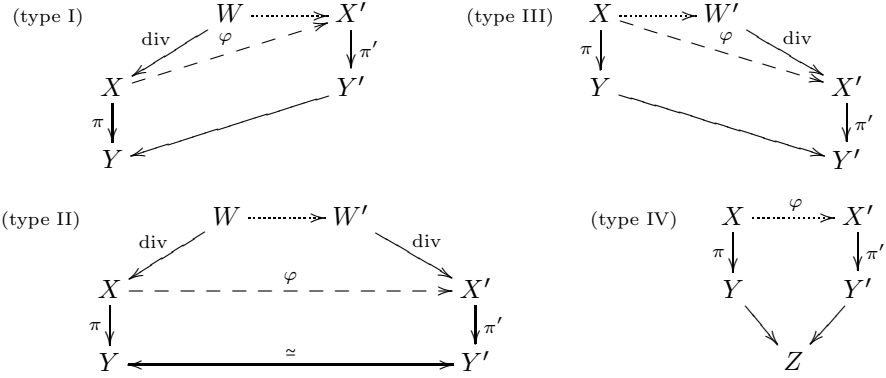
\xymatrix@R=5mm@C=1cm{
     & W  \ar@{..>}[r] \ar[dl]_-{\div} & X' \ar[d]^-{\pi'} \\
      X \ar[d]_-{\pi} \ar@{-->}[rru]^-{\varphi}  & & Y' \ar[lld] & \\
      Y &  &
    } 
    \text{\scriptsize (type \III)}\ \ 
\xymatrix@R=5mm@C=1cm{
     X \ar@{-->}[rrd]^-{\varphi} \ar@{..>}[r] \ar[d]_-{\pi} & W' \ar[dr]^{\div} \\
      Y \ar[drr]  & &  \ar[d]^-{\pi'} X' \\
    &  & Y' 
    }  
$$
$$
    \text{\scriptsize (type \II)}
\xymatrix@R=5mm@C=1cm{
    & W  \ar@{..>}[r] \ar[dl]_-{\div} & W' \ar[dr]^-{\div} & \\
     X \ar@{-->}[rrr]^-{\varphi} \ar[d]_-{\pi}  & & & X' \ar[d]^-{\pi'} & \\
     Y\ar@{<->}[rrr]^-{\simeq} & & & Y' 
  }
\text{\scriptsize (type \IV)}\ \ \
\xymatrix@R=5mm@C=.3cm{
    X  \ar@{..>}[rr]^{\varphi} \ar[d]_-{\pi}  & & X' \ar[d]^-{\pi'} \\
     Y \ar[dr]  & & Y' \ar[ld]\\
   & Z & 
    }  
$$
where:
\begin{itemize}
\item all varieties are normal;
\item all arrows that are not horizontal are elementary contractions, that is, contractions of one extremal ray, of relative Picard rank one;
\item the morphisms marked with \emph{div} are Mori divisorial contractions;
\item all the dotted arrows are \emph{small} maps, that is, compositions of Mori flips, flops and Mori anti-flips; and
\item the birational map $\varphi\colon X\dasharrow X'$ is not an isomorphism of Mori fibrations.
\end{itemize}
\end{definition}
\begin{remark}
{In a Sarkisov link, the  birational map $\varphi\colon X\dasharrow X'$ is not an isomorphism in Cases \I, \II, \III. In Case \IV, it is a pseudo-isomorphism, which can be an isomorphism or not.}
\end{remark}
\begin{remark}\label{Rem:UpToIso}
The composition of a Sarkisov link with an isomorphism of Mori fibrations is again a Sarkisov link. In the sequel we will identify two such links, and thus often say that there is a unique link, or finitely many links, which means ``up to composition at the target by an isomorphism of Mori fibrations''.
\end{remark}

\begin{remark}\label{remark:UnicitySarkisov}
In a Sarkisov link $\varphi$ as above, the morphism $W\to Y$ (type \I,\II), $W'\to Y'$ (type \III) or $X\to Z$ (type \IV) is of relative Picard rank $2$. The Sarkisov link (up to inverse and up to isomorphisms of Mori fibrations as above) is determined by this morphism, by applying a relative MMP to the two extremal rays (see \cite[Lemma 3.7]{BLZ}).
\end{remark}

Over an algebraically closed field of characteristic zero, the fact that every birational map between Mori fibrations  is a composition of elementary links as above (and of isomorphisms of Mori fibrations) was proven by Corti in  \cite[Theorem~3.7]{C95}, and generalised by \cite[Theorem~1.1]{HM} to any dimension. We need an equivariant version of this result for the action of a connected algebraic group. In dimension $3$, this follows actually from the proof of \cite[Theorem~3.7]{C95} as every step turns to be equivariant. We refer to \cite[Theorem~1.3]{Flo18} for a complete proof of the validity of the equivariant Sarkisov program in dimension $\geq 3$.

\begin{theorem} \label{th:Sarkisov decompo}\cite[Theorem~3.7]{C95}, \cite[Theorem~1.1]{HM}, \cite[Theorem~1.3]{Flo18}.\\
Assume that $\car(\k)=0$.
Let $X\to Y$ and $X'\to Y'$ be two terminal Mori fibrations and let $G=\Autz(X)$. Every $G$-equivariant birational map $\varphi\colon X\dasharrow X'$ factorises into a product of $G$-equivariant Sarkisov links and isomorphisms of Mori fibrations.\end{theorem}

\begin{notation}
From now on when we write \emph{equivariant link} we always mean \emph{non-trivial $G$-equivariant Sarkisov link}, where the connected linear algebraic group $G$ acting is clear from the context. 
\end{notation}

We now give some results that provide restrictions about the possible links that can occur in our setting (Lemmas~\ref{lem:ReducedblowUpcurve} and \ref{Lemm:AntiFlip} below).


The following result is essentially \cite[Lemma 2.13]{BLZ} {(see also \cite[Theorem 2]{ando} in the case where $W$ and $X$ are smooth threefolds and $\car(\k)=0$)}. We reproduce here the simple argument, which works over any algebraically closed field. 
\begin{lemma} \label{lem:ReducedblowUpcurve}
Let $\eta\colon W\to X$ be a divisorial contraction between $\Q$-factorial terminal projective threefolds that contracts a divisor $E$ onto a curve $C\subseteq X$.
Let $U\subset X$ be a dense open subset intersecting $C$ such that  $U\cap C$, $U$, and $\eta^{-1}(U)$ are contained in the smooth loci of $C$, $X$ and $W$ respectively. Then, $\eta\colon \eta^{-1}(U)\to U$ is the blow-up of $U\cap C$.
\end{lemma}
\begin{proof}
Let $p\in C \cap U$ be a point. {We may take  a smooth closed irreducible surface $S \subseteq X$ containing~$p$ such that the strict transform $\tilde{S}\subseteq W$ of $S$ is again a smooth surface} {(see \cite[Theorem~1.7.1]{BS95})}. {Up to shrinking $U$, we may assume that $p$ is the only intersection point of $S$ and $\Gamma$. We will prove that $\tilde{S}\to S$ is the blow-up of $p$.}

Let $C_1, \dots, C_m$ be the irreducible curves contracted by the birational morphism $\tilde S \to S$, which is the composition of $m$ blow-ups.
We now show  that $m = 1$. As $\eta$ is a divisorial contraction, we have $\rho(W/X)=1$, so all $C_i$ are numerically equivalent in $X$. Hence, for each $i,j$ we have
\[(C_i^2)_{\tilde S} = C_i \cdot E = C_j \cdot E = (C_j^2)_{\tilde S}.\]
Since at least one of the self-intersections $(C_i^2)_{\tilde S}$ must be equal to $-1$, and the exceptional locus of $\tilde S \to S$ is connected, we conclude that $m = 1$.
So $\tilde{S}\to S$ is the blow-up of $p$. 

{Repeating the argument for each point $p\in C\cap U$, we obtain that} $\pi^{-1}(U) \to U$ is the blow-up of $U\cap C$.
\end{proof}
\begin{lemma}\label{Lemm:AntiFlip}
Let $X$ be a smooth threefold endowed with a non-trivial $\PGL_2$-action.
Then there is no $\PGL_2$-equivariant anti-flip $X\dasharrow \tilde X$ to a threefold $\tilde{X}$ with terminal singularities.
\end{lemma}
\begin{proof}
Let $G=\PGL_2$ be the group acting non-trivially, and thus faithfully as $\PGL_2$ is simple, on $X$.
Assume by contradiction that there is a $G$-equivariant anti-flip $X\dasharrow \tilde{X}$, where $\tilde  X$ is a $G$-threefold with terminal singularities. This induces an isomorphism between $X\setminus \gamma$ and $\tilde X\setminus \tilde\gamma$, where 
${\gamma}\subseteq  X$ and $\tilde{\gamma}\subseteq \tilde X$ are unions of rational curves. As {$\tilde{X}\ps X$ is a flip}, $\tilde X$ must be singular (with terminal singularities). Hence,  $\tilde \gamma$ contains an isolated singular point $p$ of $\tilde X$.
The point $p$ is fixed for the $G$-action. As $G$ is simple, this implies that the $G$-action on every irreducible component $C$ of $\tilde \gamma$ containing $p$ is trivial. 

Let now $q$ be a smooth point of $\tilde X$ contained in such an irreducible component $C$ of $\tilde \gamma$. 
Let $K$ be a one-dimensional torus contained in $G$. By  Sumihiro's theorem \cite[Corollary~2]{Sum74}, there exists a $K$-invariant affine  open subset $U \subseteq \tilde X$ which is smooth and contains $q$. Then, by \cite[\S~I\!I\!I.1, Lemme]{Lun73}, there exists a $K$-equivariant morphism $U \to T_q \tilde X$, mapping $q$ to $0$ and \'etale at $q$.
Hence, if $K$ acts trivially on $T_q \tilde X$, then it acts trivially on $X$, and this contradicts the fact that $G$ acts faithfully on $X$.

Therefore, $G$ acts non-trivially on the tangent space $T_q \tilde X$, and so we obtain an injective homomorphism $G \hookrightarrow \GL(T_q \tilde X)=\GL_3$, which preserves the line $\ell\subseteq T_q$ corresponding to the tangent direction of $\tilde \gamma$. This would give an injective homomorphism into the parabolic subgroup $P\subseteq\GL_3$ that preserves the corresponding line, but this is impossible as $G\simeq \PGL_2$ does not embed in $P$. Therefore there is no $G$-equivariant anti-flip $X\dasharrow \tilde{X}$. 
\end{proof}

\subsection{Algebraic subgroups of \texorpdfstring {$\Bir(X)$}{Bir(X)}}\label{SubSec:AlgSubgroups}
The group of birational transformations $\Bir(X)$ of a variety $X$ has  no structure of algebraic group in general, but one can define a topology on it. In this subsection we recall this topology (Definition~\ref{defi: Zariski topology}) and characterise the algebraic subgroups of $\Bir(X)$.

\begin{definition}\label{Defi:AlgSubgroupsBir}
Let $X$ be a variety and let $A$ be a scheme.
\begin{enumerate}
\item
An \emph{$A$-family of birational transformations of $X$} is a birational transformation $\varphi\colon A\times X\dasharrow A\times X$ such that there is a commutative diagram
\[  \xymatrix{
    A\times X \ar@{-->}[rr]^{\varphi} \ar[dr]_{p_1}  && A\times X \ar[ld]^{p_1}  \\
   &A
  }\]
where $p_1\colon A\times X\to A$ is the first projection, and which induces an isomorphism $U\iso V$, where $U,V\subseteq   A\times X$ are two dense open subsets such that $p_1(U)=p_1(V)=A$.
\item
Every $A$-family of birational transformations of $X$ induces a map from $A$ (or more precisely from the $\k$-points of $A$) to $\Bir(X)$; this map $\rho\colon A\to \Bir(X)$ is called a \emph{morphism from $A$ to $\Bir(X)$.}
\item \label{item alg subgroup}
If $A$ is moreover an algebraic group and if $\rho$ is a group homomorphism, the rational map $A\times X\dasharrow X$ obtained by $p_2\circ \varphi$ (where $p_2\colon A\times X\to X$ is the second projection) is called a \emph{rational action of $A$ on $X$}, the morphism $\rho\colon A\to \Bir(X)$ is called a \emph{algebraic group homomorphism}, and the image of $A$ by the morphism is called \emph{an algebraic subgroup of $\Bir(X)$}. 

If, in addition, the map $\varphi$ is an automorphism, we say that the rational action of $A$ on $X$ is a \emph{regular action}, that the morphism $\rho\colon A\to \Aut(X)$ is an \emph{algebraic group homomorphism} and that the image of $A$ by the morphism $\rho$ is \emph{an algebraic subgroup of $\Aut(X)$}.
\end{enumerate}
\end{definition}
\begin{remark}
Let $\psi\colon X\dasharrow Y$ be a birational map between two varieties. For each scheme $A$, the map $\psi$ induces a bijection between $A$-families of birational transformations of $X$ and $A$-families of birational transformations of $Y$. In particular, morphisms $A\to \Bir(X)$ correspond, via $\psi$, to morphisms $A\to \Bir(Y)$.
\end{remark}
\begin{example}\label{Example:UnipotentAnyDim} Let $n\ge 2$, $d\ge 1$ two integers and let $X=\A^n$, $A=\A^d$. The next isomorphism corresponds to an $A$-family of birational transformations of $X$:
\[\begin{array}{ccc}
A\times X & \iso & A\times X\\
((t_1,\dots,t_d),(x_1,\dots,x_n)) & \mapsto &((t_1,\dots,t_d),(x_1,\dots,x_{n-1},x_n+\sum_{i=1}^d t_ix_1^i))
\end{array}\]
Since $(A,+)$ is an algebraic group and because the corresponding morphism $A\to \Aut(X)$ is a group homomorphism, there is a regular action of $A$ on $X$.
\end{example}

There is a natural contravariant functor, say $\mathfrak{Bir}_X$, from the category of schemes to the category of groups; it is defined at the level of objects by 
\[\mathfrak{Bir}_X(A)=\{ \text{morphisms from }A \text{ to } \Bir(X) \},\]
{where the group law on this set is given by pointwise multiplication.}
In the case where $X$ is rational, of dimension $\ge 2$, this functor is not representable by an algebraic group; this is not surprising and essentially follows from Example~\ref{Example:UnipotentAnyDim}, as the dimension of the corresponding algebraic group would be unbounded. In fact, this functor is not even representable by an ind-variety (inductive limit of varieties) by \cite[Theorem~1]{BF}, and the same holds when replacing ind-varieties by ind-stacks. However, the natural contravariant subfunctor, say $\mathfrak{Aut}_X$, from the category of schemes to the category of groups, defined at the level of objects by 
\[
\mathfrak{Aut}_X(A)= \Aut_A(X \times A),
\]
is representable by a group scheme when $X$ is proper \cite{MO67}.

In any case, even if $\Bir(X)$ has no structure of algebraic group or ind-algebraic group associated with the above families / morphisms, we can define a topology on $\Bir(X)$. This was done implicitly in \cite{Dem70} and explicitly in \cite{Ser00}.
\begin{definition}  \label{defi: Zariski topology}
Let $X$ be a variety. A subset~$F\subseteq   \Bir(X)$ is \emph{closed in the Zariski topology}
if for any variety~$A$ $($or more generally any $\k$-scheme locally of finite type $A)$ and any morphism~$A\to \Bir(X)$ the preimage of~$F$ is closed.
\end{definition}

We can then characterise algebraic subgroups of $\Bir(X)$, when $X$ is rational. 
\begin{definition}\label{Defi:BoundedDeg1}
Each element $f\in \Bir(\p^n)$ can be written as 
\[[x_0:\cdots:x_n]\mapsto [f_0(x_0,\dots,x_n):\cdots :f_n(x_0,\dots,x_n)]\]
where the $f_0,\dots,f_n\in \k[x_0,\dots,x_n]$ are homogeneous of the same degree $d$. Taking the polynomials without a common factor, the \emph{degree of $f$} is equal to $d$. 

This degree leads to the notion of subgroups of $\Bir(\p^n)$ of \emph{bounded degree}, which is invariant under conjugation {by an element of $\Bir(\p^n)$}. 

Similarly, when $X$ is a rational variety of dimension $n$, a subgroup $H\subseteq   \Bir(X)$ is of bounded degree if $\varphi H\varphi^{-1}\subseteq \Bir(\p^n)$ is of bounded degree for some (equivalently for each) birational map $\varphi\colon X\dasharrow \p^n$.
\end{definition}

This, together with the Zariski topology of $\Bir(X)$, allows to give the following characterisation. In particular it gives a unique structure of algebraic group to each algebraic subgroup of $\Bir(X)$.
\begin{proposition}\label{Prop:CharactAlgSubgroupsBirPn}\cite[Corollaire~2.18, Lemme~2.19, and Remarque~2.20]{BF}.\\
Let $X$ be a rational variety. The following hold.
\begin{enumerate}
\item
Every algebraic subgroup $G\subseteq   \Bir(X)$ is closed and of bounded degree.
\item
For each subgroup $H\subseteq   \Bir(X)$ which is closed and of bounded degree, there is an algebraic group $G$ and an algebraic group homomorphism $\rho\colon G\to \Bir(X)$ whose image is $H$ and such that, for each irreducible variety $A$, morphisms $A\to \Bir(X)$ whose image is contained in $H$ correspond, via $\rho$, to morphisms of varieties $A\to G$. In particular, $\rho$ induces an homeomorphism $G\to H$.
\end{enumerate}
\end{proposition}

We can in fact generalise the notion of bounded degree subgroups to any variety. This will be used to show that some elements of $\Bir(X)$ do not belong to algebraic subgroups. (More precisely, Lemma~\ref{Lemm:AlgebraickS} will be used in the proofs of Proposition~\ref{vertical_conic} and Lemma~\ref{Lem:ActiononQQg2}.)
\begin{definition}\label{Defi:BoundedDeg2}
{Let $X$ be a projective variety and let $H$ be an ample divisor on $X$}. To every birational map $\varphi\in \Bir(X)$, we associate its degree, with respect to $H$, given by
\[\deg_H(\varphi)=\varphi^*(H)\cdot H^{\dim(X)-1}=(\pi_1)^*(H^{\dim(X)-1})\cdot (\pi_2)^*(H),\]
where $\pi_1,\pi_2\colon \Gamma\to X$ are the two projections from the graph $\Gamma\subseteq   X\times X$ of $\varphi$.

We say that a subset $G\subseteq   \Bir(X)$ has \emph{bounded degree} if the subset $\{\deg_H(g)\mid g\in G\} \subseteq   \N$ admits an upper bound, for each ample divisor $H$ on $X$.
\end{definition}
\begin{remark}In order to verify the boundedness of the degree, it is sufficient to check it with respect to one ample divisor $H$ on $X$ (see \cite[Theorem~2]{1701.07760}). Moreover, the boundedness of the degree is invariant under conjugation by a birational map $X\dasharrow Y$.

 In particular, the notions of bounded degrees introduced in Definition~\ref{Defi:BoundedDeg1} and~\ref{Defi:BoundedDeg2} are the same, when $X=\p^n$, since we can choose $H$ to be a hyperplane and obtain the classical degree $\deg_H$. More generally, these two notions coincide if $X$ is rational.\end{remark}

\begin{definition}
Let $X$ be a variety.
We say that an element $\varphi\in\Bir(X)$ is \emph{algebraic} if it is contained in an algebraic subgroup $G\subseteq   \Bir(X)$.
\end{definition}

\begin{lemma}\label{Lem:AlgBoundedDeg}
Let $X$ be a {projective variety}.  Then the following hold.
\begin{enumerate}
\item\label{AlgGroupBoundedDeg}
Every algebraic subgroup $G\subseteq   \Bir(X)$ is of bounded degree.
\item\label{BoundedDegAlgGroup}
If $X$ is rational and $G\subseteq   \Bir(X)$ is a subgroup of bounded degree, then $G$ is contained in an algebraic subgroup of $\Bir(X)$.
\item\label{ElementAlgebraic}
If $X$ is rational, then an element $\varphi\in \Bir(X)$ is algebraic if and only if the subgroup generated by $\varphi$ is of bounded degree.
\end{enumerate}
\end{lemma}
\begin{proof}
\ref{AlgGroupBoundedDeg}: Let $G$ be an algebraic group, let $G\to \Bir(X)$ be an algebraic group homomorphism, and let $H$ be an ample divisor on $X$. We want to show that $\{\deg_H(g)\mid g\in G\}$ admits an upper bound. Replacing $H$ with a multiple of it, we may assume that $H$ is very ample, and thus that $X$ is a closed subvariety of $\p^n$, such that $H$ is a hyperplane section. The action of $G$ on $X$ yields a rational map  $G\times X\dasharrow X$, which extends to a rational map $G\times\p^n\dasharrow \p^n$. This latter is defined by polynomials of fixed degree, so $\{\deg_H(g)\mid g\in G\}$ is bounded. 

\ref{BoundedDegAlgGroup}: We may assume that $X=\p^n$. Then, the closure $\overline{G}\subseteq  \Bir(\p^n)$ of $G$ is of bounded degree \cite[Corollary~2.8]{BF}. Let us moreover observe that $\overline{G}$ is a subgroup of $\Bir(\p^n)$: the proof follows the same arguments as in an algebraic group, see \cite[\upshape\S 7.4, Proposition~A]{Hum75}. Hence, $\overline{G}$ is an algebraic subgroup of $\Bir(\p^n)$ by Proposition~\ref{Prop:CharactAlgSubgroupsBirPn}.

\ref{ElementAlgebraic}: Follows from \ref{AlgGroupBoundedDeg} and \ref{BoundedDegAlgGroup}.
\end{proof}

\begin{lemma}\label{Lemm:AlgebraickS}
Let $X$ be a variety, let $n\ge 1$, and let $f_1,\ldots,f_n\in \k(X)^*$. Then, the following conditions are equivalent:
\begin{enumerate}
\item\label{VarphiFAlg}
The birational map $\varphi_f\in \Bir(X\times (\p^1)^n)$ given by \[(x,[u_1:v_1],\cdots,[u_n:v_n])\mapsto (x,[u_1:f_1(x)v_1],\cdots,[u_n:f_n(x)v_n])\] is algebraic.
\item\label{VarphiFAlgConst}
Each rational function $f_i\in \k(X)^*$ $(\text{with }i=1,\ldots,n)$ is a constant, i.e.~is an element of $\k^*$.
\end{enumerate}
\end{lemma}
\begin{proof}
$\ref{VarphiFAlgConst}\Rightarrow \ref{VarphiFAlg}$ follows from the fact that the group $\G_m^n$ acts on $X\times (\p^1)^n$, via $(x,[u_1:v_1],\cdots,[u_n:v_n])\mapsto (x,[u_1:t_1v_1],\cdots,[u_n:t_nv_n])$, $(t_1,\ldots,t_n)\in \G_m^n$.

$\ref{VarphiFAlg}\Rightarrow \ref{VarphiFAlgConst}$: We suppose that at least one of the $f_i$, is not constant, and show that $\varphi_f$ is not algebraic, by showing that the group generated by $\varphi_f$ is not of bounded degree. Let $H_X$ be the class of an ample divisor on $X$, up to linear equivalence. For $i=1,\ldots,n$, we denote by $Z_i,P_i$  the divisors given by the zeros and poles of $f_i$: if $f_i\in \k^*$ then $Z_i=P_i=0$ and otherwise $Z_i,P_i$ are effective divisors on $X$ with no common support, such that $\div(f_i)=Z_i-P_i$.

We write $Y= X\times (\p^1)^n$, denote by $\pr_X\colon Y\to X$ the first projections and $\pr_i\colon Y\to \p^1$ $(i=1,\dots,n)$ the other projections. {We denote by $H_{\p^1}$ the equivalence class of a divisor of $\p^1$ of degree $1$.} The {class} $H=(\pr_X)^*(H_X)+\sum_{i=1}^n (\pr_i)^*(H_{\P^1})$ is then {ample on $Y$}.

We fix an integer $d\ge 1$ and write $\psi=(\varphi_f)^d\in \Bir(Y)$. As $\psi$ acts trivially on $X$, we obtain $(\psi)^*((\pr_X)^*(H_X))=(\pr_X)^*(H_X)$. For $i=1,\ldots,n$, {an element of} the {class} $(\pr_i)^*(H_{\p^1})$ {is given by the} equation $\mu u_i+\lambda v_i=0$ for some $[\mu:\lambda]\in \p^1$. Its pullback is then given by $\mu u_if_i(x)^d+\lambda_i v=0$, and thus {it belongs to the class} $(\pr_i)^*(H_{\p^1})+d(\pr_i)^*(Z_i)=(\pr_{i})^*(H_{\p^1})+d(\pr_i)^*(P_i)$. 

This yields $\psi^*(H)=H+d\sum_{i=1}^n (\pr_i)^*(P_i)$ and thus 
\[\deg_H(\psi)=(H+d\sum_{i=1}^n (\pr_i)^*(P_i))\cdot (H^{\dim(Y)-1}),\]
which is then not bounded when $d$ goes to infinity, if at least one of the $f_i$ is not constant.
\end{proof}

\subsection{Regularisation and reduction to automorphisms of Mori fibre spaces}\label{SubSec:Reg}
We now put together the notion of algebraic subgroups of $\Bir(X)$ introduced in {\upshape\S~\ref{SubSec:AlgSubgroups}} together with the results on Mori fibrations of~{\upshape\S~\ref{Mori fibrations}}.

\begin{theorem}\emph{\cite[Theorem]{Wei55} (see also \cite{Zai95,Kra} for a modern proof)}  \label{th Wei55}
Let $G$ be an algebraic group acting rationally on a variety~$V$. Then there exists a variety $W$ birational to $V$ such that the rational action of $G$ on $W$ obtained by conjugation is regular.
\end{theorem}

Therefore, for every algebraic subgroup $G \subseteq   \Bir(\p^n)$, there exists a birational map $\P^n \dashrightarrow X$, where $X$ is a (smooth, otherwise remove the singular locus) rational variety, which conjugates $G$ to a subgroup of $\Aut(X)$ (and of $\Autz(X)$ if moreover $G$ is connected). The following fact is well-known by the specialists but worth being mentioned.  

\begin{lemma}\label{lemma:linear}
Let $X$ be a rationally connected variety $($two general points of $X$ are connected by a rational curve$)$. Then every algebraic subgroup $G \subseteq   \Bir(X)$ is a linear algebraic group.
\end{lemma}

\begin{proof}
Since an algebraic group is linear if and only if its neutral component is linear, we can replace $G$ by $G^\circ$ and assume that $G$ is connected.
By Theorem \ref{th Wei55}, there is a variety $Y$ birational to $X$ such that $G$ identifies with a subgroup of $\Autz(Y)$. As $X$ is rationally connected, so is $Y$. As before, we may assume that $Y$ is smooth.
Let $\alpha_Y: Y \to A(Y)$ be the Albanese morphism, that is, the universal morphism to an abelian variety \cite{Ser58}. 
Then $G$ acts on $A(Y)$ by translations, compatibly with its action on $Y$, and the Nishi-Matsumura theorem (see \cite{Mat63,Bri10}) asserts that the induced homomorphism $G \to A(Y)$ factors through a homomorphism $A(G) \to A(Y)$ with finite kernel. 
However, since $Y$ is rationally connected, $A(Y)$, and then $A(G)$ are trivial.
Hence, $G$ is linear by the Chevalley's structure theorem; see for instance \cite[Theorem~1.1.1]{BSU13}.
\end{proof}

\begin{remark}
The previous result in the case $X=\P^n$, with almost the same proof, already appeared in \cite[Remark~2.21]{BF}.
\end{remark}

Under the extra assumption that $\car(\k)=0$, which ensures the existence of an equivariant resolution of singularities (see \cite[Proposition~3.9.1]{Kol07}), we have the following more precise result for algebraic subgroups of $\Bir(\P^n)$.

\begin{theorem}  \label{th alg subg of Cr3 are aut of Mori fib}
Assume that $\car(\k)=0$.
Every connected algebraic subgroup $G \subseteq  {\Bir(\p^n)}$ is conjugated to an algebraic subgroup of $\Autz(X)$, where $X$ is an {$n$-dimensional} rational Mori fibre space. 
\end{theorem}

\begin{proof}
By Lemma \ref{lemma:linear} the group $G$ is linear, and by Theorem \ref{th Wei55}, the group $G$ is conjugated to a subgroup of $\Autz(X')$, where $X'$ is a smooth rational variety. 
By \cite[Lemma~8]{Sum74} the variety $X'$ has an open covering that consists of $G$-invariant quasi-projective open subsets of $X'$. Replacing $X'$ by one of these $G$-invariant quasi-projective open subsets, we can assume that $X'$ is quasi-projective.
Taking a $G$-equivariant compactification \cite[Theorem~1]{Sum74} and then a $G$-equivariant resolution of singularities \cite[Proposition~3.9.1]{Kol07}, we may assume that the rational variety $X'$ is smooth and projective. Then we can {run} an MMP to $X'$ {(see \cite[Corollary~1.3.2]{BCHM10})} and check that this one is $G$-equivariant as $G$ is connected (see Remark~\ref{rk:MMP G eq}). 
We obtain a Mori fibre space $X$ birational to $\P^n$ on which $G$ acts faithfully.
\end{proof}

\subsection{Tori and additive groups in the Cremona groups}\label{tori_cremona}
In this subsection we prove that two tori of the same dimension resp. two additive groups are conjugate in $\Bir(\p^3)$.
These results are well-known from the specialists but we chose to recall the proof as it is quite elementary and needed in an essential way later in this article.

\begin{proposition} \label{cylinder}
Let $G$ be a torus or the additive group $\G_a$,
and let $X$ be a variety with a faithful action of $G$.
Then there exists a $G$-invariant affine dense open subset $X' \subseteq   X$ which is a $G$-\emph{cylinder}, that is, a  $G$-variety $G$-isomorphic to $G\times U$, where $G$ acts on itself by multiplication and $U$ is a smooth affine variety on which $G$ acts trivially. 
\end{proposition}

\begin{proof}
Removing the singular locus, we may assume that $X$ is smooth.
We first assume that $G$ is a torus. 
By \cite[Corollary~2]{Sum74}, the variety $X$ is covered by $G$-invariant affine open subsets, and thus we may assume that $X$ is affine. Then, the result follows from \cite[Part II, \upshape\S I.5, Proposition~9]{SB00}. 

We now assume that $G=\G_{a,k}$. By a theorem of Rosenlicht \cite{Ros63}, there exists a $G$-invariant dense open subset $V \subseteq  X$ that admits a geometric quotient $q: V \to W=V/\G_{a,k}$. Then, another theorem of Rosenlicht \cite[Theorem~10]{Ros56} gives the existence of a rational section $\sigma: W \dashrightarrow V$ of the geometric  quotient map $q$. If we consider the cartesian square
\[\xymatrix@R=4mm@C=2cm{
    Y \ar[r] \ar[d]  & V \ar[d]_{q} \\
    \Spec(K) \ar[r] & W \ar@{.>}@/_2pc/[u] _{\sigma}
  }\]  
where $K=\k(W)$ and $\Spec(K) \to W$ is the generic point of $W$, then it means that the $K$-variety $Y$ has a $K$-rational point. Moreover, $\G_{a,k}$ acts faithfully on $Y$  (because $Y$ is a $\G_{a,k}$-invariant dense open subset of $V$)  and since $Y \to \Spec(K)$ is a geometric quotient, this action is actually defined over $\Spec(K)$, and so $\G_{a,K}=\G_{a,k}\times_{\Spec(\k)} \Spec(K)$ acts on $Y$. Hence, $Y$ is an irreducible curve (over $\Spec(K)$) that contains a closed $\G_{a,K}$-orbit isomorphic to $\G_{a,K}$, and so $Y \simeq \G_{a,K}$. As $Y$ is a dense open subset of $V$  and $\G_{a,K}$ is an affine $\G_{a,\k}$-cylinder, we obtain the existence of an affine $\G_{a,\k}$-cylinder inside $X$. 
\end{proof}

\smallskip

\begin{remark} \item
\begin{itemize}
\item Assume that $\car(\k)=0$, and let $X$ be a variety with a $\G_a$-action. By Rosenlicht's theorem there exists a $\G_a$-invariant dense open subset $V \subseteq  X$ that admits a geometric quotient $q:V \to V/\G_{a}$. Since $\G_a$ has no non-trivial subgroup in characteristic zero, $q$ is in fact a $\G_a$-torsor. Then the existence of an affine $\G_a$-cylinder inside $X$ follows from the fact that  $\G_a$ is a \emph{special group} \cite[\upshape\S 3]{Gro58}. 
\item {Proposition~\ref{cylinder} does not hold for $G=\G_a^{r+1}$ with $r \geq 1$. For instance, $\G_a^{r+1}$ acts faithfully on the Hirzebruch surface $\F_r$ (with the notation of \S~\ref{sec:first classification}~\ref{first_family}) through 
\[
\begin{array}{ccc}
\G_a^{r+1} \times \F_{r} &\to &\F_{r}  \\ 
\left((a_0,\ldots,a_{r}),[y_0:y_1;z_0:z1]\right) &\mapsto &[y_0:y_1+y_0 \sum_{i=0}^{r} a_i z_0^iz_1^{r-i};z_0:z1]
\end{array}
\]
but $\F_{r}$ does not contain a $\G_a^{r+1}$-cylinder when $r \geq 1$ (for $r\geq 2$ this is clear for dimensional reasons, and for $r=1$ this follows from the fact that the $\G_a^{r+1}$-orbits are at most one-dimensional).
 }
\item Proposition~\ref{cylinder} can also be deduced from \cite[Theorem~3]{Pop16}.
\end{itemize}
\end{remark}

\begin{proposition}\label{Prop tori conjugate General}
Let $X$ be a rationally connected variety of dimension $n$ and let $G \subseteq  \Bir(X)$ be an algebraic subgroup of dimension $d$. We suppose that $G$ is a torus or the additive group.
Then, there is a $G$-equivariant birational map $X\dasharrow G\times Y$ where $Y$ is a rationally connected variety of dimension $n-d$, and the group $G$ acts on itself by left multiplication and trivially on $Y$. 
\end{proposition}

\begin{proof}
By Theorem \ref{th Wei55}, we may assume that $G$ acts faithfully and regularly on $X$.  By Proposition \ref{cylinder}, we may assume that $X=G\times Y$, for some variety $Y$, where $G$ acts on itself by multiplication and acts trivially on $Y$. Since $X$ is rationally connected, so is $Y$.
\end{proof}

\begin{corollary}\label{Coro:ConjclassToriAdd}
Let $X$ be a rationally connected variety of dimension $n$ and let $G$ be an algebraic group of dimension $d$, which is a torus or the additive group.\\
We have a bijection 
\[
\begin{array}{ccc}
\left\{\begin{array}{c}
\text{ birational classes of }\\
\text{ varieties $Y$ such that }\\
\text{  $Y\times\p^d$ is birational to }X
\end{array}\right\}
&\to & 
\left\{\begin{array}{c}
\text{ conjugacy classes of algebraic }\\
\text{subgroups of }\Bir(X)\\
\text{isomorphic to }G
\end{array}\right\}
\end{array}\]
that sends $Y$ onto the subgroup of $\Bir(X)$ obtained by conjugating the action of $G$ on  $G\times Y$ $($by left multiplication on $G$ and trivially on $Y)$ via a birational map $G\times Y\dasharrow X$.
\end{corollary}
\begin{proof}
For each variety $Y$ such that $Y\times\p^d$ is birational to $X$, the variety $Y\times G$ is birational to $X$, as $G$ is birational to $\p^d$. We then obtain an algebraic subgroup of $\Bir(X)$ isomorphic to $G$, unique up to birational conjugation. Proposition~\ref{Prop tori conjugate General} shows that every algebraic  subgroup of $\Bir(X)$ isomorphic to $G$ is obtained in this way. 

Let us then take another variety $Y'$ such that $Y'\times\p^d$ is birational to $X$. If $Y$ is birational to $Y'$, the actions of $G$ on $Y\times G$ and $Y'\times G$ are conjugate by a birational map, so we obtain the same conjugacy class in $\Bir(X)$. Conversely, if the actions are conjugate, there is a $G$-equivariant birational $\varphi\colon Y\times G\dasharrow Y'\times G$. As the fibres of the projections $\pi_Y\colon Y\times G\to Y$ and $\pi_{Y'}\colon Y'\times G\to Y'$ onto $Y$ or $Y'$ are the orbits of $G$, we obtain a birational map $\psi\colon Y\dasharrow Y'$ that makes the following diagram commutative:
\[\xymatrix@R=4mm@C=2cm{
    Y\times G \ar@{-->}[r]^{\varphi} \ar[d]_{\pi_{Y}}  & Y'\times G \ar[d]^{\pi_{Y'}} \\
    Y \ar@{-->}[r]^{\psi} & Y'.
  }\] 
\end{proof}

Corollary~\ref{Coro:ConjclassToriAdd} implies that studying conjugacy classes of tori in the Cremona groups is the same as studying birational maps between stably rational varieties. In particular, one gets the following:
\begin{corollary} \label{counter-example tori}
For each $d\ge 3$, there exist two $d$-dimensional tori in $\Bir(\p_\C^{d+3})$ which are not conjugate.
\end{corollary}
\begin{proof}
In \cite{BCSS} an example is given of a complex variety $Y$ of dimension $3$ which is not rational but such that $Y\times\p^3$ is rational.
The result then follows from Corollary~\ref{Coro:ConjclassToriAdd}{, applied to $Y$ and $\p^3$}.
\end{proof}

Recall that unirational surfaces over algebraically closed field of characteristic zero are rational, but this is not true in positive characteristic (e.g.~Shioda's surfaces in \cite{Shi74}). However, we have the following characteristic-free classical result; it will be used in the proofs of Theorems~\ref{th:A} and~\ref{Thm:conic bundles}.
\begin{proposition}\label{Prop:RatRat}
Let $Y$ be an irreducible variety of dimension $\le 2$ and suppose that one of the following holds:
\begin{enumerate}
\item\label{StablyRatisRat}
The variety $Y$ is stably rational $($i.e.~such that $Y\times\p^{m}$ is rational for some $m\ge 1)$; or
\item\label{BaseMfsRatisRat}
$\car(\k)\not=2$ and there is a Mori fibration $X\to Y$ with $X$ rational of dimension $\le 3$.
\end{enumerate}
Then $Y$ is rational.
\end{proposition}
\begin{proof}
If $Y$ is a curve, this is a consequence of the fact that $Y$ is unirational, by L\"uroth's theorem, so we may assume that $Y$ is a surface.

In case~\ref{StablyRatisRat}, we denote by $\psi$ a birational map $\psi\colon \p^{n}\dasharrow Y\times\p^m$, with $n=m+2$ and write $\pi\colon Y\times \p^m\to Y$ the first projection. As $\psi$ is birational, the differential map of $\pi\circ \psi\colon \p^n\dasharrow Y$ is surjective at a general point of $\p^n$.

In case~\ref{BaseMfsRatisRat}, we denote by $\psi\colon \p^n\dasharrow X$ a birational map and by $\pi\colon X\to Y$ the Mori fibration, which is a conic bundle since $Y$ is a surface. 
There is a dense open subset $Y'$ of $Y$ over which $\pi$ is flat (by generic flatness) and its fibres are geometrically regular (see \cite[Proposition 9.3.16]{Liu02} and \cite{MS03} for examples in characteristic $2$ of conic bundles whose a general fibre is non-reduced). Hence $\pi$ is smooth  over $Y'$ (\cite[Chapter~I\!I, Theorem~10.2]{Har77}), and so the differential map of $\pi\circ \psi\colon \p^n\dasharrow Y$ is surjective at a general point of $\p^n$.

In both cases, taking a general plane $P\subseteq  \p^n$, the restriction of $\psi$ yields a dominant rational map $f\colon P\dasharrow Y$ with a surjective differential map for a general $s \in P$. Therefore, \cite[\upshape\S 5.5, Theorem]{Hum75} implies that $f$ is a separable morphism, that is, the induced field extension $\k(Y)  \hookrightarrow \k(P)\simeq \k(\p^2)$ is separable, and thus  $Y$ is rational by Castelnuovo's theorem \cite[\upshape\S 1]{Zar58}. 
\end{proof}

\begin{corollary}\label{tori conjugate} \item
\begin{enumerate}
\item \label{item:tori conj} For each $n \geq 1$ and each $d\in \{n,n-1,n-2\}$, two tori of dimension $d$ in the Cremona group $\Bir(\p^n)$ are conjugate.
\item \label{item: Ga cong}For each $n\in \{1,2,3\}$, two additive groups are conjugate in the Cremona group $\Bir(\p^n)$.
\end{enumerate}
\end{corollary}
\begin{proof}
Follows from Corollary~\ref{Coro:ConjclassToriAdd} and Proposition~\ref{Prop:RatRat}\ref{StablyRatisRat}. 
\end{proof}

\begin{remark}
Corollary~\ref{tori conjugate} is a classical result when $\car(\k)=0$; see \cite[Theorem~2]{Pop13} for~\ref{item:tori conj} and \cite[Corollary~5]{Pop17} for~\ref{item: Ga cong}. 
\end{remark}

Note that we have the following result, well-known to specialists (see \cite[Proposition 4.1]{BL} for a similar argument), but that we did not see written explicitly in the following form.
\begin{corollary} \label{cor Auts cubic threefold}
Suppose that $\car(\k)=0$ and let $X$ be a rationally connected variety of dimension $3$, which is not rational $($for instance a smooth projective cubic threefold, or more generally every non-rational Fano threefold$)$. Then every connected algebraic subgroup of $\Bir(X)$ is trivial. In particular, $\Autz(X)$ is trivial.
\end{corollary}
\begin{proof}
By  Lemma~\ref{lemma:linear} every algebraic subgroup of $\Bir(X)$ is linear. The Jordan-Chevalley decomposition implies that any connected linear algebraic group is generated by tori and unipotent subgroups.
Let $G$ be  a connected linear algebraic subgroup of $\Bir(X)$. To prove the statement it suffices then to show that $G$ contains no non-trivial tori and no additive groups. Assume that $H$ is a subgroup of $G$ that is a non-trivial torus or an additive group. Then by Proposition \ref{Prop tori conjugate General}, the variety $X$ is birational to $H\times Y$, where $Y$ is a rationally connected variety of dimension at most $2$. If $\dim(Y)=1$, Lur\"oth theorem implies that $Y$ is rational.
If $\dim(Y)=2$, since $\car(\k)=0$, Castelnuovo's theorem implies that $Y$ is rational (see \cite[Theorem~13.27]{Bad01}). So $X$ must be rational which is false by assumption. Therefore, $G$ does not contain a non-trivial torus or an additive group, and so $G$ must be trivial.
\end{proof}

\section{Mori conic bundles over surfaces}\label{Sec:ConicBundles}
The main goal of this section is to prove Theorem~\ref{Thm:conic bundles} (in \upshape\S~\ref{sec:proof of Prop B}). We will first reduce to the case where $X \to S$ is a standard conic bundle (in \upshape\S \ref{first reduction}). If the generic fibre of this conic bundle is $\p^1$, we will obtain a $\p^1$-bundle  (Lemma~\ref{2_intersections}). Otherwise, we will see that $\Autz(X)$ is a torus of dimension at most two (in \upshape\S \ref{conic bundles}).

\subsection{Standard conic bundles}  \label{first reduction}
To study the automorphism group of a conic bundle over a surface we can reduce to the case of a standard conic bundle, which is a Mori conic bundle with nice geometric features. In this subsection we explain this reduction and some consequences.

\begin{definition} \label{def standard conic bundle}
A morphism $\pi\colon X \to S$ is called \emph{standard conic bundle} if
  \begin{enumerate}
  \item The varieties $X$ and $S$ are smooth projective, and $\dim(X)=1+\dim(S)$.
  \item\label{StdConicBundle1}
   The morphism $\pi$ is induced by the inclusion of $X$ (given by an equation of degree $2$) in a $\p^2$-bundle over $S$. The \emph{discriminant divisor} $\Delta\subseteq   S$ is reduced, and all its components are smooth and intersect in normal crossings $($i.e.~$\Delta$ is an SNC divisor$)$. For each $p\in S$, the rank of the $3\times 3$-matrix corresponding the quadric equation is $3$, $2$, $1$ respectively when $p\notin \Delta$, $p\in \Delta\setminus \mathrm{sing}(\Delta)$, $p\in \mathrm{sing}(\Delta)$.
  \item\label{StdConicBundle2}
  The relative Picard rank is $\rho(X/S)=1$.
  \end{enumerate}
\end{definition}
\begin{remark}
We may observe that the $\p^2$-bundle over $S$ that appears in the definition of a standard conic bundle is unique; indeed, it is given by $\mathbb{P}_S(\pi_*\omega_X^{-1})$ (see \cite[Proposition~1.2]{Bea77} in the case $S=\p^2$).
\end{remark}
\begin{remark}

A standard conic bundle is always a Mori fibration. Indeed, the only non-trivial condition to check is that $\pi_* \O_X=\O_S$, but this follows from \cite[Tag 0AY8]{stacks-project}, since the generic fibre of $\pi$ is assumed to be geometrically reduced.
\end{remark}

The following result, without the connected algebraic group action, is due to Sarkisov  \cite{sarkisov}. It was generalised in the equivariant setting for finite group actions by Avilov \cite{avilov}. {The assumption $\car(\k)\neq 2$ is assumed in \cite{sarkisov} and needed at different steps of the proof, as it deals with conics and symmetric matrices.}

\begin{theorem} \label{th sarkisov_standard_conic}
Assume $\car(\k)\neq 2$. Let  $S$ be a surface, let $X$ be a normal variety, let $\pi\colon X\rightarrow S$ be a conic bundle, and let $G=\Autz(X)$. 
Then there is a $G$-equivariant commutative diagram
\[\xymatrix@R=4mm@C=2cm{
    \hat X \ar@{-->}[r]^{\psi} \ar[d]_{\hat\pi}  & X \ar[d]^\pi \\
    \hat{S} \ar[r]_{\eta} & S
  }\]  
where $\psi$ is a birational map, $\eta$ is a birational morphism, and the morphism $\hat\pi\colon \hat X\to \hat S$ is a standard conic bundle.  
\end{theorem}

\begin{proof}
We follow the proof of \cite[Theorem~1.13]{sarkisov}, and simply check that all steps are $G$-equivariant.

Firstly, we may assume that $S$ is normal, by replacing $S$ with its normalisation $\tilde{S}$ and $X$ by $X\times_S \tilde{S}$. We then denote by $U\subseteq   S$  the maximal open subset of smooth points of $S$ over which $\pi$ is smooth, which is $G$-invariant (as a general fibre being $\p^1$, the set $U$ is dense in $S$). We embed $X_U=X\times_S U$ into $\mathrm{Proj}({\pi_*}(\omega_{X_U}^{\vee}))$ (see \cite[\upshape\S 1.5]{sarkisov}); the closure is another conic bundle $(X',S,\pi')$ embedded in a $\p^2$-bundle.
These steps of \cite[p.362]{sarkisov} are naturally $G$-equivariant.

For each point $p\in S$ we can then find a neighbourhood in which the fibre of $\pi'$ is a conic in $\p^2$. The discriminant of the conic yields a local equation, which produces a divisor $\tilde\Delta$ on $S$. We now observe that this one is reduced. This can be shown only on the open subset of $S$ consisting of regular points since its complement only consists of points, and follows then from \cite[Corollary~1.9]{sarkisov}.

As $G$ is connected, we can resolve the singularities of $S$ in a $G$-equivariant way: it suffices to repeatedly alternate normalising the surface with blowing-up singular points (see \cite{Art84}).
Let $\hat{S}\to S$ be a $G$-equivariant resolution of singularities of $S$, and consider $\hat\pi\colon \hat{X}=X\times_S \hat{S}\to \hat{S}$. By blowing-up the points of the new discriminant curve where this one does not have simple normal crossings and replacing it by its strict transform, one may assume, after finitely steps, that the discriminant curve $\Delta$ of $\hat\pi$ satisfies Assumption \ref{StdConicBundle1} of Definition~\ref{def standard conic bundle} \cite[Proposition~1.16 and 1.8]{sarkisov}. Again, this process is $G$-equivariant.

It remains to observe that if $\rho(\hat{X}/\hat{S})>1$ then the preimage of some components of $\Delta$ split into two surfaces, and one can contract one of these. As our group $G$ is connected, both components are invariant, so this step is again $G$-equivariant. This achieves the proof.
\end{proof}

The following result is well-known by experts (see for instance \cite[Proposition~1.2]{Bea77}, \cite{sarkisov} or \cite[Lemma~1]{Isk87} for similar results), but we chose to recall the guidelines of the proof for the sake of completeness.
\begin{lemma}\label{2_intersections}
Assume $\car(\k)\neq 2$. Let $S$ be a smooth projective  rational surface, let $\pi\colon X\to S$ be a \emph{standard conic bundle} $($as in Definition ~$\ref{def standard conic bundle})$, let $\Delta\subseteq   S$ be the discriminant curve, and let $K:=\k(S)$. Then the following are equivalent:
\begin{enumerate}
\item \label{XP1B} $X$ is a $\p^1$-bundle over $S$;
\item \label{GenericFibP1K} the generic fibre $X_K$ is isomorphic to $\p^1_K$;
\item \label{PiRatSection} $\pi$ has a rational section; and
\item \label{NoDelta}$\Delta = \emptyset$.
\end{enumerate}
Moreover, if $\Delta$ is non-empty and reducible, then each rational irreducible component $C$ of $\Delta$ intersects the complement $\overline{\Delta\setminus C}$ into at least two distinct points. If $\Delta$ is non-empty and irreducible, then $g(\Delta)\ge 1$.
\end{lemma}

\begin{proof} The implication $\ref{XP1B}\Rightarrow \ref{GenericFibP1K}$ is direct. The equivalence $\ref{GenericFibP1K}\Leftrightarrow \ref{PiRatSection}$ is classical:
The map $\pi$ has a rational section if and only if the generic fibre $X_K$ has a $K$-point. By hypothesis $X_K$ is a conic and it is isomorphic to $\p^1_K$ if and only if it has a $K$-point (just consider the projection from the point). 

The proof of $\ref{XP1B}\Leftrightarrow \ref{NoDelta}$ is more subtle; it is a consequence of the Artin-Mumford exact sequence and can be found in \cite[Corollary~5.4]{sarkisov}.

The implication $\ref{PiRatSection}\Rightarrow \ref{NoDelta}$ is a consequence of $\rho(X/S)=1$ (Condition~\ref{StdConicBundle2} of Definition~\ref{def standard conic bundle}). Indeed, assume by contradiction that $\Delta \neq \emptyset$ and let $p\in \Delta$ be a general  point  of $\Delta$ such that $\pi^{-1}(p)\simeq F_1 \cup F_2 = \p^1 \cup \p^1$. Since $\pi$ has a rational section $\sigma$, we know that $\NS(X/S)=\Z[\sigma]$ is generated by it. As $\sigma\cdot (F_1\cup F_2)=1$, we get $\sigma\cdot F_i=0$, for some $i$, which gives $D\cdot F_i=0$ for each divisor $D$ on $X$, a contradiction.

For the second part of the statement, assume $(C\cdot \overline{\Delta\ \backslash \ C})\le1$. Then the curve $\tilde C$ that parametrises the irreducible components of $\pi^{-1}(C)\rightarrow C$ gives a double cover of $C$; see for instance \cite[Proposition~1.5]{Bea77} for the construction of $\tilde C$. The condition on the intersection numbers together with the Hurwitz formula \cite[Corollary~IV.2.4]{Har77} implies that $\tilde C$ is reducible. This contradicts the minimality condition $\rho(X/S)=1$. Similarly, Hurwitz formula implies that, if $\Delta$ is irreducible, $g(\Delta)$ must be positive (the double cover of $\Delta$ is \'etale in this case).
\end{proof}

\begin{remark}
Theorem~\ref{th sarkisov_standard_conic} reduces the study of automorphisms of conic bundles to the case of standard ones. Lemma~\ref{2_intersections} then implies that a standard conic bundle with a generic fibre isomorphic to $\p^1$ is a $\p^1$-bundle, a case studied in \cite{BFT}. 
\end{remark}

\subsection{Conic bundles whose generic fibre is not \texorpdfstring{$\p^1$}{P1}}  \label{conic bundles}
Let $\pi\colon X \to S$ be a standard conic bundle (see Definition \ref{def standard conic bundle}). In this subsection we study the case where  the generic fibre of $\pi$ is not rational and prove that $G=\Autz(X)$ is a torus of dimension $\leq 2$.

\begin{proposition}\label{Prop:ConicAutoFieldK}
Let $K$ be a field, with $\car(K) \not=2$  and algebraic closure $\overline{K}$. 
Let $\Gamma\subseteq   \p^2$ be a geometrically irreducible conic, defined over $K$ and let $g\in \Aut(\Gamma)$ be a non-trivial $K$-automorphism. Then, $g$ extends to a unique element of $\Aut(\p^2)=\PGL_3(K)$ and one of the following holds:
\begin{enumerate}
\item
There is exactly one point of $\p^2(\overline{K})$ fixed by $g$. This point is defined over $K$ and belongs to $\Gamma$.
\item
There is one point $p\in \p^2(K)$ fixed by $g$ and a line $\ell\subseteq   \p^2$ invariant by $g$ such that $p\not\in \ell\cup \Gamma$ and such that $\ell\cap \Gamma$ consists of two $\overline{K}$-points fixed by~$g$.
\end{enumerate}
\end{proposition}

\begin{proof}
 Since the anti-canonical divisor of $\Gamma$ is given by an element in the linear system of hyperplane sections, there is a surjective group scheme homomorphism
 \[\nu\colon \Aut(\p^2,\Gamma)\to \Aut(\Gamma),\]
 where $\Aut(\p^2,\Gamma)$ is the algebraic subgroup of $\Aut(\p^2)$ consisting of elements preserving $\Gamma$. The group homomorphism $\nu$ is moreover an isomorphism, since no {non-trivial} element of $\Aut(\p^2)$ fixes pointwise a conic.
 
After applying an element of $\PGL_3(\overline{K})$, the curve  $\Gamma$ is given by $x^2-yz=0$. We consider the group homomorphism
\begin{equation*} 
\begin{array}{rccc}
\hat\rho\colon &\GL_2(\overline{K})&\rightarrow &\GL_3(\overline{K})\\
& \begin{bmatrix}
a & b\\
c & d\end{bmatrix} &\mapsto& \frac{1}{ad-bc}
\begin{bmatrix}
ad+bc&ac&bd\\
2ab&a^2&b^2\\
2cd& c^2& d^2\end{bmatrix} \end{array},
\end{equation*}
which induces an injective group homomorphism $\rho\colon \PGL_2(\overline{K})\hookrightarrow \PGL_3(\overline{K})$. Writing $Q=
\begin{bmatrix}
-2&0&0\\
0&0&1\\
0&1&0\end{bmatrix} \in \GL_3(K)$ the matrix corresponding to the conic $\Gamma$ (defined up to multiple), we observe that
\begin{equation}\label{Qform}
\tr{(\hat{\rho}(h))}Q\hat\rho(h)=Q\end{equation}
for each $h\in \GL_2(\overline{K})$. This shows that the action of $\PGL_2(\overline{K})$ on $\p^2$ induced by $\hat\rho$ preserves $\Gamma$. Moreover, the isomorphism $\p^1\to \Gamma$ given by $[u:v]\mapsto [uv:u^2:v^2]$ is $\PGL_2(\overline{K})$-equivariant. Hence, $\rho$ {factorises through} the group isomorphism $\nu^{-1}\colon \PGL_2(\overline{K})=\Aut(\Gamma)(\overline{K}) \to  \Aut(\p^2,\Gamma)(\overline{K}).$

The element $g$ is conjugate, in $\PGL_2(\overline{K})$, to 
$s=\begin{bmatrix}
1 & 0\\
1 & 1\end{bmatrix}$ or $t_\lambda=\begin{bmatrix}
\lambda & 0\\
0 & 1\end{bmatrix}$, for some $\lambda \in \overline{K}\setminus \{0,1\}$. Hence $g$ {maps in $\PGL_3(\overline{K})$} to the class of
\begin{equation}\label{Twomatrices}\hat{\rho}(s)=\begin{bmatrix}
1&1&0\\
0&1&0\\
2& 1& 1\end{bmatrix}\mbox{ or }
\hat{\rho}(t_\lambda)=\begin{bmatrix}
1&0&0\\
0&\lambda&0\\
0& 0& \lambda^{-1}\end{bmatrix}.\end{equation}
In the first case, the only fixed point of $\p^2(\overline{K})$ is $p_1=[0:0:1]$, which belongs to $\Gamma$. In the second case, the points $p=[1:0:0]$, $q_1=[0:1:0]$, $q_2=[0:0:1]$ are fixed (and are the only ones except if $\lambda=-1$). The last two belong to $\Gamma$ but the first does not belong to it.

It remains to show that $p_1$, $p$ and the line $\ell$ through $q_1$ and $q_2$ are defined over $K$ (before the change of coordinates). 
To do this, we associate with $g\in \Aut(\Gamma)$ its extension $F\in \PGL_3(K)=\Aut(\p^2)$, and take a matrix $M\in \GL_3(K)$ that represents the class $F$. There exists thus $R\in \SL_3(\overline{K})$ and $\mu\in \overline{K}^*$ such that \[RMR^{-1}=\mu \hat\rho(h),\] where $h\in \GL_2(\overline{K})$ is such that $h=s$ or $h=t_\lambda$ for some $\lambda\in \overline{K}$. We obtain in particular $\det(M)=\mu^3\in K$. Moreover, Equation~(\ref{Qform}) yields $\tr{M} \tilde{Q} M=\mu^2 \tilde{Q}$, where $\tilde{Q}=\tr{R}QR$ is a symmetric matrix that defines the conic $\Gamma$ before the change of coordinates. Since $\Gamma$ is defined over $K$, there is $\xi\in \overline{K}$ such that $Q'=\xi \tilde{Q}\in \GL_3(K)$. As $\tr{M} Q' M=\mu^2 Q'$ and $M,Q\in \GL_3(K)$, one finds that $\mu^2\in K$, whence $\mu\in K$ (because $\mu^3\in K$). Replacing $M$ with $\mu M$, one can thus assume that $\mu=1$. The points $p_1,p$ correspond then, before the change of coordinates, to eigenvectors of eigenvalue $1$ and are thus defined over $K$. The same holds for the line $\ell$, which also corresponds to an eigenvector of eigenvalue $1$, for the dual action of $M$.
\end{proof}

\begin{corollary}\label{Coro:FormAutoConic}
 Let $K$ be a infinite field of characteristic $\not=2$ such that $-1\in K$ is a square. Let $\Gamma\subseteq   \p^2$ be a geometrically irreducible conic defined over $K$, with no $K$-rational point, and let $g\in \Aut(\Gamma)$ be a non-trivial $K$-automorphism of $\Gamma$. Then, up to a $K$-automorphism of $\p^2$, the equation of $\Gamma$ is given by
 \[\lambda x^2+y^2-\mu z^2\]
 for some $\lambda,\mu\in K^*$ that are not squares, and $g$ is given by 
\[[x:y:z]\mapsto [x:a y+c\mu z:c y+a z]\text{ or }[x:y:z]\mapsto [x:a y-c\mu z:c y-az],\]
for some $a,c\in K$ satisfying $1=a^2-c^2\mu$.
\end{corollary}

 \begin{proof}
 Proposition~\ref{Prop:ConicAutoFieldK} provides the existence of a point $p\in \p^2\setminus \Gamma$ and a line $\ell\subseteq   \p^2$, both defined over $K$ and invariant by $g$, such that $p\not\in \ell$ and $\ell\cap \Gamma$ consists of two $\overline{K}$-points $q_1,q_2$ fixed by $g$.

 We may assume, applying an element of $\PGL_3(K)$, that $p=[1:0:0]$ and that $\ell$ is the line $x=0$.  The automorphism $g$ is then of the form $[x:y:z]\mapsto [x:ay+bz:cy+dz]$, for some $a,b,c,d\in K$, $ad-bc\not=0$. Writing the equation of the conic $\Gamma$ as 
$$\lambda_1x^2+x(\lambda_2y+\lambda_3z)+\lambda_4y^2+\lambda_5yz+\lambda_6z^2=0,$$
where $\lambda_1,\dots,\lambda_6\in K$, we find that $\lambda_2y+\lambda_3z$ is fixed by the action given by $(y,z)\mapsto (ay+bz,cy+dz)$. Conjugating by $[x:y:z]\mapsto [x-(\lambda_2y+\lambda_3z):y:z]$ we may assume that $\lambda_2=\lambda_3=0$.

The points $q_1,q_2$ correspond to the roots of the polynomial 
$\lambda_4y^2+\lambda_5yz+\lambda_6z^2$, which is therefore irreducible in $K[y,z]$. In particular, $\lambda_4\lambda_6\not=0$. Dividing the equation by $\lambda_4$ and completing the square, we may assume that $\lambda_4=1$ and $\lambda_5=0$. We then write $\mu=-\lambda_6$, $\lambda=\lambda_1$.  The equation is then 
$\lambda x^2+y^2-\mu z^2,$ and $\lambda,\mu\in K^*$. Moreover, none of the $\pm \lambda,\pm\mu$ is a square in $K$, as the curve does not contain any $K$-point.
The equation being preserved, we find \[\begin{array}{ll}
 \lambda x^2+y^2-\mu z^2=\lambda x^2+(ay+bz)^2-\mu(cy+dz)^2,\\
 1=a^2-c^2\mu ,\ ab=cd\mu ,\ 1 =d^2-b^2\mu^{-1}.\end{array}\]
It remains to show that $d=a, b=c\mu $ or that $d=-a, b=-c\mu$.

Since $-\mu$ is not a square in $K$, we have $d\not=0$. If $c=0$, we find $b=0$ and then $a^2=1$, $d^2=1$, so $d=\pm a$ as we wanted. We can then assume that $abcd\not=0$ and write $\mu=\frac{ab}{cd}$, which yields
$0=a^2-c^2\mu-d^2+b^2\mu^{-1}=\frac{(a^2-d^2)\cdot (ad-bc)}{ad}$. Hence, $d=\pm a$, which yields $b=\pm c\mu$ as we wanted.
 \end{proof}

\begin{proposition}\label{vertical_conic}
Assume that $\car(\k)\not=2$. Let $\pi\colon X\to S$ be a morphism of algebraic varieties whose generic fibre is a smooth conic in $\p^2_{\k(S)}$, not isomorphic to $\p^1_{\k(S)}$. If $G$ is an algebraic subgroup of $\Autz(X)_S$, then $G$ is a finite group isomorphic to $(\Z/2\Z)^r$ for some $r\in \{0,1,2\}$.
\end{proposition}
\begin{proof}
The generic fibre $X_{\k(S)}$ is a smooth conic in $\p^2_{\k(S)}$ not isomorphic to $\p^1_{\k(S)}$ and thus does not have any $\k(S)$-rational point. The action of $G$ yields an injective group homomorphism $G\hookrightarrow \Aut(X_{\k(S)})$ by Remark~\ref{rk: inclusion into aut group of generic fibre}. As $\Aut(X_{\k(S)})$ is isomorphic to $\PGL_2$ over the algebraic closure $\overline{\k(S)}$ of $\k(S)$, it suffices to show that every non-trivial element  $g\in G$ has order $2$.

By Corollary~\ref{Coro:FormAutoConic}, there is a $S$-birational map $X \dashrightarrow Y$, where $Y\subseteq   \p^2\times S$ is given by $\lambda x^2+y^2-\mu z^2$  for some $\lambda,\mu\in \k(S)^*$ which are not squares, and the action of $g$ on $Y$ is given by  
$$h_{a,c}\colon [x:y:z]\mapsto [x:a y+c\mu z:c y+a z]\text{  or  }h_{a,c}':[x:y:z]\mapsto [x:a y-c\mu z:c y-az],$$
for some $a,c\in \k(S)$ satisfying $1=a^2-c^2\mu$. As $(h_{a,c}')^2$ is the identity for all $a,c$ as above, one may assume that $g$ belongs to the set 
\[H=\{h_{a,c}\mid  a,c\in \k(S), 1=a^2-c^2\mu\}\subseteq   \Bir(Y/S).\] Note that $H$ is a subgroup of $\Bir(Y/S)$, isomorphic to a subgroup of the multiplicative group $(\k(S)[\sqrt{\mu}])^*$ via $h_{a,c}\mapsto a+c\sqrt{\mu}$. Hence, $H$ is an abelian group and its elements of finite order are contained in $\k^*\subseteq   (\k(S)[\sqrt{\mu}])^*$, and thus contained in the finite group $\k\cap H=\k(S)\cap H=\{h_{1,0},h_{-1,0}\}$ of order $2$. 

It remains then to see that $g$ cannot be of infinite order. We denote by $G_0\subseteq   G$ the smallest algebraic group containing $g$, which is given by $G_0=\overline{\langle g\rangle}$, and is an abelian algebraic subgroup of $G$.

 The degree $2$ field extension  $\k(S)[\sqrt{\mu}]/\k(S)$ corresponds to a double covering $\hat{S}\to S$. Since $g$ acts trivially on $S$, we can then lift the {rational} action of $g$, and of $G_0$, to the fibred product $Z=Y\times_S \hat{S}$, by acting trivially on $\hat{S}$. The {automorphism $g$ yields therefore an algebraic element $\hat{g}\in \Bir(Z)$}. To write this one explicitly, we identify $\k(S)$ with a subfield of $\k(\hat{S})$, denote by $\kappa\in \k(\hat{S})$ the element corresponding to $\sqrt{\mu}$, and view locally  $Z$ as
 \[\{([x:y:z],s)\in \p^2\times\hat{S} \mid \lambda(s)x^2+y^2-\kappa(s)^2z^2=0\}.\] The variety $Z$ is birational to $\p^1\times\hat{S}$, via $([x:y:z],s)\mapsto ([x:y+\kappa(s)z],s)$, and the action of $\hat{g}$ is now given by $([u:v],s)\mapsto ([u:(a(s)+c(s)\kappa(s))v],s)$. By Lemma~\ref{Lemm:AlgebraickS}, the element $a+c\kappa\in \k(\hat{S})$ belongs to $\k\subseteq   \k(S)$. Hence, $c=0$ and $a\in \k$. This implies that $a=\pm 1$, since $a^2-c^2\mu=1$.
\end{proof}

\begin{remark}
The case of Proposition~\ref{vertical_conic} when $\car(\k)=0$ and $G$ is finite was proven in \cite[Corollary~4.12]{BZ17} and then used to show Jordan properties of birational maps of $X$. \end{remark}

\begin{proposition}\label{conic_horizontal}
Let $S$ be a smooth projective rational surface and let $\Delta$ be an effective reduced divisor on $S$ having at least two components and such that all its components are smooth and intersecting with normal crossings $($i.e.~$\Delta$ is an SNC divisor$)$. If each rational irreducible component $C$ of $\Delta$ intersects $\overline{\Delta\setminus C}$ in at least two points, then every connected algebraic subgroup $G\subseteq  \Autz(S, \Delta)=\{g\in \Autz(S)\mid g(\Delta)=\Delta\}$ is a torus of dimension $\le 2$.
\end{proposition}
\begin{proof}Since $\Autz(S)$ is a linear algebraic group (Lemma~\ref{lemma:linear}), so is $\Autz(S, \Delta)$.
For each irreducible component $C\subseteq   \Delta$, one gets a homomorphism of algebraic groups $\rho_C\colon G\to \Aut(C)$. Suppose that there exists a component $C$ such that $\rho_C$ is injective. If $C$ is rational, then $\rho(G)$ is a subgroup of $\Aut(C)=\Aut(\p^1)$ that fixes at least two points by hypothesis, and thus is a torus of dimension $\le 1$. If $g(C)\ge 1$, then $G$ is trivial, as the identity is the only connected linear algebraic subgroup of $\Aut(C)$.

One can thus assume that each irreducible component $C$ of $\Delta$ is pointwise fixed by a non-trivial element of $G$. If $S=\p^2$, this implies that all components of $\Delta$ are lines {(the fixed locus of any non-trivial element of $\Aut(\p^2)\simeq \PGL_3$ is a union of points and lines)}, and the hypothesis implies that $\Delta$ contains the union of three lines not having one common point. This implies that $G$ is a torus of dimension $\le 2$. {If $S\not=\p^2$, there exists a birational morphism $\eta\colon S\to \F_n$ (contract the $(-1)$-curves until reaching a minimal surface, and stop at $\F_1$ if you end up at $\p^2$).} This morphism is $\Autz(S)$-equivariant and thus $G$-equivariant.

The group $G$ acts then faithfully on $\F_n$, preserving $\eta(\Delta)$. Denoting by $\pi\colon \F_n\to \p^1$ a $\p^1$-bundle {structure}, the action of $G$ yields an exact sequence $1\to G'\to G \to G''\to 1$ where $G''\subseteq   \Aut(\p^1)$ and $G'\subseteq   \Autz(\F_n)_{\p^1}$. It remains to show that {both $G'$ and $G''$} are tori of dimension $\le 1$. Denote by $C\subseteq   \Delta$ an irreducible component such that $\eta(C)$ is not contained in a fibre (which always exists because of the hypothesis on $\Delta$). The action of $G$ on $C$ fixes at least two points by hypothesis. If these are sent onto two different points of $\p^1$, then $G''$ fixes at least two points of $\p^1$ and thus is a torus. Otherwise, $\eta(C)$ is not a section and  $G''$ fixes the points of $\p^1$ corresponding to the ramification of $\pi\circ \eta\colon C\to \p^1$, so $G''$ is a torus in both cases. It remains to see that $G'$ is a torus. As $\eta(C)\subseteq  \F_n$ is not contained in a fibre and $G'$ preserves each fibre, $G'$ pointwise fixes $\eta(C)$. Hence, the only possibility to study is when $\eta(C)$ is a section. But then the hypothesis on $\Delta$ implies that we get another irreducible component $D\subseteq   \Delta$ such that $\eta(D)$ is not contained in a fibre. This implies that $G'$ fixes at least two points on a general fibre of $\F_n\to \p^1$, which implies that $G'$ is a torus of dimension $\le 1$. {Indeed, if $n=0$ then $G'\subset  \Autz(\F_n)_{\p^1}\simeq\PGL_2$ and if $n\ge 1$ then $G'\subset  \Autz(\F_n)_{\p^1}\simeq \G_m\rtimes \k[x_0,x_1]_n$.}
\end{proof}

\begin{remark}
In Proposition~\ref{conic_horizontal} the two-dimensional tori do appear by taking for $S$ any smooth projective toric surface and choosing for $\Delta$ the complement of the two-dimensional torus.
\end{remark}

\subsection{Proof of Theorem~\ref{Thm:conic bundles}} \label{sec:proof of Prop B}
First, using Theorem \ref{th sarkisov_standard_conic}, we may assume that $\pi\colon X \to S$ is a standard conic bundle. In particular, the discriminant divisor $\Delta$ is an SNC divisor. We then consider two cases, depending on the generic fibre of $\pi$:

\ref{Thm:conic bundlesCaseP1b}: If the generic fibre is isomorphic to $\p^1_{\k(S)}$, the morphism $\pi\colon X\to S$ is a $\p^1$-bundle (Lemma~\ref{2_intersections}). We reduce to the case where $S$ is smooth by applying \cite[Lemma~2.2.1]{BFT}, and  then we contract all $(-1)$-curves on $S$ to obtain a surface with no $(-1)$-curve (follows from the ``Descent lemma", see \cite[Lemma~2.3.2]{BFT}).

\ref{Thm:conic bundlesCaseTorus}: If the generic fibre is not isomorphic to $\p_{\k(S)}^{1}$, we keep the same notation as in \upshape\S \ref{Mori fibrations} and consider the short exact sequence (\ref{exact sequence}).  Proposition \ref{vertical_conic} yields that the group $\Autz(X)_S$ is finite, isomorphic to $(\Z/2\Z)^r$ for some $r\in \{0,1,2\}$. 

\ref{Thm:conic bundlesCaseTorus}\ref{Thm:conic bundlesCaseTorusSrat}: We now suppose that $S$ is rational (which is always true if $\car(\k)=0$, as $X$ is rationally connected, so $S$ is rationally connected, hence rational  \cite[Theorem~13.27]{Bad01}). By Lemma~\ref{2_intersections}, the discriminant curve $\Delta$ is not empty. If $\Delta$ is irreducible, then $g(\Delta)\ge 1$ (Lemma~\ref{2_intersections}), so the action of $H$ gives an injective homomorphism $H\hookrightarrow \Aut(\Delta)$, as the set of fixed points of every non-trivial element of $\Autz(S)$ is the union of finitely many points and rational curves. This implies that $H$ (and thus $\Autz(X)$) is trivial, as $H\subseteq \Autz(S)$ is a linear algebraic group. If $\Delta$ is not irreducible, then each rational irreducible component $C$ of $\Delta$ intersects the complement $\overline{\Delta\setminus C}$ into at least two distinct points (Lemma~\ref{2_intersections}). Proposition \ref{conic_horizontal} implies that $H$ is a torus of dimension $\le 2$. By \cite[IV.11.14, Corollary~1]{Bor91}, the group $G=\Autz(X)$ contains a subtorus of dimension $\dim(H)=\dim(G)$, hence $G$ is a torus of dimension $\le 2$. This achieves the proof of \ref{Thm:conic bundlesCaseTorus}\ref{Thm:conic bundlesCaseTorusSrat}.

\ref{Thm:conic bundlesCaseTorus}\ref{Thm:conic bundlesCaseTorusXrat}: If $X$ is rational, then $S$ is rational too (Proposition~\ref{Prop:RatRat}\ref{BaseMfsRatisRat}), so $G$ is a torus of dimension $\le 2$ by \ref{Thm:conic bundlesCaseTorus}\ref{Thm:conic bundlesCaseTorusSrat}. By Corollary \ref{tori conjugate} the group $G$ is conjugated to a subtorus of $\Aut(\p^3)$, that is, there is a $G$-equivariant birational map $\varphi\colon X \dashrightarrow \P^3$ such that $\varphi G \varphi^{-1} \subsetneq \Aut(\p^3)=\PGL_4$.
\qed

\section{Mori del Pezzo fibrations over \texorpdfstring{$\p^1$}{P1}} \label{Sec:dP}
Our main goal in this section is to prove Theorem~\ref{Thm:MainQuadric} and then to deduce Theorem~\ref{th:A}. This will be done in \upshape\S~\ref{proof of thA}.

\subsection{Some generalities} \label{subsec generalities}
{As before we fix once and for all an algebraically closed field $\k$.}
We recall that a smooth del Pezzo surface defined {over $\k$} is isomorphic to $\p^2$ (degree $9$), $\p^1\times\p^1$ (degree $8$) or to the blow-up of a set of $r$ points in $\p^2$, with $1 \leq r \leq 8$, in general position (degree $9-r$); see for instance \cite[Corollary~8.1.14]{Dol12}.
We can associate a degree $d\in \{1,\dots,9\}$ with any del Pezzo fibration, defined as the degree of the (geometric) generic fibre and this degree coincides with the degree of a general fibre.

We also recall that a \emph{Mori del Pezzo fibration} over $\P^1$ is a Mori fibration $\pi\colon X \to \p^1$ whose general fibres are del Pezzo surfaces. We remark that the generic fibre of $\pi$ is a \emph{regular} del Pezzo surface over $\k(\p^1)$ \cite[Lemma~15.1]{FS19}, which is not necessarily smooth if $\car(\k)>0$ by \cite[Theorem~14.8]{FS19}.

\begin{lemma}\label{Lem:WeakDP}
Let $S$ be a normal Gorenstein del Pezzo surface of degree $d$ {defined over $\k$, and let $\tau\colon \hat{S} \to S$ be a minimal resolution. Then, either $\hat{S}$ is  isomorphic to a Hirzebruch surface $\F_0$ or $\F_2$, or there is a birational morphism $\hat{S} \to \P^2$ which is the composition of $0\le 9-d\le 8$ smooth blow-ups. 
Moreover, every curve of $\hat{S}$ contracted by $\tau$ is a $(-2)$-curve.}
\end{lemma}
\begin{proof}
Follows from \cite[Proposition~8.1.10 and Theorem~8.1.15]{Dol12}.
\end{proof}
\begin{lemma}\label{Lem:bound_char_dP}
Let $\pi\colon X \to \p^1$ be a Mori del Pezzo fibration {over $\k$}, let $K=\k(\p^1)$, let $X_K$ be the generic fibre and let $\overline{K}$ be an algebraic closure of $K$. Then, $X_K$ has a rational point, it is regular, it satisfies $\rho(X_K)=\rk \NS(X_K)=1$, and one of the following conditions holds:
\begin{enumerate}
\item
$X_{\overline{K}}$ is regular $($which is equivalent to say that $X_{K}$ is smooth$)$;
\item
$\car(\k)=7$, the surface $X_{\overline{K}}$ is a del Pezzo surface of degree $d\in \{1,2\}$ with exactly one isolated singularity, of type $A_6$, and $\rho(X_{\overline{K}})=4-d>1$; or
\item\label{carsmaller7}
$\car(\k)\in \{2,3,5\}$.
\end{enumerate}
\end{lemma}

\begin{proof}
The existence of a rational point for $X_K$ is a consequence of \cite[Theorem~1.2]{GHS03} in characteristic zero and of \cite[Theorem]{dJS03} in positive characteristic.
As explained before, the regularity of $X_K$ follows from \cite[Lemme~15.1]{FS19}. By definition of a Mori fibration, the relative Picard rank is $1$, so $\rho(X_K)=1$. If $X_{\overline{K}}$ is again regular (or equivalently $X_{K}$ is smooth, since smoothness is equivalent to geometric regularity in our setting; see \cite[Tag~038X]{stacks-project}), we are done. We can then assume that  $X_{\overline{K}}$ is a singular del Pezzo surface, which implies that $K$ is not a perfect field. We can moreover assume that $\car(\k)\ge 7$, otherwise we are in case~\ref{carsmaller7}.

The generic fibre $X_K$ is geometrically normal, see \cite[Theorem~15.2]{FS19} and \cite[Corollary~1.4]{PW_2017}. Moreover, $X_K$ is also Gorenstein, since this property is invariant by base-change \cite[Tag 0C07]{stacks-project} and the only singularities of normal projective Gorenstein surfaces are rational double points or simple elliptic singularities \cite[Theorem~2.2]{HW81}. As $X_{\overline{K}}$ is singular and $\car(\k)\ge 7$, the only possibility is that $\car(\k)=7$ and all singularities are of type $A_6$ (follows for instance from \cite{H04} or from \cite[Theorem~6.1, Remark~ 2.8]{S08}). Each of the singular points gives rise to a chain of six $(-2)$-curves in the minimal resolution $W$ of $X_{\overline{K}}$. Hence, we can only have one such isolated singularity and  $W$ is the blow-up of $9-d\in \{6,7,8\}$ points of $\p^2_{\overline{K}}$, where $d$ is the degree of the del Pezzo fibration by Lemma~\ref{Lem:WeakDP}. Moreover $\rho(X_{\overline{K}})=4-d$, so we only need to show that $d=3$ is impossible. {In this case, $X_{\overline{K}}$ would be a singular del Pezzo surface of degree $3$, hence a cubic surface in $\p^3$, with an $A_6$-singularity.} This does not exist by \cite[Lemma~9.2.4]{Dol12}.
\end{proof}

\begin{definition}
Let $K$ be a field and let $\overline{K}$ be an algebraic closure of $K$.  The \emph{perfect closure} of $K$ is the field $L\subseteq  \overline{K}$, which is equal to $K$ if $\car(\k)=0$ and equal to the field $K^{p^{-\infty}}=\{a\in \overline{K} \mid a^{p^n}\in K\text{ for some }n\ge 1\}$ if $\car(\k)=p>0$. The field $L$ is then a perfect field and the extension $\overline{K}/L$ is separable.
\end{definition}

\begin{lemma}\label{Lemm:PerfectClosure}
{Let $K\subseteq  L$ be a field extension and let $X$ be a proper scheme over $K$. Then the following hold.}
\begin{enumerate}
\item\label{PC1} If $K\subseteq  L$ is purely inseparable $($in particular when $L$ is a perfect closure of~$K)$, then $\rho(X)=\rho(X_L)$.
\item\label{PC2} If $K\subseteq  L$ is {Galois} with Galois group $\G=\Gal(L/K)$, then $\rho(X)=\rk \NS(X_L)^\G$. 
\end{enumerate}
\end{lemma}
\begin{proof}
\ref{PC1} is in \cite[Proposition~2.4(2)]{Tanaka_2015}, while \ref{PC2} is in  \cite[Chapter I\!I, Proposition~4.3]{kol96}.
\end{proof}

{
\begin{remark}
In the following, we will write $\rk \NS(X_L)^\G=\rho(X_L)^\G$ (with the notation of Lemma~\ref{Lemm:PerfectClosure}) to shorten the notation, even if this does not refer to the fixed part of $\rho(X_L)$, as this latter is a number.
\end{remark}
}

\subsection{Mori del Pezzo fibrations of small degree}\label{red_delPezzo}
In this subsection we prove that if $X \to \P^1$ is a Mori del Pezzo fibration of degree $\leq 6$, then $\Autz(X)$ is a torus of dimension $\leq 3$ (Proposition~\ref{cor Aut dP are tori}).

The next result is classical in Mori theory (see e.g.~\cite[Theorem 3.5]{Mor82} for the smooth case and  \cite{CFST16,CFST18} for the case of terminal singularities, when $\k=\mathbb{C}$). We recall the proof due to a lack of a precise reference.

\begin{lemma}\label{lem:MoriQuad8}
Assume that $\car(\k)\notin \{2,3,5\}$ and let $\pi\colon X\to \p^1$ be a Mori del Pezzo fibration of degree $d$. Then either $d\in \{1,2,3,4,5,6,9\}$ or $d=8$ and $\pi$ is a Mori quadric fibration $($i.e.~a Mori del Pezzo fibration whose generic fibre is a smooth quadric$)$.
\end{lemma}

\begin{proof}
Let $K=\k(\p^1)$, let $X_K \to \Spec(K)$ be the generic fibre of $\pi$, let $\overline{K}$ be an algebraic closure of $K$, let $L\subseteq  \overline{K}$ be the perfect closure of $K$, and let $\G=\Gal(\overline{K}/L)$ be the associated Galois group. We have $\rho(X_K)=1$ (Lemma~\ref{Lem:bound_char_dP}), which implies that $\rho(X_L)^\G=1$ by Lemma~\ref{Lemm:PerfectClosure}. If $d\not\in \{7,8\}$, the result holds, so we may assume that $d\in \{7,8\}$, which implies that $X_{\overline K}$ is smooth (Lemma~\ref{Lem:bound_char_dP}).

If $d=8$, then either  $X_{\overline K}$ is isomorphic to $\p^1_{\overline K}\times\p^1_{\overline K}$ or to the blow-up of a point in $\p^2_{\overline K}$. Let us observe that this latter case is impossible. Indeed, the group $\NS(X_{\overline K})=\Z^2$ is generated by the unique $(-1)$--curve  $E$, obtained by blowing-up a point in $\p^2$, and the pull-back of any line $\ell$ not intersecting $E$. As $\G$ preserves $E$ and the canonical class $-3\ell+E$, we have $\rk \NS(X_{\overline K})^\G=2$.

If $d=7$, then $X_{\overline K}$ is isomorphic to the blow-up of two distinct points $p_1,p_2\in \p^2_{\overline K}$ and so $\NS(X_{\overline K})=\Z^3$ is generated by the two $(-1)$--curves $E_1$ and $E_2$ contracted on  $p_1$ and $p_2\in\p^2$, together with the $(-1)$--curve $\ell$ defined as the strict transform of the line through $p_1$ and $p_2$. These are the only three $(-1)$--curves, and as $\ell$ is the only $(-1)$--curve intersecting the two others, it must be $\G$-invariant, and so is the set $\{E_1,E_2\}$. This implies that $\rho(X_L)^\G \geq 2$, which is not possible.
\end{proof}

\begin{lemma} \label{two_sing_fib}
Suppose that $\car(\k)\notin\{2,3,5\}$.
Let $\pi\colon X\to \p^1$ be a del Pezzo fibration of degree $d\le8$ and let $H\subseteq   \Aut(\p^1)$ be the image of the natural homomorphism $\Autz(X)\to \Aut(\p^1)$. Then $H$ is a torus of dimension at most $1$.
\end{lemma}
\begin{proof}
Let $K=\k(\p^1)$ and let $X_K$ be the generic fibre, which is a regular del Pezzo surface with $\rho(X_K)=1$  (Lemma~\ref{Lem:bound_char_dP}). Denote by $\overline{K}$ an algebraic closure of $K$ and by $L\subseteq  \overline{K}$ the perfect closure of $K$. We have $\rho(X_L)=1$ by Lemma~\ref{Lemm:PerfectClosure}. We now observe that $\rho(X_{\overline{K}})>1$. If $X_{\overline{K}}$ is smooth, this is because $\rho(X_{\overline{K}})=10-d$. If $X_{\overline{K}}$ is singular, then $\car(\k)=7$ and $\rho(X_{\overline{K}})\in \{2,3\}$ by Lemma~\ref{Lem:bound_char_dP}.

We then denote by $\hat{L}\subseteq  \overline{K}$ the splitting field of $L$, that is, the intersection of all subfields containing $L$ and such that all extremal rays of $\NE(X_{\overline{K}})$ are defined over $\hat{L}$ (or equivalently the minimal field $L\subseteq  \hat{L}\subseteq  \overline{K}$ such that $\rho(X_{\hat{L}})=\rho(X_{\overline{K}}))$. Replacing $\hat{L}$ with its normal closure, we may assume that $\hat{L}/L$ is a Galois extension. We then denote by $\G$ the Galois group $\Gal(\hat{L}/L)$, which is non-trivial, since $\rho(X_{\hat{L}})=\rho(X_{\overline{K}})>1$, $\rho(X_L)=1$ and $\rho(X_L)=\rho(X_{\hat{L}})^{\G}$ by Lemma~\ref{Lemm:PerfectClosure}. 

This implies that the (unique) cover $\tau\colon C \to \p^1$ associated with the field extension $\hat{L}/L$ is non-trivial. Moreover the branch locus of $\tau$ is preserved by $H$. As a consequence of Hurwitz's formula \cite[Corollary IV.2.4]{Har77}, we deduce that $H$ preserves at least two points of $\p^1$. This finishes the proof.
\end{proof}

\begin{proposition}  \label{cor Aut dP are tori}
Assume that $\car(\k)\notin \{2,3,5\}$. If $\pi\colon X\to \p^1$ is a Mori del Pezzo fibration of degree $d\leq 5$ $($resp.~$d=6)$, then $G=\Autz(X)$ is a torus of dimension $\le 1$ $($resp.~$\le 3)$. 
\end{proposition}

\begin{proof}
Let $K=k(\p^1)$, and let  $\overline{K}$ be an algebraic closure of $K$.  By Remark \ref{rk: inclusion into aut group of generic fibre}, there is an injective group homomorphism $G_0:=\Autz(X)_{\p^1} \hookrightarrow \Aut(X_K)$, where $X_K$ is the generic fibre of $\pi\colon X \to \p^1$. Also, there is an injective group homomorphism $\Aut(X_K) \hookrightarrow \Aut(X_{\overline{K}})$.

If $d \leq 5$, we will show that $\Aut(X_{\overline{K}})$ is a finite group. This will imply that $G_0$ is finite and thus that $\Autz(X)$ is an extension of a finite group with a torus of dimension $\le 1$ (Lemma~\ref{two_sing_fib}), and so that $\Autz(X)$ is a torus of dimension $\le 1$. {We now prove that $\Aut(X_{\overline{K}})$ is finite if $d\le 5$. To do so, we distinguish between two cases, depending on whether $X_{\overline{K}}$ is smooth or not.} 

If $X_{\overline{K}}$ is smooth, then it is isomorphic to the blow-up of $9-d\ge 4$ points of $\p^{2}_{\overline{K}}$ (Lemma~\ref{Lem:WeakDP}).
{The group $\Aut(X_{\overline{K}})$ acts on the finite set of $(-1)$-curves of $X_{\overline{K}}$. The kernel $H\subseteq \Aut(X_{\overline{K}})$ of this action is the lift of the group of automorphisms of $\p^{2}_{\overline{K}}$ that fix the $9-d\ge 4$ points blown-up. As $X_{\overline{K}}$ is a del Pezzo surface, no three of the points are collinear, so $H$ is trivial and $\Aut(X_{\overline{K}})$ is finite.}

If $X_{\overline{K}}$ is not smooth, then $d=2$ and
$X_{\overline{K}}$ has a unique singular point, which is of type $A_6$ (Lemma~\ref{Lem:bound_char_dP}). Note that $X_{\overline{K}}$ is a double cover of $\p^2_{\overline{K}}$ ramified over a quartic curve $\Gamma$ (see \cite[\S8.7]{Dol12}). The group $\Aut(X_{\overline{K}})$ then acts regularly on $\p^2$ and preserves $\Gamma$. Moreover, $\Gamma$ has a unique singular point, which is of type $A_6$, and thus an ``oscular rhamphoid cusp'', which means that the minimal embedded resolution of the singularity $Y\to \p^2_{\overline{K}}$ is given by blowing-up three infinitely near points $p_1,p_2,p_3$, and that the strict transform $\tilde{\Gamma}$ is tangent to the divisor $E_3$ associated with $p_3$ with multiplicity $2$, and does not intersect {$E_1$ or $E_2$}. Choosing coordinates such that $p_1=[0:0:1]$, and that $p_2$ and $p_3$ lie on the conic $yz+x^2$, the equation of $\Gamma$ is given by the homogeneous polynomial $F=(x^2+yz)^2+\lambda y^2(x^2+yz)+\mu xy^3+ \nu y^4\in \overline{K}[x,y,z]$ for some $\lambda,\mu,\nu\in \overline{K}$. We moreover have $\mu\not=0$, as otherwise $F\in \overline{K}[x^2+yz,y^2]$ {cannot be irreducible}, and thus the singularity cannot be of type $A_6$. Replacing $z$ with $z-\frac{\lambda}{2}y$, we may assume that $\lambda=0$. We then replace $x$ and $z$ with $x+\xi y$ and $z-\xi^2y-2\xi x$ respectively, where $\xi=-\frac{\nu}{\mu}$, and may assume that $\nu=0$. Finally, replacing $y$ and $z$ with $y/\kappa$ and $z\kappa$, where $\kappa^3=\mu$, we may assume that the equation of $\Gamma$ is given by $(x^2+yz)^2+xy^3$. The group of automorphisms of $\p^2$ that preserves the quartic is then the finite group given by
$\{[x:y:z]\mapsto [\omega x:\omega^2y:z] \mid \omega^3=1\}$. {This implies that $\Aut(X_{\overline{K}})$ is finite.}

If $d=6$, then $X_{\overline{K}}$ is smooth (by Lemma~\ref{Lem:bound_char_dP}) and is thus isomorphic to the blow-up of three non-collinear points of $\p^2$.  We take a finite field extension $K\subseteq  L$ such that the six $(-1)$-curves of $X_{\overline{K}}$ are defined over $L$. We then denote by $C\to \p^1$ the finite morphism corresponding to the field extension $L/K$. Then, the $G_0$-action on $X$ lifts to a $G_0$-action on  $Y:=X\times_{\p^1} C$ by letting $G_0$ act trivially on $C$. The generic fibre of $Y\to C$ is now a del Pezzo surface of degree $6$ with all six $(-1)$-curves defined over $\k(C)$. Hence, $Y$ is birational to $S\times C$, where $S$ is the blow-up of three general points of $\p^2$ and $G_0$ identifies with an algebraic subgroup of $\Aut(S) \simeq \Aut(S \times C)_C$. As $\Aut(S)=\Gm^2 \rtimes D_6$ \cite[Theorem 8.4.2]{Dol12}, we conclude that the neutral component of $G_0$ is a torus $T$ of dimension $\le 2$.

Let us recall a classical fact: if there is a non-constant morphism from a one-dimensional connected linear algebraic group $J$ to $\Gm$, then $J \simeq \Gm$. Indeed, the other one-dimensional connected linear algebraic group is the additive group $\G_a$, and there is no non-trivial morphism $\G_a \to \G_m$. 

 If we mod out $G_0$ and $G$ by $T$ in the exact sequence \eqref{exact sequence} (in \upshape\S~\ref{Mori fibrations}), we see that $G/T$ is an extension of a finite group with a torus of dimension $\le 1$. By the previous paragraph, $G/T$ is $\Gm$ or the trivial group. But there are no non-trivial extensions between algebraic tori (consequence of \cite[\upshape\S 11.14, Corollary 1]{Bor91}), and so $G$ must be a torus of dimension $\le 3$. 
\end{proof}

\subsection{\texorpdfstring{$\p^2$}{P2}-fibrations over \texorpdfstring{$\p^1$}{P1}} \label{subsec:P2 fibrations}
Let $\pi\colon X \to \p^1$ be a Mori del Pezzo fibration of degree~9, that is, a \emph{$\p^2$-fibration} with terminal singularities, and let $G=\Autz(X)$. In this subsection we prove that there is a $\p^2$-bundle $\tau:Y \to \p^1$ and a $G$-equivariant commutative diagram
\[\xymatrix@R=3mm@C=2cm{
    X \ar@{-->}[rr]^{\varphi} \ar[rd]_{\pi}  & &Y \ar[ld]^{\tau} \\
    & \p^1&
  }\]
where $\varphi$ is a $G$-equivariant birational map (Proposition \ref{prop: main result P2 fibrations}).

\begin{lemma}\label{Lemm:Tsen}
Assume that $\car(\k)\notin \{2,3,5\}$ and let $\pi\colon X\to \p^1$ be a Mori del Pezzo fibration of degree $9$. Then the generic fibre $X_K$ is isomorphic to $\p^2_K$, where $K=\k(\p^1)$.
\end{lemma}

\begin{proof}
Let $K=\k(\p^1)$, let $X_K$ be the generic fibre, and let $\overline{K}$ be an algebraic closure. The surface $X_{\overline{K}}$ is smooth (Lemma~\ref{Lem:bound_char_dP}). Hence, by Lemma~\ref{Lem:WeakDP}, the surface $X_{\overline{K}}$ is isomorphic to $\p^2_{\overline{K}}$. To show that $X_K$ is isomorphic to $\p^2_K$, it suffices to show that the Brauer group of $K$ is trivial (see \cite[Theorem~5.2.1]{GS06}). This follows from Tsen's theorem  as $K=\k(\p^1)$ (see \cite[Tag 03RF]{stacks-project}).
\end{proof}

With Lemma~\ref{Lemm:Tsen}, it is then natural to ask whether a Mori del Pezzo fibration $X \to \P^1$ whose generic fibre is isomorphic to $\P^2$ is necessary a $\P^2$-bundle. The answer is unfortunately negative as the next example shows. But we will see with Proposition~\ref{prop: main result P2 fibrations} that we can always reduce to the case of $\P^2$-bundles.

\begin{example}Assume that $\car(\k)\not=2$.
Let $\sigma\in \Aut(\p^2)$ be an involution and let $C$ be a smooth projective curve with a $\mu_2$-action such that $C/\mu_2=\p^1$. Let $X=(\p^2\times C)/\mu_2$, where $\mu_2=\{\pm 1\}$ acts on $\p^2$ via the involution $\sigma$. Then the induced morphism $\pi\colon X \to \p^1$ is a Mori del Pezzo fibration of degree $9$ (the only singularities of $X$ are double points). But $\pi$ is not a $\p^2$-bundle as $\mu_2$ acts on $C$ with at least two fixed points (by Hurwitz's formula) and over a fixed point a fibre of $\pi$ is generically non-reduced.
\end{example}

The next elementary lemma will be useful in the proof of Proposition~\ref{prop: main result P2 fibrations}.

\begin{lemma}\label{Lem:TrivialToZn}
Let $G$ be a connected linear algebraic group. For each integer $n\ge 2$ that is not a multiple of the characteristic of the ground field $\k$, every $($abstract$)$ group homomorphism $\nu\colon G\to \Z/n\Z$ is trivial.
\end{lemma}
\begin{proof}
As $G$ is generated by its subtori and unipotent subgroups, we only need to show the result in the two cases where $G$ is isomorphic to $\G_m$ or $\G_a$. 

Let $G \simeq \G_m$. As the base field is algebraically closed, for any $t \in G$ we can find $t' \in G$ such that ${t'}^n=t$. Then $\nu(t)=\nu(t'^n)=n\nu(t')=0$, and so $\nu$ is trivial.
 
Let $G \simeq \G_a$. As $\car(\k)$ does not divide $n$, for any $u \in G$ we can find $u' \in G$ such that $nu'=u$. Then $\nu(u)=\nu(nu')=n\nu(u')=0$, and so $\nu$ is again trivial.
\end{proof}

\begin{remark}
If $\car(\k)=p$, then Zorn's lemma provides a basis $\mathcal{B}$ of $\k$ as a $\F_p$-vector space. This  implies that $\G_a \simeq \bigoplus_{i \in \mathcal{B}} \Z/p\Z$ as an abstract group, and so there are non-trivial homomorphisms $\G_a \to \Z/p\Z$. {Hence, Lemma~\ref{Lem:TrivialToZn} is false if $p=n$.}
\end{remark}

\begin{proposition}  \label{prop: main result P2 fibrations}
Assume that $\car(\k)\not=3$. Let $\pi \colon X \to \p^1$ be a morphism whose generic fibre is isomorphic to $\p^2$ $($this is for instance the case when $\car(\k) \notin \{2,3,5\}$ and $\pi$ is Mori del Pezzo fibration of degree $9$ by Lemma~$\ref{Lemm:Tsen})$. There is a regular action of $\Autz(X)$ on a $\p^2$-bundle $\tau\colon Y\to \p^1$, and an $\Autz(X)$-equivariant birational map $\varphi\colon X\dasharrow Y$ such that $\tau\circ \varphi=\pi$.
\end{proposition}

\begin{proof}There is an open subset $V\subseteq   \p^1$ over which $\pi$ is a $\p^2$-bundle. Let us write $G=\Autz(X)$. If $V=\p^1$, we are done. Otherwise, we take a point $p\in \p^1\setminus V$, necessarily fixed by $G$ (as $G$ is connected and $\p^1\setminus V$ is finite), and take affine coordinates in a neighbourhood of $p$ such that $p$ has for equation $t=0$ in $\A^1$. We use a local trivialisation on a open subset of $\p^1$, and obtain a rational action   of $G$ onto $\A^1\times\p^2$ (a group homomorphism $G\to \Bir(\A^1\times \p^2$) such that the corresponding map $G\times \A^1\times \p^2\dasharrow \A^1\times \p^2$ is rational), given by
\[\begin{array}{ccc}
G\times\A^1\times\p^2 &\dasharrow &\A^1\times\p^2\\
(g, t,[x:y:z])&\mapsto& \left(\frac{a(g)t}{1+b(g)t},\mathcal{M}(g,t)([x:y:z]),\right),
\end{array}\]
where $a\in \k[G]^*$, $b\in \k[G]$ and $\mathcal{M}\in \PGL_{3}(\k[G](t))$.

For each $g\in G$, we then obtain an element $\mathcal{M}_g\in \PGL_{3}(\k(t))$, represented by a matrix $M_g\in \mathrm{Mat}_{3}(\k(t))$, that we may assume to be in $\mathrm{Mat}_{3}(\k[t])$ and such that $M_g(0)=M_g|_{t=0}\not=0$. We then denote by $n_g\in \N$ the biggest integer such that $t^{n_g}$ divides $\det(M_g)$. If $n_g=0$ for each $g\in G$, then the rational action of $G$ on $\A^1\times\p^2$ is regular on a neighbourhood of $ \{0\}\times\p^2$, so we can replace $X$ with another projective variety $X'$ for which the open set $V$ is replaced with $V\cup \{p\}$. We will show that we can reduce to this case by conjugating the action with an element of $\PGL_3(\k(t))$ coming from a matrix in $\mathrm{Mat}_3(\k[t])$ with determinant having zeroes only at $t=0$. After  finitely many such steps, we reach the case $V=\p^1$ which achieves the proof.

We now study the case where $n_g>0$ for some $g\in G$. 
For each $g\in G$, we define two matrices $R_g,N_g\in \mathrm{Mat}_{3}(\k)$ such that 
\[M_g \equiv R_g+tN_g\pmod{t^2},\]
and obtain $R_g=M_g(0)\not=0$.

 Since the action of $G$ on $\A^1\times\p^2$ has to satisfy the axioms of a group action, there is, for all $g,h\in G$, an element $\nu_{g,h}\in \k(t)$ such that the following equality holds in $\mathrm{Mat}_{3}(\k(t)):$
 \vspace{-3mm}
\begin{equation}\label{EqProd}
\nu_{g,h}\cdot M_{gh}(t)=M_g(\frac{a(h)t}{1+b(h)t})\cdot M_h(t). \end{equation}
In particular, taking determinants, the element 
\begin{equation}\label{EqProdVal}
t^{n_{gh}-n_g-n_h}(\nu_{g,h})^3\in \k(t)\text{ is invertible at }t=0.\end{equation}

We define $n_{\max}=\max\{n_g\mid g\in G\}\ge 1$ (which is finite because $G\times \A^1\times \p^2\dasharrow \A^1\times \p^2$ is rational), and define subsets of $G$ as follows: 
\[G_{\max}=\{g\in G\mid n_g= n_{\max}\}\subseteq   G_{+}=\{g\in G\mid n_g>\frac{1}{2} n_{\max}\}\subseteq   G_{>0}=\{g\in G\mid n_g>0\}\subseteq   G.\] 
We then observe that
\begin{equation}\label{RgRh}
R_gR_h=0\text{ for all }g,h\in G \text{ such that }n_g+n_h>n_{\max} \text{ (in particular for all }g,h\in G_{+}).\end{equation}
Indeed, the inequality $n_g+n_h>n_{\max}$ implies that $n_g+n_h> n_{gh}$, which implies that $\nu_{g,h}\in \k(t)$ vanishes at $t=0$ \eqref{EqProdVal}, and thus that $M_g(0)M_h(0)=R_g R_h=0$ \eqref{EqProd}.

We now fix an element $s\in G_{\max}$. By \eqref{RgRh} we have $R_s^2=0$. After conjugating all $R_g,g\in G$, by a common element of $\GL_3(\k)$ (this corresponds to conjugate the action by an element of $\PGL_3(\k)$), we can then assume that 
\begin{equation}\label{Rs}R_s=\begin{bmatrix}
0 & 0 & 1\\
0&0&0\\
0&0&0\end{bmatrix}.\end{equation}

We now prove that 
\begin{equation}\label{Rgbigger0}
R_g=\begin{bmatrix} 0 & \star & \star\\
0 & 0& \star\\
0&0&0\end{bmatrix}\text{ for each }g\in G_{>0}.\end{equation} 
(This means that the non-zero coefficients of the matrix $R_g=M_g(0)$ are above the diagonal). To prove this, we take $g\in G_{>0}$ and observe that $n_g+n_s>n_s=n_{\max}$, which implies that  $R_sR_g=R_gR_s=0$. The matrix $R_g$ has thus a first column and a last row equal to zero. It remains to see that the middle coefficient of $R_g$ is equal to zero. 

Suppose by contradiction that some $g\in G_{>0}$ satisfies $R_g=\begin{bmatrix} 0 & \star & \star\\
0 & \mu& \star\\
0&0&0\end{bmatrix}$ for some $\mu\in \k^*$, and choose $g$ with this condition such that $n_g$ is maximal. As $(R_g)^2=\begin{bmatrix} 0 & \star & \star\\
0 & \mu^2& \star\\
0&0&0\end{bmatrix}\not=0$, we get $(R_g)^2=\lambda R_{g^2}$ for some $\lambda\in \k^*$ and $n_{g^2}=2n_g>n_g$, contradicting the maximality assumption. This achieves the proof of \eqref{Rgbigger0}.

We now prove that, {after conjugating all $R_g,g\in G$ with a fixed element of $\GL_3(\k)$}, one of the following two situations holds.
\begin{align}
\text{Either}\ 
  R_g=\begin{bmatrix} 0 & \star & \star\\
0 & 0& 0\\
0&0&0\end{bmatrix}, \forall g\in G_{+},&\text{ and }R_h=\begin{bmatrix} \star & \star & \star\\
0 & \star& \star\\
0&\star&\star\end{bmatrix},\forall h\in G;\text{ or} \label{CaseA0} \tag{$A_0$}\\
    R_g=\begin{bmatrix} 0 & 0 & \star\\
0 & 0& \star\\
0&0&0\end{bmatrix},\forall g\in G_{+},&\text{ and }R_h=\begin{bmatrix} \star & \star & \star\\
\star & \star& \star\\
0&0&\star\end{bmatrix},\forall h\in G. \label{CaseB0} \tag{$B_0$}
\end{align}
To do this, we first recall that $R_gR_h=0$ for all $g,h\in G_{+}$ \eqref{RgRh}, which means that $\Im(R_g)\subseteq   \Ker(R_h)$. This being true for all $g,h\in G_{+}$, we get 
\[\langle \begin{bmatrix} 1 \\ 0 \\ 0 \end{bmatrix}\rangle\subseteq\Im(R_s)=\bigcup_{g\in G_{+}}\Im(R_g)\subseteq  \bigcap_{g\in G_{+}}\Ker(R_g)\subseteq   \Ker(R_s)=\langle \begin{bmatrix} 1 \\ 0 \\ 0 \end{bmatrix},\begin{bmatrix} 0 \\ 1 \\ 0 \end{bmatrix}\rangle.\] We then denote by $V\subseteq   \k^3$ the linear subspace $V=\bigcap_{g\in G_{+}}\Ker(R_g)$, and obtain two possibilities for $V$.

Either $V=\langle \begin{bmatrix} 1 \\ 0 \\ 0 \end{bmatrix}\rangle$, then $R_g=\begin{bmatrix} 0 & \star & \star\\
0 & 0& 0\\
0&0&0\end{bmatrix}$ for all $g\in G_{+}$; or $V=\langle \begin{bmatrix} 1 \\ 0 \\ 0 \end{bmatrix},\begin{bmatrix} 0 \\ 1 \\ 0 \end{bmatrix}\rangle$, then $R_g=\begin{bmatrix} 0 & 0 & \star\\
0 & 0& \star\\
0&0&0\end{bmatrix}$ for all $g\in G_{+}$. This gives the first part of \eqref{CaseA0} or of \eqref{CaseB0}.

To obtain the second part, we take $h\in G$, and prove that it has the desired form. If $n_h>0$, this follows from \eqref{Rgbigger0}. We then assume that $n_h=0$, which means that $\det(M_h)$ is not divisible by $t$, and gives $R_h=M_h(0)=M_h|_{t=0}\in \GL_3(\k)$. This implies that $(M_s(\frac{a(h)t}{1+b(h)t})\cdot M_h(t))|_{t=0}=R_sR_h$ and $(M_h(\frac{a(s)t}{1+b(s)t})\cdot M_s(t))|_{t=0}=R_hR_s$ are not equal to zero, so $t$ does not divide $\nu_{s,h}$ nor $\nu_{h,s}$. In particular, $n_{sh}=n_{hs}=n_s+n_h=n_{\max}$, and so $sh$ and $hs$ are in $G_+$. If the first part of \eqref{CaseA0} holds, the second and third rows of  $R_hR_s=\lambda R_{hs}$ ($\lambda \in \k^*$) are equal to zero, which yields $R_h=\begin{bmatrix} \star & \star & \star\\
0 & \star& \star\\
0&\star&\star\end{bmatrix}$, and gives \eqref{CaseA0}. If the first part of \eqref{CaseB0} holds, the first and second columns of   $R_sR_h=\lambda R_{sh}$  ($\lambda \in \k^*$) are equal to zero, which yields $R_h=\begin{bmatrix} \star & \star & \star\\
\star & \star& \star\\
0&0&\star\end{bmatrix}$, and gives \eqref{CaseB0}.

We now prove that moreover one of the following two cases holds:
\begin{align}
  R_g=\begin{bmatrix} 0 & \star & \star\\
0 & 0& 0\\
0&0&0\end{bmatrix}\forall g\in G_{+}, R_h=\begin{bmatrix} \star & \star & \star\\
0 & \star& \star\\
0&\star&\star\end{bmatrix}\forall h\in G&\text{ and } N_g=\begin{bmatrix} \star & \star & \star \\
0 &\star & \star \\
0 &\star & \star \end{bmatrix}\forall g\in G_{\max};\text{ or} \label{CaseA} \tag{$A$}\\
    R_g=\begin{bmatrix} 0 & 0 & \star\\
0 & 0& \star\\
0&0&0\end{bmatrix}\forall g\in G_{+}, R_h=\begin{bmatrix} \star & \star & \star\\
\star & \star& \star\\
0&0&\star\end{bmatrix}\forall h\in G&\text{ and } N_g=\begin{bmatrix} \star & \star& \star\\
\star & \star& \star\\
0&0&\star \end{bmatrix}\forall g\in G_{\max} \label{CaseB} \tag{$B$}
\end{align}
{These correspond to \eqref{CaseA0} and \eqref{CaseB0}, with additional properties on $N_g$ for $g\in G_{\max}$. To prove these}, we first calculate, for all $g,h\in G_{\max}$ (recall that $R_gR_h=0$), 
\[\begin{array}{rcl}
{\nu_{g,h}\cdot M_{gh}(t)\stackrel{\eqref{EqProd}}{=}}M_g(\frac{a(h)t}{1+b(h)t})\cdot M_h(t)&\equiv &(R_g+\frac{a(h)t}{1+b(h)t} N_g)(R_h+tN_h)\\
&\equiv& (R_g+ ta(h)N_g)(R_h+tN_h) \\ &\equiv &t(a(h)N_gR_h+R_gN_h)\pmod{t^2}\end{array}\]
{and use it to prove that \begin{equation}\label{ProdNGHtriang}
a(h)N_gR_h+R_gN_h=\begin{bmatrix} 0 & \star & \star\\
0 & 0& \star\\
0&0&0\end{bmatrix}\text{ for all }g,h\in G_{\max}.\end{equation}
To prove \eqref{ProdNGHtriang} we may assume that $a(h)N_gR_h+R_gN_h\not=0$.} In this case, $t^{-1}\nu_{g,h}$ is invertible at $t=0$ and $R_{gh}=\lambda(a(h)N_gR_h+R_gN_h)$ for some $\lambda\in \k^*$ \eqref{EqProd}, which implies that $n_{gh}=n_g+n_h-3$ \eqref{EqProdVal}. As $n_g=n_h=n_{\max}$, we have $n_{gh}=2n_{\max}-3\ge 0$. Hence $n_{\max}\ge 2$ and $2n_{\max}-3>0$, so $n_{gh}>0$. {This, together with ~\eqref{Rgbigger0}, gives~\eqref{ProdNGHtriang}.}

The first column of $R_h$ being zero \eqref{Rgbigger0}, the first column of $a(h)N_gR_h$ is zero, so \eqref{ProdNGHtriang} shows that the first column of $R_gN_h$ is zero. Choosing $g=s$, we obtain that the lower-left coefficient of $N_h$ is zero for each $h\in G_{\max}$ (using the explicit form of $R_s$ given in \eqref{Rs}).

We now write \eqref{ProdNGHtriang} with explicit coefficients. We can write\\ $R_g=\begin{bmatrix} 0 & \alpha(g) & \beta(g)\\
0 & 0& \gamma(g)\\
0&0&0\end{bmatrix}$, $R_h=\begin{bmatrix} 0 & \alpha(h) & \beta(h)\\
0 & 0& \gamma(h)\\
0&0&0\end{bmatrix}$ by \eqref{Rgbigger0}, and also\\  $N_g=\begin{bmatrix} \star & \star & \star\\
\lambda(g) & \star& \star\\
0&\theta(g)&\star\end{bmatrix}$, $N_h=\begin{bmatrix} \star & \star & \star\\
\lambda(h) & \star& \star\\
0&\theta(h)&\star\end{bmatrix}$. Then, \eqref{ProdNGHtriang} gives:

\begin{align}\label{ProdNGHtriang2}
a(h)N_gR_h\!+\!R_gN_h &=\begin{bmatrix} \alpha(g)\lambda(h) & \star & \star\\
0 & a(h) \alpha(h)\lambda(g)\!+\!\gamma(g)\theta(h)& \star\\
0&0&a(h)\gamma(h)\theta(g)\end{bmatrix}\\ 
&=\begin{bmatrix} 0 & \star & \star\\ 0 & 0 & \star\\  0 & 0& 0 \end{bmatrix},\forall g,h\in G_{\max}.\end{align}

If there is $g\in G_{\max}$ such that $\alpha(g)\not=0$, then \eqref{ProdNGHtriang2} yields $\lambda(h)=0$ for all $h\in G_{\max}$. As \eqref{CaseB0} is impossible, we are in case \eqref{CaseA0}, and thus in case \eqref{CaseA}.

If there is $h\in G_{\max}$ such that $\gamma(h)\not=0$, then \eqref{ProdNGHtriang2} yields $\theta(g)=0$ for all $g\in G_{\max}$ (since $a(h)\not=0$). As \eqref{CaseA0} is impossible, we are in case \eqref{CaseB0}, and thus in case \eqref{CaseB}.

The only remaining case is when $\alpha(h)=\gamma(h)=0$ for each $h\in G_{\max}$. This implies that $\beta(h)\not=0$ for each $h\in G_{\max}$ (as $R_g\not=0$). Moreover,
\[a(h)N_gR_h+R_gN_h=\begin{bmatrix} 0& \beta(g)\theta(h) & \star\\
0& 0 & a(h) \beta(h)\lambda(g)\\
0&0&0\end{bmatrix}\text{ for all }g,h\in G_{\max}.\]
If $a(h)N_gR_h+R_gN_h=0$, then $\lambda(g)=\theta(h)=0$. If $a(h)N_gR_h+R_gN_h\not=0$, then as we observed before we have $R_{gh}=\lambda(a(h)N_gR_h+R_gN_h)$ for some $\lambda\in \k^*$, and $0\le n_{gh}=2n_{\max}-3\le n_{\max}$. There are then only two possibilities for $n_{\max}$, namely $2$ or $3$. If $n_{\max}=3$, then $gh\in G_{\max}\subseteq   G_{+}$. We then obtain $\lambda(g)=0$ if \eqref{CaseA0} holds, and $\theta(h)=0$ if \eqref{CaseB0} holds. This shows that if $n_{\max}=3$, then \eqref{CaseA} follows from \eqref{CaseA0} and \eqref{CaseB} follows from \eqref{CaseB0}. 

It remains, to prove that either \eqref{CaseA} or \eqref{CaseB} holds, to exclude the case where $n_{\max}=2$. Suppose $n_{\max}=2$ and derive a contradiction. We define a map $\nu\colon G\to \Z/3\Z$ sending $g$ onto the class of $n_{\max}$. We then observe that $\nu$ is a non-trivial group homomorphism (follows from~\eqref{EqProdVal}). This contradicts Lemma~\ref{Lem:TrivialToZn} as $\car(\k)\ne 3$.

To achieve the proof, we conjugate the group by $(t,[x:y:z])\mapsto (t,[tx:y:z])$ in case \eqref{CaseA} and by $(t,[x:y:z])\mapsto (t,[tx:ty:z])$ in case \eqref{CaseB}.

In Case \eqref{CaseA}, this replaces $M_g=\begin{bmatrix} a_{11} & a_{12} & a_{13} \\ a_{21} & a_{22} & a_{23}\\ a_{31} & a_{32} & a_{33}\end{bmatrix}$ with 
$$\begin{bmatrix} t & 0 & 0 \\ 0 & 1 & 0 \\ 0 & 0 &1\end{bmatrix}\begin{bmatrix} a_{11} & a_{12} & a_{13} \\ a_{21} & a_{22} & a_{23}\\ a_{31} & a_{32} & a_{33}\end{bmatrix}\begin{bmatrix} t & 0 & 0 \\ 0 & 1 & 0 \\ 0 & 0 &1\end{bmatrix}^{-1}=\begin{bmatrix} a_{11} & ta_{12} & ta_{13} \\ \frac{1}{t}a_{21} & a_{22} & a_{23}\\ \frac{1}{t}a_{31} & a_{32} & a_{33}\end{bmatrix}$$ 
for each $g\in G$. The new matrix has the same determinant, and has still coefficients in $\k[t]$. In particular, $n_g$ stays fixed, unless the new matrix is divisible by $t$, in which case it decreases. As this latter always hold for every $g\in G_{\max}$ (because of \eqref{CaseA}), the integer $n_{\max}$ is decreased by this process.

Case \eqref{CaseB} is similar. We replace $M_g=\begin{bmatrix} a_{11} & a_{12} & a_{13} \\ a_{21} & a_{22} & a_{23}\\ a_{31} & a_{32} & a_{33}\end{bmatrix}$ with 
$$\begin{bmatrix} 1 & 0 & 0 \\ 0 & 1 & 0 \\ 0 & 0 &\frac{1}{t}\end{bmatrix}\begin{bmatrix} a_{11} & a_{12} & a_{13} \\ a_{21} & a_{22} & a_{23}\\ a_{31} & a_{32} & a_{33}\end{bmatrix}\begin{bmatrix} 1 & 0 & 0 \\ 0 & 1 & 0 \\ 0 & 0 &\frac{1}{t}\end{bmatrix}^{-1}=\begin{bmatrix} a_{11} & a_{12} & ta_{13} \\ a_{21} & a_{22} & ta_{23}\\ \frac{1}{t}a_{31} & \frac{1}{t}a_{32} & a_{33}\end{bmatrix}$$ for each $g\in G$.\end{proof}

\subsection{Quadric fibrations over \texorpdfstring{$\p^1$}{P1}}  \label{quadric fibrations}

In this subsection we study Mori quadric fibrations over $\P^1$. We introduce the \emph{Umemura quadric fibrations} (Definition~\ref{def:QQg} and Lemma~\ref{Qgproperties}). Then we prove a series of lemmas that will be useful to prove Theorem~\ref{Thm:MainQuadric} in the next subsection.

We first recall a  fact about quadric surfaces, probably well-known, which will be useful later on, in order to study the generic fibre of a quadric fibration.

\begin{lemma}\label{quadric1}
Let $K$ be an arbitrary field with $\car (K)\neq 2$, and let $Q\subseteq   \P_K^3$ be a smooth quadric surface. Assume that $Q(K)\neq\emptyset$. 
{
Then there is $\mu \in K^*$ and a linear projective change of coordinates such that}
the quadric surface $Q$ is defined in $\P^3_K$ by the equation
\[x_0^2-x_1x_2-\mu x^2_3=0.\]
Moreover, the following hold:
\begin{enumerate}
\item\label{Qembedding}
If there exists $r\in K$ such that $\mu=r^2$, then $Q$ is isomorphic to $\p^1_K\times\p^1_K$ via
\[
\begin{array}{rccc}
\kappa_r\colon & \P_K^1\times\P_K^1 & \iso & Q\\
 &([u_0:u_1],[v_0:v_1]) &\mapsto & [r (u_0v_1+u_1v_0):2 r u_0v_0:2 r u_1v_1:u_0v_1-u_1v_0]\end{array}\]
\item\label{MuNotsquare}
If $\mu$ is not a square in $K$, then $Q$ is not isomorphic to $\p^1_K\times\p^1_K$. Let $L/K$ be the unique field extension of degree $2$, with $L=K[r]$ for some $r\in L$ with $\mu=r^2$, and let $\iota$ be the generator of $\Gal(L/K) \simeq \Z/2\Z$. 
Let
\[\sigma\colon (x,y)\mapsto (y,x),\ \text{and let}\]
\[H_0= \{(x,y) \mapsto (A(x),\iota(A)(y))\mid A\in \PGL_2(L)\}.\]
Then, the isomorphism $Q_L\iso \p^1_L\times\p^1_L$ given by \ref{Qembedding} conjugates $\Aut_K(Q)$ to the subgroup $H=H_0\rtimes \langle \sigma\rangle\subseteq  \Aut_L(\p^1_L\times\p^1_L)$.
In particular, $\Aut_K(Q)$ is isomorphic to $\PGL_2(L)\rtimes \Z/2\Z$, where the action of $\Z/2\Z$ on $\PGL_2(L)$ is the one induced by $\iota$.
\end{enumerate}
\end{lemma}

\begin{proof}
By hypothesis, $Q$ has a $K$-point, which can be assumed to be $P:=[0:1:0:0]$. Moreover, since $Q$ is smooth, we may assume that the tangent space of $Q$ at $P$ is given by the equation $x_2=0$. 
So $Q$ has the form $x_1x_2+h=0$, for some  homogeneous polynomial $h\in K[x_0,x_2,x_3]$ of degree $2$. Replacing $x_1$ with $x_1+q$ where $q\in K[x_0,x_2,x_3]$ is a homogeneous polynomial of degree $1$, we may assume that $h\in K[x_0,x_3]$.
By completing the square and applying diagonal automorphisms we may assume that $Q$ has the required form. 

\ref{Qembedding} If there exists $r\in K$ such that $\mu=r^2$, we check that $\kappa_r$ is an isomorphism, with inverse given by $x\mapsto (s_1(x),s_2(x))$, where $s_1,s_2\colon Q\to \p^1$ are given by 
\[
\begin{array}{rccc}
s_1\colon & Q& \to &\p^1\\
& [x_0:\cdots:x_3] &\mapsto &\left\{\begin{array}{lll}
\ [x_1:x_0-rx_3] & \text{ if } & (x_1,x_0-rx_3)\not=(0,0)\\
\ [x_0+rx_3:x_2] & \text{ if } & (x_0+rx_3,x_2)\not=(0,0)\end{array}\right.\\
s_2\colon & Q& \mapsto &\p^1\\
& [x_0:\cdots:x_3] &\to &\left\{\begin{array}{lll}
\ [x_1:x_0+rx_3] & \text{ if } & (x_1,x_0+rx_3)\not=(0,0)\\
\ [x_0-rx_3:x_2] & \text{ if } & (x_0-rx_3,x_2)\not=(0,0)\\
\end{array}\right. \end{array}\]

\ref{MuNotsquare} If $\mu$ is not a square, and $L$ is the degree $2$ extension of $K$ such that $L=K[r]$ and $r\in L$ is such that $r^2=\mu$, then $\Gal(L/K)$ is of order $2$, generated by  $\iota\colon L\to L$ sending $\mu$ onto $-\mu$.  The isomorphism $\kappa_r^{-1}\colon  Q_L\iso \p^1_L\times\p^1_L$ of \ref{Qembedding} conjugates the action of $\Gal(L/K)$ on $Q_L$ to the involution
\[ ([u_0:u_1],[v_0:v_1]) \mapsto ([\iota(v_0):\iota(v_1)],[\iota(u_0):\iota(u_1)]).\]
This implies that $\rho(Q)=\rho( \NS(Q_L))^{\Gal(L/K)}=1$ (Lemma~\ref{Lemm:PerfectClosure}). In particular $Q$ is not isomorphic to $\p^1_K\times\p^1_K$. Computing the elements of $\Aut_L(\p^1_L\times\p^1_L)$ that commute with the above involution, we obtain the group $H=H_0\rtimes \langle \sigma\rangle$ described in the statement of the lemma.
\end{proof}

\begin{definition} \label{def:QQg}
Let $n \geq 0$ and let $g\in \k[u_0,u_1]$ be a homogeneous polynomial of degree $2n$. We denote by $\QQ_g$ the projective threefold given by 
\[\{[x_0:x_1:x_2:x_3;u_0:u_1]\in \P(\O_{\p^1}^{\oplus 3}\oplus \O_{\p^1}(n))\mid x_0^2-x_1x_2-g(u_0,u_1)x_3^2=0\}.\]
and we denote by $\pi_g\colon \QQ_g\to \p^1$ the morphism $ [x_0:x_1:x_2:x_3;u_0:u_1]\mapsto  [u_0:u_1].$ \\
Note that $X=\P(\O_{\p^1}^{\oplus 3}\oplus \O_{\p^1}(n))$ is the quotient of $(\A^4\setminus \{0\} )\times (\A^2\setminus \{0\})$ by the action of $\G_m^2$ given by\[\begin{array}{rcc}
{\G_m^2\times (\A^4\setminus \{0\} )\times (\A^2\setminus \{0\})} & \rightarrow & {(\A^4\setminus \{0\} )\times (\A^2\setminus \{0\})}\\
((\lambda,\mu),(x_0,x_1,x_2,x_3,u_0,u_1))&\mapsto& (\mu x_0,\mu x_1,\mu x_2,\rho^{-n}\mu x_3,\rho u_0,\rho u_1).\end{array}\]
\end{definition}

The next lemma gives some basic properties of the variety $\QQ_g$. In particular, if $g$ is not a square, then Lemma~\ref{Qgproperties}\ref{QQgMfs} yields a structure of Mori quadric fibration $\pi_g\colon \QQ_g\to \p^1$; we will call such a fibration an \emph{Umemura quadric fibration}. It is a \emph{quadric bundle} in the sense of \cite[Definition 1.2]{Bea77} (except that our base is $\p^1$ and not $\p^2$ and that our varieties are not necessarily smooth). Note that we do not work with analytic coordinates, since we work in any characteristic and we can give a precise description of the singularities in Zariski local coordinates.

\begin{lemma}\label{Qgproperties}
Let $g\in \k[u_0,u_1]$ be a non-zero homogeneous polynomial of degree $2n$, for some $n\ge 0$. Denote by $H,F\subseteq \QQ_g$ the hypersurfaces given respectively by $x_3=0$ and $u_1=0$.
\begin{enumerate}
\item\label{QQgterminalrat}
The variety $\QQ_g$ is an irreducible normal rational projective threefold with  terminal singularities. Every singularity is Zariski locally given by the $cA_1$-singularity $\{(x,y,z,t)\in \A_k^4 \mid x^2-yz-t^mp(t)=0\}$ for some $m\ge 2$ and some polynomial $p(t)$ with $p(0)\neq 0$. Moreover $\QQ_g$ is $\Q$-factorial if and only if $g$ is not a square or $g\in \k^*$ and it is smooth if and only if $g$ is square-free.
\item\label{PicQQg}
If $g$ is not a square, then $\Pic(\QQ_g)=\Z H\oplus \Z F$. The cone of curves is moreover generated by the curves  $f=H\cap F$ and $h\subseteq H$, where $h$ is  given by $x_0=x_1=x_3=0$.
\item\label{canQQg}
The canonical divisor of $\QQ_g$ is given by $-2H-(n+2)F$ and satisfies $h\cdot K_{\QQ_g}=n-2$.
\item\label{QQgMfs}
The morphism
\[\begin{array}{rccc}
\pi_g\colon& \QQ_g&\to& \p^1\\
& [x_0:x_1:x_2:x_3;u_0:u_1]&\mapsto & [u_0:u_1]\end{array}\]
is a Mori quadric fibration $($i.e.~a Mori fibration whose generic fibre is a smooth quadric$)$ if and only if $g \in \k[u_0,u_1]$ is not a square.
\end{enumerate}
\end{lemma}
\begin{proof}
For $i=0,\dots,3$, we denote by $H_i=\QQ_g\cap \{x_i=0\}\subseteq \QQ_g$ the hypersurface given by $x_i=0$ and by $F_i$ the fibre given by $\QQ_g\cap \{u_i=0\}$, so that $F=F_1$ and $H=H_3$.
We then observe that $F_0\sim F_1=F$ and that $H_0\sim H_1\sim H_2 \sim H_3+nF_0=H+nF$.

We observe that $F_0$ and $F_1$ are irreducible quadric surfaces, as well as each fibre of $\pi_g\colon \QQ_g\to \p^1$. Moreover, $H_2$ is a surface, which is irreducible if and only if $g$ is not a square. As $\QQ_g\setminus (H_2\cup F)$ is isomorphic to $\A^3$, via $(x,y,z)\mapsto [x:x^2-g(1,z)y^2:1:y;z:1]$, the Picard group of $\QQ_g$ is generated by the irreducible components of $H_2$ and $F$. The same holds with $H_1$ instead of $H_2$ (by exchanging $x_1$ with $x_2$), so $\QQ_g$ is irreducible and rational. Moreover, if $g$ is not a square, the Picard group of $\QQ_g$ is  generated by $H_2$ and $F$, and thus by $H$ and $F$.

The singular locus is the finite (possibly empty) set given by 
\[\left\{x_0=x_1=x_2=0 \text{ and } g(u_0,u_1)=\frac{\partial g}{\partial u_0}(u_0,u_1) = \frac{\partial g}{\partial u_1}(u_0,u_1)=0\right\}. \]
Hence, $\QQ_g$ is smooth if and only if $g$ is square-free. Around a singular point $q=[0:0:0:1;u_0:u_1]$, the variety $\QQ_g$ is (Zariski) locally defined by  
\[\{(x,y,z,t)\in \A_k^4 \mid x^2-yz-t^mp(t)=0\},\]
where $m \geq 2$ is the multiplicity of $[u_0:u_1]$ as a root of $g$ and $p\in \k[t]$ is a polynomial that does not vanish at the origin. 
This is the equation of a {(normal)} $cA_1$-singularity, which is terminal; see \cite[$\S$1.42]{K13}. The singularity is moreover factorial if and only if $x^2-t^mp(t)$ is irreducible  \cite[(13.2)]{JK13}. This corresponds to ask that $t^mp(t)$ is not a square.

Note that $H=H_3$ is isomorphic to $\p^1\times\p^1$, with two rulings given by $f=H\cap F$ and $h=H\cap H_0\cap H_1$. 
As $h\cdot F=1$ and $h\cdot H_2=0$, we obtain $h\cdot H_3=h\cdot (H_2-nF)=-n$. We moreover have $f\cdot F=f\cdot F_0=0$ and $f\cdot H=f\cdot H_0=1$.  This implies that $\Pic(\QQ_g)=\Z H\oplus \Z F$ if $g$ is not a square, and that each irreducible curve $c\subseteq \QQ_g$,  is numerically equivalent to $a h+bf$ for some $a,b\in \Q$, with $c\cdot F= a$ and $c\cdot H= b-an$. To achieve the proof of \ref{PicQQg}, we prove that $a,b\ge 0$. If $c\subseteq H$, this is because $h,f\subseteq H$ generate the cone of curves of $H\simeq \p^1\times\p^1$. If $c\not\subseteq H$, then $0\le c\cdot H=b-an$. As $a=h\cdot F=h\cdot F_1\ge 0$, we get  \ref{PicQQg}.

The differential form $\sum\limits_{i=0}^3 \frac{dx_i}{x_i}+\sum\limits_{j=0}^1 \frac{du_j}{u_j}$ on $(\A^4\setminus \{0\} )\times (\A^2\setminus \{0\})$ being $\G_m^2$-invariant, it corresponds to a differential form on $X$, which has poles at $-K_{X}=\sum_{i=0}^3 \hat{H}_i+\hat{F}_0+\hat{F}_1\sim 4\hat{H}_3+(3n+2)\hat{F}_0$, where $\hat{H}_i,\hat{F}_j\subseteq X$ are given by $x_i=0$ and $u_j=0$ respectively. As $\QQ_g\sim 2\hat{H}_0\sim 2\hat{H}_3+2n\hat{F}_0$, we get by adjunction that $K_{\QQ_g}=-2H-(n+2)F,$ since $H=\hat{H}_3|_{\QQ_g}$ and $F=\hat{F}_0|_{\QQ_g}$. Computing moreover $K_{\QQ_g}\cdot h=-2H\cdot h-(n+2)F\cdot h=-2(-n)-(n+2)=n-2$ gives \ref{canQQg}.

It remains to prove~\ref{QQgMfs}. The map $\pi_g$ is a dominant projective morphism of normal projective varieties, and it remains to check if the conditions \ref{MoriFibrationDefa}, \ref{MoriFibrationDefb}, \ref{MoriFibrationDefc} in Definition \ref{Df:MoriFibration} are fulfilled.

\ref{MoriFibrationDefa}: For a projective morphism $f: X \to Y$ with $Y$ normal, the condition $f_* \O_X=\O_Y$  is equivalent to the connectedness of the fibres of $f$ (this can be seen for instance as a consequence of Stein factorisation \cite[Corollary~III.11.5]{Har77}). As $\pi_g$ is a projective morphism with connected fibres, this condition holds.

\ref{MoriFibrationDefb}: was proven in  \ref{QQgterminalrat}.

\ref{MoriFibrationDefc}: If $g$ is a square, then Lemma~\ref{quadric1}\ref{Qembedding} yields that 
$\rho((\QQ_{g})_{\k(\p^1)})=\rho(\p^1\times\p^1)=2$, and so $\pi_g$ is not a Mori fibration. If $g$ is not a square, \ref{PicQQg} gives $\rho(\QQ_g)=2$, which proves~\ref{MoriFibrationDefc}.
\end{proof}

\begin{lemma}\label{Lem:ActiononQQg1}
Assume that $\car(\k)\not=2$ and let $g\in \k[u_0,u_1]$ be a homogeneous polynomial of degree $2n$ that is not a square, for some $n\ge 1$. 
\begin{enumerate}
\item\label{UniqueDoubleCover}
There is a unique double cover $\tau\colon C \to \p^1$ such that the generic fibre of $\QQ_g\times_{\p^1} C\to C$ is isomorphic to $\p^1_{L}\times\p^1_{L}$, with $L=\k(C)$.
\item\label{ActionPGL2} 
 The group $\PGL_2(\k)$ acts regularly on $\QQ_g$
via 
\[\begin{array}{ccl}
\PGL_2(\k)\times\QQ_g &\to & \QQ_g\\
\left(\begin{bmatrix} a & b \\ c & d \end{bmatrix},[x_0:x_1:x_2:x_3;u_0:u_1]\right)&\mapsto& [(ad+bc)x_0+acx_1+bdx_2:\\
&&2abx_0+a^2x_1+b^2x_2:2cdx_0+c^2x_1+d^2x_2:\\
&&(ad-bc)x_3;u_0:u_1].\end{array}\]
\end{enumerate}
\end{lemma}

\begin{proof} 
As $g$ is not a square, the generic fibre $(\QQ_g)_{K}$, with $K=\k(\p^1)$, is not isomorphic to $\p^1\times\p^1$ but is isomorphic to it after a base change via a unique field extension $L/K$ of degree $2$ (Lemma~$\ref{quadric1}$). The ramified double cover $\tau\colon C \to \p^1$ associated with the field extension $L/K$ corresponds to the morphism of \ref{UniqueDoubleCover}.

The action of $\PGL_2(\k)$ given in \ref{ActionPGL2} 
{is obtained by restriction of the $\PGL_2(\k)$-action on $\P(\O_{\P^1}^{\oplus 3} \oplus \O_{\P^1}(n))$, given by the group embedding
\[
\PGL_2(\k) \hookrightarrow \PGL_4(\k),\ \ \begin{bmatrix}
a & b \\ c & d
\end{bmatrix} \mapsto \frac{1}{ad-bc}\begin{bmatrix}
ad+bc &ac & bd &0\\
2ab &a^2 & b^2 & 0\\
2cd& c^2& d^2 & 0 \\
0 & 0 & 0& 1
\end{bmatrix},\]
as the hypersurface $x_0^2-x_1x_2-g(u_0,u_1)x_3^2=0$ is invariant for this $\PGL_2(\k)$-action.} Therefore, we have a regular action on $\QQ_g$ giving an inclusion of algebraic groups
$\PGL_2(\k)\subseteq \Autz(\QQ_g)_{\p^1}.$ This achieves the proof of \ref{ActionPGL2}.
\end{proof}

\begin{lemma}\label{Lem:ActiononQQg2}Let us take the notation of Lemma~$\ref{Lem:ActiononQQg1}\ref{ActionPGL2}$.
\begin{enumerate}
\item\label{DiagramQglift}
There is a $\PGL_2(\k)$-equivariant birational map
\[
   \varphi\colon  \QQ_g\times_{\p^1} C \dasharrow \p^{1}\times\p^{1}\times C
 \]
 that induces an isomorphism between the generic fibres of the natural projections \[\QQ_g\times_{\p^1} C\to C\text{ and }\p^{1}\times\p^{1}\times C\to C,\]
where the action of $\PGL_2(\k)$ on $C$ is trivial, the action on $\p^{1}\times\p^{1}$ is the diagonal action,  and the action on $\QQ_g$ is the one given in Lemma~$\ref{Lem:ActiononQQg1}\ref{ActionPGL2}$.

\item\label{ActionGeneric}
Denote by 
\begin{itemize}
\item $H\subseteq \Bir_{\p^{1}}(\QQ_g)$ the subgroup of elements corresponding to automorphisms of the generic fibre;
\item $\hat{H}\subseteq \Bir(\QQ_g\times_{\p^1} C)$ the lift of $H$ $($obtained by acting trivially on $C)$;
\item $H_0'= \{(x,y,c) \mapsto (A(x),\iota(A)(y),c)\mid A\in \PGL_2(\k(C))\}$;
\item $\sigma\colon (x,y,c)\mapsto (y,x,c)$; and
\item $\iota\in \Aut(\k(C))$ is the involution induced by the $(2\colon 1)$-cover $C\to \p^{1}$.
\end{itemize}
Then the group $\varphi \hat{H} \varphi^{-1}\subseteq \Bir_{C}(\p^1\times\p^{1}\times C)$ is equal to $H'=H_0'\rtimes \langle \sigma\rangle$.
Also, every element of $H_{0}'$ that is algebraic and not of order $2$ is conjugate, in $H_{0}'$, to an element of $\PGL_2(\k)$. Furthermore, $\sigma$ acts on $\QQ_{g}$ as the biregular involution  $[x_0:x_1:x_2:x_3;u_0:u_1]\mapsto [x_0:x_1:x_2:-x_3;u_0:u_1]$.
\end{enumerate}

\end{lemma}
\begin{proof}
We can see $C$ as the curve given by $g(u_0,u_1)=u_2^2$ in the weighted projective space $\p(1,1,n)$, quotient of $\A^3\setminus \{0\}$ by $\G_m$ via $(u_0,u_1,u_2)\mapsto (\lambda u_0,\lambda u_1,\lambda^n u_2)$. Then, $\QQ_g\times_{\p^1} C$ is given by $x_0^2-x_1x_2-u_2^2x_3^2=0$ in $\P(\O_C^{\oplus 3} \oplus \O_C(n))$.
We then have a birational map  $\varphi\colon \QQ_g\times_{\p^1} C\dasharrow\p^1\times\p^1\times C$ given by
\[
\begin{array}{ccc}
 \QQ_g\times_{\p^1} C& \dashrightarrow & \p^1\times\p^1\times C\\
\! [x_0:x_1:x_2:x_3 ; u_0:u_1:u_2] &\mapsto & ([x_0+u_2x_3:x_2],[x_0-u_2 x_3:x_2],[u_0:u_1:u_2])\\
& & =([x_1:x_0-u_2 x_3],[x_1: x_0+u_2 x_3],[u_0:u_1:u_2])
 \end{array}
 \]
 whose inverse is given by $ \varphi^{-1}\colon ([x_{0}:x_{1}],[y_{0}:y_{1}],[u_0:u_1:u_2])\mapsto [u_{2}(x_{0}y_{1}+x_{1}y_{0}):2u_2x_{0}y_{0}:2u_{2}x_1y_{1}:x_0y_{1}-x_{1}y_{0} ; u_0:u_1:u_2] $. We observe that $\varphi$ induces an isomorphism on the generic fibres of the natural projections $\QQ_g\times_{\p^1} C\to C$ and $\p^{1}\times\p^{1}\times C\to C$.
 
The group $\PGL_2(\k)$ acts on $\QQ_g\times_{\p^1} C$ via its action on the coordinates $x_0,x_1,x_2,x_3$ given by Lemma~\ref{Lem:ActiononQQg1}\ref{ActionPGL2} and trivially on the coordinates $u_0,u_1,u_2$. On the other hand, $\PGL_2(\k)$ acts on $\p^1\times\p^1\times C $ by acting diagonally on $\p^1\times\p^1$ and trivially on $C$.  We now check that $\varphi$ is $\PGL_2(\k)$-equivariant with respect to these two actions. It suffices to check that for $\epsilon\in \{\pm 1\}$ the rational map $\QQ_g\times_{\p^1} C\dasharrow \p^{1}$ that sends $[x_0:x_1:x_2:x_3 ; u_0:u_1:u_2]$ onto $[x_0+\epsilon u_2x_3:x_2]$ is $\PGL_2(\k)$-equivariant.  This is a straightforward calculation (left to the reader) which achieves the proof of~\ref{DiagramQglift}.

We now focus on the subgroup $H\subseteq \Bir(\QQ_g/\p^{1})$ of elements corresponding to automorphisms of the generic fibre, whose action lifts to a subgroup $\hat{H}\subseteq \Bir(\QQ_g\times_{\p^1} C)$ isomorphic to $H$ (by acting trivially on $C$). As $\varphi$ gives an isomorphism on the generic fibres, the group $\varphi \hat{H} \varphi^{-1}\subseteq \Bir(\p^1\times\p^{1}\times C)$ is  contained in the group of birational maps of $\p^{1}\times\p^{1}\times C$ inducing automorphisms on the generic fibre $\p^{1}_{L}\times\p^{1}_{L}$. This latter is naturally isomorphic to $\PGL_2(L)\times\PGL_2(L)\rtimes  \langle \sigma\rangle$, where $\sigma$ is the involution $(x,y,c)\mapsto (y,x,c)$. The description of $H'=H_0'\rtimes \langle \sigma\rangle$ in the statement then follows from Lemma~\ref{quadric1}\ref{MuNotsquare}.

It remains to see that each element  $h\in H_{0}'$ that is algebraic and not of order $2$ is conjugate, in $H_{0}'$, to an element of $\PGL_2(\k)$. We denote by $A\in \PGL_2(\k(C))=\PGL_2(L)$ the element such that $h=(x,y,c) \mapsto (A(x),\iota(A)(y),c)$, and by $\hat{A}\in \GL_2(L)$ a lift. We write the characteristic polynomial of $\hat{A}$ as $\chi_{\hat{A}}=(X-\alpha)(X-\beta)$, where $\alpha,\beta\in \hat{L}^{*}$, and $\hat{L}$ is a field extension of $L$, of degree $1$ or $2$, such that $\hat{L}=L[\alpha,\beta]$. Suppose first that $\alpha=\beta$. This implies that $\chi_{\hat{A}}=X^{2}-2\alpha X+\alpha^{2}\in L[X]$, so $\alpha\in L$ as $\car(L)=\car(\k)\not=2$. We can thus conjugate $\hat{A}$ to $ \begin{bmatrix}\alpha & \alpha \\ 0 & \alpha \end{bmatrix}=\alpha\cdot  \begin{bmatrix} 1 & 1 \\ 0 & 1 \end{bmatrix}$ or $ \begin{bmatrix}\alpha & 0 \\ 0 & \alpha \end{bmatrix}=\alpha\cdot  \begin{bmatrix} 1 & 0 \\ 0 & 1 \end{bmatrix}$. This achieves the proof in this case, so we may assume that $\alpha\not=\beta$. We then write $\beta=\mu \alpha$ for some $\mu\in \hat{L}\setminus \{0,1\}$. If $\mu\in \k$, then $\chi_{\hat{A}}=X^{2}-\alpha(\mu+1)X+\alpha^{2}\mu\in L[X]$, which implies that $\alpha\in L$ (we use here that $\mu\not=-1$ as $h$ is not an involution). We can then replace $\hat{A}$ with $\frac{1}{\alpha}\hat{A}$ and obtain $\chi_{\hat{A}}=(X-1)(X-\mu)$, which implies that $\hat{A}$ is conjugated to $\begin{bmatrix} 1 & 0 \\ 0 & \mu \end{bmatrix}$  in $\GL_2(L)$.

It remains to show that $\mu\in\k$. We take again a finite cover $\hat{C}\to C$ associated with the extension $\hat{L}/L$ and consider the lift of $g$ on $\hat{X}\times_C \hat{C}$. Take a birational map $\hat{X}\times_C \hat{C}\dasharrow \p^1\times\p^1\times\hat{C}$ that conjugates $g$ to $([u_0:u_1],[v_0:v_1],t)\dasharrow 
([u_0:\mu(t) u_1],[v_0:\mu'(t)  v_1],t)$ where $\mu'\in \hat{L}$. By Lemma~\ref{Lemm:AlgebraickS}, the element $\mu$ is in $\k$. This achieves the proof of~\ref{ActionGeneric}.
\end{proof}

\begin{example}\label{Example:Qgu0u1}
Let $a,b\ge 1$ be two odd numbers and let us consider the variety $\QQ_{g}$ of Definition~\ref{def:QQg} with $g=u_{0}^au_{1}^b$. It is equal to 
\[\QQ_{g}=\{[x_0:x_1:x_2:x_3;u_0:u_1]\in \P(\O_{\p^1}^{\oplus 3}\oplus \O_{\p^1}(n))\mid x_0^2-x_1x_2-u_{0}^au_{1}^bx_3^2=0\},\]
where $n=(a+b)/2$.
\begin{enumerate}
\item
We have an algebraic subgroup $\G_m\subseteq \Autz(\QQ_{g})$ given by 
\[\G_m\times\QQ_{g}\to \QQ_{g}, \quad (t,[x_0:x_1:x_2:x_3;u_0:u_1])\mapsto [x_0:x_1:x_2:t^{-a}x_3;t^{2}u_0:u_1].\] 
{This $\G_m$-action}, together with the subgroup $\PGL_2(\k) \subseteq \Autz(\QQ_{g})$ given in Lem\-ma~\ref{Lem:ActiononQQg1}\ref{ActionPGL2}, gives an inclusion 
\[ \PGL_2(\k)\times\G_m\subseteq \Autz(\QQ_{g}).\]
We will prove in Corollary~\ref{Cor:AutQQg} that it is in fact an equality. The action of $\PGL_2(\k)\times\G_m$ on $\p^{1}$ gives an exact sequence 
\[1 \longrightarrow \PGL_2(\k)\times  \langle \sigma\rangle\longrightarrow \PGL_2(\k)\times\G_m\longrightarrow \G_m\to 1,\]
where $\sigma\in \G_m\subseteq {\Autz(\QQ_{g})}$ is the involution given by \[\sigma\colon [x_0:x_1:x_2:x_3;u_0:u_1]\mapsto [x_0:x_1:x_2:-x_3;u_0:u_1].\]
\item
When $n=1$, the morphism $\QQ_g\to \p^{4}$ given by $[x_0:x_1:x_2:x_3;u_0:u_1]\mapsto [x_0:x_1:x_2:x_3u_{0}:x_{3}u_{1}]$ is the blow-up of the plane $P\subseteq \p^{4}$ given by $P=\{[x_0:x_1:x_2:x_3:x_{4}]\in \p^{4}\mid x_{3}=x_{4}=0\}$, so  $\QQ_{g}$  is the blow-up of the smooth quadric $Q=\{[x_0:x_1:x_2:x_3:x_{4}]\in \p^{4}\mid x_0^2-x_1x_2-x_{3}x_{4}=0\}$ along the smooth conic $\Gamma=P\cap Q$. This conjugates $\Autz(\QQ_{g})$ to the connected group $\Aut(Q,\Gamma)=\{g\in \Aut(Q)\mid g(\Gamma)=\Gamma\}$, strict subgroup of $\Autz(Q)$.
\end{enumerate}
\end{example}

The following corollary will be useful in \upshape\S~\ref{subsec Umemura quadric fib}  when studying the equivariant links between the Umemura quadric fibrations.

\begin{corollary}\label{Cor:AutQQg}
Assume that $\car(\k)\not=2$, let $n\ge 1$ and let $g\in \k[u_0,u_1]$ be a homogeneous polynomial of degree $2n$ that is not a square.

If $g$ has only two roots, we may change coordinates to get $g=u_0^au_1^b$ for some odd $a,b\ge 1$ and then $\Autz(\QQ_g)$ is equal to the group $\PGL_2(\k)\times\G_m$, given by Example~$\ref{Example:Qgu0u1}$. Otherwise, $\Autz(\QQ_g)$ is equal to the group $\PGL_2(\k)$ given in Lemma~$\ref{Lem:ActiononQQg1}\ref{ActionPGL2}$.
\end{corollary}
\begin{proof}
The group $\PGL_2(\k)$ embeds into $\Autz(\QQ_g)$, via the action given in Lem\-ma~$\ref{Lem:ActiononQQg1}\ref{ActionPGL2}$.
If moreover $g=u_0^au_1^b$, we obtain an embedding of $\PGL_2(\k)\times\G_m$ into $\Autz(\QQ_g)$, as in Example~$\ref{Example:Qgu0u1}$. It remains to see that this gives in both cases the whole group $\Autz(\QQ_g)$.

As $\PGL_2(\k)$ does not fix any section of $\QQ_g\to \p^1$, Theorem~\ref{Thm:MainQuadric} gives a square-free homogeneous polynomial $\tilde{g}\in \k[u_0,u_1]$ of degree $2n'$, for some $n'\ge 1$ and a birational map $\psi\colon \QQ_g\dasharrow \QQ_{\tilde{g}}$, inducing an  isomorphism between the generic fibres and such that the group $\psi\Autz(\QQ_g)\psi^{-1}\subseteq \Autz(\QQ_{\tilde g})$ is either equal to the group $\PGL_2(\k)$ given in Lemma~$\ref{Lem:ActiononQQg1}\ref{ActionPGL2}$ or to $\PGL_2(\k)\times\G_m$ given (after change of coordinates) in Example~$\ref{Example:Qgu0u1}$. Moreover, in this latter case we have $n'=1$.

In the first case, the group $\psi\Autz(\QQ_g)\psi^{-1}$ is equal to $\PGL_2(\k)$ and contains the image of $\PGL_2(\k)\subseteq \Autz(\QQ_g)$. As $\PGL_2(\k)$ does not contain any proper algebraic subgroup isomorphic to itself, the group $\Autz(\QQ_g)$ is equal to the $\PGL_2(\k)$ given in Lemma~$\ref{Lem:ActiononQQg1}\ref{ActionPGL2}$. In the second case, the action of $\psi\Autz(\QQ_g)\psi^{-1}$ on $\p^1$ is isomorphic to $\G_m$, fixing exactly two fibres. So $\Autz(\QQ_g)$ also acts on $\p^1$ fixing exactly two fibres, which implies that $g$ has only two roots. Changing coordinates, $g=u_0^au_1^b$ for some odd number $a,b\ge 1$, we obtain an inclusion of $\PGL_2(\k)\times\G_m$ into $\Autz(\QQ_g)$, as in Example~$\ref{Example:Qgu0u1}$.
The group $\Autz(\QQ_g)$ is then sent, via $\psi$, to a subgroup of $\PGL_2(\k)\times\G_m$. As the kernel of the action on $\p^1$ is in both cases $\PGL_2(\k)\times  \langle \sigma\rangle$ and the image are the same, both groups are equal. 
\end{proof}

Motivated by Lemma~\ref{Lem:ActiononQQg2}, we now study  the subgroups of $\PGL_2(L)$ having the property that each element is conjugated to an element of $\PGL_2(\k)$.
\begin{lemma}\label{Lemm:PGL2KL}
{Let $K\subseteq L$ be an arbitrary field extension, with $K$ an algebraically closed field.} Let $G\subseteq   \PGL_2(L)$ be a subgroup.
If every element of $G$ is conjugated to an element of $\PGL_2(\k)$ by an element of $\PGL_2(L)$, then $G$ is conjugate in $\PGL_2(L)$ to either a subgroup of $\PGL_2(K)$ or to a subgroup of 
\begin{equation}\label{group_H}
T=\left \{ \begin{bmatrix} a & P \\ 0 & 1 \end{bmatrix} \in \PGL_2(L) \middle|\ a\in K^*, P\in L \right \}\simeq L\rtimes K^*.
\end{equation}
\end{lemma}
\begin{proof}
Observe that the determinant map $\det\colon \GL(2,L)\to L^*$ induces a group homomorphism 
\[\overline{\det}\colon \PGL_2(L)\to L^*/(L^*)^2,\]where $(L^*)^2=\{P^2\mid P\in L^*\}$. Since every element of $K^*$ is a square, we have $\overline{\det}(\PGL_2(K))=\{1\}$. 

For each element $g\in G$, we have then $\overline{\det}(g)=1$, because $g$ is conjugated to an element of $\PGL_2(K)$ by assumption. There exist then exactly two elements $\hat{g}, -\hat{g}\in \SL_2(L)$ that represent the element $g\in\PGL_2(L)$ (or one element $\hat{g}=-\hat{g}$ if $\mathrm{char}(K)=2$). These elements satisfy furthermore $\mathrm{Trace}(\hat{g})\in K$. As $K$ is algebraically closed, $\hat{g}$ is conjugate, in $\GL_2(L)$, to either a diagonal element of the form $\begin{bmatrix} \lambda & 0 \\ 0 & \lambda^{-1} \end{bmatrix}$, $\lambda\in K^*$, or to $\pm \tau$, where $\tau=\begin{bmatrix}1 & 1 \\ 0 & 1 \end{bmatrix} \in \SL_2(K)$.

$(a)$ Suppose first that $\hat{g}$ is conjugated to $\pm \tau$ for each $g\in G\setminus\{1\}$. If $G$ is trivial, we do not need to prove anything, so we can conjugate by an element of $\PGL_2(L)$ and assume that there is $g_0\in G$ such that $\hat{g_0}=\tau$. For each $h\in G$, {the element $\hat{h}=\begin{bmatrix} a & b \\ c & d \end{bmatrix} \in \SL_2(L)$ has then trace equal to $\pm 2$ (as it is $\pm \mathrm{id}$ or conjugate to $\pm \tau$); we may thus assume that $a+d=2$ by replacing $\hat{h}$ with $-\hat{h}$ if needed}. We then compute $\mathrm{Trace}((\hat{g_0})^n\hat{h})=2+nc\in \{\pm 2\}$, for each $n\in \Z$, which implies that $c=0$. As $ad=1$ and $a+d=2$, we get $a=d=1$, so $G$ is contained in $T$.

$(b)$ We can now assume, after conjugating by an element of $\PGL_2(L)$, that one element $g_0\in G$ satisfies $\hat{g}_0=\begin{bmatrix} \lambda & 0 \\ 0 & \lambda^{-1} \end{bmatrix}$, for some $\lambda\in K\setminus \{0,\pm 1\}$.

We take an element $h\in G$ and write $\hat{h}=\begin{bmatrix} a & P \\ Q & b \end{bmatrix}$, where $a,b,P,Q\in L$ satisfy $ab-PQ=1$ and $a+b\in K$. Since $\mathrm{Trace}(\hat{g_0}\hat{h})=\lambda a +\lambda^{-1}b\in K$, we find that $a=\frac{1}{1-\lambda^2}(a+b)-\frac{\lambda}{1-\lambda^2}(\lambda a +\lambda^{-1}b)\in K$. This implies that $a,b,PQ\in K$.

If $P=0$ for each $h\in G$ as above, then conjugating by an anti-diagonal matrix, we find that $G\subseteq   T$. {Otherwise we can} choose $h_1\in G$ with $\hat{h}_1=\begin{bmatrix} a_1 & P_1 \\ Q_1 & b_1 \end{bmatrix}$, where $a_1,b_1,P_1Q_1=c_1\in K$, $P_1\not=0$. Conjugating with $\begin{bmatrix} 1 & 0 \\ 0 & P_1 \end{bmatrix}$, we may assume that $P_1=1\in K^*$ and $Q_1\in K$. If $G$ is not contained in $T$, there exists $h_2\in G$ such that $\hat{h}_2=\begin{bmatrix} a_2 & P_2 \\ Q_2 & b_2 \end{bmatrix}$, with $a_2,b_2,P_2Q_2\in K$ and $Q_2\not=0$. The diagonal of $\widehat{h_1h_2}=\pm \hat{h}_1\cdot \hat{h}_2$ being $(a_1a_2+P_1Q_2,Q_1P_2+b_1b_2)\in K^2$, we find that $Q_2\in K^*$. 

We finish the proof by taking any element $g\in G$, by computing $\hat{g}\cdot \hat{h}_i$ for $i=1,2$ which shows that $g\in \PGL_2(K)$.
\end{proof}
We will also need the following associated result:
\begin{lemma}\label{Lemm:PGL2KLnormalised}
Let $K\subseteq L$ be a field extension, where $K$ contains at least three elements.
The group $\PGL_2(K)$ is its own normaliser in $\PGL_2(L)$.
\end{lemma}
\begin{proof}
We take $A=\begin{bmatrix} a & b \\ c & d \end{bmatrix}\in \PGL_2(L)$ that normalises $\PGL_2(K)$ and prove that $A\in \PGL_2(K)$.

$(i)$ We first study the case where $c=0$, in which case $ad\not=0$, so we can assume that $a=1$. For each $\mu \in K\setminus \{0,1\}$, we compute 
\[ A^{-1}\begin{bmatrix} 1 & 0 \\ 0 & \mu \end{bmatrix}A=\begin{bmatrix}1& b(1-\mu)\\ 0& 1\end{bmatrix},\ A^{-1}\begin{bmatrix} 1 & \mu \\ 0 & 1 \end{bmatrix} A=\begin{bmatrix}1& d\mu\\ 0& 1\end{bmatrix}\in \PGL_2(K),\]
which imply that $b,d\in K$, achieving the proof.

$(ii)$ We then study the case where $c\not=0$. We then assume that $c=1$ and compute, for each $\mu\in K^{*}$,
\[ A^{-1} \begin{bmatrix} 1 & \mu \\ 0 & 1 \end{bmatrix} A=\begin{bmatrix}ad-b+d\mu& d^{2}\mu\\ -\mu& ad-b-d\mu \end{bmatrix}\in \PGL_2(K).\]
{This yields $ad-b + d\mu\in K$ for each $\mu\in K^{*}$, so $d\in K$ as $K^{*}$ contains at least $2$ elements.} The matrix $S=\begin{bmatrix} d & 1 \\ -1 & 0 \end{bmatrix}\in \PGL_2(K)$ normalises $\PGL_2(K)$, so does also $A S=\begin{bmatrix} ad-b & a \\ 0 & 1 \end{bmatrix}$. By the previous argument, $A S$ belongs to $\PGL_2(K)$, so also $A$.
\end{proof}

The next result will be crucial in the proof of Proposition \ref{Thm:MainQuadric} (in \upshape\S~\ref{proof of thA}).

\begin{lemma} \label{Lemm:QuadricPGL2Lk}
Assume that $\car(\k)\not=2$. Let $\pi \colon X \to \p^1$ be a Mori del Pezzo fibration of degree $8$.
\begin{enumerate}
\item\label{QuadricPGL2LkMap}
There is a square-free homogeneous polynomial $g\in \k[u_0,u_1]$ of degree $2n$, for some $n\ge 1$, and a commutative diagram
\[\xymatrix@R=3mm@C=1cm{
    X \ar@{-->}[rr]^{\psi} \ar[rd]_{\pi}  && \QQ_g \ar[ld]^{\pi_g} \\
    & \p^1
  }\]
for some birational map $\psi\colon X\dasharrow \QQ_g$ inducing an isomorphism on the generic fibres.
\item\label{QuadricPGL2Lktorus}
The image of the natural homomorphism $\Autz(X) \to \Aut(\p^1)=\PGL_2$ is either trivial or $\G_m$. In this latter case, every polynomial $g$ as in \ref{QuadricPGL2LkMap} has to satisfy $\deg(g)=2$.
\end{enumerate}\end{lemma}

\begin{proof}
Let us write $K=\k(\p^1)$. By Lemma~\ref{lem:MoriQuad8}, the morphism $\pi$ is a Mori quadric fibration, that is, a Mori del Pezzo fibration whose generic fibre is a smooth quadric. This quadric having a $K$-point by Tsen theorem ($K$ is a $C_1$ field, see \cite[Tag 03RD]{stacks-project}) or by 
Lemma~\ref{Lem:bound_char_dP}, it is isomorphic, over $K$, to a quadric surface of $\p^3_K$ given by $x_0^2-x_1x_2-\mu x^2_3=0,$ for some $\mu\in K$ (Lemma~\ref{quadric1}).  Moreover, $\rho(X_K)=1$, as $\pi$ is a Mori fibration, which implies that $\mu$ is not a square in $K$ (Lemma~\ref{quadric1}\ref{Qembedding}).

Replacing $x_3$ with $x_3\alpha$ for some $\alpha\in K^*$ does not change the isomorphism class, so we may assume that the restriction of $\mu$ to $\A^1=\p^1\setminus \{[1:0]\}=\{[t:1]\mid t\in \A^1\}$ is a square-free polynomial.  We then choose $g\in \k[u_0,u_1]$ homogeneous of even degree $d\in \{\deg(\mu),\deg(\mu)+1\}$ such that $g(t,1)=\mu(t)$ and obtain that $X_K$ is isomorphic to the generic fibre $(\QQ_g)_K$. As $\mu$ is not a square in $K^*$, we get $d\ge 1$. We then obtain a birational map $\psi\colon X\dasharrow \QQ_g$ such that $\pi_g\psi=\pi$. This yields~\ref{QuadricPGL2LkMap}.

Assertion~\ref{QuadricPGL2Lktorus} follows from the fact that the double covering $\tau\colon C\to \p^1$ is uniquely determined by the generic fibre of $X\to \p^1$, which is also the generic fibre of $\QQ_g\to \p^1$, and that this one is ramified over the zeros of the polynomial $g$ of \ref{QuadricPGL2LkMap}, which are $\deg(g) \geq 2$ distinct points of $\p^1$. The action of $\Autz(X)$ on $\p^1$ then fixes these points. If $\deg(g)=2$, then the image is a subgroup of $\G_m\subseteq \PGL_2(\k)$, but if $\deg(g) \geq 3$, then  it fixes three distinct points of $\p^1$, so one gets the trivial group in $\PGL_2(\k)$.
\end{proof}

\subsection{Proof of Theorems~\ref{Thm:MainQuadric} and~\ref{th:A}}  \label{proof of thA}
In this subsection we gather the results obtained so far and give the proofs of Theorems~\ref{Thm:MainQuadric} and~\ref{th:A}, both stated in the introduction.\\

\noindent \emph{Proof of Theorem~\ref{Thm:MainQuadric}:}
The fact that $d\in \{1,\ldots,9\}$ and $d\not=7$ is given by Lemma~\ref{lem:MoriQuad8}. Part~\ref{PropMainQuadric_le5} of Theorem~\ref{Thm:MainQuadric} is given by Proposition~\ref{cor Aut dP are tori}. It remains then to prove~\ref{PropMainQuadric_mainpoint}, so we now assume that $\Autz(X)$ is not a torus. By Lemma~\ref{lem:MoriQuad8}, either $d=8$ and $\pi_X$ is a Mori quadric fibration, or $d=9$ and $\pi_X$ is a $\P^2$-fibration (Lemma~\ref{Lemm:Tsen}). If $\pi_X$ is a $\P^2$-fibration, then the result follows from Proposition~\ref{prop: main result P2 fibrations}. It remains only to consider the case where $\pi_X\colon X \to \P^1$ is a Mori quadric fibration.

Using Lemma~\ref{Lemm:QuadricPGL2Lk}\ref{QuadricPGL2LkMap}, we have a square-free homogeneous polynomial $g\in \k[u_0,u_1]$ of degree $2n$, for some $n\ge 1$, and a birational map $\psi\colon X\dasharrow \QQ_g$ such that $\pi_g\psi=\pi_X$, inducing an isomorphism on the generic fibres. We then use the unique double cover $\tau\colon C \to \p^1$ such that the generic fibre of $\QQ_g\times_{\p^1} C\to C$ is isomorphic to $\p^1_{L}\times\p^1_{L}$, with $L=\k(C)$ (Lemma~\ref{Lem:ActiononQQg1}\ref{UniqueDoubleCover}), and the birational map $ \varphi\colon  \QQ_g\times_{\p^1} C \dasharrow \p^{1}\times\p^{1}\times C$ given by Lemma~\ref{Lem:ActiononQQg2}\ref{DiagramQglift}.

We denote by $G$ the connected component of $ \Autz(X)_{\p^{1}}$. Note that $G$ is equal to  $\Autz(X)$ if $n>1$ (Lemma~\ref{Lemm:QuadricPGL2Lk}\ref{QuadricPGL2Lktorus}). Since $\Autz(X)$ normalises $ \Autz(X)_{\p^{1}}$, it also normalises its neutral component and thus 
the group $G$ is normal in $\Autz(X)$. The whole group $G$ acts rationally on $\QQ_{g}$, via $\psi$, inducing automorphisms on the generic fibre of $\QQ_{g}\to \p^{1}$. It then also acts rationally 
 on $\QQ_g\times_{\p^1} C$ by acting trivially on $C$. The conjugation by $\varphi$ then yields a rational action of $G$ on $\p^1\times\p^{1}\times C$, and thus an inclusion $G\hookrightarrow \Bir_{C}(\p^1\times\p^{1}\times C)$. As $\varphi$ induces an isomorphism between the generic fibres and $G$ is connected, we get an inclusion $G\subseteq \PGL_2(L)\times\PGL_2(L)\subseteq \Bir_C(\p^1\times\p^1\times C)$.  More precisely, Lemma~\ref{Lem:ActiononQQg2}\ref{ActionGeneric} implies that $G\subseteq H_0'$, where the group $H_{0}'\subseteq \Bir_{C}(\p^1\times\p^{1}\times C)$ is given by
\[H_0'= \{(x,y,c) \mapsto (A(x),\iota(A)(y),c)\mid A\in \PGL_2(\k(C))\}\simeq \PGL_2(\k(C))=\PGL_2(L).\]
{From now on, we identify $H_{0}'$ with $\PGL_2(\k(C))=\PGL_2(L)$ by sending}
$(x,y,c) \mapsto (A(x),\iota(A)(y),c)$ onto $A$.
By Lemma~\ref{Lem:ActiononQQg2}\ref{ActionGeneric} every element of $G$ that is not an involution is conjugate in $H_{0}'$ to an element of $\PGL_2(\k)\subseteq H_{0}'$, acting diagonally on the two factors and thus corresponding to $\PGL_2(\k)\subseteq \PGL_2(L)$. As $G$ is a connected algebraic group and $\car(\k)\not=2$, every involution is a semisimple element of $G$ and thus contained in a torus of $G$ by \cite[\S19.3 and \S22.2]{Hum75}. This implies that every involution is the square of an element of order $4$, hence also conjugate  in $H_{0}'\simeq \PGL_2(L)$ to an element of $\PGL_2(\k)$.

We then obtain two cases (by Lemma~\ref{Lemm:PGL2KL}):
$G$  is conjugate in $H_{0}'\simeq\PGL_2(L)$ (we recall that the isomorphism between the two groups is fixed) to either a subgroup of $T\subseteq \PGL_2(L)$ or of $\PGL_2(\k)\subseteq \PGL_2(L)$, where \[T=\left \{ \begin{bmatrix} a & P \\ 0 & 1 \end{bmatrix} \in \PGL_2(L) \middle|\ a\in \k^*, P\in L \right \}\simeq L\rtimes \k^*.\]

We first assume that $G$ is conjugate in $H_{0}'$ to a subgroup of the triangular group $T$. {There is then a rational section of $\p^{1}\times\p^{1}\times C\to C$, which is fixed by $G$, and its image in $X$ is a rational section of $X\to \p^1$, also fixed by $G$. We now prove that if $G$ is non-trivial, then there are only finitely many rational sections  of $X\to \p^1$ that are fixed by $G$. The preimage of any such section gives either  two rational sections of $\p^{1}\times\p^{1}\times C\to C$, both fixed by $G$, or an irreducible curve $\Gamma\subset \p^{1}\times\p^{1}\times C$, invariant by $G$ and whose projection to $C$ gives a $2:1$-map. (This follows from the fact that there are only finitely many orbits of size $\le 2$ in a general fibre of $\p^{1}\times\p^{1}\times C\to C$.)}

 If $G$ fixes finitely many rational sections, $\Autz(X)$ acts on this finite set of sections, as $G$ is normal in $\Autz(X)$. As  $\Autz(X)$ is connected, each of these sections is $\Autz(X)$-invariant. The projection away from one section {(that we can do in family, using the trivialisation given by $\tau\colon C \to \P^1$, \'etale on a dense open subset of $\p^1$)} gives an $\Autz(X)$-equivariant birational map to a $\p^2$-fibration over $\p^1$. Applying Proposition~\ref{prop: main result P2 fibrations} we reduce to the case of a $\p^2$-bundle over $\p^1$. If $\Autz(X)$ is trivial, we can do the same with any rational section (which exists by  Lemma~\ref{Lem:bound_char_dP}). The remaining case is when $G$ is trivial, so $\Autz(X)_{\p^{1}}$ is finite, but $\Autz(X)$ is not trivial. In this case $\Autz(X)/ \Autz(X)_{\p^{1}}$ is isomorphic to $\G_m$ and $n=1$ (Lemma~\ref{Lemm:QuadricPGL2Lk}\ref{QuadricPGL2Lktorus}). As $\Autz(X)$ is an algebraic group of dimension $1$ having $\G_m$ as a quotient, the  group $\Autz(X)$ is isomorphic to $\G_m$, a case excluded by assumption.

According to Lemma~\ref{Lemm:PGL2KL} and the previous case, the remaining case is when $G\subseteq H_{0}'\simeq \PGL_2(L)$ is conjugated to a subgroup of $\PGL_2(\k)$, not conjugated to a subgroup of $T$. The group $G$ corresponds to a subgroup of $\PGL_2(\k)$ acting diagonally on $\p^1\times\p^1\times C$. If $G$ is a strict subgroup of $\PGL_2(k)$, then it fixes again a point of the generic fibre, and so $G$ is conjugated to a subgroup of the triangular group $T$. We may then assume that $G=\PGL_2(\k)$. By Lemma~\ref{Lem:ActiononQQg2}\ref{ActionGeneric}, the rational action of $G$ on $\QQ_{g}$ is then exactly the biregular action of $\PGL_2(\k)$ on $\QQ_{g}$ given in Lemma~\ref{Lem:ActiononQQg1}\ref{ActionPGL2}.  This proves that $\psi\colon X\dasharrow \QQ_{g}$ is $G$-equivariant, for the biregular action of $G\simeq \PGL_2(\k)$ on $\QQ_{g}$ given in Lemma~\ref{Lem:ActiononQQg1}\ref{ActionPGL2}.

If $G=\Autz(X)$, we obtain that $\psi$ is $\Autz(X)$-equivariant, and $\psi\Autz(X)\psi^{-1}\subseteq \Autz(\QQ_g)$ is the group $\PGL_2$ given in Lemma~$\ref{Lem:ActiononQQg1}\ref{ActionPGL2}$.

The remaining case to study is when $G\subsetneq \Autz(X)$, which implies that $n=1$, and that $\Autz(X)/\Autz(X)_{\p^{1}}\simeq \G_m$  (Lemma~\ref{Lemm:QuadricPGL2Lk}\ref{QuadricPGL2Lktorus}). As $n=1$ and $g$ is square-free, the polynomial $g$ has two distinct roots, so we may assume that $g=u_0 u_1$. We then use the group $\G_m\subseteq \Autz(X)$ of Example~\ref{Example:Qgu0u1}, which contains an involution $\sigma$, being the kernel of the action of $\G_m$ on $\p^{1}$ that fixes the two ramification points. It remains to see that $\psi\Autz(X)\psi^{-1}\subseteq \Autz(\QQ_g)$ is the group $\PGL_2\times\G_m$ given in Example~$\ref{Example:Qgu0u1}$ to conclude the proof. The rational action of $\Autz(X)_{\p^{1}}$ on $\QQ_{g}$, via $\psi$, gives a group of automorphisms of the generic fibre that normalises $\PGL_2(\k)$. Looking at the group $\Autz(X)_{\p^{1}}$ in $H'=H_0'\rtimes \langle \sigma\rangle\simeq \PGL_2(L)\rtimes \langle \sigma\rangle$ (Lemma~\ref{Lem:ActiononQQg2}\ref{ActionGeneric}), it should normalise $\PGL_2(\k)$ and is thus contained in $\PGL_2(\k)\times\langle \sigma\rangle$; this follows from the fact that $\sigma$ normalises $\PGL_2(\k)$ (commuting with it) and that $\PGL_2(\k)$ is its own normaliser in  $\PGL_2(L)$ (Lemma~\ref{Lemm:PGL2KLnormalised}). This gives $\psi\Autz(X)_{\p^{1}}\psi^{-1}\subseteq  \PGL_2\times\langle \sigma\rangle \subseteq \PGL_2\times\G_m\subseteq \Autz(\QQ_g)$

Every element $\alpha\in \psi\Autz(X)\psi^{-1}$ acts on $\p^{1}$ in the same way as an element $d\in \G_m$, so $\beta=\alpha d^{-1}\in \Bir(X)$ acts trivially on $\p^{1}$ and yields an automorphism of the generic fibre. Since $\alpha$ and $d$ normalise $\PGL_2(\k)\subseteq \Autz(\QQ_{g})$, the same holds for $\beta$, which again then belongs to $\PGL_2\times\langle \sigma\rangle\subseteq \Aut(\QQ_{g})$. This proves that $\psi\Autz(X)\psi^{-1}\subseteq \PGL_2\times\G_m$. As $\psi\Autz(X)\psi^{-1}$ contains $\psi G\psi^{-1}=\PGL_2$ and as its action on $\p^1$ is the same as $\G_m$, we get $\psi\Autz(X)\psi^{-1}= \PGL_2\times\G_m$ as desired.
\qedsymbol \\

\noindent \emph{Proof of Theorem~\ref{th:A}.}
As explained in the introduction, we {run} an MMP  to $\hat{X}$ {(see \cite[Corollary 1.3.2]{BCHM10} and \cite[Theorem 1.7]{BW})}; this gives a birational map $\hat{X}\dasharrow Y$, which is $\Autz(\hat{X})$-equivariant (see Remark~\ref{rk:MMP G eq}), and is such that $Y$ has a structure of Mori fibre space $Y\to S$. 

If $\dim(S)=0$, we obtain Case~\ref{thAFano} of Theorem~\ref{th:A}. 

If $\dim(S)=1$, then $Y\to S$ is a del Pezzo fibration. If $\Autz(Y)$ is not isomorphic to a torus, we can apply Theorem~\ref{Thm:MainQuadric}, and replace $Y\to S$ with either a $\p^2$-bundle or a smooth Umemura quadric fibration $\QQ_g\to \p^1$, so get Case~\ref{thAQuadricP2} of Theorem~\ref{th:A}.

If $\dim(S)=2$, then $Y\to S$ is a conic bundle. Proposition~\ref{Prop:RatRat}\ref{BaseMfsRatisRat} implies that $S$ is rational. We can apply Theorem~\ref{Thm:conic bundles} and either reduce to the case where $Y\to S$ is a $\p^1$-bundle and where $S$ a smooth projective rational surface with no $(-1)$-curve (Case~\ref{Thm:conic bundlesCaseP1b} of Theorem~\ref{Thm:conic bundles}), or obtain that $\Autz(Y)$ is a torus of dimension at most $2$ (Case~\ref{Thm:conic bundlesCaseTorus} of Theorem~\ref{Thm:conic bundles}). The first possibility gives rise to Case~\ref{thAP1} of Theorem~\ref{th:A}. It remains to consider the case where $\Autz(Y)$ is a torus. {As  all tori of $\Bir(\p^3)$ of the same dimension are conjugate (Corollary~\ref{tori conjugate}), there exists a birational map $Y\dasharrow \p^3$ which conjugates $\Autz(Y)$ to a diagonal torus of $\Aut(\P^3)$, and so we are in Case~\ref{thAFano} of Theorem~\ref{th:A}. {(In fact, we could as well have used Corollary~\ref{tori conjugate} to conjugate $\Autz(Y)$ to a subgroup of $\Autz(Z)$,  with $Z$ any rational Mori fibre space of dimension $3$ endowed with a faithful regular action of a two-dimensional torus, and end up in any of the three Cases~\ref{thAP1}-\ref{thAQuadricP2}-\ref{thAFano} of Theorem~\ref{th:A}.)}}\qed

\section{First refinement of Theorem~\ref{th:A}, first links, and non-maximality results} \label{subsec:first refinement}

\subsection{Some families of \texorpdfstring{$\P^1$}{P1}-bundles and first step towards Theorem~\ref{th:Ea}} \label{sec:first classification}
We first introduce certain families of Mori fibrations, then we put Theorems~\ref{th:A} and~\ref{thBFT} together to get Theorem~\ref{th:first classification maximality}; the latter, which is {a refiniment of Theorem~\ref{th:A} in the case where the base field $\k$ is of characteristic zero, is} the first step to prove Theorem~\ref{th:Ea}.

\smallskip

We now define seven families of Mori fibrations that will play an important role in the rest of this article: Families  (1)--(5) are actually families of $\P^1$-bundles over smooth rational surfaces (that will appear in the statement of Theorem~\ref{thBFT}), Family (6) is formed by the $\P^2$-bundles over $\P^1$, and Family (7) is formed by some $\P^1$-fibrations over a (singular) rational surface.  
 {The threefolds of Families (1)-(2)-(6)-(7) are toric while the threefolds of Families (3)-(4)-(5) are not.}
\begin{enumerate}
\item\label{first_family} Let $a,b,c\in \Z$. 
{ The \emph{$a$-th Hirzebruch surface} $\F_a$ can be defined as the quotient of $(\A^2\setminus \{0\})^2$ by the action of $(\G_m)^2$ given by
\[\begin{array}{ccc}
(\G_m)^2 \times (\A^2\setminus \{0\})^2 & \to & (\A^2\setminus \{0\})^2\\
((\mu,\rho), (y_0,y_1,z_0,z_1))&\mapsto& (\mu\rho^{-a} y_0,\mu y_1,\rho z_0,\rho z_1)\end{array}\]
The class of $(y_0,y_1,z_0,z_1)$ will be written $[y_0:y_1;z_0:z_1]$. The projection 
\[\tau_a\colon\F_a\to\p^1, \ \ [y_0:y_1;z_0:z_1]\mapsto [z_0:z_1]\]
identifies $\F_a$ with $\P(\OP(a) \oplus \OP)$ as a $\p^1$-bundle over $\p^1$. }

{The disjoint sections $\s{-a},\s{a}\subset \F_a$ given by $y_0=0$ and $y_1=0$ have self-intersection $-a$ and $a$ respectively. The fibres $f\subset \F_a$ given by $z_0=0$ and $z_1=0$ are linearly equivalent and of self-intersection $0$. We moreover get $\Pic(\F_a)=\Z f\bigoplus \Z \s{-a}= \Z f\bigoplus \Z \s{a}$, since $\s{a}\sim\s{-a}+af$.}

\smallskip

We now define $\FF_a^{b,c}$ to be the quotient of $(\A^2\setminus \{0\})^3$ by the action of $\G_m^3$ given by
\[\begin{array}{ccc}
\G_m^3\times (\A^2\setminus \{0\})^3 & \to & (\A^2\setminus \{0\})^3\\
((\lambda,\mu,\rho), (x_0,x_1,y_0,y_1,z_0,z_1))&\mapsto& 
(\lambda\mu^{-b} x_0, \lambda\rho^{-c} x_1,\mu\rho^{-a} y_0,\mu y_1,\rho z_0,\rho z_1)\end{array}\]
The class of $(x_0,x_1,y_0,y_1,z_0,z_1)$ will be written $[x_0:x_1;y_0:y_1;z_0:z_1]$. The projection 
\[\FF_a^{b,c}\to \F_{a}, \ \ [x_0:x_1;y_0:y_1;z_0:z_1]\mapsto [y_0:y_1;z_0:z_1]\]
identifies $\FF^{b,c}_a$ with 
\[\P(\OFa(b \s{a}) \oplus \OFa(c f))=\P(\OFa \oplus \OFa(-b\s{a}+cf))\]
as a $\p^1$-bundle over $\F_a$.

Moreover, every fibre of the composed morphism $\FF^{b,c}_a \to \F_a \to \p^1$ given by the $z$-projection is isomorphic to $\F_b$ and the restriction of $\FF^{b,c}_a$ on the curves $\s{-a}$ and $\s{a}$ is isomorphic to $\F_c$ and $\F_{c-ab}$. 

As for Hirzebruch surfaces, one can reduce to the case $a\ge 0$, without changing the isomorphism class, by exchanging $y_0$ and $y_1$. We then observe that the exchange of $x_0$ and $x_1$ yields an isomorphism $\FF^{b,c}_a\simeq \FF^{-b,-c}_a$. We will then assume most of the time $a,b\ge 0$ in the following. If $b=0$, we can moreover assume $c\le 0$.
\smallskip

\item\label{DefiPPb} Let $b\in \Z$. We define $\PP_b$ to be the quotient of $(\A^2\setminus \{0\})\times (\A^3\setminus \{0\})$ by the action of $\G_m^2$ given by
\[\begin{array}{ccc}
\G_m^2\times (\A^2\setminus \{0\})\times (\A^3\setminus \{0\}) & \to & (\A^2\setminus \{0\})\times (\A^3\setminus \{0\})\\
((\mu,\rho), (y_0,y_1;z_0,z_1,z_2))&\mapsto& (\mu\rho^{-b} y_0,\mu y_1;\rho z_0,\rho z_1,\rho z_2)\end{array}\]
The class of $(y_0,y_1,z_0,z_1,z_2)$ will be written $[y_0:y_1;z_0:z_1:z_2]$. The projection 
\[\PP_b\to \p^2, \ \ [y_0:y_1;z_0:z_1:z_2]\mapsto [z_0:z_1:z_2]\]
identifies $\PP_b$ with 
\[\P(\O_{\P^2}(b) \oplus \O_{\P^2})=\P(\O_{\P^2} \oplus \O_{\P^2}(-b))\]
as a $\p^1$-bundle over $\p^2$. As before, we get an isomorphism of $\p^1$-bundles $\PP_{b}\simeq \PP_{-b}$ by exchanging $y_0$ with $y_1$, and will then often assume $b\ge0$ in the following.
\smallskip

\item\label{DefiUme} Let $a,b\ge 1$ and $c\ge 2$ be such that $c=ak+2$ with $0\le k\le b$. We 
call \emph{Umemura $\p^1$-bundle} the $\p^1$-bundle $\U_{a}^{b,c} \to \F_a$ 
obtained by the gluing of two copies of $\F_b\times\A^1$  along $\F_b\times\A^1 \setminus \{0\}$ by the automorphism $\nu\in \Aut(\F_b\times\A^1 \setminus \{0\})$,
\[\begin{array}{ccl}\nu\colon([x_0:x_1;y_0:y_1],z) &\mapsto &\left([x_0:x_1z^{c}+x_0 y_0^ky_1^{b-k}z^{c-1};y_0z^a:y_1],\frac{1}{z}\right),\\
&=&\left([x_0:x_1z^{c-ab}+x_0 y_0^ky_1^{b-k}z^{c-ab-1};y_0:y_1z^{-a}],\frac{1}{z}\right)\end{array}.\] 
The structure morphism  $\U_{a}^{b,c} \to \F_a$ sends $([x_0:x_1;y_0:y_1],z)\in \F_b\times\A^1$ onto respectively 
$[y_0:y_1;1:z]\in \F_a$ and $[y_0:y_1;z:1]\in \F_a$ on the two charts. 
\smallskip

\item Let $b \geq 1$. The $\P^1$-bundle $\V_b \to \P^2$ is the $\P^1$-bundle obtained from $\U_{1}^{b,2} \to \F_1$ by contracting the $-1$-section $\F_1 \to \P^2$. The existence of the $\P^1$-bundle $\V_b \to \P^2$ follows from the descent Lemma obtained in \cite[Lemma~2.3.2]{BFT}  (see also Lemma~\ref{Lem:DescentF1P2} below).
\smallskip

\item\label{DefiSSb} Let $b\ge -1$ and let $\kappa\colon~\p^1\times\p^1 \to \p^2$ be the $(2:1)$-cover defined by
\[\begin{array}{rccc}
\kappa\colon~& \p^1\times\p^1 & \to &\p^2\\
& ([y_0:y_1],[z_0:z_1]) &\mapsto &[y_0 z_0:y_0 z_1+y_1 z_0:y_1z_1],\end{array}\]whose ramification locus is the diagonal $\Delta\subseteq\p^1\times\p^1$ and whose branch locus is the smooth conic $\Gamma=\{ [X:Y:Z] \mid Y^2=4XZ\}\subseteq\p^2.$
 The \emph{$b$-th Schwarzenberger $\p^1$-bundle} $\SS_b\to \p^2$ is the $\p^1$-bundle defined by \[\SS_b=\P(\kappa_* \O_{\p^1\times\p^1}(-b-1,0))\to \p^2.\]
Note that $\SS_b$ is the projectivisation of the classical Schwarzenberger vector bundle $\kappa_* \O_{\p^1\times\p^1}(-b-1,0)$ introduced in \cite{Sch61}. Moreover, the preimage of a tangent line to $\Gamma$ by $\SS_b \to \P^2$ is isomorphic to $\F_b$ for each $b \geq 0$ (see \cite[Lemma~4.2.5(1)]{BFT}). This explains the shift in the notation.
\smallskip

\item\label{DefiRmn} We recall that any vector bundle over $\p^1$ is split (see e.g.~\cite{HM82}), and so a $\p^2$-bundle over $\p^1$ is isomorphic to 
\[\RR_{m,n}= \P(\O_{\p^1}(-m) \oplus \O_{\p^1}(-n) \oplus \O_{\p^1}) \text{ for some } m,n \in \Z.\]
The $\p^2$-bundle $\RR_{m,n}$ identifies with the quotient of $(\A^3\setminus \{0\})\times (\A^2\setminus \{0\})$ by the action of $\G_m^2$ given by
\[\begin{array}{ccc}
\G_m^2\times (\A^3\setminus \{0\})\times (\A^2\setminus \{0\}) & \to & (\A^3\setminus \{0\})\times (\A^2\setminus \{0\}) \\
((\lambda,\mu), (x_0,x_1,x_2,y_0,y_1))&\mapsto& 
(\lambda \mu^{-m} x_0, \lambda \mu^{-n} x_1, \lambda x_2,\mu y_0,\mu y_1)\end{array}\]
The class of $(x_0,x_1,x_2,y_0,y_1)$ is written $[x_0:x_1:x_2; y_0:y_1]$. 
Then the structure morphism $\RR_{m,n} \to \p^1$ identifies with the projection $[x_0:x_1:x_2;y_0:y_1]\mapsto [y_0:y_1]$. Also, the permutations of $x_0,x_1$ and  $x_1,x_2$ give isomorphisms of $\p^2$-bundles $\RR_{m,n}\iso \RR_{n,m}$ and $\RR_{m,n}\iso \RR_{m-n,-n}$. Hence,  up to an isomorphism that permutes the coordinates $x_0,x_1,x_2$, we may assume that $m \geq n \geq 0$.
\smallskip

\item\label{DefiWb}
For each $b\ge 2$, and when the field $\k$ has characteristic $\not=2$, the variety $\W_b$ is the toric threefold defined as the quotient of $(\A^2\setminus \{0\})\times (\A^3\setminus \{0\})$ by the action of $\G_m^2$ given by
\[\begin{array}{ccc}
\G_m^2\times (\A^2\setminus \{0\})\times (\A^3\setminus \{0\}) & \to & (\A^2\setminus \{0\})\times (\A^3\setminus \{0\})\\
((\mu,\rho), (y_0,y_1;z_0,z_1,z_2))&\mapsto& (\mu\rho^{-(2b-1)} y_0,\mu y_1;\rho z_0,\rho z_1,\rho^2 z_2)\end{array}.\]
The class of $(y_0,y_1,z_0,z_1,z_2)$ will be written $[y_0:y_1;z_0:z_1:z_2]$. The projection 
\[\W_b\to \p(1,1,2), \ \ [y_0:y_1;z_0:z_1:z_2]\mapsto [z_0:z_1:z_2]\]
yields a $\p^1$-fibration over $\P(1,1,2)$ which is a $\P^1$-bundle over $\P(1,1,2) \setminus [0:0:1]$. Moreover, using tools from toric geometry (see e.g.~\cite[Chapter~14]{Mat02}), we verify that $\W_b\to \P(1,1,2)$ is a Mori fibration. Indeed, the conditions of Definition~\ref{Df:MoriFibration} are easily checked from the fans $\Sigma_1$ and $\Sigma_2$ of $\W_b$ and $\P(1,1,2)$ respectively; for instance, the variety $\W_b$ is $\Q$-factorial with  terminal singularities if and only if each cone of $\Sigma_1$ is simplicial and, for each cone $\sigma$ of $\Sigma_1$, the only lattice points contained in the convex hull of the vertices of $\sigma$ are the vertices of $\sigma$.
Moreover, we verify that $\W_b$ has exactly two singular points, namely $[1:0;0:0:1]$ and $[0:1;0:0:1]$, both located in the fibre over $[0:0:1]$ and both having a neighbourhood isomorphic to $\A^3/\{\pm\mathrm{id}\}$  locally isomorphic to the vertex of the cone over the Veronese surface in $\P^5$; in particular, $\W_b$ is $\Q$-Gorenstein of index $2$.

\end{enumerate}
\smallskip

\begin{remark}\label{Rem:OpenSubsetsUV}
The open subsets $U_1,U_2 \subseteq   \RR_{m,n}$ given respectively by $y_0 \not=0$ and $y_1 \not=0$ are canonically isomorphic to $\p^2\times\A^1$, via $([x_0:x_1:x_2],t)\mapsto ([x_0:x_1:x_2;1:t])$ and $([x_0:yx_1:x_2],t)\mapsto ([x_0:x_1:x_2;t:1]).$ 
On the intersection, the gluing function is the birational involution of $\p^2\times\A^1\setminus \{0\}$ given by $([x_0:x_1:x_2],t)\mapsto ([t^mx_0:t^nx_1:x_2],\frac{1}{t})$.
\end{remark}

The following result, proven in \cite{BFT}, is a first reduction result in characteristic zero. We will use it to refine Theorem~\ref{th:A} into Theorem~\ref{th:first classification maximality}.
\begin{theorem}   \label{thBFT} \emph{(weak version of \cite[Theorem~A]{BFT})} \\
Assume that $\car(k)=0$.
Let $\pi\colon~\hat{X}\to S$ be a $\p^1$-bundle over a smooth projective rational surface $S$. Then there is an $\Autz(\hat X)$-equivariant birational map $ \hat X \dashrightarrow X$, where $X$ is one of the following $\P^1$-bundles $($with the notation above$)$:  
\begin{center}
\begin{tabular}{llllllll}
$(a)$& a decomposable &$\p^1$-bundle & $\FF_a^{b,c}$&\hspace{-0.3cm}$\longrightarrow$& \hspace{-0.2cm}$\F_a$& with $a,b\ge 0$, $a\not=1$, $c\in \Z$, \\
&&&&&&$c\le 0$ if $b=0$,\\
&&&&&&and where $a=0$ or $b=c=0$\\
&&&&&&or $-a<c<ab$;\\
$(b)$& a decomposable &$\p^1$-bundle &$\PP_b$&\hspace{-0.3cm}$\longrightarrow$& \hspace{-0.2cm}$\p^2$& for some $b\ge 0$;\\
$(c)$& an Umemura &$\p^1$-bundle &$\U_a^{b,c}$&\hspace{-0.3cm}$\longrightarrow$& \hspace{-0.2cm}$\F_a$& for some $a,b\ge 1, c\ge 2$,\\
&&&&&& with $c-ab<2$ if $a \geq 2$,\\
&&&&&& and $c-ab<1$ if $a=1$;\\
$(d)$ &a Schwarzenberger\!\! &$\p^1$-bundle &$\SS_b$&\hspace{-0.3cm}$\longrightarrow$& \hspace{-0.2cm}$\p^2$& for some $b\ge 1$; or\\
$(e)$ &a &$\p^1$-bundle &$\V_{b}$&\hspace{-0.3cm}$\longrightarrow$& \hspace{-0.2cm}$\p^2$& for some $b\ge 2$.
\end{tabular}
\end{center}
\end{theorem}

\begin{theorem} \label{th:first classification maximality} 
Assume that $\car(\k)=0$, and let $\hat{X}$ be a rational projective threefold. Then there is an $\Autz(\hat X)$-equivariant birational map $\hat{X} \dashrightarrow X$, where $X$ is a Mori fibre space that satisfies one of the following conditions:
\begin{enumerate}
\item \label{surface} $X \to S$ is one of the $\P^1$-bundles listed in Theorem~$\ref{thBFT}$;  
\item\label{line} $X \to \p^1$ is either a $\p^2$-bundle or a smooth Umemura quadric fibration $\QQ_g\to \p^1$ $($see Definition~$\ref{def:QQg})$ with $g \in \k[u_0,u_1]$ a square-free homogeneous polynomial of degree $2n \geq 2$; or 
\item\label{point} $X$ is a rational $\Q$-factorial Fano threefold of Picard rank $1$ with  terminal singularities.
\end{enumerate}
\end{theorem}

\begin{proof}
First, since $\car(\k)=0$, we may always assume that $\hat{X}$ is smooth by applying an equivariant resolution of singularities. Then we apply Theorem~\ref{th:A}, followed by Theorem~\ref{thBFT} to reduce the case  Theorem~\ref{th:A}\ref{thAP1} to the case of the $\P^1$-bundles over $\P^2$, $\P^1 \times \P^1$, or $\F_a$ (with $a \geq 2$) listed in Theorem~\ref{thBFT}. This gives the list \ref{surface}-\ref{line}-\ref{point} of Mori fibre spaces given in the statement of Theorem~\ref{th:first classification maximality}.
\end{proof}

\subsection{General results on \texorpdfstring{$\p^1$}{P1}-bundles over Hirzebruch surfaces} \label{subsecP1 bundles over Fa}
In this subsection we collect some results on equivariant links starting from $\P^1$-bundles over Hirzebruch surfaces (mostly proven in \cite{BFT}) that we will use to prove Theorems~\ref{th:Ea} and~\ref{th:Eb}.

\smallskip

For each $\p^1$-bundle $\pi\colon X\to S$ over a smooth surface $S$, there are Sarkisov links obtained by blowing-up a section $s\subseteq X$ over a smooth curve $\Gamma\subseteq S$, and then contracting the strict transform of $\pi^{-1}(\Gamma)$. In the next lemma we apply this observation to the $\p^1$-bundles over Hirzebruch surfaces listed in Theorem~\ref{thBFT}. Two types of such $\p^1$-bundles  arise, namely the decomposable $\p^1$-bundles $\FF_a^{b,c} \to \F_a$ and the Umemura $\p^1$-bundles $\U_a^{b,c} \to \F_a$.
\begin{lemma}\label{LinksIIbetweenFU}\item\begin{enumerate}
\item\label{LinksIIbetweenFabc}
For all $a,b,c\in \Z$, $a,b\ge 0$, the blow-up of the curve $l_{00}\subseteq \FF_a^{b,c}$, given by $x_0=y_0=0$, followed by the contraction of the strict transform of the surface $\pi^{-1}(\s{-a})$ onto $l_{10}\subseteq \FF_a^{b+1,c+a}$, given by $x_1=y_0=0$, yields a type \II Sarkisov link 
\[\begin{array}{rccc}
\varphi\colon &\FF_a^{b,c}&\dasharrow &\FF_a^{b+1,c+a}\\
&([x_0:x_1;y_0:y_1;z_0:z_1]) &\mapsto & ([x_0:x_1y_0;y_0:y_1;z_0:z_1]).\end{array}\]
 We have then $\varphi\Autz(\FF_a^{b,c})\varphi^{-1}\subseteq \Autz(\FF_a^{b+1,c+a})$ if and only if $ab>0$ or $ac<0$, and $\varphi^{-1}\Autz(\FF_a^{b+1,c+a})\varphi\subseteq \Autz(\FF_a^{b,c})$ if and only if $a(c+a)>0$.
 \item\label{LinksIIbetweenUabc} For each Umemura $\p^1$-bundle $\U_a^{b,c}\to \F_a$, the blow-up of the curve $l_{00}\subseteq \U_a^{b,c}$, given by $x_0=y_0$ on both charts, followed by the contraction of the strict transform of the surface $\pi^{-1}(\s{-a})$ onto the curve $l_{10}$, given by $x_1=y_0=0$ on both charts, yields a type \II equivariant link
\[\varphi\colon \U_a^{b,c}\dasharrow \U_a^{b+1,c+a}\] satisfying $\varphi \Autz(\U_a^{b,c})\varphi^{-1}=\Autz(\U_a^{b+1,c+a})$.
 \end{enumerate}
\end{lemma}
\begin{proof}
Part~\ref{LinksIIbetweenFabc} is given by \cite[Lemma 5.4.2]{BFT} and 
Part~\ref{LinksIIbetweenUabc} is given by \cite[Lemma 5.5.3]{BFT}.
\end{proof}
The next lemma applies to any $\p^1$-bundle over a surface $\hat{S}$ obtained by blowing-up a surface $S$. As we will only use it for $\hat{S}=\F_1$ and $S=\p^2$, we prefer to state it only in this particular situation.
\begin{lemma}\label{Lem:DescentF1P2}
Let $\pi\colon X\to \F_1$ be a $\p^1$-bundle, and let $\tau\colon \F_1\to \p^2$ be the contraction of $s_{-1}$ onto the point $p\in \p^2$. There is then a $\p^1$-bundle $\pi'\colon X'\to \p^2$ and an  $\Autz(X)$-equivariant birational map $\varphi\colon  X\dasharrow X'$, unique up to isomorphisms of $\p^1$-bundles,  such that $\tau\pi=\pi'\varphi$. 
\end{lemma}
\begin{proof}
The existence and uniqueness of $\varphi$, together with the fact that $\varphi$ is $\Autz(X)$-equivariant, are given by the ``descent lemma'' \cite[Lemma 2.3.2]{BFT}. 
\end{proof}
Lemma~\ref{Lem:DescentF1P2} applied to $\U_1^{b,2}\to \F_1$ gives a type \III equivariant link $\U_1^{b,2}\to \V_b$, which is a divisorial contraction already described in \cite[Lemma 5.5.1]{BFT}.
\begin{lemma} \label{lem:sequence of links for Umemura bundles}
For each $b\ge 2$, there is a sequence of type \II equivariant links $\U_1^{b+n,2+n}\dasharrow \U_1^{b+n-1,2+n-1}$, for $n\ge 1$,  and a type \III equivariant link that is a birational morphism $\psi\colon \U_1^{b,2}\to \V_b$, that fit into a commutative diagram
\[\xymatrix@R=5mm@C=1cm{
    \cdots\ar@{-->}[r] &\U_1^{b+2,4}\ar@{-->}[r] \ar[drr] &\U_1^{b+1,3}\ar@{-->}[r]\ar[dr]  &\U_1^{b,2}  \ar[d]\ar[r]^-{\psi}& \V_b\ar[d]  \\
&&&\F_1 \ar[r]&   \p^2.}  \]
Moreover, for each $n\ge 0$, the induced birational map $\varphi_n\colon \U_1^{b+n,2+n}\dasharrow \V_b$ 
satisfies  $\varphi_n \Autz(\U_1^{b+n,2+n}) \varphi_{n}^{-1}=\Autz(\V_b)$.
\end{lemma}
\begin{proof}
The existence of the sequence of type \II equivariant links that conjugates $\Autz(\U_1^{b+n,2+n})$ to $\Autz(\U_1^{b+n-1,2+n-1})$
 follows from Lemma~\ref{LinksIIbetweenFU}\ref{LinksIIbetweenUabc}. The type \III link $\U_1^{b,2}\to \V_b$ is given by Lemma~\ref{Lem:DescentF1P2}, and the fact that $\psi \Autz(\U_1^{b,2}) \psi^{-1}=\Autz(\V_b)$ is given by \cite[Lemma~5.5.1(4)]{BFT}.
Thus, for each $n\ge 0$, the composition gives a birational map $\varphi_n\colon \U_1^{b+n,2+n}\dasharrow \V_b$ that conjugates $\Autz(\U_1^{b+n,2+n})$ to $\Autz(\V_b)$. 
\end{proof}

\subsection{Non-maximality results} \label{sec:non-maximality}
In this subsection we {give the existence of explicit equivariant birational maps between varieties listed in Theorems~\ref{th:A} and~\ref{th:first classification maximality}, proving the non-maximality of some of their automorphism groups in $\Bir(\p^3)$. This series} of lemmas that will be useful to prove Theorem~\ref{th:Ea} in \upshape\S~\ref{subsec:proof of th D}. We keep the same notation as in {\upshape\S~\ref{sec:first classification}}. 
\begin{lemma} \label{lem:aut P2 bundle}
{For each $i \in \Z$ we denote by $\k[y_0,y_1]_{i}\subseteq \k[y_0,y_1]$ the vector subspace of homogeneous  polynomials of degree $i$ $($which is $\{0\}$ if $i<0)$. For each $m,n \in \Z$, the group $\Autz(\RR_{m,n})$ consists of all elements of the form
\[ [x_0:x_1:x_2; y_0:y_1] \mapsto \left[  \begin{bmatrix}
 p_{1,0}& p_{2,n-m} & p_{3,-m}\\ p_{4,m-n} & p_{5,0} & p_{6,-n} \\ 
 p_{7,m} & p_{8,n} & p_{9,0}  \end{bmatrix} \begin{bmatrix} x_0 \\ x_1 \\ x_2 \end{bmatrix} ; ay_0+by_1:cy_0+dy_1\right].\]
 where $ p_{k,i}\in \k[y_0,y_1]_{i}$, for $k=1,\ldots,9$, $\left[\begin{smallmatrix}
 p_{1,0}& p_{2,n-m} & p_{3,-m}\\ p_{4,m-n} & p_{5,0} & p_{6,-n} \\ 
 p_{7,m} & p_{8,n} & p_{9,0}  \end{smallmatrix}\right] \in \GL_3(\k[y_0,y_1])$ and $\left[\begin{smallmatrix} a & b \\ c & d \end{smallmatrix}\right] \in \GL_2(\k)$. In particular, the variety $\RR_{m,n}$ is toric and the morphism $\RR_{m,n} \to \p^1$ yields a surjective group homomorphism 
\[\rho\colon~\Autz(\RR_{m,n}) \twoheadrightarrow \Aut(\p^1)=\PGL_2.\]}
\end{lemma}
\begin{proof}
{The existence of $\rho\colon\Autz(\RR_{m,n}) \to \Aut(\p^1)$ is given by Proposition~\ref{blanchard}. In addition, the group $\GL_2(\k)$ acts on $\RR_{m,n}$ by 
\[
\left(\left[\begin{smallmatrix} a & b \\ c & d \end{smallmatrix}\right],[x_0:x_1:x_2; y_0:y_1]\right)\mapsto [x_0:x_1:x_2; ay_0+by_1:cy_0+dy_1],
\] which proves the surjectivity of $\rho$. Then to determine $\Autz(\RR_{m,n})_{\P^1}$, the kernel of $\rho$ (case $a=d=1$, $b=c=0$ in the above description), we use the global description of $\RR_{m,n}$ given in \S~\ref{sec:first classification}\ref{DefiRmn} and the fact that the $\P^2$-bundle $\RR_{m,n} \to \P^1$ is trivial on any open subset of $\P^1$ isomorphic to $\A^1$.}
\end{proof}

\begin{lemma} \label{lem:Rmn cases non max}
Let $G=\Autz (\RR_{m,n})$ with $m \geq n \geq 0$.
\begin{enumerate}
\item \label{item R10}
Let $\varphi \colon~\RR_{1,0} \to \P^3$ be the blow-up of the line $[0:0:*:*]$. 
Then we have $\varphi \Autz(\RR_{1,0}) \varphi^{-1} \subsetneq \Aut(\P^3)$.
\item \label{item Rmn pas max}
If $2n\geq m>n \geq 1$, then there is a $\P^1$-bundle $X \to S$ over a smooth projective rational surface $S$ and an $\Autz(\RR_{m,n})$-equivariant birational map $\delta\colon \RR_{m,n}\dashrightarrow X $ such that $\delta \Autz(\RR_{m,n}) \delta^{-1} \subsetneq  \Autz(X)$.
\end{enumerate}
\end{lemma}

\begin{proof}
\ref{item R10}: The blow-up of the line $\ell=[0:0:*:*]$ in $\P^3$ can be written
\[ \varphi\colon~\RR_{1,0} \to \P^3, \ \ [x_0:x_1:x_2;y_0:y_1] \mapsto [x_0 y_0 : x_0 y_1: x_1 : x_2].\]
By Proposition~\ref{blanchard}, $\varphi$ is $\Autz(\RR_{1,0})$-equivariant. Therefore $\varphi \Autz(\RR_{1,0}) \varphi^{-1} \subsetneq \Aut(\P^3)$ as $\varphi \Autz(\RR_{1,0}) \varphi^{-1}$ must preserve the line $\ell$.

\ref{item Rmn pas max}:
If $2n \geq m>n \geq 1$, then $[0:0:1;*:*]$ is a $G$-invariant curve of $\RR_{m,n}$ {(follows from Lemma~\ref{lem:aut P2 bundle})} whose blow-up can be written 
\begin{equation*} 
\begin{array}{cccc}
\psi\colon~&\FF_{a}^{1,c}& \to & \RR_{m,n}\\
&\ [x_0:x_1;y_0:y_1;z_0:z_1]&\mapsto& [x_0y_0:x_0y_1:x_1;z_0:z_1].\end{array}
\end{equation*}
where $a=m-n>0$, and $c=-n<0$. Thus $  \psi^{-1} G \psi = \Autz(\FF_{a}^{1,c}) $ by Proposition \ref{blanchard}.
{
On the other hand, according to Lemma~\ref{LinksIIbetweenFU}\ref{LinksIIbetweenFabc}, there exists a $G$-equivariant birational map $\varphi\colon \FF_a^{1,c} \dashedrightarrow \FF_a^{2,c+1}$ such that $\varphi \Autz(\FF_a^{1,c}) \varphi^{-1} \subsetneq \Autz(\FF_a^{2,c+a})$. Hence, composing $\psi^{-1}$ and $\varphi$ gives a $G$-equivariant birational map $\delta\colon \RR_{m,n} \dashedrightarrow \FF_a^{2,c+a}$ such that $\delta G \delta^{-1}  \subsetneq \Autz(\FF_{a}^{2,c+a})$.
}
\end{proof}

\begin{lemma}\label{Lemm:Fa1cRac}
Let $a\ge 0$ and let $c\in \Z$. There is a birational morphism 
\[ \varphi\colon \FF_a^{1, c} \to \RR_{a, c},\ [x_0:x_1; y_0: y_1; z_0: z_1] \mapsto [x_0 y_0: x_1: x_0 y_1; z_0: z_1],\] 
which contracts the divisor $H_{x_0}=(x_0=0)$ onto the section $\ell=[0:1:0;*:*]$ of $\RR_{a,c}\to \p^1$. We then have $\varphi\Autz(\FF_a^{1,c})\varphi^{-1}\subseteq \Autz(\RR_{a,c})$ with an equality if and only if $c<0$.
\end{lemma}
\begin{proof}
We check that {$\varphi$} is the blow-up of the section of $\RR_{a,c}$ given by $x_0=x_2=0$. The inclusion $\varphi\Autz(\FF_a^{1,c})\varphi^{-1}\subseteq \Autz(\RR_{a,c})$ is given by Proposition~\ref{blanchard}. The equality holds if and only if $\ell$ is fixed by $\Autz(\RR_{a,c})$, which is equivalent to the condition $c<0$ (Lemma~\ref{lem:aut P2 bundle}).
\end{proof}

\begin{lemma} \label{lem:Fabc non maximality}
Let $X=\FF_a^{b,c}$ with $a\geq 0$, $b \geq 0$, and $c \in \Z$, and let $G=\Autz(X)$.
\begin{enumerate}
\item  \label{Fabc-itemb0} If $b=0$ and $-a<c<0$, then there is a birational map $\varphi\colon~X \dasharrow \RR_{a,c}$ such that $\varphi G \varphi^{-1} \subsetneq \Autz(\RR_{a,c+a})$.
\item  \label{Fabc-itemb1} If $b=1$ and $c\geq 0$, then there is a birational morphism $\varphi\colon~X \to \RR_{a,c}$ such that $\varphi G \varphi^{-1} \subsetneq \Autz(\RR_{a,c})$.
\item  \label{Fabc-itemb2} If $b \geq 2$, and $ab-a \leq c<ab$, then there is a birational map $\varphi\colon~X \dasharrow \RR_{a,c-a(b-1)}$ such that $\varphi G \varphi^{-1} \subsetneq \Autz(\RR_{a,c-a(b-1)})$.

\item  \label{Fabc-itemc1} If $a=0$ and $c=1$, then there is a birational morphism $\varphi\colon~X \to \RR_{0,b}$ such that $\varphi G \varphi^{-1} \subsetneq \Autz(\RR_{0,b})$.
\end{enumerate}
\end{lemma}
\begin{proof}
Assertion \ref{Fabc-itemb1} follows directly from Lemma~\ref{Lemm:Fa1cRac}.

\ref{Fabc-itemb0}: As $ac<0$, \cite[Lemma 5.4.2]{BFT} gives a $G$-equivariant birational map $\varphi\colon\FF_a^{0,c} \dashrightarrow \FF_{a}^{1,c+a}$ such that $\varphi G \varphi^{-1}=\Autz(\FF_{a}^{1,c+a})$, and then \ref{Fabc-itemb0} follows from \ref{Fabc-itemb1}.

\ref{Fabc-itemb2}:   As $ab-a \leq c<ab$, we have $a>0$. For all integers $b',c'\ge 1$, there is an $\Autz(\FF_a^{b',c'})$-equivariant link  $\FF_a^{b',c'} \dashrightarrow \FF_{a}^{b'-1,c'-a}$,  by \cite[Lemma 5.4.2]{BFT}. Iterating these links yields a  $G$-equivariant birational map $\varphi\colon\FF_a^{b,c} \dashrightarrow \FF_{a}^{b-r,c-ra}$ with $r$ the smallest non-negative integer such that $b-r\le 0$ or $c-ra\le 0$.
If $c=a(b-1)$, then $r=b-1$ and $\varphi G \varphi^{-1}=\Autz(\FF_{a}^{1,c-a(b-1)})$, so the result follows from \ref{Fabc-itemb1}. 
If $a(b-1)<c<ab$, then $r=b$ and $\varphi G \varphi^{-1}=\Autz(\FF_{a}^{0,c-ab})$, so the result follows from~\ref{Fabc-itemb0}.

\ref{Fabc-itemc1} As $\FF_0^{b,c}$ and $\FF_0^{c,b}$ are isomorphic, via $[x_0:x_1; y_0: y_1; z_0: z_1] \mapsto[x_1:x_0;  z_0: z_1;y_0: y_1]$, Assertion \ref{Fabc-itemc1} follows from \ref{Fabc-itemb1}.
\end{proof}

\begin{lemma} \label{lem:non-max case b=1}
Let $\varphi \colon~\PP_1 \to \P^3$ be the blow-up of the point $[0:0:0:1]$. Then $\varphi \Autz(\PP_1) \varphi^{-1} \subsetneq  \Aut(\P^3)$. 
\end{lemma}

\begin{proof}
The blow-up of the point $p_0=[0:0:0:1]$ in $\P^3$ is given by
\[\varphi\colon~\PP_1 \to \P^3,\ \ [y_0:y_1; z_0 : z_1: z_2] \mapsto [y_0 z_0 : y_0 z_1: y_0 z_2 : y_1].\]
By Proposition \ref{blanchard}, the morphism $\varphi$ is $\Autz(\PP_1) $-equivariant, hence $\varphi \Autz(\PP_1)  \varphi^{-1}\\ \subseteq  \Aut(\P^3)$. The inclusion is strict as $\varphi \Autz(\PP_1)  \varphi^{-1}$ fixes $p_0$.
\end{proof}

\begin{lemma} \label{lem:UmeNotMax}
Let $a,b \geq 1$ and $c  \geq 2$ be such that $c=ak+2$ with $0 \leq k \leq b$.
Let $\pi\colon~\U_{a}^{b,c} \to \F_a$ be an Umemura $\p^1$-bundle  and let $G=\Autz(\U_{a}^{b,c})$.
\begin{enumerate}
\item \label{UmeMax2} If $a=1$ and $b=c \geq 2$, then there is a birational map $\varphi\colon~\U_{a}^{b,c} \dasharrow Q_3 \subseteq \P^4$ to a smooth quadric $Q_3$, such that $\varphi G \varphi^{-1} \subsetneq \Autz(Q_3)=\PSO_5$. Moreover, there is an $\Autz(\U_1^{2,2})$-equivariant birational morphism $\U_1^{2,2}\to Q_3$ that contracts the two divisors $x_0=0$ and $y_0=0$ onto a line $\ell\subset Q_3$.
\item \label{UmeMax3} If $a \geq 2$, $b \geq 1$, and $c=2+a(b-1)$, then there is a birational map $\varphi\colon~\U_{a}^{b,c} \dasharrow \RR_{a-1,0}$ such that $\varphi G \varphi^{-1} \subsetneq \Autz(\RR_{a-1,0})$. In particular, $\U_a^{1,2}$ is isomorphic to the blow-up of $\RR_{a-1,0}$ along a section of the $\p^2$-bundle $\RR_{a-1,0}\to \p^1$.
\end{enumerate}
\end{lemma}

\begin{proof}
\ref{UmeMax2}: If $b=c \geq 2$, then \cite[Corollary 5.5.4]{BFT} yields the existence of a $G$-equivariant birational map $\delta\colon~\U_{1}^{b,c} \dashrightarrow \U_{1}^{2,2}$. Hence $\delta G \delta^{-1} \subseteq \Autz(\U_{1}^{2,2})$, and so it suffices to prove  \ref{UmeMax2} for $G=\Autz(\U_{1}^{2,2})$. 

We use the notation of {\upshape\S~\ref{sec:first classification}\ref{DefiUme}}. 
There is a morphism $\varphi$ from $\U_1^{2,2}$ to $\p^4$ given by
\[ \F_2\times\A^1\to \p^4, ([x_0:x_1;y_0:y_1],z)\to [x_0y_1^2+x_1z:x_1:x_0y_0^2z:x_0y_0^2:x_0y_0y_1]\] 
on the first chart, and thus by
\[ \F_2\times\A^1\to \p^4, ([x_0:x_1;y_0:y_1],z)\to [x_1:x_1z-x_0y_1^2:x_0y_0^2:x_0y_0^2z:x_0y_0y_1]\] on the second chart.
Outside of $x_0y_0=0$, it is an isomorphism with its image:
On each chart we obtain $\A^3\hookrightarrow \F_2\times\A^1$, $(x,y,z)\mapsto ([1:x;1:y],z)$, and the composition yields respectively
\[(x,y,z)\mapsto [y^2+xz:x:z:1:y]\]
\[(x,y,z)\mapsto [x:xz-y^2:1:z:y]\]
on the first and the second chart.
Hence, $\varphi$ is a birational morphism whose image is the quadric
\[Q_3=\{[x_0:\cdots:x_4]\in \p^4\mid x_0x_3-x_4^2-x_1x_2=0\} \subseteq \P^4.\]
More precisely, $\varphi$ is the blow-up  $\eta_1\colon \mathrm{Bl}_\ell(Q_3)\to Q_3$ of the line $\ell=(x_2=x_3=x_4=0)$ in $Q_3$ followed by the blow-up of a section $\ell_1\subset \mathrm{Bl}_\ell(Q_3)$ of $\eta_1^{-1}(\ell)\to \ell$. According to Proposition~\ref{blanchard}, it is $G$-equivariant, and so $\varphi G \varphi^{-1}$ is the seven-dimensional parabolic subgroup $P=\{ g \in \PSO_5 \ | \ g \cdot \ell=\ell\} \subsetneq \PSO_5$.

\ref{UmeMax3}: By \cite[Corollary 5.5.4]{BFT}, there is a $G$-equivariant birational map $\delta\colon\U_{a}^{b,c} \dashrightarrow \U_{a}^{1,2}$. Therefore $\delta G \delta^{-1} \subseteq \Autz(\U_{a}^{1,2})$ and so  it suffices to prove  \ref{UmeMax3} for $G=\Autz(\U_{1}^{1,2})$. 

Let $\ell$ be the closure in $X=\U_{a}^{1,2}$ of the line defined by $X$ $([0:1;*:*],0) \subseteq \F_1\times\A^1$ in the first chart of $X$ (see {\upshape\S~\ref{sec:first classification}\ref{DefiUme}} for the definition of $X$ with the two charts and the transition function). 
Then the class of $\ell$ generates an extremal ray of the cone of effective curves $\NE(X)$ and $K_X \cdot \ell=-1$ (see Lemma~\ref{lem:UmePicN}). 
The variety $X$ can be viewed as a $\F_1$-bundle over $\p^1$, and contracting the numerical class of $\ell$ corresponds to contract fibrewise the $(-1)$--section   $\F_1 \to \p^2, \, [x_0:x_1;y_0:y_1] \mapsto [x_0y_0:x_0y_1:x_1]$.
Using the transition function of $X$ given in {\upshape\S~\ref{sec:first classification}\ref{DefiUme}} we obtain a $\p^2$-bundle $X' \to \p^1$ whose transition function is
\[  \begin{array}{ccc}
\p^2\times\A^1 \setminus \{0\}& \dashrightarrow & \p^2\times\A^1 \setminus \{0\} \\
([u_0:u_1:u_2],t) & \mapsto & ([z^au_0:u_1:z^2u_2+zu_1],\frac{1}{t})
\end{array} \] 
Composing with automorphisms on both sides, we obtain
\[  \begin{array}{ccc}
\p^2\times\A^1 & \dashrightarrow & \p^2\times\A^1 \\
([u_0:u_1:u_2],t) & \mapsto & ([z^au_0:zu_1:zu_2],\frac{1}{t})
\end{array} \] 
and so according to Remark~\ref{Rem:OpenSubsetsUV}, we have $X' \simeq \RR_{a-1,0}$ . Therefore, there is a $G$-equivariant birational morphism $\varphi\colon~X \to \RR_{a-1,0}$, and so $\varphi G \varphi^{-1} \subseteq \Autz(\RR_{a-1,0})$. But according to Lemma~\ref{lem:aut P2 bundle}, the group $\Autz(\RR_{a-1,0})$ acts on $\RR_{a-1,0}$ with two orbits, which are a divisor and its complement, while $\varphi G \varphi^{-1}$ must preserve the section of $\RR_{a-1,0} \to \P^1$ obtained by contracting the numerical class of $\ell$. Hence, $\varphi G \varphi^{-1} \subsetneq \Autz(\RR_{a-1,0})$ and we get \ref{UmeMax3}.
\end{proof}

\begin{corollary} \label{cor:V2 not max}
There exists a birational morphism $\psi\colon \V_2 \to Q_3 \subseteq \P^4$ such that $\psi \Autz(\V_2) \psi^{-1} \subsetneq \Autz(Q_3)=\PSO_5$.
\end{corollary}

\begin{proof}
The morphism $\varphi\colon \U_{1}^{2,2} \to Q_3$ described in the proof of Lemma~\ref{lem:UmeNotMax}\ref{UmeMax2} factorises through the contraction morphism $\delta\colon \U_{1}^{2,2} \to \V_2$. Since $\delta \Autz(\U_{1}^{2,2}) \delta^{-1}=\Autz(\V_2)$, it follows that $\psi \Autz(\V_2) \psi^{-1} \subsetneq \Autz(Q_3)$.
\end{proof}

\smallskip

To study the Schwarzenberger $\p^1$-bundle $\SS_b\to \p^2$, we can consider the $\p^1$-bundle $\TT_b\to \p^1\times\p^1$ given by the pull-back $\SS_{b}\times_{\p^2} (\p^1\times\p^1)\to \p^1\times\p^1$ that comes from the double covering $\kappa\colon~\p^1\times\p^1\to \p^2$ defined in {\upshape\S~\ref{sec:first classification}\ref{DefiSSb}}. It is much easier to work with the $\p^1$-bundle $\TT_b$, which can be described by the following result:

\begin{lemma}\label{Lemm:AutTTb} $($\cite[Lemma 3.5.5]{BFT}$)$ Assume that $\car(\k)\not=2$. Let $b\ge 1$ be an integer. Let  $\pi$ and $\pi'$ denote the $\p^1$-bundles $\pi\colon~\TT_b\to \p^1\times\p^1$ and $\pi'\colon~\FF_0^{b+1,b+1}\to \p^1\times\p^1$ respectively. Then, the following hold.
\begin{enumerate}
\item\label{S1S2TTb}
For each $i=1,2$, denoting by $\mathrm{pr}_i\colon~\p^1\times\p^1\to \p^1$ the $i$-th projection, the morphism $\mathrm{pr}_i\circ \pi\colon~\TT_b\to \p^1$ is a $\F_b$-bundle.
Denoting by $S_i\subseteq \TT_b$ the union of the $(-b)$-curves of the $\F_b$'s, the intersection $C=S_1\cap S_2$ is a curve isomorphic to the diagonal $\Delta\subseteq \p^1\times\p^1$ via~$\pi$. The curve $C$ coincides with the intersection of $\pi^{-1}(\Delta)$ and the surface $x_0=0$ in both charts.
\item\label{CommDiag}
We have $\Autz(\TT_b)\simeq \PGL_2$, and a commutative diagram
\[\xymatrix@R=3mm@C=2cm{
& Z\ar[rd]^{\eta}\ar[ld]_{\epsilon}\\
     \TT_{b}\ar@{-->}[rr]^{\psi} \ar[rd]  && \FF_0^{b+1,b+1} \ar[ld] \\
    & \p^1\times\p^1
  }\]
  where all maps are $\PGL_2$-equivariant, the action of $\PGL_2$ on $\p^1\times\p^1$ is the diagonal one, the action of $\PGL_2$ on $\FF_0^{b+1,b+1}$ is given by 
\[ [x_0:x_1;y_0:y_1;z_0:z_1]\mapsto [x_0:x_1;\alpha y_0+\beta y_1:\gamma y_0+\delta y_1;\alpha z_0+\beta z_1:\gamma z_0+\delta z_1],\]
the morphism $\eta$ the blow-up of the curve $C'\subseteq \FF_0^{b+1,b+1}$ given by $\p^1\hookrightarrow \FF_0^{b+1,b+1}$, $[u:v]\mapsto [1:1:u:v:u:v]$, and the morphism $\epsilon$ is the blow-up of the curve $C$. 
  
  Moreover, {the group $\Autz(\TT_b)_{\P^1 \times \P^1}$} is trivial.
\item\label{HatS:OnlyInvariantCurve}
If $\car(\k)$ does not divide $b+1$, the curve $C$ is the unique curve invariant by $\Autz(\TT_b)$.
  \end{enumerate}
\end{lemma}

The above lemma gives a $\PGL_2$ action on $\TT_b$, which goes down to a $\PGL_2$-action on $\SS_b$, which is in fact the whole automorphism group $\Autz(\SS_b)$ for $b\ge 2$ (\cite[Lemma 4.2.5(2)]{BFT}).

We now determine the intersection form on $\SS_b$, the cone of effective curves $\NE(\SS_b)$, and the canonical divisor $K_{\SS_b}$.

\begin{lemma} \label{lem:SbPicN}
Assume that $\car(\k)\not=2$. 
Let $b \geq 2$, let $\pi\colon~X=\SS_b \to \p^2$ be the Schwarzenberger $\p^1$-bundle and let $\epsilon \colon~\TT_b\to \SS_b$ be the double covering  given by the projection $\TT_b=\SS_{b}\times_{\p^2} (\p^1\times\p^1)\to \SS_{b}$. Taking the notation of {\upshape\S~\ref{sec:first classification}\ref{DefiSSb}} and Lemma~$\ref{Lemm:AutTTb}$, the following hold:
\begin{enumerate}
\item\label{SbFourOrbits}
If $\car(\k)$ does not divide $b+1$, the action of $\Autz(\SS_b)\simeq \PGL_2$ on $\SS_b$ has four orbits: the curve $\gamma$,  the surfaces $D \setminus \gamma$ and $E \setminus \gamma$, and the open orbit $\SS_b \setminus (D \cup E)$, where $\gamma=\epsilon(C)\simeq \p^1$, $C\subseteq \TT_b$ is the curve as in Lemma~$\ref{Lemm:AutTTb}$, $E=\epsilon(S_1)=\epsilon(S_2)\simeq \p^1\times\p^1$ and $D=\pi^{-1}(\Gamma)$. If $\car(\k)$ divides $b+1$, the four subsets $\gamma$, $D \setminus \gamma$, $E \setminus \gamma$, and $\SS_b \setminus (D \cup E)$ are still $\Autz(\SS_b)$-invariant, but they are not the $\Autz(\SS_b)$-orbits anymore. 

\item\label{SbPicN1}
The group of numerical equivalence classes of $1$-cocycles $\NS_\Q(\SS_b)$ is generated by $H:=\pi^{-1}(\ell)$, with $\ell \subseteq \p^2$ a line, and by $E$.
\item\label{SbPicN2}
Let $f$ be  a fibre of $\pi$, and let $s_1$ and $s_2$ be the two curves of $E \subseteq \SS_b$ obtained respectively as the images of the two curves $s_1'=\{[1:0;0:1;u:v]\mid [u:v]\in \p^1\}$ and $s_2'=\{[0:1;0:1;u:v]\mid [u:v]\in \p^1\}$ of $\FF_{0}^{b+1,b+1}$ {through the birational map $\epsilon \circ \psi^{-1}$}. For $i=1,2$, the group of numerical equivalence classes of $1$-cycles $N_1^\Q(\SS_b)$ is generated by $f$ and $s_i$.  
\item\label{SbPicN3}
The intersection form on $X$ satisfies
\[  \begin{array}{|c|cc|}
\hline
& E & H\\
\hline
f &  2& 0\\
s_1&  1-b&1  \\
s_2& b+3 & 1 \\
\hline\end{array}\]
\item\label{SbPicN4}
The cone of effective curves $\NE(X)$ is generated by $s_1$ and $f$.
\item\label{SbPicN5}
The canonical divisor $K_X$ is $-E-2H$, and so $K_X\cdot f=-2$ and $K_X\cdot s_1=b-3$.
\end{enumerate}
\end{lemma}

\begin{proof}
We first observe that the involution $\hat\sigma\in \Aut(\FF_0^{b+1,b+1})$ given by  $[x_0:x_1;y_0:y_1;z_0:z_1] \mapsto [x_1:x_0;z_0:z_1;y_0:y_1]$ preserves the curve $C'$ of Lemma~\ref{Lemm:AutTTb}, so $\hat\sigma$ {induces an involution $\sigma':= \psi^{-1} \circ \hat\sigma \circ \psi \in \Aut(\TT_b)$}. We then prove that $\sigma'$ is equal to the involution $\sigma\in \Aut(\TT_b)$ associated with the double cover $\epsilon\colon~\TT_b\to \SS_b$. Indeed, both involutions are lifts of the involution $\tau\in \Aut(\p^1\times\p^1)$ that is the exchange of the two factors. Since every automorphism of the $\p^1$-bundle $\TT_b\to \p^1\times\p^1$ is trivial (Lemma~\ref{Lemm:AutTTb}\ref{CommDiag}), $\sigma=\sigma'$ is the unique lift of $\tau$ in $\Aut(\TT_b)$. 

Let us denote by $S'_1,S'_2\subseteq  \FF_0^{b+1,b+1}$ the sections given by $x_1=0$ and $x_0=0$ respectively. Let $q=[0:1]\in \p^1$ be a point (the same would work with another point). For $i=1,2$, we consider the preimage $F'_i\subseteq \FF_0^{b+1,b+1}$ of the fibre of the $i$-th projection $\p^1\times\p^1$ over $q$. {With the explicit description of $\FF_0^{b+1,b+1}$ given at the beginning of \upshape\S~\ref{sec:first classification}, we check} that $F'_1\simeq F'_2\simeq \F_{b+1}$, with exceptional sections given by $s_1'=F'_1\cap S'_1$ and $F'_2\cap S'_2$ and sections of self-intersection $b+1$ given by $s_2'=F'_1\cap S'_2$ and $F'_2\cap S'_1$. 
 As $C'\cap F'_1=[1:1;0:1;0:1]$ is outside of the exceptional section, the images $F_1,F_2\subseteq  \TT_b$ of $F'_1$, $F'_2$ {through $\psi^{-1}$} are isomorphic to $\F_b$, with exceptional sections being the intersections with the strict transforms of $S'_1$ and $S'_2$ respectively, which are then equal to $S_1$ and $S_2$. Hence, $\hat{s}_1=F_1\cap S_1$ and $F_2\cap S_2$ are the sections of self-intersection $-b$ of $F_1$ and $F_2$, and  $\hat{s}_2=F_1\cap S_2$ and $F_2\cap S_1$ are sections of self-intersection $b+2$ of $F_1$ and $F_2$ respectively. 

We are now ready to prove the lemma:

\ref{SbFourOrbits}: Since $\hat\sigma$ exchanges $S'_1$ and $S'_2$, the involution $\sigma$ exchanges $S_1$ and $S_2$, which yields $E=\epsilon(S_1)=\epsilon(S_2)\simeq \p^1\times\p^1$.  As $\epsilon\colon~\TT_b\to \SS_b$ is an isomorphism above the branch locus $\Gamma$, it suffices to show that the action of $\PGL_2$ on $\TT_b$ has five orbits: the curve $C$, the surfaces $S_1\setminus C$, $S_2\setminus C$, $\hat{D}\setminus C$ and the open orbit $\TT_b\setminus (\hat{D}\cup S_1\cup S_2)$, where $\hat{D}$ is the pull-back of the diagonal $\Delta\subseteq \p^1\times\p^1$. The fact that the surfaces $S_1,S_2,\hat{D}$ are invariant directly follows from the description of the action of $\PGL_2$ on $\FF_0^{b+1,b+1}$ given in Lemma~\ref{Lemm:AutTTb}, as $S_1$ and $S_2$ are correspond respectively to the surfaces $S'_1$ and $S'_2$. Hence, $C=S_1\cap S_2$ is also invariant, and is one orbit as its image in $\p^1\times\p^1$ is $\Delta$ (Lemma~\ref{Lemm:AutTTb}\ref{S1S2TTb}). The fact that  $\hat{D}\setminus C$ is an orbit follows from the fact that $\hat{D}\to \Delta$ is a $\p^1$-bundle and that $C$ is the only curve invariant in $\TT_b$ (Lemma~\ref{Lemm:AutTTb}\ref{HatS:OnlyInvariantCurve}). It remains to see that $S_1\setminus C$, $S_2\setminus C$ and $\TT_b\setminus (\hat{D}\cup S_1\cup S_2)$ are orbits, which corresponds to ask that $S'_1\setminus C'$, $S'_2\cap C'$ and $\FF_0^{b+1,b+1}\setminus (S'_1\cup S'_2\cup D')$ are orbits, where $D'\subseteq \FF_0^{b+1,b+1}$ is the preimage of the diagonal. For $S'_i\cap C'$, $i=1,2$, this directly follows from the explicit action given in Lemma~\ref{Lemm:AutTTb}. For $\FF_0^{b+1,b+1}\setminus (S'_1\cup S'_2\cup D')$, we observe that $[1:1;0:1;1:0]$ is sent onto $[1:1;\beta :\delta;\alpha:\delta ]$, so its orbit contains the whole fibre over $[0:1;1:0]$, except the two points of $S_1'$ and $S_2'$.

\ref{SbPicN1}-\ref{SbPicN2}: As $\pi\colon~\SS_b\to \p^2$ is a $\p^1$-bundle, we only need the preimage of a non-trivial element of $\Pic(\p^2)$ (respectively of a point) and a divisor (respectively a curve) not contracted by $\pi$ to generate $\NS_\Q$ (respectively $N_1^\Q$).

\ref{SbPicN3}:
The restriction $\pi_{|E}:E \to \p^2$ is a $(2:1)$-cover ramified over the diagonal $\Delta \subseteq \p^1\times\p^1$, and so $E \cdot f=2$. Also, since all fibres of $\pi$ are linearly equivalent, we can assume that $f \cap H= \emptyset$, hence $H \cdot f=0$. Choosing for $\ell$ the tangent line to the conic $\Gamma \subseteq \p^2$, such that $\kappa^{-1}(\ell)\subseteq  \p^1\times\p^1$ is the union of the fibres over $q=[0:1]$, we get $\pi^{-1}(\ell)=\epsilon(F_1)=\epsilon(F_2)\simeq \F_b$. The restriction of $E$ to $\pi^{-1}(\ell)$ corresponds to the union of the two curves $s_1=\epsilon(F_1\cap S_1)=\epsilon(F_2\cap S_2)$ and $s_2=\epsilon(F_1\cap S_2)=\epsilon(F_2\cap S_1)$. Hence, the intersection $E\cdot s_i$ in $\SS_b$ can computed by the intersection $s_i\cdot (s_1+s_2)$ in the surface $\pi^{-1}(\ell)\simeq \F_b$. As $(s_1)^2=-b$, $s_1\cdot s_2=1$ and $s_2^2=b+2$, we get $s_1\cdot E= 1-b$ and $s_2\cdot E=b+3$.

\ref{SbPicN4}: We already know that the curve $f$ can be contracted, this corresponds to $\pi\colon~\SS_b \to \p^2$, thus the cone $\NE(X)$ is generated by $f$ and some effective curve $r=\alpha s_1 + \beta f$. If the curve $r$ is contained in $E$, then $r \in \Q_+ \left \langle s_1, s_2 \right \rangle \subseteq  \Q_+ \left \langle s_1, f \right \rangle$, as $s_2=s_1+(b+1)f$, and so we must have $r=s_1$. If $r$ is not contained in $E$, then $E \cdot r =\alpha(1-b)+2\beta \geq 0$. Also, we always have $H \cdot r =\alpha \geq 0$. It follows that, if $r \notin E$, then $\alpha, \beta \geq 0$, and thus again we must have $r=s_1$.

\ref{SbPicN5}: Let $K_X=\alpha E +\beta H$ be the canonical divisor of $X$. Recall that $E \simeq \p^1\times\p^1$ and $H \simeq \F_b$, therefore $K_E=-2s_1-2s_2$ and $K_H=-2s_1-(b+2)f$. Applying the adjunction formula yields $K_H=(K_X+H)\cdot H=2 \alpha s_1+ (\alpha (b+1)+\beta+1)f$ by \ref{SbPicN3}. Hence, $\alpha=-1$ and $\beta=-2$, and we obtain $K_X=-E-2H$.
\end{proof}

\begin{lemma} \label{lem:S2 not max}
Assume that $\car(\k)=0$.
Let $G=\Autz(\SS_2) \simeq \PGL_2$. Then there is a  $G$-equivariant birational morphism $\varphi\colon~\SS_2 \to \P^3$, that is the blow-up of a twisted cubic, such that $\varphi G \varphi^{-1} \subsetneq  \Aut(\P^3) \simeq \PGL_4$.
\end{lemma}

\begin{proof}
Let $\rho\colon \p^1\to \p^3, [u:v]\mapsto [u^3:u^2v:uv^2:v^3]$ be the standard parametrisation of the twisted cubic $\Gamma=\rho(\p^1)\subset \p^3$. The natural action of $\SL_2$ on $\k[u,v]_3$ gives rise to an action of $\Aut(\p^1)=\PGL_2$ on $\p^3$ that makes $\rho$ equivariant. Let $\varphi\colon X\to \p^3$ be the blow-up of $\Gamma$. By Proposition~\ref{blanchard}, the group $\Autz(X)$ is conjugate via $\varphi$ to the group of automorphisms of $\p^3$ that preserve the twisted cubic. {The group $\Autz(X)$ is therefore} isomorphic to $\PGL_2$, as no non-trivial element of $\Aut(\p^3)$ can fix  $\Gamma$ pointwise (the fixed locus of an element of $\Aut(\p^3)$ is a union of linear subspaces). The linear system of quadrics through $\Gamma$ gives a $\p^1$-bundle $X\to \p^2$: this can be checked in coordinates and is also very classical. As $\Autz(X)=\PGL_2$, the $\p^1$-bundle is not decomposable and is in fact isomorphic to a Schwarzenberger bundle $\SS_n$ for some $n\ge 1$ \cite[Proposition 4.3.4]{BFT}. The case $n=1$ is impossible as $\Autz(\SS_1)\simeq \PGL_3$ (because $\SS_1\to \p^2$ is the projectivisation of the tangent bundle, see  \cite[Corollary 4.2.2]{BFT}). Then, Lemma~\ref{lem:SbPicN}\ref{SbPicN5} implies that $n=2$.
\end{proof}

\section{Description of the equivariant links} \label{sec:maximality}
In this section we describe the equivariant Sarkisov links between certain Mori fibrations: the $\P^2$-bundles over $\P^1$, the Umemura quadric fibrations $\QQ_g \to \P^1$, and the $\P^1$-bundles over $\P^2$, $\P^1 \times \P^1$ and $\F_n$ ($n \geq 2$) listed in Theorem~\ref{thBFT}. Then in \upshape\S~\ref{subsec:proof of th D} we give the proof of the main results of this article when the base field is assumed to be of characteristic zero (Theorems~\ref{th:Ea} and~\ref{th:Eb}).

\subsection{Homogeneous  spaces}\label{subsec:homog case}
In this subsection we consider the equivariant links starting from a Mori fibre space $X$ on which $\Autz(X)$ acts transitively.

\begin{lemma}  \label{lemma homogeneous  case}
Let $G$ be an algebraic group and let $\varphi\colon~X \dasharrow Y$ be a $G$-equivariant birational map between two projective varieties equipped with a regular $G$-action. 
If $X$ is $G$-homogeneous $($i.e.~if $G$ acts transitively on $X)$, then $\varphi$ is an isomorphism. In particular, if $\varphi$ is a Sarkisov link, then $\varphi$ is of type \IV.
\end{lemma}

\begin{proof}
Since $X$ is $G$-homogeneous, the rational map $\varphi$ has no base-point and does not contract any hypersurface. It is then an isomorphism between  $X$ and $\varphi(X)$, and since $X$ is projective we have that $\varphi(X)=Y$.
\end{proof}

\begin{proposition}\label{Prop:HomSpaces}
Let $X$ be one of the following variety: 
$$\P^3, \ Q_3\subseteq \P^4, \ \p^1\times\p^1\times\p^1, \ \p^2\times\p^1,\text{or } \SS_1 \simeq \P(T_{\P^2}).$$
Then $\Autz(X)$ acts transitively on $X$, so every $\Autz(X)$-equivariant link starting from $X$ is an isomorphism and a type \IV link. There is no such link in Cases~\ref{HomS1}-\ref{HomS2}, there are two links in Case~\ref{HomS3}, and one link in Cases~\ref{HomS4}-\ref{HomS5}.
\begin{enumerate}
\item\label{HomS1} $X=\P^3$ and $\Aut(X) \simeq \PGL_4$.
\item\label{HomS2} $X=Q_3 \subseteq   \P^4$ is a smooth quadric and $\Aut(X) \simeq \PSO_5$.    
\item\label{HomS3} $X=\p^1\times\p^1\times\p^1\simeq \FF_0^{0,0}$ and $\Autz(X)= \PGL_2\times\PGL_2\times\PGL_2$; the two links are then given by
\[\xymatrix@R=5mm@C=.3cm{
   \p^1\times\p^1\times\p^1  \ar[rr]^{\mathrm{id}}_{\simeq} \ar[d]_{\mathrm{pr}_1\times\mathrm{pr}_3}  & &  \p^1\times\p^1\times\p^1  \ar[rr]^{\mathrm{id}}_{\simeq} \ar[d]_{\mathrm{pr}_1\times\mathrm{pr}_2}  & &  \p^1\times\p^1\times\p^1 \ar[d]_{\mathrm{pr}_2\times\mathrm{pr}_3}  \\
     \p^1\times\p^1 \ar[dr]_{\mathrm{pr}_1}  & &\p^1\times\p^1 \ar[dl]_{\mathrm{pr}_1}  \ar[dr]^{\mathrm{pr}_2}   & &\p^1\times\p^1 \ar[ld]^{\mathrm{pr}_1}\\
   & \p^1&& \p^1.
    }  \]
\item\label{HomS4} $X=\p^2\times\p^1$ and $\Aut X \simeq \PGL_3\times\PGL_2$; the link is then 
\[\xymatrix@R=5mm@C=.3cm{
    \PP_0\simeq \p^1\times\p^2  \ar[rr]^{\varphi}_{\simeq} \ar[d]  & &\p^1\times\p^2\simeq \RR_{0,0} \ar[d] \\
     \p^2 \ar[dr]  & & \p^1 \ar[ld]\\
   & {\mathrm{pt}}.
    }  \]

\item\label{HomS5} $X=\{([x_0:x_1:x_2],[y_0:y_1:y_2])\in \p^2\times\p^2, \sum_{i=0}^3 x_i y_i=0\}\simeq \SS_1$, which coincides with the projectivisation of the tangent bundle of $\p^2$, and $\Autz(X)=\PGL_3$;
the link is then 
\[\xymatrix@R=5mm@C=.3cm{
    X  \ar[rr]^{\mathrm{id}}_{\simeq} \ar[d]_{\mathrm{pr}_1}  & &X \ar[d]_{\mathrm{pr}_2} \\
     \p^2 \ar[dr]  & & \p^2 \ar[ld]\\
   & {\mathrm{pt}}.
    }  \]
\end{enumerate}
\end{proposition}

\begin{proof}
In all the cases listed above, $\Autz(X)$ acts transitively on $X$. 
{This is clear in all cases except maybe when $X=\SS_1 \to \p^2$ is the projectivisation of the tangent bundle, but then $X \simeq \PGL_3/P$, where $P$ is a maximal parabolic subgroup, and so $\Autz(X) \simeq \PGL_3$ acts indeed transitively on $X$. 
}

The fact that every equivariant link starting from $X$ is an isomorphism follows from Lemma \ref{lemma homogeneous  case}. Every such link is then of type \IV, and corresponds to the different contractions of the extremal rays of $\NE(X)$ negative against $K_X$ that we can obtain, all being Mori fibre spaces. If $X= \p^3$ or $X=Q_3 \subseteq   \P^4$ is a smooth quadric, the Picard rank is equal to $1$, so there is no equivariant link. If $X=\p^1\times\p^1\times\p^1$, we have three extremal rays, one gives the Mori fibre space and the two others give two links as in \ref{HomS3}. If $X=\p^2\times\p^1$, the two contractions correspond to the $\PP_0\to \p^2$ and $\RR_{0,0}\to \p^1$, with one link as in \ref{HomS4}. 
If $X=\{([x_0:x_1:x_2],[y_0:y_1:y_2])\in \p^2\times\p^2, \sum_{i=0}^3 x_i y_i=0\}$, then $X$ it is isomorphic to $\SS_1$, or to the projectivisation of the tangent bundle of $\p^2$. The Picard rank being of rank $2$, there are exactly two contractions, and we get the link of \ref{HomS5}.
\end{proof}

\subsection{Decomposable \texorpdfstring{$\p^1$}{P1}-bundles and Schwarzenberger \texorpdfstring{$\p^1$}{P1}-bundles over \texorpdfstring{$\p^2$}{P2}}\label{Decompo P1 over P2}
In this subsection we consider the equivariant links starting from $\PP_b$, from  $\P(1,1,1,2)$, and from $\SS_b$ {(see \S~\ref{sec:first classification} for the definition of $\PP_b$ and $\SS_b$)}.

\begin{proposition}\label{LinkFromPPb}
Assume that $\car(\k) \neq 2$.
Let $b \geq 2$, and let $G=\Aut(\PP_b)$. 
There is an equivariant link from $\PP_b$ if and only if $b=2$. This link, which is unique, is the contraction of the unique $G$-invariant divisor $\varphi\colon~\PP_2 \to \P(1,1,1,2)$, and $\varphi G \varphi^{-1} =  \Autz(\P(1,1,1,2))$. Moreover,  $\varphi^{-1}$ is the unique equivariant link starting from $\P(1,1,1,2)$. 
 \end{proposition}

\begin{proof}
We use the same notation as in {\upshape\S~\ref{sec:first classification}\ref{DefiPPb}}.
For $b \geq 2$, it follows from \cite[Lemma 4.1.2 and Remark 4.1.3]{BFT} that $\PP_b$ is the union of two $G$-orbits: the divisor $D \simeq \p^2$, given by $y_0=0$, and its complement in $\PP_b$.
Therefore we can only contract $D$, which is possible if and only if $b=2$ as $K_{\PP_b} \cdot \ell=b-3$, where $K_{\PP_b}$ is the canonical divisor of $\PP_b$ and $\ell$ is a line contained in $D$. Indeed, the cone of effective curves $\NE(\PP_b)$ is generated by $\ell$ and by a fibre $f$ of the structure morphism $\PP_b \to \p^2$. 

Contracting the class of $\ell$ yields the $G$-equivariant morphism 
\[\varphi\colon~\PP_2 \to Z=\P(1,1,1,2),\  [y_0:y_1;z_0:z_1:z_2]\mapsto\left[z_0:z_1:z_2:\frac{y_1}{y_0}\right].\] Hence $\varphi G \varphi^{-1} \subseteq \Autz(Z)$.
Since $\Autz(Z)$ acts on $Z$ with two orbits, the singular point $q=[0:0:0:1]$ and its complement, and $\varphi$ is the blow-up of the singular point $q$, we have an equality $\varphi G \varphi^{-1} = \Autz(Z)$. Finally, any other blow-up $\varphi'\colon X' \to Z$ of the point $q$ in $Z$ cannot be $G$-equivariant. Indeed, we see that $\psi=\varphi'^{-1} \circ \varphi \colon X \dashrightarrow X'$ is a $G$-equivariant birational morphism (because $X$ has no $G$-orbits of codimension~$\geq 2$), thus it must a $G$-equivariant isomorphism.
\end{proof}

\begin{proposition} \label{prop:Schwarzenberg involution}
Assume that $\car(\k)\not=2$, and let $b\ge 3$ such that $\car(\k)$ does not divide $b+1$.
Let $G=\Autz(\SS_b) \simeq \PGL_2$.
There is a birational involution $\varphi\colon \SS_b\dasharrow \SS_b$, which is a type \II equivariant link such that $\varphi G\varphi^{-1}=G$. 
Moreover, $\varphi$ is the unique  equivariant link starting from $\SS_b$.
\end{proposition}

\begin{proof}
By Lemma \ref{lem:SbPicN}\ref{SbFourOrbits} the variety $\SS_b$ is the union of four $G$-orbits: a curve $\gamma$, two surfaces whose closures are divisors $D$ and $E$, and the open orbit. 
Also, by Lemma \ref{lem:SbPicN}\ref{SbPicN4}, the cone of effective curves is generated by $f$ and $s_1$, where the curves numerically equivalent to $s_1$ span the divisor $E$, but $K_{\SS_b} \cdot s_1=b-3 \geq 0$. 

We first show that there is no equivariant link of type \III or  \IV starting from $\SS_b$. As the extremal rays cover {respectively $\SS_b$ and $E$}, the small map starting from $\SS_b$ cannot be a flop or a flip. It follows from Lemma~\ref{Lemm:AntiFlip} that it also cannot be an anti-flip, so the small map is an isomorphism.
Hence, an equivariant link of  type \III or  \IV needs to start with the contraction of a negative extremal ray, but $K_{\SS_b} \cdot s_1\ge 0$ and the contraction of $f$ is the morphism $\SS_b\to \p^2$, and so there is no such link.

The only possible equivariant links then start with blowing-up the curve $\gamma$ (with its reduced structure, by Lemma~$\ref{lem:ReducedblowUpcurve}$). One such equivariant link exists: it is of type \II, obtained by blowing-up the curve $\gamma$ and then contracting the strict transform of the divisor $E$, and is the birational involution considered in \cite[Lemma~5.6.2]{BFT}. By Remark~\ref{remark:UnicitySarkisov}, this is the only possible equivariant link starting from $\SS_b$.
\end{proof}

\subsection{\texorpdfstring{$\p^2$}{P2}-bundles over \texorpdfstring{$\p^1$}{P1}}\label{sec:P2bundles}
In this subsection we study the equivariant links starting from $\RR_{m,n}$, from certain $\FF_a^{b,c}$, and from $\P(1,1,2,3)$ {(see \S~\ref{sec:first classification} for the definition of $\RR_{m,n}$ and $\FF_a^{b,c}$)}.

\begin{lemma}\label{Lemm:XmnPic}
Let $X=\RR_{m,n}$ with $m\ge n\ge 0$.
\begin{enumerate}
\item\label{XmnPicN1}
The group of numerical equivalence classes of $1$-cocycles $\NS_\Q(X)$ is generated by $F:=(y_0=0) \simeq \p^2$ $($a fibre of the $\p^{2}$-bundle$)$ and $H:=(x_0=0) \simeq \F_n$.
\item\label{XmnPicN2}
The group of numerical equivalence classes of $1$-cycles $N_1^\Q(X)$ is generated by the curve $\ell:=(x_0=x_1=0)$, which is a section of the structure morphism $\RR_{m,n} \to \p^1$, and by the curve $f=H \cap F=(x_0=y_0=0)$.
\item\label{XmnPicN3}
The intersection form on $X$ is given by
\[\begin{array}{|c|cc|}
\hline
& H & F\\
\hline
H & \ell -(m-n)f & f \\
F& f & 0 \\
\hline\end{array}\ \ \begin{array}{|c|cc|}
\hline
& H & F\\
\hline
\ell & -m & 1\\
f& 1 & 0 \\
\hline\end{array}\]
\item\label{XmnPicN4}
The cone of effective curves $\NE(X)$ is generated by $\ell$ and $f$.
\item\label{XmnPicN5}
The canonical divisor $K_X$ is $-3H-(2m-n+2)F$, and so $K_X\cdot \ell=m+n-2$ and $K_X\cdot f=-3$.
\end{enumerate}
\end{lemma}

\begin{proof}
As $X\to \p^1$ is a $\p^2$-bundle, $\NS_\Q(X)$ is generated by a fibre $F$ and by a divisor whose restriction to each fibre is a line. We can choose this divisor to be one of the following surfaces $H:=(x_0=0) \simeq \F_n$, $H' :=(x_1=0) \simeq \F_m$, or  $H'':=(x_2=0)\simeq \F_{m-n}$. In particular, \ref{XmnPicN1} is shown.

As $\frac{x_0 y_0^m}{x_2}, \frac{x_1y_0^n}{x_2}$ are two rational functions on $X$, we get $H'' = H+mF = H'+nF$ in $\NS_\Q(X)$. 
Writing $\ell=H \cap H'$, we obtain $\ell\cdot H''=0$ and $\ell\cdot F=1$, thus $\ell \cdot H=-m$. 
Moreover, $f=H\cdot F=H'\cdot F=H''\cdot F$ and $f$ satisfies $f\cdot F=0$ and $f\cdot H=1$. 
Hence, $f$ and $\ell$ generate $N_1^\Q(X)$. As $H\cdot H=H\cdot (H'-(m-n)F)=\ell -(m-n)f$. This yields \ref{XmnPicN2}-\ref{XmnPicN3}.

To show \ref{XmnPicN4}, we take an irreducible curve $\gamma\subseteq X$, which is equivalent to $a\ell +bf\in \NE(X)$ for some $a,b\in \Q$, and show that $a,b\ge 0$.  We observe that $0\le \gamma\cdot F=a$, and that $\gamma\cdot H''=b$, $\gamma\cdot H'=b-an$, $\gamma\cdot H=b-am$. 
As $\gamma$ cannot be contained in the three surfaces {$H$, $H'$, and $H''$}, we have $b\ge 0$.

Writing $K_X=\alpha H + \beta F$, for some $\alpha,\beta \in \Z$, and applying twice the adjunction formula (once for $H$ and once for $F$) yields $K_X=-3H-(2m-n+2)F$,
and then \ref{XmnPicN5} follows from \ref{XmnPicN3}.
\end{proof}

Lemma~\ref{Lem:Rm0} and Proposition~\ref{prop:Rmn list of links} below describe equivariant links starting from $\RR_{m,n}$, where we may assume $m\ge n\ge  0$. By Lemma~\ref{lem:Rmn cases non max}, the cases where $(m,n)=(1,0)$ or when $2n\geq m>n \geq 1$ can be excluded. The case of $\RR_{0,0}\simeq \p^1\times\p^2$ is done in Proposition~\ref{Prop:HomSpaces}.
\begin{lemma}\label{Lem:Rm0}
If $m\geq 2$, then there are no equivariant links starting from $\RR_{m,0}$. 
\end{lemma}

\begin{proof}
We see from Lemma~\ref{lem:aut P2 bundle} that $\RR_{m,0}$ is the union of two $G$-orbits: the divisor $H=(x_0=0) \simeq \p^1\times\p^1$ and its open complement. Hence, small maps are isomorphisms and type \I and \II equivariant links are excluded. By Lemma~\ref{Lemm:XmnPic}\ref{XmnPicN5}, $K_X\cdot \ell=m-2\ge 0$, so type \III and \IV equivariant links are also excluded, which proves the lemma.
\end{proof}

\begin{proposition} \label{prop:Rmn list of links}
Assume that $\car(\k) \notin \{2,3\}$, and
assume that $m=n \geq 1$ or $m>2n \geq 2$.
The morphism 
\[\varphi\colon~\FF_{m-n}^{1,-n} \to \RR_{m,n}, \ [x_0:x_1;y_0:y_1;z_0:z_1]\mapsto [x_0y_0:x_0y_1:x_1;z_0:z_1]\] 
is the blow-up of the curve $\ell:=(x_0=x_1=0)$ and  satisfies $\varphi \Autz(\FF_{m-n}^{1,-n}) \varphi^{-1}=\Autz(\RR_{m,n})$. Hence, $\varphi$ is a type \III equivariant link starting from $\FF_{m-n}^{1,-n}$ and $\varphi^{-1}$ is a type \I equivariant link starting from $\RR_{m,n}$. Moreover, the following hold:
\begin{enumerate}
\item\label{Only1131} If  $(m,n)\notin \{(1,1),(3,1)\}$, then $\varphi^{-1}$ is the unique equivariant link starting from $\RR_{m,n}$.
\item  \label{prop:Rmn list of links2} If $(m,n)=(1,1)$, there are two equivariant links starting from $\RR_{1,1}$: the link $\varphi^{-1}$ and the flop of $\ell$ $($type \IV link$)$ that yields a map $\RR_{1,1} \ps \RR_{1,1}$. This flop can be further factorised into the composition of equivariant links as follows:
\[
\xymatrix@R=3mm@C=.3cm{
    & & \FF_{0}^{1,-1} \ar[d]  \ar_{\varphi'}@/_1.0pc/@{->}[dll] \ar^{\varphi}@/^1.0pc/@{->}[drr]\\
    \RR_{1,1}  \ar@/^1.5pc/@{<-->}[rrrr] \ar[d]  & & \p^1\times\p^1  \ar[dll] \ar[drr] & & \RR_{1,1} \ar[d]\\
     \p^1 \ar[drr]  & & & & \p^1 \ar[lld]\\
  & & \{pt\} & &
    } 
\]
where the two morphisms $\varphi,\varphi'$ are defined by
\begin{equation*}
\begin{array}{cccccc}
&\RR_{1,1} &\stackrel{\varphi'}{\longleftarrow} &\FF_{0}^{1,-1}&\stackrel{\varphi}{\longrightarrow} & \RR_{1,1} \\ 
&[ x_0 z_0:x_0 z_1:x_1;y_0:y_1] & \leftmapsto& [x_0:x_1;y_0:y_1;z_0:z_1] & \mapsto & [x_0y_0:x_0y_1:x_1;z_0:z_1].
 \end{array}
\end{equation*}
\item \label{prop:Rmn list of links0t} If $(m,n)=(3,1)$,  there are two equivariant links starting from $\RR_{3,1}$: the link $\varphi^{-1}$ above and a type \III link $\eta\colon\RR_{3,1} \dashrightarrow \P(1,1,2,3)$, starting with the anti-flip of $\ell$ followed by the contraction of the unique invariant divisor $($the image of the divisor $x_0=0$ in $\RR_{3,1})$.  
The link satisfies $\eta \Autz(\RR_{3,1}) \eta^{-1} = \Autz(\P(1,1,2,3))$.
\[
\xymatrix@R=5mm@C=1cm{
      \RR_{3,1} \ar@{-->}[rrrd]^-{\varphi} \ar@{..>}[rr]^{\antiflip} \ar[d]_-{\pi} & & W' \ar[dr]^{\div} \\
      \P^1 \ar[drrr]  & & &  \ar[d]^-{\pi'}  \P(1,1,2,3) \\
    &  & & \{pt\}
    }
\]
\end{enumerate}
\end{proposition}

\begin{proof}
As $m \geq n \geq 1$,  {the curve $\ell\subseteq \RR_{m,n}$ is invariant by $\Autz(\RR_{m,n})$ (Lemma~\ref{lem:aut P2 bundle})}.
This yields  $\varphi^{-1}\Autz(\RR_{m,n})\varphi\subseteq \Autz(\FF_{0}^{1,-n})$, and the inclusion $\varphi \Autz(\FF_{0}^{1,-n})  \varphi^{-1}\subseteq \Autz(\RR_{m,n})$ follows from Proposition \ref{blanchard}. Hence, $\varphi$ and $\varphi^{-1}$ are equivariant links of type \III and \I respectively.

We now study other equivariant links starting from $X=\RR_{m,n}$. For each equivariant link of type \I or \II, if the divisorial contraction $W\to X$ is centered at $\ell\subseteq X$, it is the blow-up of $\ell$ with the reduced structure (Lemma~\ref{lem:ReducedblowUpcurve}), so the equivariant link is equal to $\varphi^{-1}$ (see Remark~\ref{remark:UnicitySarkisov}). Hence, $\varphi^{-1}$ is the only equivariant link of type \I and there is no equivariant link of type \II.

Moreover, by Lemma~\ref{Lemm:XmnPic}\ref{XmnPicN4} the cone of effective curves $\NE(X)$ is generated by $\ell$ and by $f$. The contraction of $f$ is the $\p^2$-bundle $X\to \p^1$ and there is no divisorial contraction associated with $\ell$, since $K_X\cdot \ell=m+n-2\ge 0$ (Lemma~\ref{Lemm:XmnPic}\ref{XmnPicN5}). Hence, every equivariant link of type \III or \IV starting from $X$ starts by an anti-flip (or a flop when $m=n=1$) centred at $\ell$.

Since $\RR_{m,n}$ is a toric variety {(Lemma~\ref{lem:aut P2 bundle})}, we can use tools from toric geometry (see e.g.~\cite[Chapter~14]{Mat02}) to
verify from the fan of $\RR_{m,n}$ that there exist non-trivial toric small maps  from $\RR_{m,n}$ to a toric threefold with  terminal singularities if and only if $n=1$. In this case $m=1$, $m=3$ or $m\ge 4$.

If $m=n=1$, the flop of $\ell$ is the type $\IV$ link $\RR_{1,1} \dashrightarrow \RR_{1,1}$ given in the statement. This gives~\ref{prop:Rmn list of links2}. 

If $m>2$ and $n=1$, denote the anti-flip of the curve $\ell$ by $X\dashrightarrow X_1$.  We then verify, 
{using again tools from toric geometry (since $X_1$ is a toric variety),} that the cone of effective curves $\NE(X_1)$ is generated by two {torus invariant} curves $\ell_5$ and $\ell_6$ satisfying $K_{X_1}\cdot \ell_5=m-4$ and $K_{X_1} \cdot \ell_6=1-m$ (and flipping $\ell_6$ brings us back to $X$). 
{Moreover, the curves in $X$ defined by the equations $x_0=x_2-\mu x_1y_0^n=0$, with $\mu \in \k$, are numerically equivalent and cover the divisor defined by the equation $x_0=0$. As these curves all intersect $\ell$ transversally, their images in $X_1$ are numerically equivalent to $\ell_5$, it follows that the set of curves equivalent to $\ell_5$ covers a divisor, namely the image of $x_0=0$.} This implies that there is no equivariant link starting by the anti-flip of $\ell_5$ if $m\ge 4$. This finishes the proof of~\ref{Only1131}.

If $m=3$ and $n=1$, we get a divisorial contraction $X_1\to \p(1,1,2,3)$ and the birational map $\RR_{3,1}\dasharrow \p(1,1,2,3)$ is given by $\eta\colon[x_0 : x_1 : x_2 ; y_0 : y_1] \mapsto  [x_0 y_0 : x_0 y_1 : x_0 x_1 : x_0^2 x_2]$. Also, by Proposition~\ref{blanchard} we have $\eta \Autz(\RR_{3,1}) \eta^{-1} \subseteq \Autz(\P(1,1,2,3))$, and computing the dimension on both sides yields an equality. {(To compute $\dim(\Autz(\RR_{3,1}))$ we use Lemma~\ref{lem:aut P2 bundle}, and for the computation of $\dim(\Autz(\P(1,1,2,3)))$ see for instance \cite[\S~8]{AlAm}.)
}
\end{proof}

\begin{lemma}\label{LemP1123} Assume that $\car(\k)\notin \{2,3\}$.
The variety $X=\P(1,1,2,3)$ is the union of four $\Autz(X)$-orbits, namely the two singular points $p_1=[0:0:1:0]$ and $p_2=[0:0:0:1]$, a curve $C$ satisfying $\overline{C}=\{[0:0:*:*]\}=C\sqcup p_1 \sqcup p_2$, and the open complement $X \setminus \overline{C}$. Moreover, there are only two equivariant links starting from $X$, both of type \I: 
\begin{enumerate}
\item The $($reduced$)$ blow-up of $p_1$ followed by the flip of the strict transform of $C$, which yields the $\P^2$-bundle $\RR_{3,1} \to \P^1$ and verifies $\varphi \Autz(\RR_{3,1}) \varphi^{-1}=\Autz(X)$.
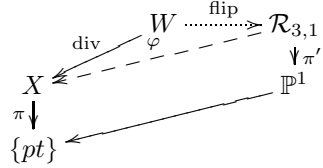
\[ \xymatrix@R=3mm@C=1cm{
     & W  \ar@{..>}[r]^-{\flip} \ar[dl]_-{\div} & \RR_{3,1} \ar[d]^-{\pi'}  \ar@{-->}[lld]_-{\varphi} \\
      X \ar[d]_-{\pi}  & & \P^1 \ar[lld] & \\
      \{pt\} &  &
    }\] 
       Explicitly, one has
   \[ \begin{array}{cccc}
  \varphi^{-1}\colon &\P(1,1,2,3)& \dasharrow & \RR_{3,1} \\
&  [x_0:x_1:x_2:x_3]& \mapsto & [1:x_2:x_3;x_0:x_1]\end{array}.\]
\item\label{item PP2} The weighted blow-up of $p_2$ $($with weights $(1,1,2))$, which yields the Mori $\P^1$-fibration $\W_2 \to \P(1,1,2)$ $($defined in {\upshape\S~\ref{sec:first classification}\ref{DefiWb})} and verifies $\varphi \Autz(\W_2) \varphi^{-1}=\Autz(X)$.
\[ \xymatrix@R=3mm@C=1cm{
     &     & \W_2 \ar[d]^-{\pi'}  \ar[lld]_-{\varphi}^-{\div} \\
      X \ar[d]_-{\pi}   & & \P(1,1,2) \ar[lld] & \\
      \{pt\} &  &
    }\] 
    Explicitly, one has
   \[ \begin{array}{cccc}
  \varphi^{-1}\colon &\P(1,1,2,3)& \dasharrow & \W_2 \\
&  [x_0:x_1:x_2:x_3]& \mapsto & [1:x_3;x_0:x_1:x_2]\end{array}.\]
   \end{enumerate}
\end{lemma}

\begin{proof} 
The description of the $\Autz(X)$-orbits of $X$ follows from Lemma~\ref{lem:aut P2 bundle}, and Proposition~\ref{prop:Rmn list of links}\ref{prop:Rmn list of links0t}. It can also be obtained by writing explicitly the $\Autz(X)$-action on $X$ in global coordinates. 

Using tools from toric geometry (see \cite{Ful93} and \cite[Chapter~14]{Mat02}), we verify from the fan of $X$ that we cannot contract $\overline{C}$, that blowing-up $\overline{C}$ (torus-equivariantly, but not necessarily with its reduced structure)  always gives a variety with non-terminal singularities, and that there is only one way to blow-up (torus-equivariantly) $p_1$ and $p_2$ respectively to get a variety with an $\Autz(X)$-action and terminal singularities; for $p_1$ it is the reduced blow-up while for $p_2$ it is a weighted blow-up with weights $(1,1,2)$ corresponding to the projection morphism from the closure of the graph of the equivariant rational map
\[\P(1,1,2,3) \dashrightarrow \P(1,1,2),\ [x_0:x_1:x_2:x_3] \mapsto [x_0:x_1:x_2].\]
This implies that there are only two equivariant links starting from $X$ which are those described in the statement.
\end{proof}

\subsection{Umemura's \texorpdfstring{$\p^1$}{P1}-bundles over \texorpdfstring{$\F_a$}{Fa}} \label{UmemuraP1 bundles over Fa}
In this subsection we first study the geometry of the $\P^1$-bundles $\U_a^{b,c}\to \F_a$ and $\V_b \to \P^2$, and then we consider the equivariant links starting from $\U_a^{b,c}$ and $\V_b$.

\begin{lemma} \label{lem:UmePicN}
Let $a,b \geq 1$ and $c  \geq 2$ be such that $c=ak+2$ with $0 \leq k \leq b$.
Let $q\colon~X=\U_{a}^{b,c} \to \F_a$ be an Umemura's $\p^1$-bundle.
\begin{enumerate}
\item\label{UmePicN1}
The group of numerical equivalence classes of $1$-cocycles $\NS_\Q(X)$ is generated by $H_x=(x_0=0) \simeq \F_a$, $H_y=(y_0=0) (\simeq \F_c \text{ if } c>2 \text{ and } \simeq \P^1\times\P^1 \text{ if } c=2)$, and $H_z \simeq \F_b$ $($given by $z=0$ on the first chart$)$.
\item\label{UmePicN2}
 The group of numerical equivalence classes of $1$-cycles $N_1^\Q(X)$ is generated by $f=H_y \cap H_z$, $s=H_x \cap H_z$, and $l_{00}=H_x \cap H_y$ $($which is the curve on $H_y$ of self-intersection $c$ if $c>2$ and is a curve of bidegree $(1,1)$ on $H_y\simeq \P^1\times\P^1$ if $c=2)$.  
\item\label{UmePicN3}
The intersection form on $X$ satisfies
\[ \begin{array}{|c|ccc|}
\hline
& H_x &  H_y & H_z\\
\hline
f & 1 & 0& 0 \\
s&  -b& 1 & 0 \\
l_{00}& \lambda & -a&1 \\
\hline\end{array} \;\;\; \text{ where $\lambda=c$ if $c>2$ and $\lambda=2$ if $c=2$.}\]
\item\label{UmePicN4}
Let $l_{10}=(x_1=y_0=0)$ that satisfies $l_{10}=l_{00}-cf$ in $N_1^\Q(X) $ and is the curve of $H_y\simeq \F_c$ of self-intersection $-c$ if $c>2$. If $c=2$, then $H_y \simeq \p^1\times\p^1$, the fibres of the two rulings are equivalent to $f$, and $r=l_{00}-f\in N_1^\Q(X)$.  Then $\NE(X)=\Q_+ \left \langle f,s,l_{10} \right \rangle$ if $c>2$. and  $\NE(X)=\Q_+ \left \langle f,s,r \right \rangle$ if $c=2$.
\item\label{UmePicN5}
The canonical divisor of $X$ is $K_X= -2 H_x-(b+2)H_y-a(b+1-k)H_z$, thus
$ K_X\cdot f=-2, \; K_X \cdot s =b-2, \; K_X \cdot  \ell_{10}=a(k+1) \text{ if } c>2, \text{ and }  K_X\cdot r =a-2 \text{ if } c=2.$
\end{enumerate}
\end{lemma}

\begin{proof}
The proof \ref{UmePicN1}-\ref{UmePicN2} is analogue to the proof of Lemma~\ref{lem:SbPicN}\ref{SbPicN1}-\ref{SbPicN2}.

\ref{UmePicN3}: Note that $H_x$ is a section of $q\colon \U_{a}^{b,c} \to \F_a$ and $H_y=q^{-1}(s_{a})$, $H_z=q^{-1}(f_a)$, where $f_a,s_a\in \F_a$ are respectively a fibre and the section of self-intersection $-a$.

As $f$ is a fibre of $\U_{a}^{b,c} \to \F_a$, we have $f\cdot H_x=1$, $f\cdot H_y=f\cdot H_z=0$.

The curves $l_{00}$ and $s$ have self-intersection $-a$ and $0$ in $H_x\simeq \F_a$ and are sections of $H_y\to s_a$ and $H_z\to f_a$ respectively. Hence, $l_{00}\cdot H_z=s\cdot H_y=1$ and $s\cdot H_z=0$. Moreover, $l_{00}\cdot  H_y=-a$, as it the self-intersection of $l_{00}$ in  $H_x$. Similarly, we get $s\cdot H_x=-b$, as it is the self-intersection of $s$ in $H_z\simeq \F_b$ respectively. 

It remains to compute the self-intersection of $l_{00}$ in $H_y$ to get $l_{00}\cdot H_x$. The surface $H_y$ is obtained by gluing two copies of $\p^1\times\A^1$ via 
\[\begin{array}{llllll}
\nu\colon ([x_0:x_1;0:1],z)&\mapsto &\left([x_0:x_1z^{c};0:1],\frac{1}{z}\right)& \text{ if }&c>2,\\
\nu\colon([x_0:x_1;0:1],z)&\mapsto &\left([x_0:x_1z^{2}+x_0 z;0:1],\frac{1}{z}\right)& \text{ if }&c=2.\end{array}\] 
Hence, if $c>2$ we get $H_y\simeq \F_c$ and $l_{00}\cdot H_x=c$. Moreover, $\ell_{10}$ is the curve of self-intersection $-c$. If $c=2$, we use the automorphisms at the source and target given respectively by
\[\begin{array}{llllll}
\alpha_1\colon&([x_0:x_1;0:1],z)&\mapsto &([x_0-x_1z:x_1;0:1],z)&\text{ and }\\
\alpha_2\colon&([x_0:x_1;0:1],z)&\mapsto& ([x_1:-x_0+x_1z;0:1],z)\end{array}.\]
We obtain $\alpha_2\circ \nu\circ \alpha_1\colon ([x_0:x_1;0:1],z)\mapsto ([x_0:x_1;0:1],\frac{1}{z})$, so $H_y\simeq \p^1\times\p^1$. Moreover, the curve $l_{00}$ is given by $x_0=0$ on the two charts and is sent by $\alpha_1^{-1}$ onto $x_0-x_1z=0$, corresponding thus to a curve of bidegree $(1,1)$ on $\p^1\times\p^1$. This gives $l_{00}\cdot H_x=2$ if $c=2$ and achieves the proof of \ref{UmePicN3}.
 
\ref{UmePicN4}: We first assume that $c>2$. Let $\ell=\alpha f + \beta s + \gamma \ell_{10}$. We want to prove that $\alpha,\beta,\gamma \in \Q_+$.
We have $\gamma=H_z \cdot \ell \geq 0$. Then we distinguish between three cases. 
\begin{itemize}
\item If $\ell \subseteq H_x$, then since $\NE(H_x)=\Q_+ \left \langle s,\ell_{00} \right \rangle$ and $\ell_{00}=\ell_{10}+cf \in \Q_+ \left \langle f, \ell_{10} \right \rangle$, we get the result.
\item If $\ell \subseteq H_y$, then since $\NE(H_y)=\Q_+ \left \langle f,\ell_{00} \right \rangle$, we get the result.
\item If $\ell \not \subseteq H_x,H_y$, then {$\alpha-b\beta=\ell \cdot H_x \geq 0$ and $\beta-a \gamma=\ell \cdot H_y \geq 0$. Therefore, we deduce that $\alpha,\beta \in \Q_+$ (since $\gamma \in \Q_+$), and so we get again the result. }

We now assume that $c=2$. Let $\ell=\alpha f + \beta s + \gamma r$. Arguing as in the case $c>2$, and using the fact that $\ell_{00}=\ell_{10}+2f$ and $\ell_{00}=r+f$ when $c=2$, we prove that $\alpha,\beta,\gamma \in \Q_+$. This yields \ref{UmePicN4}.

\ref{UmePicN5}: It remains to determine the canonical divisor $K_X$ of $X$. 
We write $K_X=\alpha H_x+\beta H_y+\gamma H_z$ with $\alpha,\beta,\gamma\in \Z$. The adjunction formula gives $K_{H_z}=(\alpha H_x+\beta H_y+(\gamma+1) H_z)|_{H_z}=\alpha s+\beta f$. As $H_z\simeq \F_b$ and $s,f$ are respectively the $(-b)$-section and a fibre, we find $\alpha=-2$ and $\beta=-b-2$. To find $\gamma$, we again use the adjunction with $H_y$ and obtain $K_{H_y}=(-2H_x-(b+1) H_y+\gamma H_z)|_{H_y}=-2l_{00}-(b+1) H_y\cdot H_y+\gamma f$. Note that the divisor $H_{y'}$ given by $y_1=0$ is linearly equivalent to $H_y+aH_z$, as the $(a)$-section $(y_1=0)$ of $\F_a$ is equivalent to $s_a+af$. Hence, $H_y\cdot H_y=H_y\cdot (H_{y'}-aH_z)=-aH_z\cdot H_y=-af$, which gives $K_{H_y}=-2 l_{00}+(\gamma+a(b+1))f$.

If $c>2$, then $H_y\simeq \F_c$, $l_{00}$ is a $(c)$-section, and $f$ is a fibre. Hence, $K_{H_y}=-2l_{00}-(2-c)f$, which yields $\gamma=c-2-a(b+1)=-a(b+1-k)$. If $c=2$, then $H_y\simeq \p^1\times\p^1$ and $l_{00}$ has bidegree $(1,1)$, so $K_{H_y}=-2l_{00}$. As this is also $-2l_{00}-(2-c)f$ (because $c=2$), we again obtain the same $\gamma$.

The intersection with the curves is then a straightforward calculation.\qedhere
\end{itemize}\end{proof}

\begin{lemma} \label{lem:VbPicN}
Let $b \geq 1$ and $\pi \colon~X=\V_b \to \p^2$ be the $\p^1$-bundle obtained by the contraction $\psi\colon~\U_{1}^{b,2} \to \V_b$ of the extremal ray of $r$ $($with the notation of Lemma~$\ref{lem:UmePicN})$.
\begin{enumerate}
\item\label{VbPicN1}
The group of numerical equivalence classes of $1$-cocycles $\NS_\Q(X)$ is generated by $H'=\psi(H_x)$ and $F'=\psi(H_z)$.
\item\label{VbPicN2}
 The group of numerical equivalence classes of $1$-cycles $N_1^\Q(X)$ is generated by $f'=\psi(f)$ and $s'=\psi(s)$.  
\item\label{VbPicN3}
The intersection form on $X$ satisfies
\[ \begin{array}{|c|cc|}
\hline
& H' & F'\\
\hline
f' & 1 & 0 \\
s'&  -(b-1) &  1\\
\hline\end{array} \]
\item\label{VbPicN4}
The cone of effective curves $\NE(X)$ is generated by $f'$ and $s'$.
\item\label{VbPicN5}
The canonical divisor of $X$ is $K_X=-2 H' -(b+1)F'$, thus
$ K_X\cdot f'=-2$ and $K_X \cdot s' =b-3.$
\end{enumerate}
\end{lemma}

\begin{proof}
The morphism $\psi\colon \U_{1}^{b,2} \to \V_b$ contracts the divisor $H_y\simeq \p^1\times\p^1$ onto the curve $f'=\psi(H_y)$, which is a fibre of the $\p^1$-bundle $\pi\colon X=\V_b \to \p^2$.
The fibres of $H_y\to f'$ are the curves equivalent to $r$. The curves equivalent to $f$ are fibres of the other projection and $l_{00}=r+f$ is of bidegree $(1,1)$. This gives $\psi_*(l_{00})=\psi_*(f)=f'$.  As the cone of effective curves of $\U_{1}^{b,2}$ is generated by $f,s,r$, the cone of effective curves of $\V_b$ is generated by $f'$ and $s'$. This gives~\ref{VbPicN1}-\ref{VbPicN2}-\ref{VbPicN4}.

The morphism $\psi|_{H_x}\colon H_x\to H'=\psi(H_x)$ is an isomorphism since $H_x\cap H_y$ is the curve $l_{00}$ that has bidegree $(1,1)$ in $H_y\simeq \p^1\times\p^1$ {and is thus a section of $H_y\to \psi(H_y)=f'$}. Hence, the curve $f'$ has self-intersection $-1$ in $H'=\psi(H_x)\simeq \F_1$ and $\pi\colon H'\to \p^2$ is the contraction of $f'$. Moreover, $s'$ is a fibre of $H'\to\p^1$, as $s$ was a fibre of $H_x\to \p^1$. 
Since $F'\cdot H'=s'+f'$, we obtain $F'\cdot s'=(s'\cdot (s'+f'))_{H'}=1$ and $F'\cdot f'=(f'\cdot (s'+f'))_{H'}=0$.

The morphism $H_z\to F'{=\psi(H_z)}$ is an isomorphism, so $F'\simeq \F_b$ and the curves $s'$ and $f'$ are the $(-b)$-curve and a fibre, as so happens for $s,f$ in $H_z$. We then compute $s'\cdot H'=(s'\cdot (s'+f'))_{F'}=-b+1$ and $f'\cdot H'=(f'\cdot (s'+f'))_{F'}=1$. This gives~\ref{VbPicN3}.

The canonical divisor of {$X=\V_b$} is $K_X=\alpha F'+\beta H'$ for some $\alpha,\beta\in \Z$. The adjunction formula gives $K_{F'}=((\alpha+1) F'+\beta H')|_{F'}=(\alpha+1)f'+\beta (f'+ s')=(\alpha+\beta+1)f'+\beta s'$. As $F'\simeq \F_b$ and $s',f'\subseteq F'$ are the $(-b)$-section and a fibre, we obtain $\beta=-2$ and $\alpha+\beta+1=-(b+2)$, so $K_X=-2H'-(b+1)F'$.

This yields $ K_X\cdot f'=-2$ and $K_X \cdot s' =b-3.$
\end{proof}

\begin{lemma} \label{lem:orbits of Uabc et Vb}  
Let $\psi\colon\U_{1}^{b,2}\to \V_b$ be the morphism induced by the contraction of the ($-1$)-section in $\F_1$. We keep the notation of Lemmas~$\ref{lem:UmePicN}$ and $\ref{lem:VbPicN}$.
\begin{enumerate}
\item \label{item:orbits Uabc}
The variety $\U_{1}^{b,2}$ is the union of four $\Autz(\U_{1}^{b,2})$-orbits:
\[\U_{1}^{b,2}=(\U_{1}^{b,2} \setminus (H_x \cup H_y)) \sqcup (H_x \setminus \ell_{00}) \sqcup (H_y \setminus \ell_{00}) \sqcup \ell_{00}.\]
\item \label{item:orbits Vb} Let $b \geq 2$. The variety $\V_b$ is the union of three $\Autz(\V_b)$-orbits: 
\[\V_b=(\V_b \setminus H') \sqcup (H' \setminus f') \sqcup f'.\]
\end{enumerate}
\end{lemma}
\begin{proof}
The variety $\U_{1}^{b,2}$ contains two $G$-invariant divisors, namely the invariant section $H_x \simeq \F_1$ of the $\P^1$-bundle $\U_1^{b,2} \to \F_1$ and the preimage $H_y \simeq \P^1\times\P^1$ of the ($-1$)-curve in $\F_1$, they intersect along the invariant curve $\ell_{00} \simeq \P^1$. It follows from the description of the automorphism group $G$ in \cite[\upshape\S 3.6]{BFT} that $\U_{1}^{b,2} \setminus (H_x \cup H_y)$ and $H_x \setminus \ell_{00}$ are two $G$-orbits. As observed in \cite[Remark~3.6.5]{BFT} the group $\GL_2$ acts on  $\U_1^{b,2}$, this induces a $\PGL_2$-action on $H_y$ given in the two charts of $\U_1^{b,2}$ by
\[\begin{array}{ccc}
\F_b\times\A^1& \hspace{-0.2cm}\to &\hspace{-0.3cm} \F_b\times\A^1\\
([x_0:x_1;0:1],z)&\hspace{-0.2cm}\mapsto &\hspace{-0.3cm}\left([x_0: x_1 (\gamma z+\delta)^2\!+\!x_0 \gamma(\gamma z+\delta) ; 0: 1],\frac{\delta z +\gamma}{\beta z +\alpha}\right); \text{ \ and}\\
([x_0:x_1;0:1],z)&\hspace{-0.2cm}\mapsto &\hspace{-0.3cm}\left([x_0: x_1 (\beta z+\alpha)^2 \!-\!x_0 \beta(\beta z+\alpha) ; 0: 1],\frac{\alpha z +\beta}{\gamma z +\delta}\right).\end{array}\]
Since this $\PGL_2$-action is non-trivial and stabilizes the curve $\ell_{00}$, which corresponds to a curve of bidgree ($1,1$) on $\P^1\times\P^1$, the image $H \simeq \PGL_2$ of the natural homomorphism $\PGL_2 \to \Autz(H_y)\simeq \PGL_2\times\PGL_2$ is conjugated to the diagonal embedding of $\PGL_2$. 
Hence $H_y=(H_y \setminus \ell_{00}) \sqcup \ell_{00}$ is the union of two $G$-orbits. This proves \ref{item:orbits Uabc}.

By \cite[Lemma~5.5.1(4)]{BFT} the morphism $\psi\colon \U_{1}^{b,2} \to \V_b$ satisfies $\psi G \psi^{-1}=\Autz(\V_b)$ since $b \geq 2$. Also, it is an isomorphism outside $H_y$. But $f'=\psi(H_y)$ is a line on which $\Autz(\V_b)$ acts transitively since $G$ acts transitively on the two rulings of $H_y \simeq \P^1\times\P^1$. This proves \ref{item:orbits Vb}. 
\end{proof}

To study the equivariant links starting from $\U_{a}^{b,c}$ in Proposition~\ref{prop:UmeMax}, we will need Proposition~\ref{prop:Sarki_dec and Ume cases} that gives restrictions on the possible links.
\begin{proposition} \label{prop:Sarki_dec and Ume cases}
Let $\pi \colon X\to \F_a$ be a decomposable or an Umemura $\p^1$-bundle with numerical invariants $(a, b, c)$ as in Theorem~$\ref{thBFT}$, let $G=\Autz(X)$, and let $\chi\colon X\dasharrow X'$ be an equivariant link. Then, one of the following three possibilities occurs:
\begin{enumerate}
\item\label{SarkiDecUF1}
The link $\chi$ is of type $\II$, and either $\chi$ or $\chi^{-1}$ is equal to the link $\varphi$ given by
\[\xymatrix@R=5mm@C=1cm{
    & W  \ar[dl]_-{\div} \ar[dr]^-{\div} &  && W  \ar[dl]_-{\div} \ar[dr]^-{\div}\\
     \FF_a^{b,c} \ar@{-->}[rr]^-{\varphi} \ar[dr]_-{\pi}  &  & \FF_a^{b+1,c+a} \ar[dl]^-{\pi'} &  \U_a^{b,c} \ar@{-->}[rr]^-{\varphi} \ar[dr]_-{\pi}  &  &  \U_a^{b+1,c+a} \ar[dl]^-{\pi'} \\
     &\F_a& &    &\F_a
  }\]
and  described in \cite[Lemma 5.4.2]{BFT} and \cite[Lemma 5.5.3]{BFT} respectively. 

    In particular, $\chi$ decomposes as the blow-up of a section $\Gamma\subseteq X$ over the $(-a)$-curve $s_{-a}$ of $\F_a$, followed by the contraction of the strict transform of $\pi^{-1}(s_{-a})$.
\item\label{SarkiDecUF2}
The link $\chi$ is of type \III $($resp.~$\IV)$ and the small map associated with it $($see Definition~$\ref{sarkisov-links})$ is an isomorphism. Hence, the map $X\dasharrow X'$ $($resp.~$X\dasharrow Y')$ is the contraction of a negative extremal ray distinct from the one given by the fibres of $\pi\colon X\to \F_a$.
\item\label{SarkiDecUF3}
The link $\chi$ is of type \III, the small map is not an isomorphism and $X=\FF_{2}^{b,c}$ with $c=\pm 1$; in particular, $a=2$.
\end{enumerate}
\end{proposition}
\begin{proof}
We keep the notation of Definition~\ref{sarkisov-links}.
The link $\chi$ starts from the Mori fibration $X\to Y$ over the surface $Y=\F_a$.

We first exclude the case of an equivariant  link of type $\I$. As $3=\dim(X')>\dim(Y')\ge \dim(Y)=2$, we find that $Y'$ and $Y$ are surfaces, so $Y'\to Y$ is a divisorial contraction, hence the blow-up of a point of $\F_a$. This is impossible, as the action of $G$ on $\F_a$ gives the whole group $\Autz(\F_a)$, which has no fixed point.

Suppose now that the link is of type $\II$. It starts with a divisorial contraction $W\to X$, which contracts the exceptional divisor on a invariant curve $\Gamma\subseteq X$, as there is no point fixed by $G$ on $X$. By Lemma~\ref{lem:ReducedblowUpcurve}, the morphism $\eta$ is the blow-up of $\Gamma$ with its reduced structure. Moreover, $\Gamma$ is a section of $X\to \F_a \to \P^1$ by \cite[Lemmas~5.4.1-5.5.2]{BFT}. The equivariant link starting from this blow-up is then unique up to automorphisms of the target (see Remark~\ref{remark:UnicitySarkisov}).  It remains to apply \cite[Lemmas 5.4.2 and 5.5.3]{BFT}, to obtain that $\chi$ is a link of type $\II$ between two $\p^1$-bundles, that the small map $W\ps W'$ is an isomorphism, and the morphism $W'\to X'$ is the contraction of the strict transform of $\pi^{-1}(\pi(\Gamma))$; we are thus in Case~\ref{SarkiDecUF1}.

If the link is of type \IV, then $Z$ has dimension $\le 1$ (\cite[Proposition~3.5]{C95}). The morphism $\F_a\to Z$ being of relative Picard rank $1$, we obtain $Z=\p^1$. {The small map $X\ps X'$ are obtained by running an MMP over $Z$, so only curves of $\mathrm{NE}(X/Z)$ are considered. Hence, the curves on which the small map is not defined are contained in fibres. But in our situation there are no such curves since $G$ maps onto $\Autz(\F_a)$, and the latter acts on $\F_a$ with no fixed point.} Hence, the small map of the type \IV link is an isomorphism.

We now finish the proof by considering type \III links. The morphism $\F_a\to Y'$ is of relative Picard rank $1$. Hence, it is either $\F_1\to \p^2$, $\F_2\to \p(1,1,2)$ or $\F_a\to \p^1$, since singularities for the base of a three-dimensional Mori fibration over a surface are canonical (see \cite[Theorem~1.2.7]{MP_2008}). If it is $\F_a\to \p^1$, the same argument as above implies that the {curves over which the small map is not defined} are contained in fibres, which is impossible as the action of $G$ on $\F_a$ gives $\Autz(\F_a)$. Hence, if the small map is not an isomorphism, the morphism $\F_a\to Y'$ is either $\F_1\to \p^2$ or $\F_2\to \p(1,1,2)$, so $a=1$ or $a=2$.
We distinguish between the two cases.

$\bullet \ \textbf{a=1}$: We have $X=\U_1^{b,c}$ with $2 \leq c \leq b$ (as the $\P^1$-bundles $\FF_1^{b,c}$ do not appear in Theorem~\ref{thBFT}).
If $c=2$, Lemma~\ref{lem:orbits of Uabc et Vb}  \ref{item:orbits Uabc} gives that the only invariant curve of $\U_1^{b,2}$ is $\ell_{00}$ (defined by $x_0=y_0=0$), and by Lemma~\ref{lem:UmePicN}\ref{UmePicN4} the latter is not extremal. Hence, if $X\dasharrow W'$ is not an isomorphism, we must have $c>2$. 
Now for $c>2$, {\cite[Lemma 5.5.2]{BFT} gives that $\U_1^{b,c}$ has exactly two invariant curves, namely $\ell_{00}$ and $\ell_{10}$ (given by $x_1=y_0=0$), but according to Lemma~\ref{lem:UmePicN}\ref{UmePicN4}, only $\ell_{10}$ is extremal.}
Thus, $X\dasharrow W'$ is the anti-flip of the curve $\ell_{10}$, and $W'\to X'$ is a divisorial contraction, which necessarily contracts the  strict transform in $W'$ of the divisor $H_y \simeq \F_c$ (as $Y\to Y'$ is the contraction $\F_1\to \p^2$). The curve $l_{10}$ on $H_y$ is the $-c$-curve of $\F_c$ (Lemma~\ref{lem:UmePicN}\ref{UmePicN4}) so the strict transform of $H_y$ in $W'$ is isomorphic to the weighted projective space $\P(1,1,c)$. The divisorial contraction $W'\to X'$ contracts then this divisor $\P(1,1,c)$ onto a point, fixed by $G$, and we obtain a $\P^1$-fibration $\varphi\colon X' \to Y'=\P^2$ on which $G$ acts with a fixed point. On the other hand, denoting by $q \in \P^2$ the fixed point for the induced $G$-action on the basis, we observe that $X' \setminus \varphi^{-1}(q) \simeq \V_b \setminus \psi^{-1}(q)$ as $\P^1$-bundles over $\P^2 \setminus \{q\}$, where $\psi\colon \V_b \to \P^2$ is the structure morphism. It follows from \cite[Proposition~3.5]{C95} that $X' \simeq \V_b$ as a $\P^1$-bundle over $\P^2$, which contradicts the fact that $G$ acts on $\V_b$ with no fixed point (Lemma~\ref{lem:orbits of Uabc et Vb}\ref{item:orbits Vb}). Hence, the case $a=1$ cannot occur if $X\dasharrow W'$ is not an isomorphism.

$\bullet \ \textbf{a=2}$: If $X=\U_{2}^{b,c}$, then $c=2k+2$, for some $k \geq 0$. As $c$ is even, the group $ \SL_2/\{\pm1\}\simeq\PGL_2$ acts non-trivially on $X$ (see \cite[Remark~3.6.5]{BFT}).
Therefore Lemma~\ref{Lemm:AntiFlip} implies that the small map $X \ps W'$ is an isomorphism. We now consider the case {of the toric variety} $X=\FF_2^{b,c}$. Recall that ($a,b,c$) satisfy the numerical conditions of Theorem~\ref{thBFT}, and so either $-2<c<2b$ or $b=c=0$. 
If $c$ is even, then by \cite[\S3.1]{BFT} the group $\PGL_2$ acts non-trivially on $X$, and we can again apply Lemma~\ref{Lemm:AntiFlip} to conclude that $X \dashrightarrow W'$ is an isomorphism. Also, applying the terminal singularity criterion for toric varieties (see e.g. \cite[Proposition~14.3.1]{Mat02}) to the variety $W'$, obtained from $X$ by antiflipping an extremal invariant curve, yields that $W'$ has terminal singularities if and only if $c \leq 1$ (see Lemma~\ref{XabPic} for a description of the extremal curves of $X$). Hence, if $X \dashrightarrow W'$ is not an isomorphism, we must have $c=\pm 1$.
\end{proof}

\begin{proposition} \label{prop:UmeMax}
Let $a,b \geq 1$ and $c  \geq 2$ be such that $c=ak+2$ with $0 \leq k \leq b$.
Let $\pi\colon~X=\U_{a}^{b,c} \to \F_a$ be an Umemura $\p^1$-bundle  and let $G=\Autz(X)$.
\begin{enumerate}
\item \label{Ume max item 1}
If $a=1$ and $c<b$, then the equivariant links starting from $X$ are the birational maps $\varphi\colon X \dashrightarrow X'=\U_{1}^{b+p,c+p}$, with $p=\pm 1$, and if $c=2$ the contraction morphism $\psi\colon X \to X'=\V_b$; all these equivariant links were  described in Proposition $\ref{prop:Sarki_dec and Ume cases}\ref{SarkiDecUF1}$ and \cite[Lemma~5.5.3]{BFT} and satisfy $\varphi \Autz(X) \varphi^{-1}=\Autz(X')$ and $\psi \Autz(X) \psi^{-1}=\Autz(X')$. 

\item \label{Vb bundles max}
If $b \geq 3$, then the only equivariant link from the $\P^1$-bundle $\V_b \to \P^2$ is the blow-up morphism $\psi\colon\U_{1}^{b,2} \to \V_b$ that satisfies $\psi \Autz(\U_{1}^{b,2}) \psi^{-1}=\Autz(\V_b)$.

\item \label{Ume max item 2}
If $a \geq 2$ and $c-ab<2$ $($with $c-ab \neq 2-a)$, then  the equivariant links starting from $X$ are the
birational maps $\varphi\colon X \dashrightarrow X'=\U_{a}^{b+p,c+pa}$, with $p=\pm 1$, described in Proposition~$\ref{prop:Sarki_dec and Ume cases}\ref{SarkiDecUF1}$ and \cite[Lemma 5.5.3]{BFT}; they all satisfy $\varphi \Autz(X) \varphi^{-1}=\Autz(X')$.
\end{enumerate}
\end{proposition}

\begin{proof}

\ref{Ume max item 1}: Case~\ref{SarkiDecUF1} in Proposition~\ref{prop:Sarki_dec and Ume cases} gives the family of type $\II$ equivariant links of the form $X=\U_{a}^{b,c} \dasharrow \U_{a}^{b+p,c+pa}$. Case~\ref{SarkiDecUF2} in Proposition~\ref{prop:Sarki_dec and Ume cases} corresponds to  links of type  \III or \IV with a small map being an isomorphism, and are thus given by the contraction of a negative extremal ray distinct from $\Q_{\ge 0}\cdot f$, corresponding to the structure morphism $X\to \F_a$. Since $b>c \geq 2$, Lemma~\ref{lem:UmePicN} shows that there is another negative ray on $X$ if only if $(a,c)=(1,2)$, namely $\Q_{\ge 0}\cdot r$. The contraction of this ray yields the type \III equivariant link $\psi\colon~\U_{1}^{b,2} \to \V_b$ considered in Lemma~\ref{lem:sequence of links for Umemura bundles} and satisfying $\psi \Autz(\U_{1}^{b,2}) \psi^{-1}=\Autz(\V_b)$.

\ref{Vb bundles max}: According to Lemma~\ref{lem:orbits of Uabc et Vb}\ref{item:orbits Vb}, the variety $\V_b$ is the union of three $\Autz(\V_b)$-orbits which are a line $\ell$, a surface $S$ whose closure is $\overline{S}=S \sqcup \ell \simeq \F_1$, and an open orbit $\V_b \setminus \overline{S}$. By Lemma~\ref{lem:VbPicN}\ref{VbPicN5}, there is only one negative extremal ray (as $b \geq 3$), namely $\Q_{\geq 0}f$, whose contraction yields the structure morphism $\V_b \to \P^2$. Hence we can only blow-up $\ell$ that yields the $\P^1$-bundle $\U_{1}^{b,2} \to \F_1$.

\ref{Ume max item 2}: According to Lemma \ref{lem:UmePicN}, there is only one negative extremal ray (as our numerical conditions on $a,b,c$ imply $b \geq 2$), namely $\Q_{\geq 0} f$, whose contraction yields the structure morphism $X \to \F_a$. Hence Case \ref{SarkiDecUF2} in Proposition~\ref{prop:Sarki_dec and Ume cases} cannot occur. Thus the only equivariant links starting from $X$ are the type \II links  described in Proposition~\ref{prop:Sarki_dec and Ume cases}\ref{SarkiDecUF1}.
\end{proof}

\subsection{Decomposable \texorpdfstring{$\P^1$}{P1}-bundles over \texorpdfstring{$\F_a$}{Fa}}\label{decompo P1 over Fa}
In this subsection we first study the geometry of the $\P^1$-bundles $\FF_a^{b,c}\to \F_a$, and then we consider the equivariant links starting from $\FF_a^{b,c}$ and from certain $\P^1$-fibrations over $\P(1,1,2)$.

\begin{lemma} \label{XabPic}
Let $X=\FF_{a}^{b,c}$ with $a \geq 0$ and $b,c \in \Z$.
\begin{enumerate}
\item\label{XabPicN1}
The group $\NS_\Q(X)$ is generated by $H_{x_0}:=(x_0=0) \simeq \F_a$,  $H_{y_0}:=(y_0=0) \simeq \F_c$, and $H_{z_0}:=(z_0=0) \simeq \F_b$. 
\item\label{XabPicN2}
The group $N_1^\Q(X) $ is generated by the following curves 
\begin{itemize}
\item $\ell_1:=H_{y_0} \cap H_{x_0}=\{ [0:1;0:1;z_0:z_1]\} \simeq \p^1$ 
\item $\ell_2:=H_{z_0}\cap H_{x_0}=\{ [0:1;y_0:y_1;0:1]\} \simeq \p^1$ 
\item $\ell_3:=H_{z_0}\cap H_{y_0}=\{ [x_0:x_1;0:1;0:1]\} \simeq \p^1$ 
\end{itemize}
\item\label{XabPicN3}
The intersection form on $X$ is given by
\[\begin{array}{|c|ccc|}
\hline
& H_{z_0} & H_{y_0} & H_{x_0} \\
\hline
H_{z_0} & 0 & \ell_3 & \ell_2 \\
H_{y_0} & \ell_3 & -a\ell_3& \ell_1\\
H_{x_0} & \ell_2 & \ell_1& -b\ell_1+(c-ab)\ell_2\\
\hline\end{array}\ \ \begin{array}{|c|ccc|}
\hline
& H_{z_0} & H_{y_0} & H_{x_0}\\
\hline
\ell_1 & 1 & -a& c\\
\ell_2 & 0 & 1 & -b \\
\ell_3 & 0 & 0 & 1\\
\hline\end{array}\]
\item\label{XabPicN4}
Let $\ell_4=(x_1=y_0=0) \simeq \p^1$ that satisfies $\ell_4=\ell_1-c\ell_3$ in $N_1^\Q(X) $. 
The cone of effective curves $\NE(X)$ is generated by $\ell_1$, $\ell_2$, and $\ell_3$ if $c \leq 0$ and by $\ell_4$, $\ell_2$, and $\ell_3$ if $c>0$.
\item\label{XabPicN5}
The canonical divisor of $X$ is $K_X=-(a(b+1)+2-c)H_{z_0}-(b+2)H_{y_0}-2 H_{x_0}$, thus $K_{X} \cdot \ell_1=a-c-2$,  $K_X \cdot \ell_2=b-2$, $K_X\cdot\ell_3=-2$, and $K_X \cdot \ell_4=a+c-2$.
\end{enumerate}
\end{lemma}

\begin{proof}
The proof is analogue to the proof of Lemma \ref{Lemm:XmnPic}.
\end{proof}

We now consider the equivariant links starting from $\FF_a^{b,c}$. By Theorem~\ref{th:first classification maximality}, we can consider $a,b\ge 0$, $a\not=1$, $c\in \Z$ and may assume that $c\le 0$ if $b=0$ and that $a=0$ or $b=c=0$ or  $-a<c<ab$. 

We first consider the case where $a=0$. As the Mori fibre spaces $\FF_0^{b,c}\to \F_0$, $\FF_0^{-b,-c}\to \F_0$, $\FF_0^{-c,-b}\to \F_0$  and $\FF_0^{c,b}\to \F_0$ are isomorphic, we may assume $b \geq \lvert c \rvert $. Moreover, the case of $(b,c)=(0,0)$ has been treated in Proposition~$\ref{Prop:HomSpaces}\ref{HomS3}$, the case where $b\ge c=1$ has been removed in Lemma~\ref{lem:Fabc non maximality}\ref{Fabc-itemc1} and the case $(b,c)=(1,0)$ in Lemma~\ref{lem:Fabc non maximality}\ref{Fabc-itemb1}. The remaining cases are done below:

\begin{lemma}\label{Lemma:F0bcListLinks}
Let $X=\FF_0^{b,c}$ with $b \geq \lvert c \rvert $, $c\not=1$ and either $b\ge 2$ or $(b,c)=(1,-1)$, and let $G=\Autz(X)$. Every $G$-equivariant link starting from $X$ is a link of type \III or \IV and the complete list is given as follows:
\begin{enumerate}
\item\label{LinkF1m1} If $(b,c)=(1,-1)$, there are exactly two equivariant links starting from $X$, namely the links $\varphi,\varphi'\colon X\to \RR_{1,1}$ of type $\III$ given in Proposition~$\ref{prop:Rmn list of links}\ref{prop:Rmn list of links2}$.
\item\label{LinkFaP1}
If $c=0$ and $b\ge 2$, the unique equivariant link is the type \IV link given by
\[\xymatrix@R=5mm@C=.3cm{
    \FF_0^{b,0}\simeq \F_b\times\p^1  \ar[rr]^{\varphi}_{\simeq} \ar[d]  & & \FF_b^{0,0}\simeq \F_b\times\p^1 \ar[d] \\
     \F_0 \ar[dr]  & & \F_b \ar[ld]\\
   & \p^1. 
    }  \]
\item\label{LinkFbcm1}
If $c=-1$ and $b\ge 2$, the unique equivariant link starting from $X$ is the type \III link $\varphi\colon \FF_0^{b,-1}\simeq \FF_0^{1,-b}\to \RR_{b,b}$ given in Proposition~$\ref{prop:Rmn list of links}$, which satisfies  $\varphi \Autz(\FF_{0}^{b,-1}) \varphi^{-1}=\Autz(\RR_{b,b})$.
\end{enumerate}
\end{lemma}

\begin{proof}
By \cite[\upshape\S 3.1]{BFT}, if $c \leq 0$, then the variety $X$ is the union of two $G$-orbits: the divisor $H_{x_0}=(x_0=0) \simeq \p^1\times\p^1$ and its open complement. And if $c>0$, then $X$ is the union of three $G$-orbits: the two divisors $H_{x_0}=(x_0=0) \simeq \p^1\times\p^1$ and $H_{x_1}=(x_1=0) \simeq \p^1\times\p^1$, and the open complement of $H_{x_0} \cup H_{x_1}$ in $X$. In both cases, there is no invariant subspace of dimension $\le 1$, so the only possible equivariant links are of type \III or \IV and start with the contraction of a negative extremal ray of $N_1(X)$, which gives respectively a divisorial contraction or a Mori fibration.  By Lemma~$\ref{XabPic}\ref{XabPicN4}$, the cone of effective curves $\NE(X)$ is generated by $\ell_1$, $\ell_2$, and $\ell_3$ if $c \leq 0$ and by $\ell_2$, $\ell_4$, and $\ell_3$ if $c>0$. Also, the contraction of $\ell_3$ gives the Mori fibration $X\to \F_0$. Moreover, $K_{X} \cdot \ell_1=-c-2$,  $K_X \cdot \ell_2=b-2$ and $K_X \cdot \ell_4=c-2$ (Lemma~$\ref{XabPic}\ref{XabPicN5}$).  As $b\ge 1$ and $c\not=1$, this shows that the only possible contractions are those of $\ell_1$ when $c\in \{0,-1\}$ and $\ell_2$ when $b=1$.

\ref{LinkF1m1}: If $(b,c)=(1,-1)$, there are exactly two contractions, (of $\ell_1$ and $\ell_2$). These are the two birational morphisms $\varphi,\varphi'\colon \FF_0^{1,-1}\to \RR_{1,1}$ given in Proposition~$\ref{prop:Rmn list of links}\ref{prop:Rmn list of links2}$.

\ref{LinkFaP1}: If $c=0$ and $b\ge 2$, the only contraction is given by contracting $\ell_1$; this gives the type \IV link of \ref{LinkFaP1}.

\ref{LinkFbcm1}: When $c=-1$ and $b\ge 2$, the unique link is the contraction of $\ell_1$, given by $\varphi\colon \FF_0^{b,-1}\simeq \FF_0^{1,-b}\to \RR_{0,-b}\simeq \RR_{b,b}$ and that satisfies  $\varphi \Autz(\FF_{0}^{b,-1}) \varphi^{-1}=\Autz(\RR_{b,b})$ (see Proposition~\ref{prop:Rmn list of links}).
\end{proof}

It remains to consider the equivariant links starting from $\FF_a^{b,c}$ when $a \geq 2$.
\begin{example}\label{WbFF2bm1}
For each $b\ge 2$ we have a  birational map 
\[\begin{array}{ccc}
\W_b & \dasharrow & \FF_2^{b-1,-1}\\
\ [y_0:y_1;z_0:z_1:z_2]& \mapsto & [y_0:y_1;1:z_2;z_0:z_1]\end{array}\]
One checks, using toric coordinates, that it is the blow-up of the point $p=[1:0;0:0:1]$ (with its reduced structure) followed by the flip of the strict transform of the curve $f\subseteq \W_b$ given by $z_0=z_1=0$. As $\Autz(\W_b)$ acts on $\p(1,1,2)$ by Proposition~\ref{blanchard}, it fixes the singular point $[0:0:1]$ and the fibre over it. The point $p$ is also fixed as it is a singular point of $\W_b$  (see {\upshape\S~\ref{sec:first classification}\ref{DefiWb}}). Hence, the birational map is $\Autz(\W_b)$-equivariant.
\end{example}
\begin{example}\label{WbFF2b}
For each $b\ge 2$ we have a  birational map 
\[\begin{array}{ccc}
\W_b & \dasharrow & \FF_2^{b,1}\\
\ [y_0:y_1;z_0:z_1:z_2]& \mapsto & [y_0:y_1;1:z_2;z_0:z_1]\end{array}\]
One checks, using toric coordinates, that it is the blow-up of the point $q=[0:1;0:0:1]$ (with its reduced structure) followed by the flip of the strict transform of the curve $f\subseteq \W_b$ given by $z_0=z_1=0$. As $\Autz(\W_b)$ acts on $\p(1,1,2)$ by Proposition~\ref{blanchard}, it fixes the singular point $[0:0:1]$ and the fibre over it. The point $q$ is also fixed as it is a singular point of $\W_b$ (see {\upshape\S~\ref{sec:first classification}\ref{DefiWb}}). Hence, the birational map is $\Autz(\W_b)$-equivariant.
\end{example}

\begin{proposition} \label{prop:Fabc maximality} 
 Let $X=\FF_a^{b,c}$ with $a\geq 2$, $b \geq 0$, and $c \in \Z$, and let $G=\Autz(X)$. 
\begin{enumerate}
\item  \label{Fabc-itema1}If $b=c=0$, then there is a unique equivariant link starting from $X$, which is the inverse of the one of Lemma~$\ref{Lemma:F0bcListLinks}\ref{LinkFaP1}$, namely $X=\FF_a^{0,0}\iso \F_a\times\p^1\iso \FF_{0}^{a,0}$.

\item  \label{Fabc-item2} If $b=1$ and $-a<c < 0$, then the equivariant links from $X$ are the following:
\begin{itemize}
\item  the type \II link $\FF_a^{1,c} \dashrightarrow \FF_{a}^{2,c+a}$, described  in \cite[Lemma 5.4.2]{BFT};
\item  the type \III link $\varphi\colon~\FF_{a}^{1,c} \to \RR_{a-c,-c}$ defined in Proposition~$\ref{prop:Rmn list of links}$;
\item if $(a,c)=(2,-1)$, then there is an extra type \III link $\eta\colon~X\dashrightarrow \W_2$ defined as the antiflip of the curve $\ell_1$ followed by the contraction of the image of the invariant divisor $y_0=0$ and such that $\eta \Autz(X) \eta^{-1}=\Autz(\W_2)$, where $\W_2 \to \P(1,1,2)$ is the Mori $\P^1$-fibration defined in {\upshape\S~\ref{sec:first classification}\ref{DefiWb}}.
 \end{itemize}
\item  \label{Fabc-item3} If $b \geq 2$ and $-a<c<a(b-1)$, then the equivariant links from $X$ are the following:
\begin{itemize}
\item the type \II  links $X=\FF_a^{b,c} \dashrightarrow \FF_{a}^{b+k,c+ka}$, with $k=\pm 1$ and $c+ka>-a$, described in \cite[Lemma 5.4.2]{BFT};
\item if $(a,c)=(2, \pm 1)$, then there is a type \III link $\eta\colon~X\dashrightarrow \W_b$ defined as the antiflip of the unique invariant extremal curve of $X$ followed by the contraction of the image of the invariant divisor $y_0=0$ and such that $\eta \Autz(X) \eta^{-1}=\Autz(\W_b)$, where $\W_b \to \P(1,1,2)$ is the Mori $\P^1$-fibration defined in {\upshape\S~\ref{sec:first classification}\ref{DefiWb}}. The link is the inverse of the one given in Example~$\ref{WbFF2bm1}$ if $c=-1$ and of Example~$\ref{WbFF2b}$ if $c=1$. Moreover, we have the following commutative diagram, for each $b\ge 2$:
 \[ 
\scalebox{0.8}{ 
 \xymatrix{ 
\FF_2^{b-1,-1} \ar@{.>}[rr]^-{\antiflip}  \ar[d]  & & Y_{b-1,-1}  \ar[dr]^{\div}& & Y_{b,1} \ar[dl]_{\div}  && \FF_2^{b,1} \ar[d]  \ar@{.>}[ll]_-{\antiflip} \\
  \F_2 \ar[rrrd]  &&& \W_b\ar[d] &&&  \F_2  \ar[llld]   \\
    &&  & \P(1,1,2) &&&&
  }
  } \]
 \end{itemize}
\end{enumerate}
\end{proposition}

\begin{proof}
By Proposition~\ref{prop:Sarki_dec and Ume cases} there are no type \I equivariant links starting from $X$ and if the small map $X \dashrightarrow X'$ is not an isomorphism, then the link is of type \III and $X=\FF_2^{b,c}$ with $c=\pm 1$. 
 Moreover, recall from Lemma~\ref{XabPic} that $\NE(X)$ is generated by $\ell_1$, $\ell_2$, and $\ell_3$ if $c \leq 0$ resp.~by $\ell_2$, $\ell_4$, and $\ell_3$ if $c>0$, and that $K_X \cdot \ell_1=a-c-2$, $K_X\cdot \ell_2=b-2$, $K_X \cdot \ell_3=-2$ and $K_X \cdot \ell_4=a+c-2$. 
 The contraction of $\ell_3$ corresponds to the Mori fibration $\FF_a^{b,c} \to \F_a$.
If $b=c=0$, then $X \simeq \p^1\times\F_a\simeq \FF_{0}^{a,0}$ and \ref{Fabc-itema1} follows from the study of contractions made in Lemma~\ref{Lemma:F0bcListLinks}\ref{LinkFaP1}. This yields \ref{Fabc-itema1}.

Assume that $b\geq 1$ and $-a<c<a(b-1)$. Type \II equivariant links are those given in Proposition~\ref{prop:Sarki_dec and Ume cases}\ref{SarkiDecUF1} and described in \cite[Lemma 5.4.2]{BFT}. Also, the only negative extremal rays are $\Q_{\geq 0}\ell_3$, corresponding to the structure morphism $X \to \F_a$, and $\Q_{\geq 0} \ell_2$ when $b=1$. The contraction of the class of $\ell_2$ (when $b=1$) is the divisorial contraction $\varphi\colon~\FF_{a}^{1,c} \to \RR_{a-c,-c}$ defined in Proposition~\ref{prop:Rmn list of links}. By Proposition~\ref{prop:Sarki_dec and Ume cases}, it remains only to consider {type \III links} where the small map is not an isomorphism to have the complete list of equivariant links from $X$.

If the small map is not an isomorphism, then $a=2$, $c=\pm 1$ and it must be the antiflip of the unique extremal invariant curve of $X$ {(see Proposition~\ref{prop:Sarki_dec and Ume cases}\ref{SarkiDecUF3})}, namely $\ell_1$ if $c \leq 0$ or $\ell_4$ if $c>0$, this gives a variety that we denote by $Y$. This antiflip can easily be described using tools from toric geometry since $X$ is a toric variety (see e.g.~\cite[\S~14.2]{Mat02}); {in particular, the image of the invariant divisor $H_{y_0} \simeq \F_1$ in $X$ is an invariant divisor $D \simeq \P^2$ in $Y$.} Then we can only contract $D$, this yields the threefold $X'=\W_b$ if $c=1$ and $X'=\W_{b+1}$ if $c=-1$, which is equipped with a $\P^1$-fibration structure $X' \to \P(1,1,2)$. {Denote by $\eta\colon X \dashedrightarrow X'$ the birational map obtained by composing the small map $\chi\colon X \dashedrightarrow Y$ with the divisorial contraction $\delta\colon Y \to X'$.} 
By Proposition~\ref{blanchard} we have $\eta \Autz(X) \eta^{-1} \subseteq \Autz(X')$. 
{On the other hand, the morphism $Y \to X'$ is the blow-up of a singular point (with its reduced structure) in $X'$, hence the $\Autz(X')$-action on $X'$ lifts to $Y$, and so we have $\delta \Autz(Y) \delta^{-1} =\Autz(X')$. Furthermore, $\chi \Autz(X) \chi^{-1}=\Autz(Y)$, and so it follows that $\eta \Autz(X) \eta^{-1} = \Autz(X')$.} 
This gives \ref{Fabc-item2} and \ref{Fabc-item3}.
\end{proof}

\begin{remark}
Putting together Propositions~\ref{prop:Fabc maximality} and \ref{prop:Rmn list of links}\ref{prop:Rmn list of links0t} yields the following commutative diagram of nested equivariant links, between five Mori fibrations, given below by thick arrows :
\[
\begin{tikzpicture}[scale=1.15,font=\small]
\node (Pt) at (0,0) {$\{\mathrm{pt}\}$};
\coordinate (Yy3) at (150:1.8*\Rb);
\coordinate (Yz3) at (210:1.8*\Rb);
\node (X1) at (-2,-0.5)  {$\p^1$};
\node (X2) at (2,0) {$\p(1,1,2)$};
\node (X3) at (-2,0.5) {$\p(1,1,2,3)$};
\node (Y1) at (2,-1.5) {$\F_2$};
\node (Y2) at (2,1.5) {$\W_2$};
\node (Yy3) at (-4,1) {$W'$};
\node (Yz3) at (-4,-1) {$\RR_{3,1}$};
\node (Z3) at (-6,-1.5) {$\FF_2^{1,-1}$};
\node (Z2) at (-6,1.5) {$Y_{1,-1}$};
\node (Z1) at (4,-1.5) {$\FF_2^{2,1} $};
\node (Z4) at (4,1.5) {$Y_{2,1}$};
\node (W) at (2,-2) {$W$};
\draw[->] (X1) to  (Pt);
\draw[->] (X2) to  (Pt);
\draw[thick,->] (X3) -- node [above=0cm] {\scriptsize{Mfs}} (Pt);
\draw[thick,->] (Yz3) -- node [above=0cm] {\scriptsize{Mfs}} (X1);
\draw[->] (Y1) to (X1);
\draw[->] (Y1) to  (X2);
\draw[dotted,->] (Yy3) -- node [left=-0.05cm] {\scriptsize{flip}} (Yz3);
\draw[dotted,->] (Z2) -- node [left=-0.05cm] {\scriptsize{flip}} (Z3);
\draw[->] (Yy3) -- node [above=-0.05cm] {\scriptsize{div}} (X3);
\draw[->] (Y2) -- node [above=-0.05cm] {\scriptsize{div}} (X3);
\draw[thick,->] (Y2) -- node [right=0cm] {\scriptsize{Mfs}} (X2);
\draw[->] (Z3) -- node [above=0cm] {\scriptsize{div}} (Yz3);
\draw[thick,->] (Z3) -- node [above=-0.05cm] {\scriptsize{Mfs}} (Y1);
\draw[->] (Z2) -- node [above=0cm] {\scriptsize{div}} (Y2);
\draw[->] (Z2) to (Yy3);
\draw[thick,->] (Z1) -- node [above=0cm] {\scriptsize{Mfs}} (Y1);
\draw[dotted,->] (Z4) -- node [right=-0.05cm] {\scriptsize{flip}} (Z1);
\draw[->] (Z4) -- node [above=0cm] {\scriptsize{div}} (Y2);
\draw[->] (W) -- node [below=-0.05cm] {\scriptsize{div}} (Z1);
\draw[->] (W) -- node [below=-0.05cm] {\scriptsize{div}} (Z3);
\end{tikzpicture}
\]
\end{remark}

\begin{lemma} \label{lem:links from Wb}
Let $b \geq 2$ and let $\W_b\to \p(1,1,2)$ be the Mori $\P^1$-fibration  introduced in {\upshape\S~\ref{sec:first classification}\ref{DefiWb}}. Then $\W_b$ is the union of five $\Autz(\W_b)$-orbits. 
There is one invariant section $D \simeq \P(1,1,2)$ and one invariant fibre $\ell$, which is the union of three orbits: two points $q_1$ and $q_2(=\ell \cap D)$ and their complement $\ell \setminus (\{q_1\} \sqcup \{q_2\})$.

Moreover, the only equivariant links starting from $\W_b \to \P(1,1,2)$ are the two type \I links described in Examples~$\ref{WbFF2bm1}$ and~$\ref{WbFF2b}$ and, if $b=2$, the type \III link corresponding to the divisorial contraction $\W_2 \to \P(1,1,2,3)$ described in Lemma~$\ref{LemP1123}\ref{item PP2}$.
\end{lemma}

\begin{proof}
The orbits description follows from the description of the $\Autz(\FF_a^{b,c})$-action on $\FF_a^{b,c}$ in \cite[\S3.1]{BFT} and Proposition~\ref{prop:Fabc maximality}\ref{Fabc-item3}.
Using tools from toric geometry (see \cite{Ful93} and \cite[Chapter~14]{Mat02}), we verify that 
\begin{itemize}
\item it is not possible to contract nor antiflip $\ell$;
\item blowing-up $\ell$ (torus-equivariantly, but not necessarily with its reduced structure) always gives a variety with non-terminal singularities; 
\item we can contract $D$ if and only if $b=2$ in which case we get $\P(1,1,2,3)$; and that
\item only the reduced blow-ups of $q_1$ and $q_2$ give varieties with  terminal singularities (the varieties $Y_{b,1}$ and $Y_{b-1,-1}$ of Proposition~$\ref{prop:Fabc maximality}\ref{Fabc-item3}$).
\end{itemize}
The description of the equivariant links starting from $\W_b$ follows from these observations.
\end{proof}

\begin{remark}\label{Rem:EuclideanModels}
Umemura classifies in \cite{Ume88} all \emph{smooth minimal} rational threefolds realising maximal connected algebraic subgroups of $\Bir(\P^3)$ and in class $[J9]$ {(defined in the table of \S~\ref{sec:Umemura classification})} certain \emph{Euclidean models} appear.
These are obtained from the $\P^1$-bundles $\FF_a^{b,c}$ via a sequence of blow-ups of $G$-invariant curves followed by a divisorial contraction. These models are \emph{not} Mori fibre spaces, but they can be recovered from our classification (Theorem~\ref{th:Ea}). For instance, one can recover the smooth Euclidean model $\mathcal{E}$ obtained from $\FF_3^{1,-1}$ via toric elementary birational maps (divisorial contractions and flips). In the following diagram the varieties $X_i$'s are smooth, while the $Y_i$'s are singular.
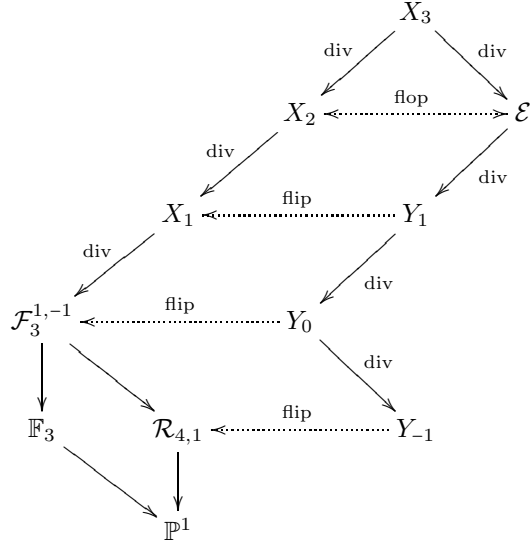
\[ \ \ \ \ \ \ \ \ \ \ \ \ \ \ \ \ \ \ \ \ \ \ \ \ \ \ \ \ \ \ \ \ \ \xymatrix{
  &  & &X_3\ar[rd]^-{\text{div}} \ar[ld]_{\text{div}} &   \\
   &&X_2 \ar[ld]_{\text{div}}  & &  \mathcal{E} \ar[ld]^-{\text{div}} \ar@{<.>}_{\flop}[ll]  \\
    &X_1\ar[ld]_-{\text{div}}  && Y_1 \ar@{.>}_{\flip}[ll]  \ar[ld]^-{\text{div}}   &  \\
        \FF_3^{1,-1} \ar[d]  \ar[rd] && Y_0 \ar@{.>}_-{\flip}[ll]  \ar[rd]^-{\text{div}} &&  \\ 
        \F_3 \ar[dr] & \RR_{4,1} \ar[d] &&  Y_{-1} \ar@{.>}_{\flip}[ll]    &   \\
         & \P^1&& &  & && & 
   }\]
\end{remark}

\subsection{Umemura's quadric fibrations over \texorpdfstring{$\p^1$}{P1}} \label{subsec Umemura quadric fib}
In this subsection we study the equivariant links between Umemura quadric fibrations. There are many equivariant links between them, we will however prove that all of them are given by Lemma~\ref{Lem:LinksQQgh}.

\begin{lemma}\label{QQgOrbits}
Assume that $\car(\k)\neq 2$. Let $g\in \k[u_0,u_1]$ be a homogeneous polynomial of degree $2n \geq 4$ with at least three roots and let $\pi \colon \QQ_g\to \p^1$ be the associated Umemura quadric fibration $($Definition~$\ref{def:QQg})$. The orbits of $\QQ_g$ for the action of $\Autz(\QQ_g)\simeq \PGL_2$ are the following:
\begin{enumerate}
\item\label{OrbitQuadricSmoothFibre}
Two orbits for each point $p\in \p^1$ that is not a root of $g$: the orbit $\Gamma_p=\pi^{-1}(p)\cap H_{x_3}\simeq \p^1$, corresponding to the diagonal of $\pi^{-1}(p)\simeq \p^1\times\p^1$, and the orbit $\pi^{-1}(p)\setminus \Gamma_p$.
\item\label{OrbitQuadricSingFibre}
Three orbits for each point $p\in \p^1$ that is a root of $g$: the fixed point $q\in\pi^{-1}(p)$, given by $x_0=x_1=x_2=0$ (the singular point of the quadric cone $\pi^{-1}(p)$), the curve $\Gamma_p=\pi^{-1}(p)\cap H_{x_3}\simeq \p^1$, and the orbit $\pi^{-1}(p)\setminus (\Gamma_p \sqcup \{q\})$.
\end{enumerate}
\end{lemma}
\begin{proof}
By Corollary~\ref{Cor:AutQQg}, the group $G=\Autz(X)$, is equal to $\PGL_2$, via the action given in Lemma~$\ref{Lem:ActiononQQg1}\ref{ActionPGL2}$. The action on $\p^1$ being trivial, every fibre $\pi^{-1}(p)$ is invariant, for each $p=[u_0:u_1]\in \p^1$. The equation of the fibre is $ x_0^2-x_1x_2=g(u_0,u_1)x_3^2$ in $\p^3$ and is thus isomorphic to $\p^1\times\p^1$ if $g(p)\not=0$ and to a quadric cone if $g(p)=0$. In both cases, the action of $\PGL_2$ preserves the conic $\Gamma_p=\pi^{-1}(p)\cap H_{x_3}$, isomorphic to $\p^1$, and acts on it via the standard action of $\PGL_2$ on $\p^1$, so $\Gamma_p$ is an orbit. As $\Gamma_p$ is a conic, it corresponds to the diagonal in $\p^1\times\p^1$ when $g(p)\not=0$, so the action is diagonal and the orbits in the fibre are $\Gamma_p$ and its complement. This achieves \ref{OrbitQuadricSmoothFibre}. If $g(p)=0$, the fibre is a quadric cone, so the vertex $q$ of the cone is fixed. It remains to observe that the action on $\pi^{-1}(p)\setminus (\Gamma_p \sqcup \{q\})$ is transitive (can be checked explicitly with the formula of Lemma~$\ref{Lem:ActiononQQg1})$. This proves \ref{OrbitQuadricSingFibre}.
\end{proof}

\begin{lemma}\label{Lem:LinksQQgh}
Assume that $\car(\k)\neq 2$. Let $g,h\in \k[u_0,u_1]$ be homogeneous polynomials of degree $2n \geq 4$ and $1$ respectively and such that $g$ is not a square. Let $p\in \p^1$ be the zero of $h$. The birational map 
\[\begin{array}{cccc}
\psi\colon & \QQ_{g}&\dasharrow & \QQ_{gh^2}\\
& [x_0:x_1:x_2:x_3;u_0:u_1]& \mapsto &  [hx_0:hx_1:hx_2:x_3;u_0:u_1]\end{array}\]
is a Sarkisov link of type \II, which decomposes as
\[\xymatrix@R=5mm@C=1cm{
    & W  \ar@{<->}[r]^{\simeq} \ar[dl]_-{\div} & W' \ar[dr]^-{\div} & \\
     \QQ_{g} \ar@{-->}[rrr]^-{\psi} \ar[d]_-{\pi}  & & & \QQ_{gh^2} \ar[d]^-{\pi'} & \\
     \p^1\ar@{<->}[rrr]^-{\mathrm{id}} & & & \p^1
  }\]
  where $W\to \QQ_g$ is the blow-up of the curve $\Gamma\subseteq \pi^{-1}(p)$ given by $x_3=0$ and $W'\to \QQ_{gh^2}$ is the blow-up of the point $q\in \pi^{-1}(p)$ given by $x_0=x_1=x_2=0$, singular point of $\QQ_{gh^2}$. 
Moreover, $\psi^{-1}\Autz(\QQ_{gh^2}) \psi \subseteq \Autz(\QQ_g)$, which is an equality if and only if either $g$ has more than two roots or $gh^2$ has two roots.
\end{lemma}
\begin{proof}
The birational map $\psi$ contracts $\pi^{-1}(p)$ onto the point $q$. Its inverse is given by $[x_0:x_1:x_2:x_3;u_0:u_1] \mapsto   [x_0:x_1:x_2:hx_3;u_0:u_1]$ and contracts $\pi'^{-1}(p)$ onto the curve $\Gamma$. One then locally checks that the map simply decomposes in the above equivariant link. By Corollary~\ref{Cor:AutQQg}, both $\Autz(\QQ_{g})$ and $\Autz(\QQ_{gh^2})$ contain the group $\PGL_2$ given in Lemma~$\ref{Lem:ActiononQQg1}\ref{ActionPGL2}$, and are equal to it if and only if $g$ and $gh^2$ have respectively more than two roots.

The explicit description of $\psi$ and of the $\PGL_2$-action given in Lemma~$\ref{Lem:ActiononQQg1}\ref{ActionPGL2}$ imply that $\psi$ is $\PGL_2$-equivariant. If $gh^2$ has more than two roots, we have $\Autz(\QQ_{gh^2})=\PGL_2$, so $\psi^{-1}\Autz(\QQ_{gh^2}) \psi^{1} \subseteq \Autz(\QQ_g)$, with an equality if and only if $g$ has more than two roots. If $gh^2$ has less than three roots, it has exactly two roots, as it is not a square, and the same holds for $g$. We may thus assume, up to a change of coordinates of $\p^1$, that $g=u_0^au_1^b$ for some odd $a,b\ge 1$, that $h=u_0$, and thus that $gh^2=u_0^{a+2}u_1^b$.  Corollary~\ref{Cor:AutQQg} implies that  both $\Autz(\QQ_{g})$ and $\Autz(\QQ_{gh^2})$ are then equal to the group $\PGL_2\times\G_m$, and so $\psi^{-1}\Autz(\QQ_{gh^2}) \psi=\Autz(\QQ_g)$; this can be checked with the explicit description provided by Example~$\ref{Example:Qgu0u1}$ or since both groups leaves invariant the centres of the blow-ups $W\to \QQ_g$ and $W'\to \QQ_{gh^2}$.
\end{proof}

We will apply to the quadric fibrations $\QQ_g$ the next lemma, which applies to \emph{compound du Val singularities of type $A_1$} (or $cA_1$); these are by definition locally analytically  isomorphic to $\{(x,y,z,t)\in \A_k^4 \mid x^2-yz-t^m=0\}$ for some $m\ge 2$. This makes sense over $\C$, but also over any algebraically closed field of characteristic zero, as the variety is defined by finitely many equations and we can then work over a subfield that embeds into $\C$. In the case of the quadric bundles $\QQ_g$, we moreover have an equation of the form $x^2-yz-t^ms(t)=0$ in \emph{Zariski} local coordinates, where $s(t)$ is a polynomial verifying $s(0)\neq 0$ (Lemma~\ref{Qgproperties}\ref{QQgterminalrat}). As explained in \cite[$\S$1.42]{K13}, a $cA_1$ singularity is terminal.

\begin{lemma} \label{Lem:Kawakita} 
Assume that $\car(\k)=0$.
Let $X$ be a $\Q$-factorial terminal variety, let $G=\Autz(X)$, and let $\eta\colon W \to X$ be a $G$-equivariant extremal divisorial contraction that contracts the exceptional divisor $E$ to a point $q$, which is either smooth or $cA_1$.
Then there exists a finite sequence of $G$-equivariant blow-ups of reduced centres $h:= h_1\circ \cdots\circ h_N \colon X_N \to X_{N-1} \to \cdots \to X_1 \to X_0=X$ and a commutative diagram
\[\xymatrix@R=4mm@C=2cm{
    X_N \ar[dr]_-{h} \ar@{-->}[r]^-g & W \ar[d]^-{\eta} \\
     & X
  }\]  
such that the strict transform $E_N:=(g^{-1})_*(E)$ is exceptional for $h_N$ and such that for any $1\le i\le N$ the centre of the blow-up $h_i$ is contained in the divisor $E_{i-1}$ contracted by $h_{i-1}$ but not contained in the strict transform of another divisor exceptional for $h_{1}\circ \cdots \circ h_{i-1}$.
\end{lemma}

\begin{proof}
The sequence of  birational morphisms $h_1,\ldots,h_N$  is described  in \cite[Construction~3.1]{kaw01} for the smooth case and in \cite[Construction~4.1]{kaw02} for the $cA_1$-singularity. We will  show that each $h_i\colon X_i  \to X_{i-1}$ is $G$-equivariant. The $(X_i,h_i)$ are defined inductively as follows: 
\begin{itemize}
\item $X_0=X$ and $Z_0=\{q\}$;
\item the morphism $h_i\colon X_i \to X_{i-1}$ is the blow up of $X_{i-1}$ along $Z_{i-1}$ followed by a $G$-equivariant resolution;
\item the centre $Z_i$ is defined as the centre of $E$ in $X_{i}$ (the image of $E$ in $X_i$ under the rational map $(h_1\circ \cdots \circ h_i)^{-1}\circ \eta$) with reduced structure, and $E_i$ is the only $h_i$-exceptional prime divisor on $X_i$ containing $Z_i$ (the unicity is explained below);
\item the process terminates when $Z_N=E_N$.
\end{itemize}
Moreover, the process always terminates as explained in \cite[Remark~3.2]{kaw01}. To show the last part of the lemma and the unicity of $E_i$ in the construction, assume that $Z_{i+1} \subseteq E_i^{(i+1)} \cap E_{i+1}$ for some $i$, where $E_i^{(i+1)}$ is the strict transform of $E_i$ in $X_{i+1}$. This means that $\O_X(-2E)=\mathfrak{m}_q$, which is impossible, as explained in \cite[Section~5.1]{kaw01} and \cite[Corollary~4.11(1)]{kaw02}.
\end{proof}

\begin{proposition}\label{Prop:QuadFibSmoothPointextremal}
Assume that $\car(\k)=0$.
Let $g\in \k[u_0,u_1]$ be a homogeneous polynomial of degree $2n \geq 4$, let $X=\QQ_g$ be the associated Umemura quadric fibration $($Definition~$\ref{def:QQg})$, and let $\eta\colon W\to X$ be an $\Autz(X)$-equivariant divisorial contraction from a terminal $\Q$-factorial variety $W$ that contracts a divisor $E\subseteq W$ onto a smooth or $cA_1$ point $q\in X$, and let us write $K_W=\eta^*(K_X)+aE$ where $a\in \Q$ is the discrepancy of $\eta$. Then, the following are equivalent:
\begin{enumerate}
\item
\label{Eapl2}
$a< 2$;
\item
\label{Ea1}
$a=1$;
\item
\label{Easing}
$q$ is a singular point of $X$ and $\eta$ is the simple blow-up of $q$ $($blow-up with the reduced structure$)$.
\end{enumerate}
\end{proposition}

\begin{proof}
Let us write $G=\Autz(X)$, which is {isomorphic} to $\PGL_2$, via the action given in Lemma~$\ref{Lem:ActiononQQg1}\ref{ActionPGL2}$ (Theorem~\ref{Thm:MainQuadric}).
We apply Lemma~\ref{Lem:Kawakita} and obtain a finite sequence of $G$-equivariant blow-ups of reduced centres $h:= h_1\circ \cdots\circ h_N \colon X_N \to X_{N-1} \to \cdots \to X_1 \to X_0=X$ and a commutative diagram
\[\xymatrix@R=4mm@C=2cm{
    X_N \ar[dr]_-{h} \ar@{-->}[r]^-g & W \ar[d]^-{{\eta}} \\
     & X
  }\]  
such that the strict transform {$\widetilde{E}:=E_N\subseteq X_N$ of $E$} is exceptional for $h_N$. Moreover, by construction (see Lemma~\ref{Lem:Kawakita}), for any $1\le i\le N$ the centre of the blow-up $h_i$ is not supported in the intersection of two divisors that are exceptional for $h_{1}\circ \cdots \circ h_{i-1}$. As the fibre $F=\pi^{-1}(p)$ is invariant by $G$, its strict transform $F_i\subseteq X_i$ is also invariant by $G$, for each $i$. Also, the exceptional divisor $E_i\subseteq X_i$ is invariant, as well as its strict transform $E_i^{(j)}\subseteq X_j$ for each $j\ge i$.

Writing $K_{X_N}=h^*(K_X)+\sum_{i=1}^N \alpha_i E_i^{(N)}$ where $\alpha_1,\ldots,\alpha_N\in \Q$ and $K_W=\eta^*(K_X)+aE$ where $a\in \Q$, we then obtain $a=\alpha_N$. This can for instance be seen by taking a resolution of $g$ and comparing the coefficients of $\widetilde{E}=E_N$ and $E$ in the ramification formula on both sides. 

{We can write $K_{X_1}=(h_1)^*(K_X)+\alpha_1 E_1$.} Moreover, as the centre of the blow-up $h_i$ is contained in the divisor $E_{i-1}$ for each $i\ge 1$, we get $\alpha_1<\alpha_2<\ldots <\alpha_N$. If $q$ is smooth, then $\alpha_1=2$, so $\alpha_N\ge 2$ in which case none of the three conditions \ref{Eapl2}-\ref{Ea1}-\ref{Easing} is satisfied. If $q$ is singular, the singularity is Zariski-locally given by $x^2-yz-t^mp(t)$ for some $m\ge 2$ and some polynomial $p(t)$ with $p(0)\neq 0$ (Lemma~\ref{Qgproperties}\ref{QQgterminalrat}). Hence, computing the simple blow-up in charts, we obtain $\alpha_1=1$. If $N=1$, then all three conditions \ref{Eapl2}-\ref{Ea1}-\ref{Easing} are satisfied. It remains to assume that $N\ge 2$ and to prove that $\alpha_2=2$, which will imply that $\alpha_3\ge 2$ and thus that none of the three conditions \ref{Eapl2}-\ref{Ea1}-\ref{Easing} is satisfied.

The exceptional divisor $E_1$ is isomorphic to the cone $\p(1,1,2)$ if $m>2$ and its singular point corresponds to a singular point of the variety $X_1$ locally analytically defined by $x_0^2-x_1x_2-t^{m-2}=0$. If $m=2$, the exceptional divisor is isomorphic to $\p^1\times\p^1$ and $X_1$ is smooth. Moreover, $E_1$ intersects the strict transform of the fibre (isomorphic to $\F_2$) in a smooth curve. The second blow-up $h_2$ being $G$-equivariant, there are two possibilities:
\begin{enumerate}[$a)$]
\item\label{singquad1} $h_2$ is the blow-up of the curve $C_{1}:=E_{1}\cap F^{(1)}$ in $X_1$; or
\item\label{singquad2} $h_2$ is the blow up if the singular point of $X_1$.
\end{enumerate}
In both cases, the discrepancy of $h_2$ is equal to $1$ so $\alpha_2=2$.
\end{proof}

\begin{proposition} \label{Prop:QuadFibLinks}
Assume that $\car(\k)=0$.
Let $g\in \k[u_0,u_1]$ be a homogeneous polynomial of degree $2n\ge 2$ that is not a square, let $X=\QQ_g$ be the associated Umemura quadric bundle $($Definition~$\ref{def:QQg})$, and let $G=\Autz(X)$.
\begin{enumerate}
\item  \label{QF-item1} If $g$ has only two roots, then $G= \PGL_2\times \G_m$ is conjugated to a strict subgroup of $\Autz(Q)$, where $Q\subseteq \p^4$ is a smooth quadric hypersurface.
\item  \label{QF-item4} If $g$ has at least three roots, then $G=\PGL_2$ and the only $G$-equivariant  links starting from $\QQ_g$ are those described in Lemma~$\ref{Lem:LinksQQgh}$ and their inverses.
\end{enumerate}
\end{proposition}

\begin{proof}
If $g$ has only two roots, we may assume that $g=u_0^au_1^b$ for some odd integers $a,b\ge 1$. By Corollary~\ref{Cor:AutQQg}, $G$  is equal to the group $\PGL_2\times\G_m$, given by Example~$\ref{Example:Qgu0u1}$. Writing $a=2r+1$ and $b=2s+1$, the birational map 
\[\begin{array}{cccc}
\psi\colon & \QQ_{g}&\dasharrow & \QQ_{u_0u_1}\\
& [x_0:x_1:x_2:x_3;u_0:u_1]& \mapsto &  [x_0:x_1:x_2:x_3u_0^{r}u_1^{s};u_0:u_1]\end{array}\]
is $\PGL_2\times\G_m$-equivariant (follows from the explicit description of the $\PGL_2\times\G_m$-action given in  Example~$\ref{Example:Qgu0u1}$). We then use the $\Autz(\QQ_{u_0u_1})$-equivariant link  $\QQ_{u_0u_1}\to Q$ of type \III given in Example~\ref{Example:Qgu0u1} to obtain \ref{QF-item1}.

We now assume that $g$ has at least three distinct roots. By Corollary~\ref{Cor:AutQQg}, $G$  is equal to to the group $\PGL_2$ given in Lemma~$\ref{Lem:ActiononQQg1}\ref{ActionPGL2}$, which preserves every fibre of $\QQ_g\to \p^1$. We now consider a $G$-equivariant link  $\chi\colon \QQ_g\dasharrow X'$. 

We first show that $\chi$ is not of type \III or \IV. Assuming the converse by contradiction, we obtain respectively that $Y'$ or $Z$ is a point {(because $Y=\p^1$ has Picard rank $1$). Hence,}  the link is given by applying an MMP on $\QQ_g$ and starts by contracting the two distinct extremal rays of $\Pic(\QQ_g)$ (see Remark~\ref{remark:UnicitySarkisov}). These two extremal rays are the fibre $f$ of $\QQ_g\to \p^1$  and the curve $h$ given by $x_0=x_1=x_3=0$ (Lemma~\ref{Qgproperties}\ref{PicQQg}). The curve $h$ satisfies $K_{\QQ_g}\cdot h=n-2\ge 0$ (Lemma~\ref{Qgproperties}\ref{canQQg}). We then cannot do a link of type \III or \IV with a pseudo-isomorphism being an isomorphism. The divisor $H\subseteq\QQ_g$, isomorphic to $\p^1\times\p^1$ and given by $x_3=0$, is covered by curves equivalent to $h$, so it is also not possible for the link to start with a non-trivial pseudo-isomorphism. 

We  then only have to consider the case where the $G$-equivariant link $\chi$ is of type \I or \II. The divisorial contraction $\eta\colon W\to X$ is then $G$-equivariant. The centre of $\eta$ is contained in a fibre $F$ of $\QQ_g\to \p^1$. {We recall that the fibres of $\QQ_g\to \p^1$ are quadrics in $\p^3$ (Definition~\ref{def:QQg}). We now distinguish two cases, depending on whether $F$ is smooth or not.}

If the fibre $F$ is smooth, then $F\simeq \p^1\times\p^1$, with $G$ acting with two orbits $($the diagonal and its complement, see Lemma~\ref{QQgOrbits}\ref{OrbitQuadricSmoothFibre}$)$. Hence, $\eta$ is the blow-up of the only invariant curve (with reduced structure, see Lemma~\ref{lem:ReducedblowUpcurve}). We can then contract the strict transform of a fibre and obtain a $G$-equivariant link $\QQ_g\dasharrow \QQ_{gh^2}$ as in Lemma~\ref{Lem:LinksQQgh}. There is no other possibilities of equivariant link starting from $\eta$ (Remark~\ref{remark:UnicitySarkisov}).

If the fibre $F$ is singular, then $F$ is isomorphic to a quadric cone, with an action whose orbits are the singular point $q\in F$ of $F$, the curve $\Gamma=F\cap H_{x_3}\simeq \p^1$ and $F\setminus (\{q\}\cup \Gamma)$ $($Lemma~\ref{QQgOrbits}\ref{OrbitQuadricSingFibre}$)$. Hence, the centre of $\eta$ is either $\Gamma$ or $q$. In the case where the centre is $\Gamma$, we similarly obtain that $\chi$ is a $G$-equivariant link $\QQ_g\dasharrow \QQ_{gh^2}$ as in Lemma~\ref{Lem:LinksQQgh}. If the centre is $q$, we write  the ramification formula
\[K_W = \eta^*K_X + a E,\]
where $E$ is the exceptional divisor of $\eta$. We then denote by $F^{(W)}$ and $\ell_{F^{(W)}}$ the strict transforms on $W$ of the fibre $F$ and of the ruling $\ell_F$ of $F$. As $E$ intersects $\ell_{F^{(W)}}$, we have $E\cdot \ell_{F^{(W)}}>0$. We then prove that $E\cdot \ell_{F^{(W)}}\ge 1$. This is because the union of curves equivalent to $\ell_{F^{(W)}}$ covers the divisor $F^{(W)}$, and as $F^{(W)}$ and $E$ are $\Q$-Cartier divisors, they intersect into a curve (by the Krull's Hauptidealsatz). Hence $E$ and $\ell_{F^{(W)}}$ intersect into a smooth point of $Y$, which gives $E\cdot \ell_{F^{(W)}}\ge 1$. Moreover, the class of $\ell_{F^{(W)}}$ in $\NE(W/\P^1)$ is extremal, since all curves contracted by $\pi$ are numerically equivalent to a positive combination of $\ell_{F^{(W)}}$ and $\ell_E$, where $\ell_E$ is a curve on $E$. Since $\ell_{F^{(W)}}$ spans a divisor, it can be contracted if and only if $K_W\cdot \ell_{F^{(W)}}<0$. So:
\[0>K_W\cdot \ell_{F^{(W)}} = (f^*K_X + a E) \cdot \ell_{F^{(W)}} = K_X \cdot \ell_F + a \ge a-2,\]
which is negative if and only if $a\le 1$. Proposition~\ref{Prop:QuadFibSmoothPointextremal} implies that $\eta$ is the simple blow up of a $cA_1$ point. Thus $\chi$ is a $G$-equivariant link $\QQ_{gh^2}\dasharrow \QQ_{g}$ as in Lemma~\ref{Lem:LinksQQgh}.
\end{proof}

\subsection{Proofs of Theorem~\ref{th:Ea}, Theorem~\ref{th:Eb}, and Corollary~\ref{corF}} \label{subsec:proof of th D}
In this subsection we combine the results obtained in \upshape\S~\S~\ref{subsec:first refinement}-\ref{sec:maximality} to prove Theorem~\ref{th:Ea}, Theorem~\ref{th:Eb}, and Corollary~\ref{corF}.\\

\noindent \emph{Proof of Theorem~\ref{th:Ea}:} We first take a rational projective threefold $\hat X$ and prove the existence of an $\Autz(\hat X)$-equivariant birational map $ \hat X \dashrightarrow X$, where $X$ is one of the Mori fibre spaces listed in Theorem~\ref{th:Ea}. Applying Theorem~\ref{th:first classification maximality}, we may assume that $\hat{X}$ belongs to one of the three cases \ref{surface}-\ref{line}-\ref{point} of Theorem~\ref{th:first classification maximality}.

Case~\ref{point} of Theorem~\ref{th:first classification maximality}: it corresponds to {Cases \hyperlink{th:D_bP3}{$(i)$}-\hyperlink{th:D_bQ3}{$(j)$}-\hyperlink{th:D_bP1112}{$(k)$}-\hyperlink{th:D_bP1123}{$(l)$}-\hyperlink{th:D_Fano}{$(m)$}} of Theorem~\ref{th:Ea}.

Case~\ref{line} of Theorem~\ref{th:first classification maximality}: {we prove that we can reduce to Cases \hyperlink{th:D_RR}{$(g)$}-\hyperlink{th:D_QQg}{$(h)$} of Theorem~\ref{th:Ea}. Indeed,} every $\p^2$-bundle over $\p^1$ is isomorphic to $X=\RR_{m,n}$ with $m \geq n \geq 0$. We can then exclude the cases $(m,n)=(1,0)$ and $2n \geq m>n \geq 1$ by applying Lemma~\ref{lem:Rmn cases non max}. In the case of a smooth quadric fibration $\QQ_g\to \p^1$, the polynomial $g$ is a square-free homogeneous polynomial of degree $2n \geq 2$. We apply Proposition~\ref{Prop:QuadFibLinks}  to remove the case where $g$ has exactly two roots (i.e.~when $n=1$), and may thus assume that $n\ge 2$, so  $g$ has at least four roots, all with multiplicity $1$. In particular, $\QQ_g$ belongs to the Family \hyperlink{th:D_QQg}{$(h)$} of Theorem~\ref{th:Ea}, as it has at least four roots of odd multiplicity.

Case~\ref{surface} of Theorem~\ref{th:first classification maximality}: {we prove that we can reduce} to Cases~\hyperlink{th:D_a}{$(a)$}--\hyperlink{th:D_e}{$(e)$} of Theorem~\ref{th:Ea}. In each case listed in Theorem~$\ref{thBFT}$, we can exclude certain numerical values using the series of lemmas proven in \upshape\S~\ref{sec:non-maximality}.

 In Case~\hyperlink{th:D_a}{$(a)$}, Theorem~$\ref{thBFT}$  gives the $\P^1$-bundle $\FF_a^{b,c} \to \F_a$ with  $a,b\ge 0$, $a\not=1$, $c\in \Z$, $c\le 0$ if $b=0$, and either $a=0$ or $b=c=0$, or $-a<c<ab$. We then apply Lemma~\ref{lem:Fabc non maximality} to remove $\FF_a^{b,c}$ in the following cases: 
\[\ref{Fabc-itemb0}\ b=0 \text{ and } -a<c<0;\quad  \ref{Fabc-itemb1}\ b=0 \text{ and } c\ge 0; \quad \text{\ and \ } \ref{Fabc-itemb2}\ b\ge 2 \text{ and } ab-a \leq c<ab.\]    

In Case~\hyperlink{th:D_b}{$(b)$}, we apply Lemma~\ref{lem:non-max case b=1}  to remove $\PP_1$ and notice that $\PP_0=\P^2 \times \P^1=\RR_{0,0}$ is contained in Case~\hyperlink{th:D_RR}{$(g)$}. 

In Case~\hyperlink{th:D_c}{$(c)$}, Theorem~$\ref{thBFT}$  gives $\U_a^{b,c}$ with $a,b\ge 1, c\ge 2$ and where $c-ab<2$ if $a \geq 2$ and $c-ab<1$ if $a=1$. We apply Lemma~\ref{lem:UmeNotMax}  to remove $\U_{a}^{b,c}$ in the following cases:
\[\ref{UmeMax2}\ a=1 \text{ and } b=c\ge 2\quad  \text{and \ } \ref{UmeMax3}\ a\ge 2, b \ge 1 \text{ and } c=2+a(b-1).\]
 
In Case~\hyperlink{th:D_d}{$(d)$},  Theorem~$\ref{thBFT}$  gives $\SS_b$ with $b\ge 1$, and $\SS_2$ is removed by Lem\-ma~\ref{lem:S2 not max}.

In Case~\hyperlink{th:D_e}{$(e)$},   Theorem~$\ref{thBFT}$  gives $\V_b$ with $b\ge 2$, and $\V_3$ is removed by Corollary~\ref{cor:V2 not max}. \qedsymbol \\

\medskip

\noindent \emph{Proof of Theorem~\ref{th:Eb}:} 
Let {$X_1$ be  a Mori fibre space belonging to one} of the families \hyperlink{th:D_a}{$(a)$}--\hyperlink{th:D_bP1123}{$(l)$} of Theorem~\ref{th:Ea}. If $\varphi \colon X_1 \dashrightarrow X_2$ is an $\Autz(X_1)$-equivariant birational map, where $X_2$ is some other Mori fibre space, then according to Theorem~\ref{th:Sarkisov decompo} the map $\varphi$ factorises into a product of equivariant Sarkisov links. It remains then to prove that each possible equivariant Sarkisov link $\chi \colon X_1 \dashrightarrow X_1'$, where $X_1$ belongs to one of the Families \hyperlink{th:D_a}{$(a)$}--\hyperlink{th:D_bP1123}{$(l)$} is, up to isomorphisms of Mori fibrations, one of the sixteen Cases~\ref{SA1}--\ref{SA16} of Theorem~\ref{th:Eb} (or its inverse) and that $\chi \Autz(X_1)\chi^{-1}=\Autz(X_1')$, which corresponds to say that $\chi^{-1}$ is also equivariant. This has been done in \upshape\S~\ref{sec:maximality} and implies in particular that $X_1'$ belongs also to the list. Note that the fact that $\chi^{-1}$ is equivariant follows from Proposition~\ref{blanchard} if $\chi^{-1}$ is a morphism (and is even trivial if $\chi$ is an isomorphism), so we only prove it when $\chi^{-1}$ is not a morphism. Let us now go into details for the ten cases \hyperlink{th:D_a}{$(a)$}--\hyperlink{th:D_bP1123}{$(l)$}: 

\hyperlink{th:D_a}{$(a)$}: here $X_1=\FF_a^{b,c}$ with $a,b\ge 0$, $a\not=1$, $c\in \Z$, and \[(a,b,c)=(0,1,-1)\text{ or }a=0, c \neq 1, b \geq 2, b\ge \lvert c\rvert \text{ or }-a<c<a(b-1)\text{ or }b=c=0.\]

If $(a,b,c)=(0,0,0)$ then $X_1= \p^1\times \p^1\times \p^1$ and by Proposition~\ref{Prop:HomSpaces} the only equivariant links are the exchanges of factors given by \ref{SA1}.

If $(a,b,c)=(0,1,-1)$, Lemma~\ref{Lemma:F0bcListLinks} implies that there are exactly two equivariant links starting from $X_1=\FF_0^{1,-1}$, namely the equivariant links $\varphi,\varphi'\colon X_1\to \RR_{1,1}$ of Proposition~$\ref{prop:Rmn list of links}\ref{prop:Rmn list of links2}$. Both are given by \ref{SA7} (up to isomorphisms of Mori fibrations), and their inverses are equivariant  by Proposition~\ref{prop:Rmn list of links}.

If $a=0$ and  $c \neq 1, b \geq 2, b\ge \lvert c\rvert $, the equivariant links starting from $X_1$ are also given by Lemma~\ref{Lemma:F0bcListLinks}. If $c=0$ and $b\ge 2$, there is a unique equivariant link starting from $X_1$ which is $\FF_0^{b,0}\iso \F_b\times\p^1  \iso \FF_b^{0,0}$, given by  \ref{SA4} (Lemma~\ref{Lemma:F0bcListLinks}\ref{LinkFaP1}). If $c=-1$ and $b\ge 2$, the unique equivariant link starting from $X_1$ is $\FF_0^{b,-1}\simeq \FF_0^{1,-b}\to \RR_{b,b}$, given by \ref{SA7}  (Lemma~\ref{Lemma:F0bcListLinks}\ref{LinkFbcm1}). Its inverse is equivariant  by Proposition~\ref{prop:Rmn list of links}.

If $-a<c<a(b-1)$, the equivariant links starting from $X_1$ are given by Proposition~\ref{prop:Fabc maximality}. If $b=1$ and $-a<c<0$, then by Proposition~\ref{prop:Fabc maximality}\ref{Fabc-item2} the equivariant links starting from $X_1$ are the link  $\FF_a^{1,c} \dashrightarrow \FF_{a}^{2,c+a}$ given by \ref{SA11}, whose inverse is equivariant by Lemma~\ref{LinksIIbetweenFU}\ref{LinksIIbetweenFabc}, the link $\FF_{a}^{1,c} \to \RR_{a-c,-c}$ given in \ref{SA7} (we have here $n=-c\ge 1$ and $m=a-c>2n\ge 2$),   whose inverse is equivariant by Proposition~\ref{prop:Rmn list of links}, and the link $\FF_2^{1,-1}\dashrightarrow \W_2$ which is the inverse of the link given by \ref{SA14}, and whose inverse is equivariant by Proposition~\ref{prop:Fabc maximality}\ref{Fabc-item2}.
If $b\ge 2$ and $-a<c<a(b-1)$, then by Proposition~\ref{prop:Fabc maximality}\ref{Fabc-item3} the equivariant links starting from $X_1$ are the link given by \ref{SA11} and its inverse,which are both equivariant by Lemma~\ref{LinksIIbetweenFU}\ref{LinksIIbetweenFabc}, and the links $\FF_2^{b,1}\dasharrow   \W_b$ and $\FF_2^{b,-1}\dasharrow \W_{b+1}$, whose inverses are given by \ref{SA15} and \ref{SA14} respectively. In this last case, the fact that the links and their inverses are equivariant follows from Proposition~\ref{prop:Fabc maximality}\ref{Fabc-item3}.

If $(b,c)=(0,0)$ and $a\ge 2$, then Proposition~\ref{prop:Fabc maximality}\ref{Fabc-itema1} implies that the only equivariant link  starting from $X_1$ is $X_1=\FF_a^{0,0}\iso \F_a\times\p^1\iso \FF_{0}^{a,0}$, which is the inverse of the link given by \ref{SA4}.

\hyperlink{th:D_b}{(b)}: here $X_1=\PP_b$, with $b \geq 2$, is a decomposable $\P^1$-bundle over $\P^2$. Proposition~\ref{LinkFromPPb} there is an equivariant link starting from $X_1$ if and only if $b=2$ in which case it is the link $\varphi\colon \PP_2 \to \P(1,1,1,2)$ given by \ref{SA6}. Moreover, Proposition~\ref{LinkFromPPb} also gives $\varphi \Autz(\PP_2) \varphi^{-1} =  \Autz(\P(1,1,1,2))$.

\hyperlink{th:D_c}{(c)}: here $X_1=\U_a^{b,c}$, for some $a,b\ge 1, c\ge 2$ with either $c<b$ if $a=1$ or $c-2<ab$ and $c-2 \neq a(b-1)$ if $a \geq 2$, is an Umemura $\P^1$-bundle over $\F_a$. The equivariant links starting from $X_1$ are given in Proposition~\ref{prop:UmeMax}; these are the links given by  \ref{SA12} and their inverses, together with the link $\U_1^{b,2}\to \V_b$ given by \ref{SA13}. Moreover, by Proposition~\ref{prop:UmeMax} the inverses of all these links are equivariant.

\hyperlink{th:D_d}{(d)}: here $X_1=\SS_b$, with $b=1$ or $b \geq 3$, is a Schwarzenberger $\P^1$-bundle over $\P^2$. If $b=1$, then $X_1$ is the projectivisation of the tangent bundle of $\P^2$ and so  by Proposition~\ref{Prop:HomSpaces}\ref{HomS5} the only equivariant link starting from $X_1$ is $\SS_1/\p^2 \iso\SS_1/\p^2 $ given by \ref{SA3} and its inverse.
If $b \geq 3$, then Proposition~\ref{prop:Schwarzenberg involution} implies that the only equivariant link starting from $X_1$ is the birational involution $\SS_b \dashrightarrow \SS_b$ given by \ref{SA5}, which conjugates $\Autz(\SS_b)$ to itself.

\hyperlink{th:D_e}{(e)}: here $\V_b$, with $b \geq 3$, is a $\P^1$-bundle over $\P^2$. By Proposition~\ref{prop:UmeMax}\ref{Vb bundles max} the only equivariant link starting from $X_1$ is the birational map $\V_b\dasharrow\U_{1}^{b,2}$ whose inverse is the blow-up morphism $\U_{1}^{b,2} \to \V_b$ given by \ref{SA13}.

\hyperlink{th:D_W}{(f)}: here $X_1=\W_b$, with $b \geq 2$, is a singular $\P^1$-fibration over $\P(1,1,2)$. By Lemma~\ref{lem:links from Wb} the only equivariant links starting from $X_1$ are the links $\W_b \dashrightarrow \FF_2^{b,1}$ given by \ref{SA15} and $\W_b \dashrightarrow \FF_2^{b-1,-1}$ given by \ref{SA14}, whose inverses are equivariant by Proposition~\ref{prop:Fabc maximality}\ref{Fabc-item3}, together with the link $\W_2 \to \P(1,1,2,3)$ given by \ref{SA10}, whose inverse is equivariant by Lemma~\ref{LemP1123}.

\hyperlink{th:D_RR}{(g)}: here $X_1=\RR_{m,n}$, where $m=n \geq 0$ or $m>2n \geq 0$ and $(m,n) \neq (1,0)$, is a decomposable $\p^2$-bundle over $\P^1$. If $m \geq 2$ and $n=0$, then by Lemma~\ref{Lem:Rm0} there are no equivariant links starting from $X_1$. If $m=n=1$, then by Proposition~\ref{prop:Rmn list of links}\ref{prop:Rmn list of links2} there are two equivariant links starting from $X_1$ which are $\RR_{1,1} \dashrightarrow \FF_0^{1-1}$ given by \ref{SA7} and $\RR_{1,1} \dashrightarrow \RR_{1,1}$ given by \ref{SA8}. The inverses are equivariant: the inverse of the first one is a morphism and the second one is an involution. If $(m,n)=(3,1)$, then by Proposition~\ref{prop:Rmn list of links}\ref{prop:Rmn list of links0t} there are two equivariant links starting from $X_1$ which are $\RR_{3,1} \dashrightarrow \FF_2^{1,-1}$ given by \ref{SA7} and $\RR_{3,1} \dashrightarrow \P(1,1,2,3)$ given by \ref{SA9}; their inverses are equivariant by Proposition~\ref{prop:Rmn list of links}. If $m=n \geq 2$ or $m>2n \geq 2$ (but $(m,n) \neq (3,1)$), then by Proposition~\ref{prop:Rmn list of links}\ref{Only1131} the only equivariant link starting from $X_1$ is $\RR_{m,n} \dashrightarrow \FF_{m-n}^{1,-n}$ given by \ref{SA7}. 
Again, the inverses of these links are equivariant by Proposition~\ref{prop:Rmn list of links}.
Finally, if $m=n=0$, then $\RR_{m,n} \simeq \P^2 \times \P^1$ and by Proposition~\ref{Prop:HomSpaces}\ref{HomS4} the only equivariant link starting from $X_1$ is $\p^1\times\p^2/\p^1\iso \p^1\times\p^2 /\p^2$ given by \ref{SA2}.

\hyperlink{th:D_QQg}{(h)}: here $X_1$ is an Umemura quadric fibration $\QQ_g$, where $g \in \k[u_0,u_1]$ is homogeneous of even degree with at least four roots of odd multiplicity.
By Proposition~\ref{Prop:QuadFibLinks}\ref{QF-item4}  the only equivariant links starting from $\QQ_g$ are the links $\QQ_g\dasharrow \QQ_{gh^2}$ and $\QQ_g\dasharrow \QQ_{g/h^2}$, which are either given  by \ref{SA16} or whose inverse is given by \ref{SA16}. As $g$ is homogeneous of even degree with at least four roots of odd multiplicity, the same holds for $gh^2$ or $g/h^2$. The equivariance of these links is given by Lemma~\ref{Lem:LinksQQgh}.

\hyperlink{th:D_bP3}{(i)}-\hyperlink{th:D_bQ3}{(j)}: {These cases is studied in Proposition~\ref{Prop:HomSpaces}.}

\hyperlink{th:D_bP1112}{(k)}: here $X_1=\P(1,1,1,2)$ and by Proposition~\ref{LinkFromPPb} the only equivariant link starting from $X_1$ is the blow-up $\PP_2 \to \P(1,1,1,2)$ given by \ref{SA6}.

\hyperlink{th:D_bP1123}{(l)}: here $X_1=\P(1,1,2,3)$ and by Lemma~\ref{LemP1123} there are only two equivariant links starting from $X_1$, namely $\P(1,1,2,3) \dashrightarrow  \RR_{3,1}$ given by \ref{SA9} and $\P(1,1,2,3) \dashrightarrow  \W_2$ given by \ref{SA10}.  Moreover, the inverse of these two links are equivariant by Lemma~\ref{LemP1123}. \qedsymbol \\

\noindent \emph{Proof of Corollary~\ref{corF}:} Theorem~\ref{th alg subg of Cr3 are aut of Mori fib} and Theorem~\ref{th:Ea} imply the following assertion:
\begin{equation}
		\label{corFstar}
		\tag{$\ast$}
\text{\it \begin{tabular}{c}
For each connected algebraic subgroup $G$ of $\Bir(\P^3)$, there exists\\
a birational map $\varphi\colon X \dashedrightarrow \P^3$ such that $\varphi^{-1} G\varphi \subseteq  \Autz(X)$ and such \\
that $X$ belongs  to one of the Families  \hyperlink{th:D_a}{$(a)$}-\hyperlink{th:D_Fano}{$(m)$} of Theorem~\ref{th:Ea}.\end{tabular}}\end{equation}

It remains to prove the following two assertions:
\begin{enumerate}[$(i)$]
\item\label{CorFa}
In \eqref{corFstar}, we can always assume $\varphi \Autz(X) \varphi^{-1}$ to be a maximal connected algebraic subgroup of $\Bir(\P^3)$; and
\item\label{CorFb}
$\psi \Autz(Y) \psi^{-1}$ is a maximal connected algebraic subgroup of $\Bir(\P^3)$ for each $Y$  that belongs to one of the Families  \hyperlink{th:D_a}{$(a)$}-\hyperlink{th:D_bP1123}{$(l)$} and each birational map $\psi\colon Y\dasharrow \p^3$. 
\end{enumerate}
We first prove \ref{CorFb} and then prove \ref{CorFa}. To simplify the notation, we only say ``is maximal'' for ``is a maximal connected algebraic subgroup of $\Bir(\P^3)$''.

\ref{CorFb}: If the connected algebraic subgroup $\psi \Autz(Y) \psi^{-1}\subseteq \Bir(\p^3)$ were not maximal, it would be strictly contained in a connected algebraic subgroup $G\subseteq \Bir(\p^3)$. Applying $\eqref{corFstar}$, we would obtain a birational map $\kappa\colon Y\dasharrow Z$, where $Z$ belongs to one of the Families  \hyperlink{th:D_a}{$(a)$}-\hyperlink{th:D_Fano}{$(m)$}, and such that $\kappa \Autz(Y)\kappa^{-1}\subsetneq \Autz(Z)$, contradicting Theorem~\ref{th:Eb}.

\ref{CorFa}: If $X$ belongs to one of the Families  \hyperlink{th:D_a}{$(a)$}-\hyperlink{th:D_bP1123}{$(l)$}, then $\varphi \Autz(X) \varphi^{-1}$ is maximal by \ref{CorFb}. We may thus assume that $X$ belongs to family \hyperlink{th:D_Fano}{$(m)$} and is thus a  rational $\Q$-factorial Fano threefold of Picard rank $1$ with terminal singularities. If $\varphi \Autz(X) \varphi^{-1}$ is not maximal, it is strictly contained in a connected algebraic subgroup $G_1\subset \Bir(\p^3)$. Applying \eqref{corFstar}, we find a birational map $\varphi_1\colon X_1 \dashedrightarrow \P^3$ such that $\varphi^{-1} G_1\varphi_1 \subseteq  \Autz(X_1)$ and such
that $X_1$ belongs  to one of the Families  \hyperlink{th:D_a}{$(a)$}-\hyperlink{th:D_Fano}{$(m)$} of Theorem~\ref{th:Ea}. Again, if $X_1$ belongs to the one of the Families  \hyperlink{th:D_a}{$(a)$}-\hyperlink{th:D_bP1123}{$(l)$}, then $\varphi_1 \Autz(X_1) \varphi_1^{-1}$ is maximal, so we may assume that $X_1$ again belongs to family \hyperlink{th:D_Fano}{$(m)$}, and continue this process. We obtain either the result or an infinite sequence $G\subsetneq G_1\subsetneq G_2 \subsetneq G_3 \subsetneq\cdots$ where each $G_i$ is equal to $\varphi_i \Autz(X_i) \varphi_i^{-1}$, and $X_i$ a  rational $\Q$-factorial Fano threefold of Picard rank $1$ with  terminal singularities and $\varphi_i\colon X_i\dasharrow \p^3$ is birational. We now explain why this is impossible, by using that the set $\mathcal{S}$ of $\Q$-factorial Fano threefolds with  terminal singularities is parametrised by a bounded family. {This is a simple case of the BAB conjecture, proven already by \cite[Theorem 1.2]{KMMT-2000} (see also \cite{BirkarS} for the more general case). As a consequence, there is an integer $r\ge 1$ such that $-rK_X$ is Cartier for  each $X\in \mathcal{S}$ (\cite[Theorem 1.2(1)]{KMMT-2000} gives the explicit bound $r=24!$). Then \cite[Theorem 1.1]{Kol93} gives the existence of an integer $m$ such that $\lvert -mK_X\rvert $ is base-point free for each $X\in \mathcal{S}$. Applying \cite[Lemma~1.2]{Kol93}, we can choose $m$ bigger if needed and assume that  $-mK_X$ is very ample for each $X\in \mathcal{S}$. The elements of $\mathcal{S}$ being isomorphic to the fibres of a morphism, there exists $d\ge 1$ such that for each $X\in \mathcal{S}$, the linear system $\lvert -mK_Z\rvert$ provides a closed embedding $\varphi_Z\colon Z\hookrightarrow \p^{h^0(-mK_Z)-1}$ with $h^0(-mK_Z)\le d$. This gives an embedding of $\Autz(X)$ into $\PGL(h^0(-mK_Z))$ and implies that there is an integer $D\ge 1$, not depending on $Z$, such that $\dim(\Autz(X))\le D$. Therefore we cannot have the above infinite sequence.}
\qedsymbol

\bigskip

\noindent \textbf{Funding.}
The first-named author acknowledges support by the Swiss National Science Foundation Grant ``Birational transformations of threefolds'' [200020\_178807]. The third-named author is supported by the Project FIBALGA [contract ANR-18-CE40-0003-01].
This work received partial support from the French "Investissements d\textquoteright Avenir" program and from project ISITE-BFC [contract ANR-lS-IDEX-OOOB]. The IMB receives support from  the EIPHI Graduate School [contract ANR-17-EURE-0002].

\bigskip

\noindent \textbf{Acknowledgments.}
We would like to thank Michel Brion, Paolo Cascini, Alessio Corti, Enrica Floris, Anne-Sophie Kaloghiros, Lucy Moser-Jauslin, and Boris Pas\-quier for interesting discussions related to this work. 
We are also very grateful to the anonymous referees for their valuable comments and suggestions which greatly helped to improve this article.

\end{document}